\documentclass[11pt]{article}

\usepackage[utf8x]{inputenc}
\usepackage[english]{babel}
\usepackage[a4paper,top=2.5cm,left=2.5cm,right=2.5cm,bottom=2.5cm]{geometry}
\usepackage{charter}
\usepackage{microtype}
\usepackage{tikz,pgfplots}
\usepackage{url}
\usepackage[plainpages=false,pdfpagelabels,colorlinks=true,citecolor=blue,hypertexnames=false]{hyperref}
\usepackage{amsmath,amssymb,amsthm}
\usepackage{stmaryrd,dsfont}
\usepackage{enumerate,array,multirow}
\usepackage{listings}
\usepackage{comment}
\usepackage{xspace}

\numberwithin{equation}{section}
\numberwithin{figure}{section}

\definecolor{darkgreen}{rgb}{0,0.5,0}
\definecolor{digit}{rgb}{0.3,0,0.5}
\definecolor{beige}{rgb}{0.9,0.9,0.7}
\definecolor{comment}{rgb}{0.8,0.3,0.3}

\makeatletter
\let\commentfullflexible\lst@column@fullflexible
\makeatother
\makeatletter
\renewcommand\paragraph{\@startsection{paragraph}{4}{\z@}%
                                    {2ex \@plus1ex \@minus.2ex}%
                                    {1ex}%
                                    {\normalfont\normalsize\bfseries}}
\renewcommand\subparagraph{\@startsection{subparagraph}{5}{\z@}%
                                       {1.5ex \@plus1ex \@minus .2ex}%
                                       {-0.5em}%
                                      {\normalfont\normalsize\it}}
\makeatother

\renewcommand{\leq}{\leqslant}
\renewcommand{\geq}{\geqslant}

\newtheorem{theo}{Theorem}[section]
\newtheorem{prop}[theo]{Proposition}
\newtheorem{lem}[theo]{Lemma}
\newtheorem{cor}[theo]{Corollary}

\theoremstyle{definition}
\newtheorem{deftn}[theo]{Definition}
\newtheorem{rem}[theo]{Remark}

\newtheorem*{ex*}{Example}

\newcommand{\N}{\mathbb N}
\newcommand{\Z}{\mathbb Z}
\newcommand{\Zp}{\Z_p}
\newcommand{\Q}{\mathbb Q}
\newcommand{\Qp}{\Q_p}
\newcommand{\F}{\mathbb F}
\newcommand{\Fp}{\F_p}
\newcommand{\R}{\mathbb R}

\newcommand{\E}{\mathbb E}

\newcommand{\calB}{\mathcal B}
\newcommand{\calF}{\mathcal F}
\newcommand{\calH}{\mathcal H}
\newcommand{\calL}{\mathcal L}
\newcommand{\calN}{\mathcal N}
\newcommand{\calR}{\mathcal R}
\newcommand{\calT}{\mathcal T}
\newcommand{\calU}{\mathcal U}
\newcommand{\calZ}{\mathcal Z}

\newcommand{\GL}{\text{\rm GL}}

\newcommand{\spec}{\text{\rm Spec}\:}

\newcommand{\Span}{\text{\rm Span}}
\renewcommand{\min}{\text{\rm min}}
\renewcommand{\max}{\text{\rm max}}
\newcommand{\pr}{\text{\rm pr}}
\newcommand{\symm}{\text{\rm s}}

\newcommand{\id}{\text{\rm id}}
\newcommand{\ev}{\text{\rm ev}}
\newcommand{\rec}{\text{\rm rec}}
\newcommand{\val}{\text{\rm val}}

\newcommand{\abs}{\text{\rm abs}}
\newcommand{\rel}{\text{\rm rel}}
\newcommand{\Res}{\text{\rm Res}}
\newcommand{\Com}{\text{\rm Comatrix}}

\newcommand{\zealous}{\text{\rm Zeal}}

\newcommand{\Value}{\text{\rm value}}
\newcommand{\ttx}{\texttt{\rm x}\xspace}
\newcommand{\tty}{\texttt{\rm y}\xspace}
\newcommand{\ttF}{\texttt{\rm F}\xspace}
\newcommand{\ttN}{\texttt{\rm N}\xspace}

\newcommand{\nan}{\text{\rm NaN}}
\newcommand{\FP}{\text{\rm FP}}
\newcommand{\QpFP}{\Qp^{\FP}}
\newcommand{\QpNan}{\Qp'}

\newcommand{\cond}{\text{\rm cond}}

\newcommand{\softO}{\tilde O}

\newcommand{\sage}{\textsc{SageMath}~\cite{sage}\xspace}
\newcommand{\sagenoref}{\textsc{SageMath}\xspace}
\newcommand{\magma}{\textsc{Magma}~\cite{magma}\xspace}
\newcommand{\pari}{\textsc{Pari}~\cite{pari}\xspace}

\newcommand{\mathemagix}{\textsc{Mathemagix}~\cite{mathemagix}\xspace}

\title{Computations with $p$-adic numbers}
\author{Xavier Caruso}
\date\today

\begin{document}

\maketitle

\begin{abstract}
This document contains the notes of a lecture I gave at the ``Journées 
Nationales du Calcul Formel\footnote{(French) National Computer Algebra 
Days}'' (JNCF) on January 2017. The aim of the lecture was to discuss 
low-level algorithmics for $p$-adic numbers. It is divided into two main 
parts: first, we present various implementations of $p$-adic numbers and 
compare them and second, we introduce a general framework for studying 
precision issues and apply it in several concrete situations.
\end{abstract}

\setcounter{tocdepth}{2}
\tableofcontents

\section*{Introduction}

The field of $p$-adic numbers, $\Qp$, was first introduced by 
Kurt Hensel at the end of the 19th century in a short paper written in 
German~\cite{He97}. From that time, the popularity of $p$-adic numbers 
has grown without interruption throughout the 20th century. Their first 
success was materialized by the famous Hasse--Minkowski's Theorem 
\cite{Se70} that 
states that a Diophantine equation of the form $P(x_1, \ldots, x_n) = 0$ 
where $P$ is a polynomial of total degree at most $2$ has a solution over $\Q$ 
if and only if it has a solution over $\R$ and a solution over $\Qp$ for 
all prime numbers $p$. This characterization is quite interesting because 
testing whether a polynomial equation has a $p$-adic solution can be 
carried out in a very efficient way using \emph{analytic} methods just 
like over the reals.
This kind of strategy is nowadays ubiquitous in many areas of Number 
Theory and Arithmetic Geometry. After Diophantine equations, other 
typical examples come from the study of number fields: we hope deriving 
interesting information about a number field $K$ by studying carefully
all its $p$-adic incarnations $K \otimes_\Q \Qp$. The ramification of
$K$, its Galois properties, \emph{etc.} can be --- and are very often 
--- studied in this manner~\cite{Sa72,Ne99}.
The class field theory, which provides a precise description of all
Abelian extensions\footnote{An abelian extension is a Galois extension
whose Galois group is abelian.} of a given number field, is also 
formulated in this language~\cite{Ne12}.
The importance of $p$-adic numbers is so prominent today that there is 
still nowadays very active research on theories which are dedicated to 
purely $p$-adic objects: one can mention for instance the
study of $p$-adic geometry and $p$-adic cohomologies~\cite{BO78,LS07},
the theory of $p$-adic differential equations~\cite{Ke10}, Coleman's theory 
of $p$-adic integration~\cite{Co98}, the $p$-adic Hodge theory~\cite{BrCo09}, 
the $p$-adic Langlands correspondence~\cite{Be11}, the study of $p$-adic 
modular forms~\cite{Go88}, $p$-adic $\zeta$-functions~\cite{Ko84} and 
$L$-functions~\cite{Co89}, \emph{etc.} The proof of Fermat's last 
Theorem by Wiles and Taylor~\cite{Wi95,TaWi95} is stamped with many of 
these ideas and developments.

Over the last decades, $p$-adic methods have taken some importance in 
Symbolic Computation as well. For a long time, $p$-adic methods have 
been used for factoring polynomials over $\Q$~\cite{LeLeLo82}. More recently, 
there has been a wide diversification of the use of $p$-adic numbers for 
effective computations: Bostan et al.~\cite{BoGoPeSc05} used Newton sums 
for polynomials over $\Zp$ to compute composed products for polynomials 
over $\Fp$; Gaudry et al.~\cite{GaHoWeRiKo06} used $p$-adic lifting 
methods to generate genus 2 CM hyperelliptic curves; Kedlaya 
\cite{Ke01}, Lauder~\cite{La04} and many followers used $p$-adic 
cohomology to count points on hyperelliptic curves over finite fields; 
Lercier and Sirvent~\cite{LeSi08} computed isogenies between elliptic 
curves over finite fields using $p$-adic differential equations. 

The need to build solid foundations to the algorithmics of $p$-adic 
numbers has then emerged. This is however not straightforward because a 
single $p$-adic number encompasses an infinite amount of information 
(the infinite sequence of its digits) and then necessarily needs to be 
truncated in order to fit in the memory of a computer. From this point 
of view, $p$-adic numbers behave very similarly to real numbers and the 
questions that emerge when we are trying to implement $p$-adic numbers 
are often the same as the questions arising when dealing with rounding 
errors in the real setting~\cite{Mu89,DaMu97,Mu09}. The algorithmic study 
of $p$-adic numbers is then located at the frontier between Symbolic 
Computation and Numerical Analysis and imports ideas and results coming 
from both of these domains.

\subparagraph{Content and organization of this course.}

This course focuses on the low-level implementation of $p$-adic numbers 
(and then voluntarily omits high-level algorithms making use of $p$-adic 
numbers) and pursues two main objectives.
The first one is to introduce and discuss the most standard strategies 
for implementing $p$-adic numbers on computers. We shall detail three of 
them, each of them having its own spirit: (1)~the \emph{zealous arithmetic} 
which is inspired by interval arithmetic in the real setting, (2)~the 
\emph{lazy arithmetic} with its \emph{relaxed} improvement and (3)~the 
\emph{$p$-adic floating-point arithmetic}, the last two being inspired
by the eponym approaches in the real setting.

The second aim of this course is to develop a general theory giving 
quite powerful tools to study the propagation of accuracy in the 
$p$-adic world. The basic underlying idea is to linearize the situation 
(and then model the propagation of accuracy using differentials); it is 
once again inspired from classical methods in the real case. 
However, it turns out that the non-archimedean nature of $\Qp$ 
(\emph{i.e.} the fact that $\Z$ is bounded in $\Qp$) is the source of 
many simplifications which will allow us to state much more accurate 
results and to go much further in the $p$-adic setting. As an example,
we shall see that the theory of $p$-adic precision yields a general
strategy for increasing the numerical stability of any given algorithm
(assuming that the problem it solves is well-conditioned).

\medskip

This course is organized as follows. 
\S \ref{sec:padicnumbers} is devoted to the introduction of $p$-adic 
numbers: we define them, prove their main properties and discuss in
more details their place in Number Theory, Arithmetic Geometry and 
Symbolic Computation.
The presentation of the standard implementations of $p$-adic numbers 
mentioned above is achieved in \S \ref{sec:implementations}. 
A careful comparison between them is moreover proposed and supported
by many examples coming from linear algebra and commutative algebra.
Finally, in \S \ref{sec:precision}, we expose the aforementioned theory 
of $p$-adic precision. We then detail its applications: we will notably
examine many very concrete situations and, for each of them, we will
explain how the theory of $p$-adic precision helps us either in quantifying 
the qualities of a given algorithm regarding to numerical stability or,
even better, in improving them.

\subparagraph{Acknowledgments.}

This document contains the (augmented) notes of a lecture I gave at the 
``Journées Nationales du Calcul Formel\footnote{(French) 
National Computer Algebra Days}'' (JNCF) on January 2017. I heartily thank 
the organizers and the scientific committee of the JNCF for giving me 
the opportunity to give these lectures and for encouraging me to write 
down these notes. I am very grateful to Delphine Boucher, 
Nicolas Brisebarre, Claude-Pierre 
Jeannerod, Marc Mezzarobba and Tristan Vaccon for their careful reading
and their helpful comments on an earlier version of these notes.

\subparagraph{Notation.}

We use standard notation for the set of numbers: $\N$ is the set of 
natural integers (\emph{including $0$}), $\Z$ is the set of relative 
integers, $\Q$ is the set of rational numbers and $\R$ is the set of 
real numbers.
We will sometimes use the soft-$O$ notation $\softO(-)$ for writing 
complexities; we recall that, given a sequence of positive real numbers 
$(u_n)$, $\softO(u_n)$ is defined as the union of the sets $O(u_n \log^k 
u_n)$ for $k$ varying in $\N$.

Throughout this course, the letter $p$ always refers to a fixed prime 
number.

\section{Introduction to $p$-adic numbers}
\label{sec:padicnumbers}

In this first section, we define $p$-adic numbers, discuss their basic 
properties and try to explain, by selecting a few relevant examples, 
their place in Number Theory, Algebraic Geometry and Symbolic 
Computation. The presentation below is voluntarily very summarized; we 
refer the interested reader to~\cite{Am75,Go97} for a more complete 
exposition of the theory of $p$-adic numbers.

\subsection{Definition and first properties}

$p$-adic numbers are very ambivalent objects which can be thought of 
under many different angles: computational, algebraic, analytic. It 
turns out that each point of view leads to its own definition of 
$p$-adic numbers: computer scientists often prefer viewing a $p$-adic 
number as a sequence of digits while algebraists prefer speaking of 
projective limits and analysts are more comfortable with Banach spaces 
and completions. Of course all these approaches have their own interest 
and understanding the intersections between them is often the key behind 
the most important advances.

In this subsection, we briefly survey all the standard definitions of 
$p$-adic numbers and provide several mental representations in order to 
try as much as possible to help the reader to develop a good $p$-adic 
intuition.

\subsubsection{Down-to-earth definition}
\label{sssec:downtoearth}

Recall that each positive integer $n$ can be written in \emph{base $p$},
that is as a finite sum:
$$n = a_0 + a_1 p + a_2 p^2 + \cdots + a_\ell p^\ell$$
where the $a_i$'s are integers between $0$ and $p{-}1$, the so-called
\emph{digits}. This writing is moreover unique assuming that the most
significant digit $a_\ell$ does not vanish.
A possible strategy to compute the expansion in base $p$ goes as 
follows. We first compute $a_0$ by noting that it is necessarily the 
remainder in the Euclidean division of $n$ by $p$: indeed it is 
congruent to $n$ modulo $p$ and lies in the range $[0,p{-}1]$ by 
definition. Once $a_0$ is known, we compute $n_1 = \frac{n-a_0}p$, which 
is also the quotient in the Euclidean division of $n$ by $p$. Clearly
$n_1 = a_1 + a_2 p + \cdots + a_\ell p^{\ell-1}$
and we can now compute $a_1$ repeating the same strategy. 
Figure~\ref{fig:decbasis} shows a simple execution of this algorithm.

\begin{figure}
\hfill$\begin{array}{rr@{\hspace{1ex}}c@{\hspace{1ex}}r@{\hspace{1ex}}c@{\hspace{1ex}}c@{\hspace{1ex}}c@{\hspace{1ex}}c}
a_0: & 1742 & = & 248 & \times & 7 & + & \mathbf{6} \\
a_1: & 248 & = & 35 & \times & 7 & + & \mathbf{3} \\
a_2: & 35 & = & 5 & \times & 7 & + & \mathbf{0} \\
a_3: & 5 & = & 0 & \times & 7 & + & \mathbf{5} \\
\end{array}$
\hfill\null

\medskip

\hfill $\Longrightarrow$ \quad $1742 = \overline{5036}^7$\hfill\null

\caption{Expansion of $1742$ in base $7$}
\label{fig:decbasis}
\end{figure}

By definition, a \emph{$p$-adic integer} is an infinite formal sum of 
the shape:
$$x = a_0 + a_1 p + a_2 p^2 + \cdots + a_i p^i + \cdots$$
where the $a_i$'s are integers between $0$ and $p{-}1$. In other words,
a $p$-adic integer 
is an integer written in base $p$ with an infinite number of 
digits. We will sometimes alternatively write $x$ as follows:
$$x = \overline{\ldots a_i \ldots a_3 a_2 a_1 a_0}^p$$
or simply
$$x = \ldots a_i \ldots a_3 a_2 a_1 a_0$$
when no confusion can arise.
The set of $p$-adic integers is denoted by $\Z_p$. It is endowed with a 
natural structure of commutative ring. Indeed, we can add, subtract and 
multiply $p$-adic integers using the schoolbook method; note that 
handling carries is possible since they propagate on the left.
\begin{figure}
\hfill
\raisebox{-\height}{
$\begin{array}{cc@{\hspace{1ex}}c@{\hspace{1ex}}c@{\hspace{1ex}}c@{\hspace{1ex}}c@{\hspace{1ex}}c@{\hspace{1ex}}c@{\hspace{1ex}}c}
  & \ldots & 2 & 3 & 0 & 6 & 2 & 4 & 4 \\{}
+ & \ldots & 1 & 6 & 5 & 2 & 3 & 3 & 2 \\
\hline
  & \mathbf \ldots & \mathbf 4 & \mathbf 2 & \mathbf 6 & \mathbf 1 & \mathbf 6 & \mathbf 0 & \mathbf 6 
\end{array}$}
\hspace{2cm}
\raisebox{-\height}{
$\begin{array}{cc@{\hspace{1ex}}c@{\hspace{1ex}}c@{\hspace{1ex}}c@{\hspace{1ex}}c@{\hspace{1ex}}c@{\hspace{1ex}}c@{\hspace{1ex}}c}
  & \ldots & 2 & 3 & 0 & 6 & 2 & 4 & 4 \\{}
\times & \ldots & 1 & 6 & 5 & 2 & 3 & 3 & 2 \\
\hline
  & \ldots & 4 & 6 & 1 & 5 & 5 & 2 & 1 \\
  & \ldots & 2 & 2 & 5 & 0 & 6 & 5 & \\
  & \ldots & 2 & 5 & 0 & 6 & 5 & & \\
  & \ldots & 5 & 5 & 2 & 1 & & & \\
  & \ldots & 6 & 1 & 6 & & & & \\
  & \ldots & 6 & 3 & & & & & \\
  & \ldots & 4 & & & & & & \\
\cline{2-9}
  & \mathbf \ldots & \mathbf 4 & \mathbf 3 & \mathbf 2 & \mathbf 0 & \mathbf 3 & \mathbf 0 & \mathbf 1 
\end{array}$}
\hfill\null

\caption{Addition and multiplication in $\Z_7$}
\end{figure}
The ring of natural integers $\N$ appears naturally as a subring of 
$\Z_p$: it consists of $p$-adic integers $\ldots a_i \ldots a_3 a_2 a_1 
a_0$ for which $a_i = 0$ when $i$ is large enough. Note in particular 
that the integer $p$ writes $\ldots 0010$ in $\Z_p$ and more generally 
$p^n$ writes $\ldots 0010\ldots 0$ with $n$ ending zeros.
As a consequence, a $p$-adic integer is a multiple of $p^n$ if and 
only if it ends with (at least) $n$ zeros.
Remark that negative integers are $p$-adic integers as well: the 
opposite of $n$ is, by definition, the result of the subtraction 
$0{-}n$.

\bigskip

Similarly, we define a \emph{$p$-adic number} as a formal infinite sum of 
the shape:
$$x = a_{-n} p^{-n} + a_{-n+1} p^{-n+1} + \cdots + a_i p^i + \cdots$$
where $n$ is an integer which may depend on $x$. 
Alternatively, we will write:
$$x = \overline{\ldots a_i \ldots a_2 a_1 a_0\,.\, 
a_{-1} a_{-2} \ldots a_{-n}}^p$$
and, when no confusion may arise, we will freely remove the bar and the 
trailing $p$.
A $p$-adic number is then nothing but a ``decimal'' number written in 
base $p$ with an infinite number of digits before the decimal mark 
and a
finite amount of digits after the decimal mark. Addition and multiplication
extend to $p$-adic numbers as well.

The set of $p$-adic numbers is denoted by $\Q_p$. Clearly $\Q_p = 
\Z_p[\frac 1 p]$. We shall see later (\emph{cf} Proposition 
\ref{prop:invZp}, page \pageref{prop:invZp}) that $\Q_p$ is actually the 
fraction field of $\Z_p$; in particular it is a field and $\Q$, which is 
the fraction field of $\Z$, naturally embeds into $\Q_p$.

\subsubsection{Second definition: projective limits}

From the point of view of addition and multiplication, the last 
digit of a $p$-adic integer behaves like an integer modulo $p$, that is 
an element of the finite field $\Fp = \Z/p\Z$.
In other words, the application $\pi_1 : \Z_p \to \Z/p\Z$ taking a 
$p$-adic integer $x = a_0 + a_1 p + a_2 p^2 + \cdots$ to the class of 
$a_0$ modulo $p$ is a ring homomorphism.
More generally, given a positive integer $n$, the map:
$$\begin{array}{rcl}
\pi_n :\quad \Z_p & \to & \Z/p^n\Z \\
a_0 + a_1 p + a_2 p^2 + \cdots & \mapsto &
(a_0 + a_1 p + \cdots + a_{n-1} p^{n-1}) \text{ mod } p^n
\end{array}$$
is a ring homomorphism. These morphisms are compatible in the following 
sense: for all $x \in \Z_p$, we have $\pi_{n+1}(x) \equiv \pi_n(x) 
\pmod{p^n}$ (and more generally $\pi_m(x) \equiv \pi_n(x) \pmod{p^n}$ 
provided that $m \geq n$). Putting the $\pi_n$'s all together, we end
up with a ring homomorphism:
$$\begin{array}{rcl}
\pi: \quad \Z_p & \to & \varprojlim_n \Z/p^n\Z \smallskip \\
x & \mapsto & (\pi_1(x), \pi_2(x), \ldots)
\end{array}$$
where $\varprojlim_n \Z/p^n\Z$ is by definition the subring of 
$\prod_{n=1}^\infty \Z/p^n\Z$ consisting of sequences $(x_1, x_2,
\ldots)$ for which $x_{n+1} \equiv x_n \pmod{p^n}$ for all $n$: it
is called the \emph{projective limit} of the $\Z/p^n\Z$'s.

Conversely, consider a sequence $(x_1, x_2, \ldots) \in \varprojlim_n 
\Z/p^n\Z$. In a slight abuse of notation, continue to write $x_n$ for 
the unique integer of the range $\llbracket 0, p^n{-}1 \rrbracket$ 
which is congruent to 
$x_n$ modulo $p^n$ and write it in base $p$:
$$x_n = a_{n,0} + a_{n,1} p + \cdots + a_{n,n-1} p^{n-1}$$
(the expansion stops at $(n{-}1)$ since $x_n < p^n$ by construction).
The condition $x_{n+1} \equiv x_n \pmod{p^n}$ implies that 
$a_{n+1,i} = a_{n,i}$ for all $i \in \llbracket 0, n{-}1\rrbracket$.
In other words, when $i$ remains fixed, the sequence $(a_{n,i})_{n > i}$ 
is constant and thus converges to some $a_i$. Set:
$$\psi(x_1, x_2, \ldots) = \ldots a_i \ldots a_2 a_1 a_0 \in \Z_p.$$
We define this way an application $\psi : \varprojlim_n \Z/p^n\Z \to 
\Z_p$ which is by construction a left and a right inverse of $\pi$. In
other words, $\pi$ and $\psi$ are isomorphisms which are inverses of
each other.

The above discussion allows us to give an alternative definition of 
$\Z_p$, which is:
$$\Z_p = \varprojlim_n \, \Z/p^n\Z.$$
The map $\pi_n$ then corresponds to the projection onto the $n$-th factor.
This definition is more abstract and it seems more difficult to 
handle as well. However it has the enormous advantage of making the ring 
structure appear clearly and, for this reason, it is often much more
useful and powerful than the down-to-earth definition of \S 
\ref{sssec:downtoearth}.
As a typical example, let us prove the following proposition.

\begin{prop}
\label{prop:invZp}
\begin{enumerate}[(a)]
\renewcommand{\itemsep}{0pt}
\item An element $x \in \Z_p$ is invertible in $\Z_p$ if and only
if $\pi_1(x)$ does not vanish.
\item The ring $\Q_p$ is the fraction field of $\Z_p$; in particular,
it is a field.
\end{enumerate}
\end{prop}

\begin{proof}
\emph{(a)}~Let $x \in \Z_p$.
Viewing $\Z_p$ as $\varprojlim_n \, \Z/p^n\Z$, we find that $x$ 
is invertible in $\Z_p$ if and only if $\pi_n(x)$ is invertible in
$\Z/p^n\Z$ for all $n$. The latest condition is equivalent to
requiring that $\pi_n(x)$ and $p^n$ are coprime for all $n$. 
Noting that $p$ is prime, this is further equivalent to the fact that 
$\pi_n(x) \text{ mod } p = \pi_1(x)$ does not vanish in $\Z/p\Z$.

\smallskip

\noindent
\emph{(b)}~By definition $\Q_p = \Z_p[\frac 1 p]$. It is then enough
to prove that any nonzero $p$-adic integer $x$ can be written as a
product $x = p^n u$ where $n$ is a nonnegative integer and $u$ is a
unit in $\Z_p$. Let $n$ be the number of zeros at the end of the 
$p$-adic expansion of $x$ (or, equivalently, the largest integer $n$
such that $\pi_n(x) = 0$). Then $x$ can be written $p^n u$ where $u$
is a $p$-adic integer whose last digit does not vanish. By the first
part of the proposition, $u$ is then invertible in $\Z_p$ and we are
done.
\end{proof}

We note that the first statement of Proposition~\ref{prop:invZp}
shows that the subset of non-invertible elements of $\Zp$ is exactly
the kernel of $\pi_1$. We deduce from this that $\Zp$ is a local ring
with maximal ideal $\ker \pi_1$.

\subsubsection{Valuation and norm}
\label{sssec:padicnorm}

We define the \emph{$p$-adic valuation} of the nonzero $p$-adic 
number 
$$x = \ldots a_i \ldots a_2 a_1 a_0\,.\, a_{-1} a_{-2} \ldots a_{-n}$$
as the smallest (possibly negative) integer $v$ for which $a_v$ 
does not vanish. We denote it $\val_p(x)$ or simply $\val(x)$ if no
confusion may arise. Alternatively $\val(x)$ 
can be defined as the largest integer $v$ such that $x \in p^v \Zp$. 
When $x = 0$, we put $\val(0) = +\infty$.
We define this way a function $\val : \Q_p \to \Z \cup \{+\infty\}$.
Writing down the computations (and remembering that $p$ is prime),
we immediately check the following compatibility properties for all
$x,y \in \Qp$:
\begin{enumerate}[(1)]
\renewcommand{\itemsep}{0pt}
\item $\val(x+y) \geq \min\big(\val(x), \val(y)\big)$,
\item $\val(xy) = \val(x) + \val(y)$.
\end{enumerate}
Note moreover that the equality $\val(x+y) 
= \min\big(\val(x), \val(y)\big)$ does hold as soon as $\val(x)
\neq \val(y)$. As we shall see later, this property reflects the
tree structure of $\Z_p$ (see \S \ref{sssec:tree}).

The \emph{$p$-adic norm} $|\cdot|_p$ is defined by
$|x|_p = p^{-\val(x)}$ for $x \in \Q_p$. In the sequel, when no
confusion can arise, we shall often write $|\cdot|$ instead of
$|\cdot|_p$. The properties (1) and (2) above immediately translate 
as follows:
\begin{enumerate}[(1)]
\renewcommand{\itemsep}{0pt}
\item[(1')] $|x+y| \leq \max\big(|x|, |y|\big)$ and equality holds
if $|x| \neq |y|$,
\item[(2')] $|xy| = |x| \cdot |y|$.
\end{enumerate}
Remark that (1') implies that $|\cdot|$ satisfies the triangular 
inequality, that is $|x+y| \leq |x| + |y|$ for all $x,y \in \Q_p$. 
It is however much stronger: we say that the $p$-adic norm is 
\emph{ultrametric} or \emph{non Archimedean}. We will see later that
ultrametricity has strong consequences on the topology of $\Q_p$ 
(see for example Corollary~\ref{cor:convseries} below) and strongly 
influences the calculus with $p$-adic (univariate and multivariate) 
functions as well (see \S \ref{sssec:preclemma}).
This is far from being anecdotic; on the contrary, this will be the 
starting point of the theory of $p$-adic precision we will
develop in \S \ref{sec:precision}.

The $p$-adic norm defines a natural distance $d$ on $\Q_p$ as follows: 
we agree that the distance between two $p$-adic numbers $x$ and $y$ is 
$|x-y|_p$. Again this distance is ultrametric in the sense that:
$$d(x,z) \leq \max\big( d(x,y), d(y,z)\big).$$
Moreover the equality holds as soon as $d(x,y) \neq d(y,z)$: all 
triangles in $\Qp$ are isosceles!
Observe also that $d$ takes its values in a proper subset of $\R^+$
(namely $\{0\} \cup \{p^n : n \in \Z\}$) whose unique accumulation point 
is $0$. This property has surprising consequences; for example, closed
balls of positive radius are also open balls and \emph{vice et versa}. 
In particular $\Zp$ is open (in $\Qp$) and compact according to the 
topology defined by the distance. From now on, we endow $\Qp$ with this 
topology.

Clearly, a $p$-adic number lies in $\Z_p$ if and only if its $p$-adic 
valuation is nonnegative, that is if and only if its $p$-adic norm is at 
most $1$. In other words, $\Z_p$ appears as the closed unit ball in 
$\Q_p$. Viewed this way, it is remarkable that it is stable under 
addition (compare with $\R$); it is however a direct consequence of the 
ultrametricity. Similarly, by Proposition \ref{prop:invZp}, a $p$-adic 
integer is invertible in $\Zp$ if and only if it has norm $1$, meaning 
that the group of units of $\Zp$ is then the unit sphere in $\Qp$. As 
for the maximal ideal of $\Z_p$, it consists of elements of positive 
valuation and then appears as the open unit ball in $\Q_p$ (which is
also the closed ball of radius $p^{-1}$).

\subsubsection{Completeness}

The following important proposition shows that $\Qp$ is nothing but the 
completion of $\Q$ according to the $p$-adic distance. In that sense, 
$\Qp$ arises in a very natural way... just as does~$\R$.

\begin{prop}
The space $\Q_p$ equipped with its natural distance is complete
(in the sense that every Cauchy sequence converges). Moreover $\Q$
is dense in $\Q_p$.
\end{prop}

\begin{proof}
We first prove that $\Qp$ is complete. Let $(u_n)_{n \geq 0}$ be
a $\Qp$-valued Cauchy sequence. It is then bounded and rescaling the
$u_n$'s by a uniform scalar, we may assume that $|u_n| \leq 1$ 
(\emph{i.e.} $u_n \in \Zp$) for all $n$. 
For each $n$, write :
$$u_n = \sum_{i=0}^\infty a_{n,i} p^i$$
with $a_{n,i} \in \{0, 1, \ldots, p{-}1\}$. Fix an integer $i_0$ and
set $\varepsilon = p^{-i_0}$. Since $(u_n)$ is a Cauchy sequence, there
exists a rank $N$ with the property that $|u_n - u_m| \leq \varepsilon$
for all $n,m \geq N$. Coming back to the definition of the $p$-adic
norm, we find that $u_n - u_m$ is divisible by $p^{i_0}$. Writing $u_n = u_m 
+ (u_n - u_m)$ and computing the sum, we get $a_{n,i} = a_{m,i}$ for all 
$i \leq i_0$. In particular the sequence $(a_{n,i_0})_{n \geq 0}$ is 
ultimately constant. Let $a_{i_0} \in \{0, 1, \ldots, p{-}1\}$ denote its 
limit.
Now define $\ell = \sum_{i=0}^\infty a_i p^i \in \Zp$ and
consider again $\varepsilon > 0$. Let $i_0$ be an integer such that
$p^{-i_0} \leq \varepsilon$. By construction, there exists a rank $N$
for which $a_{n,i} = a_i$ whenever $n \geq N$ and $i \leq i_0$. For
$n \geq N$, the difference $u_n - \ell$ is then divisible by $p^{i_0}$
and hence has norm at most $\varepsilon$. Hence $(u_n)$ converges to 
$\ell$.

We now prove that $\Q$ is dense in $\Qp$. Since $\Qp = \Zp[\frac 1 p]$, 
it is enough to prove that $\Z$ is dense in $\Zp$. Pick $a \in \Zp$ and
write $a = \sum_{i \geq 0} a_i p^i$. For a nonnegative integer $n$, set
$b_n = \sum_{i=0}^{n-1} a_i p^i$. Clearly $b_n$ is an integer and the
sequence $(b_n)_{n \geq 0}$ converges to $a$. The density follows.
\end{proof}

\begin{cor}
\label{cor:convseries}
Let $(u_n)_{n \geq 0}$ be a sequence of $p$-adic numbers. The series 
$\sum_{n \geq 0} u_n$ converges in $\Qp$ if and only if its general
term $u_n$ converges to $0$.
\end{cor}

\begin{proof}
Set $s_n = \sum_{i=0}^{n-1} u_i$. Clearly $u_n = s_{n+1} - s_n$ for
all $n$. If $(s_n)$ converges to a limit $s \in \Qp$, then $u_n$ converges 
to $s-s = 0$.
We now assume that $(u_n)$ goes to $0$. We claim that
$(s_n)$ is a Cauchy sequence (and therefore converges). 
Indeed, let $\varepsilon > 0$ and pick an integer $N$ for which $|u_i| \leq
\varepsilon$ for all $i \geq N$. Given two integers $m$ and $n$
with $m > n \geq N$, we have:
$$|s_m - s_n| = \left|\sum_{i=n}^{m-1} u_i \right| \leq
\max \big(|u_n|, |u_{n+1}|, \ldots, |u_{m-1}|\big)$$
thanks to ultrametricity. Therefore $|s_m - s_n| \leq \varepsilon$
and we are done.
\end{proof}

\subsubsection{Tree representation}
\label{sssec:tree}

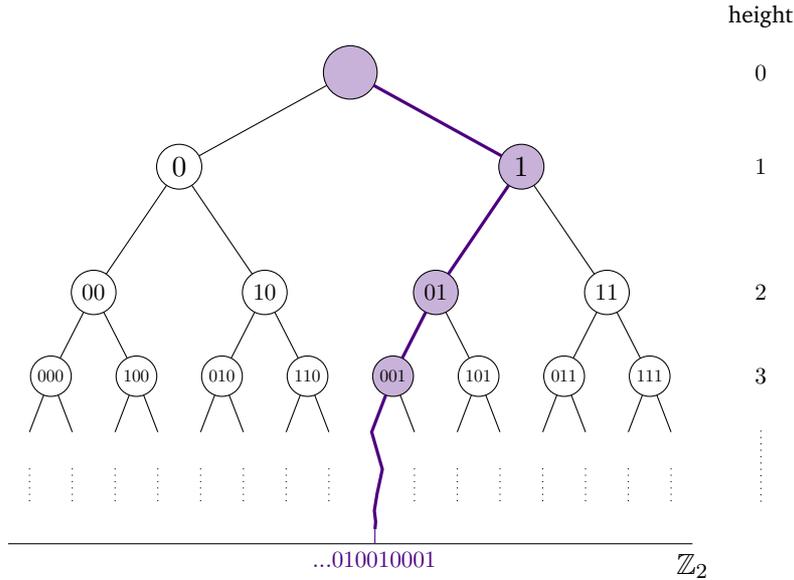
\begin{figure}
\hfill
\begin{tikzpicture}[xscale=9,yscale=-2.5]
\begin{scope}[inner sep=1mm,outer sep=0pt]
\node[scale=0.8] at (1.1, 0.2) { height };
\node[scale=0.8] at (1.1, 0.5) { $0$ };
\node[scale=0.8] at (1.1, 1) { $1$ };
\node[scale=0.8] at (1.1, 5/3) { $2$ };
\node[scale=0.8] at (1.1, 19/9) { $3$ };
\draw[thin,dotted] (1.1,2.4)--(1.1,2.8);
\node[draw,fill=digit!30,circle,inner sep=2.5mm] (void) at (1/2,0.5) { };
\node[draw,circle] (0) at (1/4,1) { $0$ };
\node[draw,fill=digit!30,circle] (1) at (3/4,1) { $1$ };
\node[scale=0.8,draw,circle] (00) at (1/8,5/3) { $00$ };
\node[scale=0.8,draw,circle] (10) at (3/8,5/3) { $10$ };
\node[scale=0.8,draw,fill=digit!30,circle] (01) at (5/8,5/3) { $01$ };
\node[scale=0.8,draw,circle] (11) at (7/8,5/3) { $11$ };
\node[scale=0.6,draw,circle] (000) at (1/16,19/9) { $000$ };
\node[scale=0.6,draw,circle] (100) at (3/16,19/9) { $100$ };
\node[scale=0.6,draw,circle] (010) at (5/16,19/9) { $010$ };
\node[scale=0.6,draw,circle] (110) at (7/16,19/9) { $110$ };
\node[scale=0.6,draw,fill=digit!30,circle] (001) at (9/16,19/9) { $001$ };
\node[scale=0.6,draw,circle] (101) at (11/16,19/9) { $101$ };
\node[scale=0.6,draw,circle] (011) at (13/16,19/9) { $011$ };
\node[scale=0.6,draw,circle] (111) at (15/16,19/9) { $111$ };
\end{scope}
\draw (void)--(0); \draw[very thick,digit] (void)--(1);
\draw (0)--(00); \draw (0)--(10);
\draw[very thick,digit] (1)--(01); \draw (1)--(11);
\draw (00)--(000); \draw (00)--(100);
\draw[very thick,digit] (01)--(001); \draw (01)--(101);
\draw (10)--(010); \draw (10)--(110);
\draw (11)--(011); \draw (11)--(111);
\draw (000)--(1/32,65/27); \draw (000)--(3/32,65/27);
\draw (100)--(5/32,65/27); \draw (100)--(7/32,65/27);
\draw (010)--(9/32,65/27); \draw (010)--(11/32,65/27);
\draw (110)--(13/32,65/27); \draw (110)--(15/32,65/27);
\draw (001)--(17/32,65/27); \draw (001)--(19/32,65/27);
\draw (101)--(21/32,65/27); \draw (101)--(23/32,65/27);
\draw (011)--(25/32,65/27); \draw (011)--(27/32,65/27);
\draw (111)--(29/32,65/27); \draw (111)--(31/32,65/27);
\draw (0,3)--(1,3);
\draw[thin,dotted] (1/32,2.6)--(1/32,2.8);
\draw[thin,dotted] (3/32,2.6)--(3/32,2.8);
\draw[thin,dotted] (5/32,2.6)--(5/32,2.8);
\draw[thin,dotted] (7/32,2.6)--(7/32,2.8);
\draw[thin,dotted] (9/32,2.6)--(9/32,2.8);
\draw[thin,dotted] (11/32,2.6)--(11/32,2.8);
\draw[thin,dotted] (13/32,2.6)--(13/32,2.8);
\draw[thin,dotted] (15/32,2.6)--(15/32,2.8);
\draw[thin,dotted] (19/32,2.6)--(19/32,2.8);
\draw[thin,dotted] (21/32,2.6)--(21/32,2.8);
\draw[thin,dotted] (23/32,2.6)--(23/32,2.8);
\draw[thin,dotted] (25/32,2.6)--(25/32,2.8);
\draw[thin,dotted] (27/32,2.6)--(27/32,2.8);
\draw[thin,dotted] (29/32,2.6)--(29/32,2.8);
\draw[thin,dotted] (31/32,2.6)--(31/32,2.8);
\draw[very thick,digit] (001)--(17/32,65/27)--(35/64,211/81)
                 --(69/128,665/243)--(137/256,2059/729)
                 --(275/512,6305/2187)--(549/1024,2.9219);
\draw[digit] (549/1024,2.9219)--(0.536,3);
\node[digit,scale=0.8,below] at (0.536,3) { $...010010001$ };
\node[below] at (1,3) { $\Z_2$ };
\end{tikzpicture}
\hfill\null

\caption{Tree representation of $\Z_2$}
\label{fig:treeZ2}
\end{figure}

Geometrically, it is often convenient and meaningful to represent $\Zp$ 
as the infinite full $p$-ary tree. In order to explain this 
representation, we need a definition.

\begin{deftn}
\label{def:intervalZp}
For $h \in \N$ and $a \in\Zp$, we set:
$$I_{h,a} = \big\{ \, x \in \Zp 
\quad \text{s.t.} \quad 
x \equiv a \pmod{p^h} \, \big\}.$$
An \emph{interval} of $\Zp$ is a subset of $\Zp$ of the form $I_{h,a}$
for some $h$ and $a$.
\end{deftn}

If $a$ decomposes in base $p$ as
$a = a_0 + a_1 p + a_2 p^2 + \cdots + a_{h-1} p^{h-1} + \cdots$,
the interval $I_{h,a}$ consists exactly of the $p$-adic integers whose
last digits are $a_{h-1} \ldots a_1 a_0$ in this order.
On the other hand, from the analytic point of view, the condition $x 
\equiv a \pmod {p^h}$ is equivalent to $|x-a| \leq p^{-h}$. Thus the
interval $I_{h,a}$ is nothing but the closed ball of centre $a$ and 
radius $p^{-h}$. Even better, the intervals of $\Zp$ are exactly the 
closed balls of $\Zp$.

Clearly $I_{h,a} = I_{h,a'}$ if and only if $a \equiv a' \pmod{p^h}$. 
In particular, given an interval $I$ of $\Zp$, there is exactly one
integer $h$ such that $I = I_{h,a}$. We will denote it by $h(I)$ and
call it the \emph{height} of $I$. We note that there exist exactly
$p^h$ intervals of $\Zp$ of height $h$ since these intervals are
indexed by the classes modulo $p^h$ (or equivalently by the sequences
of $h$ digits between $0$ and $p{-}1$).

From the topological point of view, intervals behave like $\Zp$:
they are at the same time open and compact.

\medskip

We now define the \emph{tree} of $\Zp$, denoted by $\calT(\Zp)$, as follows:
its vertices are the intervals of $\Zp$ and we put an edge $I \to J$ 
whenever $h(J) = h(I) + 1$ and $J \subset I$. 
A picture of $\calT(\Z_2)$ is represented on Figure~\ref{fig:treeZ2}.
The labels indicated on the vertices are the last $h$ digits of $a$.
Coming back to a general $p$, 
we observe that the height of an interval $I$ corresponds to the 
usual height function in the tree $\calT(\Zp)$. Moreover, given two 
intervals $I$ and $J$, the inclusion $J \subset I$ holds if and only 
if there exists a path from $I$ to $J$. 

Elements of $\Zp$ bijectively correspond to infinite paths of 
$\calT(\Zp)$ starting from the root through the following 
correspondence: an element $x \in \Zp$ is encoded by the path
$$I_{0,x} \to I_{1,x} \to I_{2,x} \to \cdots \to I_{h,x} \to \cdots.$$
Under this encoding, an infinite path of $\calT(\Zp)$ starting from the 
root 
$$I_0 \to I_1 \to I_2 \to \cdots \to I_h \to \cdots$$ 
corresponds to a uniquely determined $p$-adic integer, which is the 
unique element lying in the decreasing intersection $\bigcap_{h \in \N} 
I_h$. Concretely each new $I_h$ determines a new digit of $x$; the whole 
collection of the $I_h$'s then defines $x$ entirely.
The distance on $\Zp$ can be visualized on $\calT(\Zp)$ as well:
given $x, y \in \Zp$, we have $|x-y| = p^{-h}$ where $h$ is the height
where the paths attached to $x$ and $y$ separate.

\medskip

The above construction easily extends to $\Qp$.

\begin{deftn}
\label{def:intervalQp}
For $h \in \Z$ and $a \in\Qp$, we set:
$$I_{h,a} = \big\{ \, x \in \Qp 
\quad \text{s.t.} \quad 
|x - a| \leq p^{-h} \, \big\}.$$
A \emph{bounded interval} of $\Qp$ is a subset of $\Qp$ of the form
$I_{h,a}$ for some $h$ and $a$.
\end{deftn}

Similarly to the case of $\Zp$, a bounded interval of $\Qp$ of height 
$h$ is a subset of $\Qp$ consisting of $p$-adic numbers whose digits at 
the positions $< h$ are fixed (they have to agree with the digits
of $a$ at the same positions).

The graph $\calT(\Qp)$ is defined as follows: its vertices are the 
intervals of $\Qp$ while there is an edge $I \to J$ if $h(J) = h(I) 
+ 1$ and $J \subset I$. We draw the attention of the reader to the
fact that $\calT(\Qp)$ is a tree but it is not rooted: there does not
exist a largest bounded interval in $\Qp$.
To understand better the structure of $\calT(\Qp)$, let us define, 
for any integer $v$, the subgraph $\calT(p^{-v} \Zp)$ of $\calT(\Qp)$ 
consisting of intervals which are contained in $p^{-v}\Zp$. From the
fact that $\Qp$ is the union of all $p^{-v}\Zp$, we derive that
$\calT(\Qp) = \bigcup_{v \geq 0} \calT(p^{-v}\Zp)$. Moreover, for all
$v$, $\calT(p^{-v} \Zp)$ is a rooted tree (with root $p^{-v}\Zp$)
which is isomorphic to $\calT(\Zp)$ except that the height function is 
shifted by $-v$. The tree $\calT(p^{-v-1}\Zp)$ is thus obtained by 
juxtaposing $p$ copies of $\calT(p^{-v}\Zp)$ and linking the roots of 
them to a common parent $p^{-v}\Zp$ (which then becomes the new root).

\subsection{Newton iteration over the $p$-adic numbers}

Newton iteration is a well-known tool in Numerical Analysis for 
approximating a zero of a ``nice'' function defined on a real interval.
More precisely, given a differentiable function $f : [a,b] \to \R$, we
define a recursive sequence $(x_i)_{i \geq 0}$ by:
\begin{equation}
\label{eq:Newtoniter}
x_0 \in [a,b] \quad ; \quad 
x_{i+1} = x_i - \frac{f(x_i)}{f'(x_i)}, \, i = 0, 1, 2, \ldots
\end{equation}
Under some assumptions, one can prove that the sequence $(x_i)$ 
converges to a zero of $f$, namely $x_\infty$. Moreover the convergence is 
very rapid since, assuming that $f$ is twice differentiable, we usually 
have an inequality of the shape $|x_\infty - x_i| \leq \rho^{2^i}$ for 
some $\rho \in (0,1)$. In other words, the number of correct digits 
roughly doubles at each iteration. The Newton recurrence 
\eqref{eq:Newtoniter} has a nice geometrical interpretation as well: the 
value $x_{i+1}$ is the $x$-coordinate of the intersection point of the 
$x$-axis with the tangent to the curve $y = f(x)$ at the point $x_i$ 
(see Figure~\ref{fig:Newtoniter}).

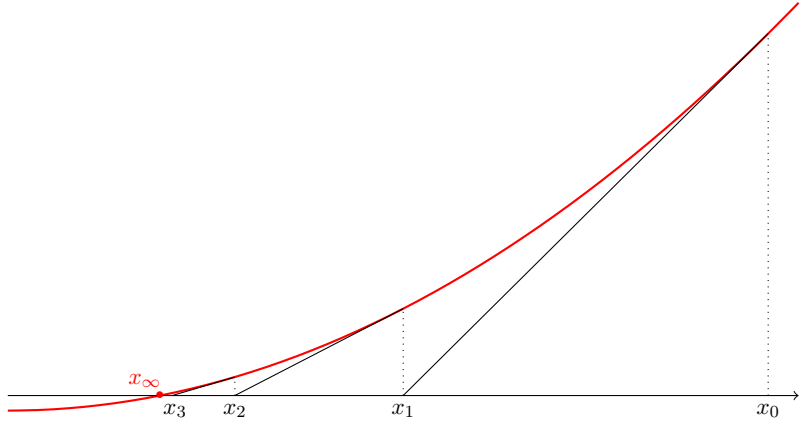
\begin{figure}
\hfill
\begin{tikzpicture}[xscale=2,yscale=0.2]
\draw[thick,red] plot[smooth] coordinates { 
  (0, -1)
  (0.20, -0.96)
  (0.40, -0.84)
  (0.60, -0.64)
  (0.80, -0.36)
  (1.0, 0)
  (1.2, 0.44)
  (1.4, 0.96)
  (1.6, 1.56)
  (1.8, 2.24)
  (2.0, 3)
  (2.2, 3.84)
  (2.4, 4.76)
  (2.6, 5.76)
  (2.8, 6.84)
  (3.0, 8.0)
  (3.2, 9.24)
  (3.4, 10.56)
  (3.6, 11.96)
  (3.8, 13.44)
  (4.0, 15)
  (4.2, 16.64)
  (4.4, 18.36)
  (4.6, 20.16)
  (4.8, 22.04)
  (5, 24)
  (5.2, 26.04)
};
\draw[->] (0,0)--(5.2,0);
\draw[dotted] (5,0)--(5,24);
\draw (5,24)--(2.6,0);
\draw[dotted] (2.6,0)--(2.6,5.76);
\draw (2.6,5.76)--(1.493,0);
\draw[dotted] (1.493,0)--(1.493,1.229);
\draw (1.493,1.229)--(1.081,0);
\node[below, scale=0.8] at (5,0) { $x_0$ };
\node[below, scale=0.8] at (2.6,0) { $x_1$ };
\node[below, scale=0.8] at (1.493,0) { $x_2$ };
\node[below, scale=0.8] at (1.1,0) { $x_3$ };
\node[red, scale=0.6] at (1,0) { $\bullet$ };
\node[red,above, scale=0.8] at (0.9,0) { $x_\infty$ };
\end{tikzpicture}
\hfill\null

\caption{Newton iteration over the reals}
\label{fig:Newtoniter}
\end{figure}

\subsubsection{Hensel's Lemma}
\label{sssec:hensel}

It is quite remarkable that the above discussion extends almost 
\emph{verbatim} when $\R$ is replaced by $\Qp$. Actually, extending the 
notion of differentiability to $p$-adic functions is quite subtle and 
probably the most difficult part. This will be achieved in \S 
\ref{sssec:classC1} (for functions of class $C^1$) and \S 
\ref{sssec:NewtonC2} (for functions of class $C^2$). For now, we prefer 
avoiding these technicalities and restricting ourselves to the simpler 
(but still interesting) case of polynomials. For this particular case, 
the Newton iteration is known as \emph{Hensel's Lemma} and already 
appears in Hensel's seminal paper~\cite{He97} in which $p$-adic 
numbers are introduced.

Let $f(X) = a_0 + a_1 X + \cdots + a_n X^n$ be a polynomial in the 
variable $X$ with coefficients in $\Qp$. Recall that the derivative 
of $f$ can be defined in a purely algebraic way as $f'(X) = a_1 + 
2 a_2 X + \cdots + n a_n X^{n-1}$.

\begin{theo}[Hensel's Lemma]
\label{th:Hensel}
Let $f \in \Zp[X]$ be a polynomial with coefficients in $\Zp$.
We suppose that we are given some $a \in \Zp$ with the property that
$|f(a)| < |f'(a)|^2$. 
Then the sequence $(x_i)_{i \geq 0}$
defined by the recurrence:
$$x_0 = a \quad ; \quad x_{i+1} = x_i - \frac{f(x_i)}{f'(x_i)}$$
is well defined and converges to $x_\infty \in \Zp$ with $f(x_\infty) = 
0$.
The rate of convergence is given by:
$$|x_\infty - x_i| \leq |f'(a)| \cdot 
\left(\frac{|f(a)|}{|f'(a)|^2}\right)^{2^i}.$$
Moreover $x_\infty$ is the unique root of $f$ in the open ball of
centre $a$ and radius $|f'(a)|$.
\end{theo}

The proof of the above theorem is based on the next lemma:

\begin{lem}
\label{lem:polyC2}
Given $f \in \Zp[X]$ and $x, h \in \Zp$, we have:
\begin{enumerate}[(i)]
\item $|f(x+h) - f(x)| \leq |h|$.
\item $|f(x+h) - f(x) - h f'(x)| \leq |h|^2$.
\end{enumerate}
\end{lem}

\begin{proof}
For any nonnegative integer $i$, define $f^{[i]} = \frac 1{i!} \: 
f^{(i)}$ where $f^{(i)}$ stands for the $i$-th derivative of $f$.
Taylor's formula then reads:
\begin{equation}
\label{eq:taylor2}
f(x+h) - f(x) = h f'(x) + h^2 f^{[2]}(x) + \cdots + h^n f^{[n]}(x).
\end{equation}
Moreover, a direct computation shows that the coefficients of $f^{[i]}$ 
are obtained from that of $f$ by multiplying by binomial coefficients. 
Therefore $f^{[i]}$ has coefficients in $\Zp$. Hence $f^{[i]}(x) \in 
\Zp$, \emph{i.e.} $|f^{[i]}(x)| \leq 1$, for all $i$.
We deduce that each summand of the right hand
side of \eqref{eq:taylor2} has norm at most $|h|$. The first assertion
follows while the second is proved similarly.
\end{proof}

\begin{proof}[Proof of Theorem~\ref{th:Hensel}]
Define $\rho = \frac{|f(a)|}{|f'(a)|^2}$. 
We first prove by induction on $i$ the following conjunction:
$$(H_i) : \quad |f'(x_i)| = |f'(a)|
\quad \text{and} \quad |f(x_i)| \leq |f'(a)|^2 \cdot \rho^{2^i}.$$
Clearly $(H_0)$ holds. We assume now that $(H_i)$ holds for some $i 
\geq 0$. We put $h
= -\frac{f(x_i)}{f'(x_i)}$ so that $x_{i+1} = x_i + h$. We write
$f'(x_{i+1}) = \big(f'(x_{i+1}) - f'(x_i)\big) + f'(x_i)$.
Observe that the first summand $f'(x_{i+1}) - f'(x_i)$ has norm at 
most $|h|$ by the first assertion of Lemma~\ref{lem:polyC2}, while 
$|f'(x_i)| \geq |h| \cdot \rho^{-2^i} > |h|$ by the induction 
hypothesis.
The norm of $f'(x_{i+1})$ is then the maximum of the norms of the two
summands, which is $|f'(x_i)| = |f'(a)|$.
Now, applying again Lemma~\ref{lem:polyC2}, we get $|f(x_{i+1})| \leq 
|h|^2 \leq |f'(a)|^2 \cdot \rho^{2^i}$ and the induction goes.

Coming back to the recurrence defining the $x_i$'s, we get:
\begin{equation}
\label{eq:rateNewton}
|x_{i+1} - x_i| = \frac{|f(x_i)|}{|f'(x_i)|} \leq |f'(a)| \cdot 
\rho^{2^i}.
\end{equation}
By Corollary~\ref{cor:convseries}, this implies the convergence
of the sequence $(x_i)_{i \geq 0}$. Its limit $x_\infty$ is a 
solution to the equation $x_\infty = x_\infty + 
\frac{f(x_\infty)}{f'(x_\infty)}$. Thus $f(x_\infty)$ has to vanish. 
The announced rate of convergence follows from Eq.~\eqref{eq:rateNewton} 
thanks to ultrametricity.

It remains to prove uniqueness. For this, consider $y \in \Zp$
with $f(y) = 0$ and $|y - x_0| < |f'(a)|$. Since $|x_\infty - x_0|
\leq |f'(a)| \cdot \rho < |f'(a)|$, we deduce $|x_\infty - y| <
|f'(a)|$ as well. Applying Lemma~\ref{lem:polyC2} with $x = x_\infty$
and $h = y - x_\infty$, we find $|h f'(x_\infty)| \leq |h|^2$. Since 
$|h| < |f'(a)| = |f'(x_\infty)|$, this implies $|h| = 0$, 
\emph{i.e.} $x = x_\infty$. Uniqueness is proved.
\end{proof}

\begin{rem}
\label{rem:Hensel}
All conclusions of Theorem~\ref{th:Hensel} are still valid for any
sequence $(x_i)$ satisfying the weaker assumption:
$$\left|x_{i+1} - x_i + \frac{f(x_i)}{f'(x_i)}\right| 
\leq |f'(a)| \cdot \rho^{2^{i+1}}$$
(the proof is entirely similar). Roughly speaking, this stronger version 
allows us to work with approximations at \emph{each} iteration. It will 
play a quite important role for algorithmic purpose (notably in \S 
\ref{sssec:intervalNewton}) since computers cannot handle exact $p$-adic 
numbers but always need to work with truncations.
\end{rem}

\subsubsection{Computation of the inverse}
\label{sssec:inverse}

A classical application of Newton iteration is the computation of the 
inverse: for computing the inverse of a real number $c$, we introduce 
the function $x \mapsto \frac 1 x - c$ and the associated Newton scheme. 
This leads to the recurrence $x_{i+1} = 2 x_i - c x_i^2$ with initial 
value $x_0 = a$ where $a$ is sufficiently close to $\frac 1 c$.

In the $p$-adic setting, the same strategy applies (although it does not 
appear as a direct consequence of Hensel's Lemma since the mapping $x 
\mapsto \frac 1 x - c$ is \emph{not} polynomial). Anyway, let us pick 
an invertible element $c \in \Zp$, \emph{i.e.} $|c| = 1$, and define the
sequence $(x_i)_{i \geq 0}$ by:
$$x_0 = a \quad ; \quad x_{i+1} = 2 x_i - c x_i^2, \, i = 0, 1, 2, \ldots$$
where $a$ is any $p$-adic number whose last digit is the inverse modulo
$p$ of the last digit of $c$. Computing such an element $a$ reduces to
computing a modular inverse and thus is efficiently feasible.

\begin{prop}
\label{prop:Newtoninverse}
The sequence $(x_i)_{i \geq 0}$ defined above converges to $\frac 1 c 
\in \Zp$. More precisely, we have $|c x_i - 1| \leq p^{-2^i}$ for all 
$i$.
\end{prop}

\begin{proof}
We prove the last statement of the proposition by induction on $i$. By
construction of $a$, it holds for $i = 0$. Now observe that
$c x_{i+1} - 1 = 2 c x_i - c^2 x_i^2 - 1 = - (c x_i - 1)^2$.
Taking norms on both sides, we get $|c x_{i+1} - 1| = |c x_i - 1|^2$
and the induction goes this way.
\end{proof}

Quite interestingly, this method extends readily to 
the non-commutative case. Let us illustrate this by showing how it can be 
used to compute inverses of matrices with coefficients in $\Zp$. Assume
then that we are starting with a matrix $C \in M_n(\Zp)$ whose reduction
modulo $p$ is invertible, \emph{i.e.} there exists a matrix $A$ such
that $AC \equiv I_n \pmod p$ where $I_n$ is the identity matrix of 
size $n$. Note that the computation of $A$ reduces to the inversion 
of an $n \times n$ matrix over the finite field $\Fp$ and so can be 
done efficiently.
We now define the sequence $(X_i)_{i \geq 0}$ by the recurrence:
$$X_0 = A \quad ; \quad X_{i+1} = 2 X_i - X_i C X_i$$
(be careful with the order in the last term). Mimicking the proof of
Proposition~\ref{prop:Newtoninverse}, we write:
$$C X_{i+1} - I_n = 2 C X_i - C X_i C X_i - I_n = - (C X_i - I_n)^2$$
and obtain this way that each entry of $C X_i - 1$ has norm at most
$p^{2^i}$. Therefore $X_i$ converges to a matrix $X_\infty$ satisfying
$C X_\infty = I_n$, \emph{i.e.} $X_\infty = C^{-1}$ (which in particular
implies that $C$ is invertible). A similar argument works for $p$-adic 
skew polynomials and $p$-adic differential operators as well.

\subsubsection{Square roots in $\Qp$}
\label{sssec:sqrt}

Another important application of Newton iteration is the computation
of square roots. Again, the classical scheme over $\R$ applies without
(substantial) modification over $\Qp$.

Let $c$ be a $p$-adic number. If the $p$-adic valuation of $c$ is odd, 
then $c$ is clearly not a square in $\Qp$ and it does not make sense to 
compute its square root. On the contrary, if the $p$-adic valuation of 
$c$ is even, then we can write $c = p^{2v} c'$ where $v$ is an integer and 
$c'$ is a unit in $\Zp$. Computing a square root then reduces to 
computing a square root of $c'$; in other words, we may assume that $c$ 
is invertible in $\Zp$, \emph{i.e.} $|c| = 1$.

We now introduce the polynomial function $f(x) = x^2 - c$. In order to 
apply Hensel's Lemma (Theorem~\ref{th:Hensel}), we need a first rough
approximation $a$ of $\sqrt c$. Precisely, we need to find some $a \in 
\Zp$ with the property that $|f(a)| < |f'(a)|^2 = |2a|^2$. From $|c| = 
1$, the above inequality implies $|a| = 1$ as well and therefore can be
rewritten as 
$a^2 \equiv c \pmod q$ where $q = 8$ if $p = 2$ and $q = p$ otherwise.
We then first need to compute a square root of $c$ modulo $q$. If $p=2$,
this can be achieved simply by looking at the table of squares modulo $8$:
$$\begin{array}{|c||c|c|c|c|c|c|c|c|}
\hline
a & 0 & 1 & 2 & 3 & 4 & 5 & 6 & 7 \\
\hline
a^2 & 0 & 1 & 4 & 1 & 0 & 1 & 4 & 1 \\
\hline
\end{array}$$

Observe moreover that $c$ is necessarily odd (since it is assumed to be
invertible in $\Z_2$). If it is congruent to $1$
modulo $8$, $a = 1$; otherwise, there is no solution
and $c$ has no square root in $\Q_2$.
When $p > 2$, we have to compute a square root in the finite field $\Fp$
for which efficient algorithms are known~\cite[\S 1.5]{Co93}. If 
$c \text{ mod }p$ is not a square in $\Fp$, then $c$ does not admit a square root 
in $\Qp$ either.

Once we have computed the initial value $a$, we consider the recursive
sequence $(x_i)_{i \geq 0}$ defined by $x_0 = a$ and
$$x_{i+1} = x_i - \frac{f(x_i)}{f'(x_i)} = \frac 1 2  \left( x_i
+ \frac c {x_i} \right), \, i = 0, 1, 2, \ldots$$
By Hensel's Lemma, it converges to a limit $x_\infty$ which is a square
root of $c$. Moreover the rate of convergence is given by:
$$\begin{array}{r@{\hspace{0.5ex}}ll}
|x_\infty - x_i| & \leq p^{-2^i} & \text{if } p > 2 \smallskip \\
|x_\infty - x_i| & \leq 2^{-(2^i+1)} & \text{if } p = 2
\end{array}$$
meaning that the number of correct digits of $x_i$ is at least $2^i$ 
(resp. $2^i + 1$) when $p > 2$ (resp. $p = 2$).

\subsection{Similarities with formal and Laurent series}

According to the very first definition of $\Zp$ we have given (see \S 
\ref{sssec:downtoearth}), $p$-adic integers have formal 
similitudes with formal series over $\Fp$: they are both described as
infinite series with coefficients between $0$ and $p{-}1$, the main
difference being that additions and multiplications involve carries
in the $p$-adic case while they do not for formal series. 
From a more abstract point of view, the parallel between $\Zp$ and 
$\Fp[\![t]\!]$ is also apparent. For example $\Fp [\![t]\!]$ is also 
endowed with a valuation; this valuation defines a distance on 
$\Fp[\![t]\!]$ for which $\Fp[\![t]\!]$ is complete and Hensel's Lemma 
(together with Newton iteration) extends \emph{verbatim} to formal 
series. In the same fashion, the analogue of $\Qp$ is the field of 
Laurent series $\Fp(\!(t)\!)$. Noticing further that $\Fp[t]$ is dense 
in $\Fp[\![t]\!]$ and similarly that $\Fp(t)$ is dense in $\Fp(\!(t)\!)$ 
\footnote{The natural embedding $\Fp(t) \to \Fp(\!(t)\!)$ takes a 
rational function to its Taylor expansion at $0$.}, we can even 
supplement the correspondence between the characteristic $0$ and the 
characteristic $p$, ending up with the ``dictionary'' of 
Figure~\ref{fig:char0p}.

\begin{figure}
\hfill
\begin{tabular}{ccc}
char. $0$ & & char. $p$ \smallskip \\
$\Z$ & $\longleftrightarrow$ & $\Fp[t]$ \\
$\Q$ & $\longleftrightarrow$ & $\Fp(t)$ \\
$\Z_p$ & $\longleftrightarrow$ & $\Fp[\![t]\!]$ \\
$\Q_p$ & $\longleftrightarrow$ & $\Fp(\!(t)\!)$ 
\end{tabular}
\hfill\null

\caption{Similarities between characteristic $0$ and characteristic $p$}
\label{fig:char0p}
\end{figure}

Algebraists have actually managed to capture the essence of there 
resemblances in the notion of \emph{discrete valuation rings/fields} and 
\emph{completion} of those~\cite{Se62}. Recently, Scholze defined the 
notion of \emph{perfectoid} which allows us (under several additional 
assumptions) to build an actual bridge between the characteristic $0$ 
and the characteristic $p$. It is not the purpose of this lecture to go 
further in this direction; we nevertheless refer interested people to 
\cite{Sc12,Sc14,Bh14} for the exposition of the theory.

\subsubsection{The point of view of algebraic geometry}

Algebraic Geometry has an original and interesting point of view 
on the various rings and fields introduced before, showing in 
particular how $p$-adic numbers can arise in many natural problems in 
arithmetics. The underlying ubiquitous idea in algebraic geometry is
to associate to any ring $A$ a geometrical space $\spec A$ --- its 
so-called \emph{spectrum} --- on which the functions are the elements
of $A$. We will not detail here the general construction of $\spec A$
but just try to explain informally what it looks like when $A$ is 
one of the rings appearing in Figure~\ref{fig:char0p}.

Let us start with $k[t]$ where $k$ is a field (of any characteristic). 
For the simplicity of the exposition, let us assume further that $k$ is 
algebraically closed.
One thinks of elements of $k[t]$ as (polynomial) functions over $k$, 
meaning that this spectrum should be thought of as $k$ itself. $\spec 
k[t]$ is called the affine line over $k$ and is usually drawn as a 
straight line. The spectrum of $k(t)$ can be understood in similar 
terms: a rational fraction $f \in k(t)$ defines a function on $k$ as 
well, except that it can be undefined at some points. Therefore $\spec 
k(t)$ might be thought as the affine line over $k$ with a ``moving'' 
finite set of points removed. It is called the \emph{generic point} of 
the affine line.

What about $k[\![t]\!]$? If $k$ has characteristic $0$ (which we assume 
for the simplicity of the discussion), the datum of $f \in k[\![t]\!]$ 
is equivalent to the datum of the values of the $f^{(n)}(0)$'s (for $n$ 
varying in $\N$); we sometimes say that $f$ defines a function on a 
\emph{formal neighborhood} of $0$. This formal neighborhood is the 
spectrum of $k[\![t]\!]$; it should then be thought as a kind of 
thickening of the point $0 \in \spec k[t]$ which does not include any 
other point (since a formal series $f \in k[\![t]\!]$ cannot be 
evaluated in general at any other point than $0$). 
Finally $\spec k(\!(t)\!)$ is the \emph{punctured formal neighborhood} 
of $0$; it is obtained from $\spec k[\![t]\!]$ by removing the point
$0$ but not its neighborhood!

The embedding $k[t] \to k[\![t]\!]$ (resp. $k(t) \to k(\!(t)\!)$) given 
by Taylor expansion at $0$ corresponds to the restriction of a function 
$f$ defined on the affine line (resp. the generic point of the affine 
line) to the formal neighborhood (resp. the punctured formal 
neighborhood) of $0$.

Of course, one can define similarly formal neighborhoods and punctured 
formal neighborhoods around other points: for $a\in k$, the 
corresponding rings are respectively $k[\![h_a]\!]$ and $k(\!(h_a)\!)$ 
where $h_a$ is a formal variable which plays the role $t{-}a$. The 
algebraic incarnation of the restrictions to $\spec k[\![h_a]\!]$ and 
$k(\!(h_a)\!)$ are the Taylor expansions at $a$.

The above discussion is illustrated by the drawing at the top of 
Figure~\ref{fig:alggeom}; it represents more precisely the rational 
fraction $f(t) = \frac{t^2}{1-t}$ viewed as a function defined on $\spec 
k(t)$ (with $1$ as point of indeterminacy) together its Taylor expansion 
around $0$ (in blue) and $2$ (in purple) which are defined on formal 
neighborhoods.

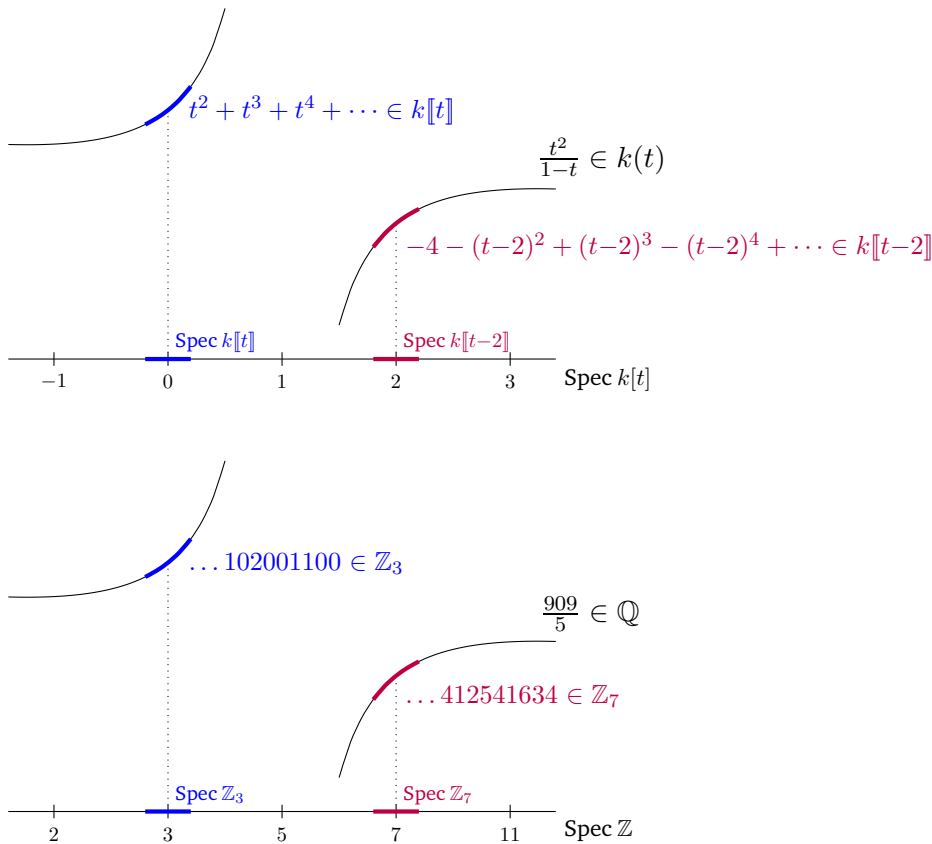
\begin{figure}
\hfill
\begin{tikzpicture}[xscale=1.5,yscale=0.3]
\draw plot[smooth] coordinates { 
  (-1.40000000000000,0.483333333333333)
  (-1.30000000000000,0.473913043478261)
  (-1.20000000000000,0.472727272727273)
  (-1.10000000000000,0.480952380952381)
  (-1.00000000000000,0.500000000000000)
  (-0.900000000000000,0.531578947368421)
  (-0.800000000000000,0.577777777777778)
  (-0.700000000000000,0.641176470588235)
  (-0.600000000000000,0.725000000000000)
  (-0.500000000000000,0.833333333333333)
  (-0.400000000000000,0.971428571428571)
  (-0.300000000000000,1.14615384615385)
  (-0.200000000000000,1.36666666666667)
  (-0.100000000000000,1.64545454545455)
  (0,2.00000000000000)
  (0.0999999999999999,2.45555555555555)
  (0.200000000000000,3.05000000000000)
  (0.300000000000000,3.84285714285714)
  (0.400000000000000,4.93333333333333)
  (0.500000000000000,6.50000000000000)
};
\draw plot[smooth] coordinates { 
  (1.50000000000000,-7.50000000000000)
  (1.60000000000000,-5.93333333333333)
  (1.70000000000000,-4.84285714285714)
  (1.80000000000000,-4.05000000000000)
  (1.90000000000000,-3.45555555555555)
  (2.00000000000000,-3.00000000000000)
  (2.10000000000000,-2.64545454545454)
  (2.20000000000000,-2.36666666666667)
  (2.30000000000000,-2.14615384615384)
  (2.40000000000000,-1.97142857142857)
  (2.50000000000000,-1.83333333333333)
  (2.60000000000000,-1.72500000000000)
  (2.70000000000000,-1.64117647058823)
  (2.80000000000000,-1.57777777777778)
  (2.90000000000000,-1.53157894736842)
  (3.00000000000000,-1.50000000000000)
  (3.10000000000000,-1.48095238095238)
  (3.20000000000000,-1.47272727272727)
  (3.30000000000000,-1.47391304347826)
  (3.40000000000000,-1.48333333333333)
};
\node[above] at (3.8,-1.483) { $\frac{t^2}{1-t} \in k(t)$ };

\draw[dotted] (0,-9)--(0,2);
\draw[ultra thick, blue] plot[smooth] coordinates { 
  (-0.200000000000000,1.36666666666667)
  (-0.100000000000000,1.64545454545455)
  (0,2.00000000000000)
  (0.0999999999999999,2.45555555555555)
  (0.200000000000000,3.05000000000000)
};
\node[blue,scale=0.9,right] at (0.1,2) { $t^2 + t^3 + t^4 + \cdots \in k[\![t]\!]$ };

\draw[dotted] (2,-9)--(2,-3);
\draw[ultra thick, purple] plot[smooth] coordinates { 
  (1.80000000000000,-4.05000000000000)
  (1.90000000000000,-3.45555555555555)
  (2.00000000000000,-3.00000000000000)
  (2.10000000000000,-2.64545454545454)
  (2.20000000000000,-2.36666666666667)
};
\node[purple,scale=0.9,below right] at (2,-3) { $-4 - (t{-}2)^2 + (t{-}2)^3 - (t{-}2)^4 + \cdots \in k[\![t{-}2]\!]$ };

\begin{scope}[yshift=-9cm]
\draw (-1.4,0)--(3.4,0);
\node[below right,scale=0.8] at (3.4,0) { $\spec k[t]$ };
\draw (-1,0.3)--(-1,-0.3);
\node[scale=0.7] at (-1,-1) { ${-}1$ };
\draw (0,0.3)--(0,-0.3);
\node[scale=0.7] at (0,-1) { $0$ };
\draw (1,0.3)--(1,-0.3);
\node[scale=0.7] at (1,-1) { $1$ };
\draw (2,0.3)--(2,-0.3);
\node[scale=0.7] at (2,-1) { $2$ };
\draw (3,0.3)--(3,-0.3);
\node[scale=0.7] at (3,-1) { $3$ };
\draw[ultra thick, blue] (-0.2,0)--(0.2,0);
\node[above right,blue,scale=0.7] at (0,0) { $\spec k[\![t]\!]$ };
\draw[ultra thick, purple] (1.8,0)--(2.2,0);
\node[above right,purple,scale=0.7] at (2,0) { $\spec k[\![t{-}2]\!]$ };
\end{scope}

\begin{scope}[yshift=-20cm]
\draw plot[smooth] coordinates { 
  (-1.40000000000000,0.483333333333333)
  (-1.30000000000000,0.473913043478261)
  (-1.20000000000000,0.472727272727273)
  (-1.10000000000000,0.480952380952381)
  (-1.00000000000000,0.500000000000000)
  (-0.900000000000000,0.531578947368421)
  (-0.800000000000000,0.577777777777778)
  (-0.700000000000000,0.641176470588235)
  (-0.600000000000000,0.725000000000000)
  (-0.500000000000000,0.833333333333333)
  (-0.400000000000000,0.971428571428571)
  (-0.300000000000000,1.14615384615385)
  (-0.200000000000000,1.36666666666667)
  (-0.100000000000000,1.64545454545455)
  (0,2.00000000000000)
  (0.0999999999999999,2.45555555555555)
  (0.200000000000000,3.05000000000000)
  (0.300000000000000,3.84285714285714)
  (0.400000000000000,4.93333333333333)
  (0.500000000000000,6.50000000000000)
};
\draw plot[smooth] coordinates { 
  (1.50000000000000,-7.50000000000000)
  (1.60000000000000,-5.93333333333333)
  (1.70000000000000,-4.84285714285714)
  (1.80000000000000,-4.05000000000000)
  (1.90000000000000,-3.45555555555555)
  (2.00000000000000,-3.00000000000000)
  (2.10000000000000,-2.64545454545454)
  (2.20000000000000,-2.36666666666667)
  (2.30000000000000,-2.14615384615384)
  (2.40000000000000,-1.97142857142857)
  (2.50000000000000,-1.83333333333333)
  (2.60000000000000,-1.72500000000000)
  (2.70000000000000,-1.64117647058823)
  (2.80000000000000,-1.57777777777778)
  (2.90000000000000,-1.53157894736842)
  (3.00000000000000,-1.50000000000000)
  (3.10000000000000,-1.48095238095238)
  (3.20000000000000,-1.47272727272727)
  (3.30000000000000,-1.47391304347826)
  (3.40000000000000,-1.48333333333333)
};
\node[above] at (3.7,-1.483) { $\frac{909}{5} \in \Q$ };

\draw[dotted] (0,-9)--(0,2);
\draw[ultra thick, blue] plot[smooth] coordinates { 
  (-0.200000000000000,1.36666666666667)
  (-0.100000000000000,1.64545454545455)
  (0,2.00000000000000)
  (0.0999999999999999,2.45555555555555)
  (0.200000000000000,3.05000000000000)
};
\node[blue,scale=0.9,right] at (0.1,2) { $\ldots102001100 \in \Z_3$ };

\draw[dotted] (2,-9)--(2,-3);
\draw[ultra thick, purple] plot[smooth] coordinates { 
  (1.80000000000000,-4.05000000000000)
  (1.90000000000000,-3.45555555555555)
  (2.00000000000000,-3.00000000000000)
  (2.10000000000000,-2.64545454545454)
  (2.20000000000000,-2.36666666666667)
};
\node[purple,scale=0.9,below right] at (2,-3) { $\ldots412541634 \in \Z_7$ };

\begin{scope}[yshift=-9cm]
\draw (-1.4,0)--(3.4,0);
\node[below right,scale=0.8] at (3.4,0) { $\spec \Z$ };
\draw (-1,0.3)--(-1,-0.3);
\node[scale=0.7] at (-1,-1) { $2$ };
\draw (0,0.3)--(0,-0.3);
\node[scale=0.7] at (0,-1) { $3$ };
\draw (1,0.3)--(1,-0.3);
\node[scale=0.7] at (1,-1) { $5$ };
\draw (2,0.3)--(2,-0.3);
\node[scale=0.7] at (2,-1) { $7$ };
\draw (3,0.3)--(3,-0.3);
\node[scale=0.7] at (3,-1) { $11$ };
\draw[ultra thick, blue] (-0.2,0)--(0.2,0);
\node[above right,blue,scale=0.7] at (0,0) { $\spec \Z_3$ };
\draw[ultra thick, purple] (1.8,0)--(2.2,0);
\node[above right,purple,scale=0.7] at (2,0) { $\spec \Z_7$ };
\end{scope}
\end{scope}
\end{tikzpicture}
\hfill \null

\caption{The point of view of algebraic geometry}
\label{fig:alggeom}
\end{figure}

\medskip

We now move to the arithmetical rings ($\Z$, $\Q$, $\Zp$ and $\Qp$) for 
which the picture is surprisingly very similar. In order to explain what 
is the spectrum of $\Z$, we first need to reinterpret $\spec k[t]$ in a 
more algebraic way. Given $f \in k[t]$, its evaluation at $a$, namely 
$f(a)$, can alternatively be defined as the remainder in the Euclidean 
division of $a$ by $t{-}a$. The affine line over $k$ then appears more 
intrinsically as the set of monic polynomials of degree $1$ over $k$ or, 
even better, as the set of irreducible monic polynomials (recall that we 
had assumed that $k$ is algebraically closed).

Translating this to the arithmetical world, we are inclined to think at 
the spectrum of $\Z$ as the set of prime numbers. Moreover, the 
evaluation of an integer $a \in \Z$ at the prime $p$ should be nothing 
but the remainder in the Euclidean division of $a$ by $p$. The integer 
$a$ defines this way a function over the spectrum of $\Z$.

A rational number $\frac a b$ can be reduced modulo $p$ for all prime 
$p$ not dividing $b$. It then defines a function on the spectrum of $\Z$ 
except that there are points of indeterminacy. $\spec \Q$ then appears as 
the generic point of $\spec \Z$ exactly as $\spec k(t)$ appeared as the
generic point of $\spec k[t]$.

Now sending an integer or a rational number to $\Qp$ is the exact
algebraic analogue of writing the Taylor expansion of a polynomial or
a rational fraction. That is the reason why $\spec \Zp$ should be
thought of as a formal neighborhood of the point $p \in \spec \Z$.
Similarly $\spec \Qp$ is the punctured formal neighborhood around $p$.

The second drawing of Figure~\ref{fig:alggeom} (which really looks like 
the first one) displays the rational number $\frac{909} 5$ viewed as a 
function over $\spec \Z$ (with $5$ as point of indeterminacy) together 
with its local $p$-adic/Taylor expansion at $p = 3$ and $p = 7$.

\subsubsection{Local-global principle}
\label{sssec:localglobal}

When studying equations where the unknown is a function, it is often 
interesting to look at local properties of a potential solution.
Typically, if we have to solve a differential equation:
$$a_d(t) y^{(d)}(t) + a_{d-1}(t) y^{(d-1)}(t) + \cdots + a_1(t) y'(t)
+ a_0 y(t) = b(t)$$
a standard strategy consists in looking for analytic solutions of the 
shape $y(t) = \sum_n c_n (t{-}a)^n$ for some $a$ lying in the base 
field. The differential equation then rewrites as a recurrence on the 
coefficients $c_n$ which sometimes can be solved. This reasoning yields
local solutions which have to be glued afterwards.

Keeping in mind the analogy between functions and integers/rationals, we 
would like to use a similar strategy for studying Diophantine equations
over $\Q$. Consider then a Diophantine equation of the shape:
\begin{equation}
\label{eq:polyeq}
P(x_1, \ldots, x_n) = 0
\end{equation}
where $P$ is a polynomial with rational coefficients and the $x_i$'s are
the unknowns. If Eq.~\eqref{eq:polyeq} has a global solution, \emph{i.e.}
a solution in $\Q^n$, then it must have local solutions everywhere,
\emph{i.e.} a solution in $\Qp^n$ for each prime number $p$. Indeed $\Q$
embeds into each $\Qp$. We are interested in the converse: assuming that
Eq.~\eqref{eq:polyeq} has local solutions everywhere, can we glue them 
in order to build a global solution? 
Unfortunately, the answer is negative in general. There is nevertheless
one remarkable case for which this principle works well.

\begin{theo}[Hasse--Minkowski]
Let $P(x_1, \ldots, x_n)$ be a multivariate polynomial. We assume that
$P$ is quadratic, \emph{i.e.} that the total degree of $P$ is $2$.
Then,
the equation \eqref{eq:polyeq} has a solution in $\Q^n$ if and only if
it has a solution in $\R^n$ and in $\Qp^n$ for all prime numbers $p$.
\end{theo}

We refer to~\cite{Se70} for the proof of this theorem (which is not
really the purpose of this course).
Understanding the local-global principle beyond the case of quadratic 
polynomials has motivated a lot of research for more than 50 years.
In 1970, at the International Congress of Mathematicians in Nice, Manin 
highlighted a new obstruction of cohomological nature to the possibility of glueing local 
solutions~\cite{Ma71}. This obstruction is called nowadays the 
\emph{Brauer--Manin obstruction}\footnote{The naming comes from the fact 
this obstruction is written in the language of Brauer groups... the 
latter being defined by Grothendieck in the context we are interested 
in.}. Exhibiting situations where it can explain, on its own, the 
non-existence of rational solutions is still an active domain of 
research today, in which very recent breakthrough progresses have been 
done~\cite{HaWi16,Ha16}.

\subsection{Why should we implement $p$-adic numbers?}

We have seen that $p$-adic numbers are a wonderful mathematical object 
which might be quite useful for arithmeticians. However it is still not 
clear that it is worth implementing them in mathematical software.
The aim of this subsection is to convince the reader that it is
definitely worth it. Here are three strong arguments supporting this 
thesis:
\begin{enumerate}[(A)]
\renewcommand{\itemsep}{0pt}
\item \emph{$p$-adic numbers provide sometimes better numerical stability;}
\item \emph{$p$-adic numbers provide a solution for ``allowing for
division by $p$ in $\Fp$'';}
\item \emph{$p$-adic numbers really appear in nature.}
\end{enumerate}
In the next paragraphs (\S\S \ref{sssec:Hilbert}--\ref{sssec:nature}), 
we detail several examples illustrating the above arguments and showing 
that $p$-adic numbers appear as an essential tool in many questions 
of algorithmic nature.

\subsubsection{The Hilbert matrix}
\label{sssec:Hilbert}

The first aforementioned argument, namely the possibility using $p$-adic 
numbers to have better numerical stability, is used in several contexts 
as factorization of polynomials over $\Q$ or computation of Galois 
groups of number fields.
In order to avoid having to introduce too advanced concepts here, we 
have chosen to focus on a simpler example which was pointed out in 
Vaccon's PhD thesis \cite[\S 1.3.4]{Va15}: the computation of 
the inverse of the Hilbert matrix. Although this example is not directly 
related to the most concrete concerns of arithmeticians, it already very 
well highlights the potential of $p$-adic numbers when we are willing
to apply numerical methods to a problem which is initially stated over 
the rationals.

We recall that the Hilbert matrix of size $n$ is the square matrix $H_n$ 
whose $(i,j)$ entry is $\frac 1 {i+j-1}$ ($1 \leq i,j \leq n$). For 
example:
$$H_4 = \left( \begin{matrix}
1 & 1/2 & 1/3 & 1/4 \\
1/2 & 1/3 & 1/4 & 1/5 \\
1/3 & 1/4 & 1/5 & 1/6 \\
1/4 & 1/5 & 1/6 & 1/7 
\end{matrix}\right).$$
Hilbert matrices are famous for many reasons. One of them is that 
they are very ill-conditioned, meaning that numerical computations 
involving Hilbert matrices may lead to important numerical errors.
A typical example is the computation of the inverse of $H_n$. Let us
first mention that an exact formula giving the entries of $H_n^{-1}$
is known:
\begin{equation}
\label{eq:Hninv}
(H_n^{-1})_{i,j} = (-1)^{i+j} \cdot (i+j-1) 
\cdot \binom{n+i-1}{n-j}
\cdot \binom{n+j-1}{n-i}
\cdot \binom{i+j-2}{i-1}^2.
\end{equation}
(see for example~\cite{Ch83}).
We observe in particular than $H_n^{-1}$ has integral coefficients.

We now move to the numerical approach: we consider $H_n$ as a matrix 
with real coefficients and compute its inverse using standard Gaussian 
elimination (with choice of pivot) and IEEE floating-point arithmetics 
(with $53$ bits of precision)~\cite{IEEE08}. Here is the result we get 
with \sage for $n = 4$:
$$\small
H_4^{-1} \approx \left( \begin{matrix}
\phantom{+}15.9999999999998 & -119.999999999997 & \phantom{+}239.999999999992 & -139.999999999995 \\
-119.999999999997 & \phantom{+}1199.99999999996 & -2699.99999999989 & \phantom{+}1679.99999999993 \\
\phantom{+}239.999999999992 & -2699.99999999989 & \phantom{+}6479.99999999972 & -4199.99999999981 \\
-139.999999999995 & \phantom{+}1679.99999999993 & -4199.99999999981 & \phantom{+}2799.99999999987
\end{matrix} \right).$$
We observe that the accuracy of the computed result is acceptable but not so
high: the number of correct binary digits is about $45$ (on average 
on each entry of the matrix), meaning then that the number of incorrect 
digits is about $8$. Let us now increase the size of the matrix and observe 
how the accuracy behaves:

\begin{center}
\begin{tabular}{|c||c|c|c|c|c|c|c|c|c|c|}
\hline
{\small \begin{tabular}{@{}c@{}} size of \vspace{-0.3ex}\\the matrix\end{tabular}} 
& $5$ & $6$ & $7$ & $8$ & $9$ & $10$ & $11$ & $12$ & $13$ \\
\hline
{\small \begin{tabular}{@{}c@{}} number of \vspace{-0.3ex}\\ correct digits\end{tabular}}
& $40$ & $34$ & $28$ & $25$ & $19$ & $14$ & $9$ & $4$ & $0$ \\
\hline
\end{tabular}
\end{center}

\noindent
We see that the losses of accuracy are enormous.

\smallskip

On the other hand, let us examine now the computation goes when 
$H_n^{-1}$ is viewed as a matrix over $\Q_2$. Making experimentations in 
\sagenoref using again Gaussian elimination and the straightforward 
analogue of floating-point arithmetic 
(see \S \ref{ssec:pfloat}) with $53$ bits of precision, we 
observe the following behavior:
\begin{center}
\begin{tabular}{|c||c|c|c|c|c|c|c|c|c|c|c|c|}
\hline
{\small \begin{tabular}{@{}c@{}} size of \vspace{-0.3ex}\\the matrix\end{tabular}} 
& $5$ & $6$ & $7$ & $8$ & $9$ & $10$ & $11$ & $12$ & $13$ & $50$ & $100$ \\
\hline
{\small \begin{tabular}{@{}c@{}} number of \vspace{-0.3ex}\\ correct digits\end{tabular}}
& $52$ & $52$ & $51$ & $51$ & $51$ & $51$ & $51$ & $51$ & $51$ & $49$ & $48$ \\
\hline
\end{tabular}
\end{center}

\noindent
The computation of $H_n^{-1}$ seems quite accurate over $\Q_2$ whereas it 
was on the contrary highly inaccurate over $\R$. As a consequence, if we 
want to use numerical methods to compute the exact inverse of $H_n$ over 
$\Q$, it is much more interesting to go through the $2$-adic numbers. Of 
course this approach does not make any sense if we want a \emph{real} 
approximation of the entries of $H_n^{-1}$; in particular, if we are 
only interesting in the size of the entries of $H_n^{-1}$ (but not in
their exact values) passing through the $p$-adics is absurd since two 
integers of different sizes might be very close in the $p$-adic world.

\medskip

The phenomena occurring here are actually easy to analyze. The accuracy
of the inversion of a real matrix is governed by the 
condition number which is defined by:
$$\cond_\R(H_n) = \Vert H_n \Vert_\R \cdot \Vert H_n^{-1} \Vert_\R$$
where $\Vert \cdot \Vert_\R$ is some norm on $M_n(\R)$. 
According to Eq.~\eqref{eq:Hninv}, the entries of $H_n$ are large: 
as an example, the bottom right entry of $H_n$ is equal to $(2n-1) \cdot 
\binom{2n-2}{n-1}^2$ and thus grows exponentially fast with respect to 
$n$. As a consequence the condition number $\cond_\R(H_n)$ is quite 
large as well; this is the source of the loss of accuracy observed
for the computation of $H_n^{-1}$ over $\R$.

Over $\Qp$, one can argue similarly and consider the 
$p$-adic condition number:
$$\cond_{\Qp}(H_n) = 
  \Vert H_n \Vert_{\Qp} \cdot \Vert H_n^{-1} \Vert_{\Qp}$$
where $\Vert \cdot \Vert_{\Qp}$ is the infinite norm over $M_n(\Qp)$
(other norms do not make sense over $\Qp$ because of the ultrametricity).
Since $H_n^{-1}$ has integral coefficients, all its entries have their
$p$-adic norm bounded by $1$. Thus $\Vert H_n^{-1} \Vert_{\Qp} \leq 1$. 
As for the Hilbert matrix $H_n$ itself, its norm is equal to $p^v$ where 
$v$ is the highest power of $p$ appearing in a denominator of an entry 
of $H_n$, \emph{i.e.} $v$ is the unique integer such that $p^v \leq 2n-1 
< p^{v+1}$. Therefore $\Vert H_n \Vert_{\Qp} \leq 2n$ and 
$\cond_{\Qp}(H_n) = O(n)$. The growth of the $p$-adic condition number 
in then rather slow, explaining why the computation of $H_n^{-1}$ is 
accurate over $\Qp$.

\subsubsection{Lifting the positive characteristic}
\label{sssec:diffeq}

In this paragraph, we give more details about the mysterious sentence: 
\emph{$p$-adic numbers provide a solution for ``allowing for divisions by $p$ 
in $\Fp$''}.
Let us assume that we are working on a problem that makes sense over any 
field and always admits a unique solution (we will give examples later). 
To fix ideas, let us agree that our problem consists in designing a fast 
algorithm for computing the aforementioned unique solution. Assume 
further that such an algorithm is available when $k$ has characteristic 
$0$ but does not extend to the positive characteristic (because it
involves a division by $p$ at some point).
In this situation, one can sometimes take advantage of $p$-adic numbers to attack the 
problem over the finite field $\Fp$, proceeding concretely in three
steps as follows:
\begin{enumerate}
\renewcommand{\itemsep}{0pt}
\item we lift our 
problem over $\Qp$ (meaning that we introduce a $p$-adic instance of 
our problem whose reduction modulo $p$ is the problem we have started
with),
\item we solve the problem over $\Qp$ (which has characteristic $0$),
\item we finally reduce the solution modulo $p$. 
\end{enumerate}
The existence and uniqueness of the solution together ensure that the
solution of the $p$-adic problem is defined over $\Zp$ and reduces 
modulo $p$ to the correct answer.
Concretely, here are two significant examples where the above strategy 
was used with success: the fast computation of composed products 
\cite{BoGoPeSc05}
and the fast computation of isogenies in positive characteristic 
\cite{LeSi08,LaVa16}.

\medskip

In order to give more substance to our thesis, we briefly detail the 
example of composed products. Let $k$ be a field. Given two monic 
polynomials $P$ and $Q$ with coefficients in $k$, we recall that the 
composed product of $P$ and $Q$ is defined by:
$$P \otimes Q = \sum_\alpha \sum_\beta (X - \alpha \beta)$$
where $\alpha$ (resp. $\beta$) runs over the roots of $P$ (resp. 
of $Q$) in an algebraic closure of $k$, with the convention that a
root is repeated a number of times equal to its multiplicity. Note
that $P \otimes Q$ is always defined over $k$ since all its coefficients
are written as symmetric expressions in the $\alpha$'s and the $\beta$'s.
We address the question of the fast multiplication of composed products.

If $k$ has characteristic $0$, Dvornicich and Traverso~\cite{DvTr89} proposed
a solution based on Newton sums. Indeed, letting $S_{P,n}$ and $S_{Q,n}$ 
be the Newton sums of $P$ and $Q$ respectively:
$$S_{P,n} = \sum_\alpha \alpha^n \quad and \quad 
S_{Q,n} = \sum_\beta \beta^n.$$
It is obvious that the product $S_{P,n} \cdot S_{Q,n}$ is the $n$-th 
Newton sum of $P \otimes Q$. Therefore, if we design fast algorithms for 
going back and forth between the coefficients of a polynomial and its 
Newton sums, the problem of fast computation of the composed product 
will be solved.

Schönhage~\cite{Sc82} proposed a nice solution to the latter problem. 
It relies on a remarkable differential equation relating a polynomial
and the generating function of the sequence of its Newton sums.
Precisely, let $A$ be a monic polynomial of degree $d$ over $k$. 
We define the formal series:
$$S_A(t) = \sum_{n=0}^\infty S_{A,n+1} t^n.$$
A straightforward computation then leads to the following remarkable relation:
$$(E) : \quad A_\rec'(t) = - S_A(t) \cdot A_\rec(t)$$
where $A_\rec$ is the reciprocal polynomial of $A$ defined by $A_\rec(X) 
= X^d \cdot A(\frac 1 X)$.
If the polynomial $A$ is known, then one can compute $S_A(t)$ 
by performing a simple division in $k[\![t]\!]$. Using fast 
multiplication and Newton iteration for computing the inverse (see \S 
\ref{sssec:inverse}), this method leads to an algorithm that computes 
$S_1, S_2, \ldots, S_N$ for a cost of $\softO(N)$ operations in $k$.

Going in the other direction reduces to solving the differential equation 
$(E)$ where the unknown is now the polynomial $A_\rec$. When $k$ has 
characteristic $0$, Newton iteration can be used in this context (see 
for instance~\cite{BoChLeSaSc07}). This leads to the following 
recurrence:
$$B_1(t) = 1 \quad ; \quad
B_{i+1}(t) = B_i(t) - B_i(t) 
\int \left(\frac{B'_i(t)}{B_i(t)} + S_A(t) \right)dt$$
where the integral sign stands for the linear operator acting on 
$k[\![t]\!]$ by $t^n \mapsto \frac{t^{n+1}}{n+1}$.
One proves that $B_i(t) \equiv A_\rec(t) \pmod {t^{2^i}}$ for all 
$i$. Since $A_\rec(t)$ is a polynomial of degree $d$, it is then enough 
to compute it modulo $t^{d+1}$. Hence after $\log_2 (d{+}1)$ iterations 
of the Newton scheme, we have computed all the relevant coefficients. 
Moreover all the computations can be carried by omitting all terms of 
degree $> d$. This leads to an algorithm that computes $A_\rec$ (and 
hence $A$) for a cost of $\softO(d)$ operations in $k$.

In positive characteristic, this strategy no longer works because the 
integration involves divisions by integers. When $k$ is the finite 
field with $p$ elements\footnote{The same strategy actually extends to 
all finite fields.}, one can tackle this issue by lifting the problem
over $\Zp$, following the general strategy discussed here. First, we 
choose two polynomials $\hat P$ and $\hat Q$ with coefficients in $\Zp$ 
with the property that $\hat P$ (resp. $\hat Q$) reduces to $P$ (resp. 
$Q$) modulo $p$. Second, applying the characteristic $0$ method to
$\hat P$ and $\hat Q$, we compute $\hat P \otimes \hat Q$. Third, we
reduce the obtained polynomial modulo $p$, thus recovering $P \otimes
Q$. This strategy was concretized by the work of Bostan, Gonzalez-Vega, 
Perdry and Schost~\cite{BoGoPeSc05}.
The main difficulty in their work is to keep control on the $p$-adic 
precision we need in order to carry out the computation and be sure to 
get the correct result at the end; this analysis is quite subtle and was 
done by hand in~\cite{BoGoPeSc05}. Very general and powerful tools are now
available for dealing with such questions; they will be presented in
\S \ref{sec:precision}.

\subsubsection{$p$-adic numbers in nature}
\label{sssec:nature}

We have already seen in the introduction of this course that $p$-adic 
numbers are involved in many recent developments in Number Theory and 
Arithmetic Geometry (see also \S \ref{sssec:localglobal}). Throughout 
the 20th century, many $p$-adic objects have been defined and studied as 
well: $p$-adic modular/automorphic forms, $p$-adic differential 
equations, $p$-adic cohomologies, $p$-adic Galois representations, 
\emph{etc.} It turns out that mathematicians, more and more often, feel 
the need to make ``numerical'' experimentations on these objects for 
which having a nice implementation of $p$-adic numbers is of course a 
prerequisite.

Moreover $p$-adic theories sometimes have direct consequences in other 
areas of mathematics. A good example in this direction is the famous 
problem of counting points on algebraic curves defined over finite 
fields (\emph{i.e.} roughly speaking counting the number of solutions in 
finite fields of equations of the shape $P(X,Y) = 0$ where $P$ is 
bivariate polynomial). This question has received much attention during
the last decade because of its applications to cryptography (it serves
as a primitive in the selection of secure elliptic curves).
Since the brilliant intuition of Weil~\cite{We49} followed by the revolution 
of Algebraic Geometry conducted by Grothendieck in the 20th century 
\cite{SGA}, the approach to counting problems is now often cohomological.
Roughly speaking, if $C$ is an algebraic variety (with some additional 
assumptions) defined over a finite field $\F_q$, then the number of points
of $C$ is related to the traces of the Frobenius map acting on the
cohomology of $C$. Counting points on $C$ then ``reduces'' to computing
the Frobenius map on the cohomology. Now it turns out that traditional
algebraic cohomology theory yields vector spaces over $\Q_\ell$ (for an 
auxiliary prime number $\ell$ which does not need to be equal to the 
characteristic of the finite field we are working with).
This is the way $\ell$-adic numbers enter naturally into the scene.

\section{Several implementations of $p$-adic numbers}
\label{sec:implementations}

Now we are all convinced that it is worth implementing $p$-adic numbers, 
we need to discuss the details of the implementation. The main problem 
arising when we are trying to put $p$-adic numbers on computers is the 
precision. Indeed, remember that a $p$-adic number is defined as an 
infinite sequence of digits, so that it \emph{a priori} cannot be stored 
entirely in the memory of a computer. From this point of view, $p$-adic 
numbers really behave like real numbers; the reader should therefore not 
be surprised if he/she often detects similarities between the solutions 
we are going to propose for the implementation of $p$-adic numbers and 
the usual implementation of reals.

In this section, we design and discuss three totally different 
paradigms: zealous arithmetic (\S \ref{ssec:zealous}), lazy 
arithmetic together with the relaxed improvement (\S \ref{ssec:lazy}) and 
$p$-adic floating-point arithmetic (\S \ref{ssec:pfloat}). Each of 
these ways of thinking has of course its own advantages and 
disadvantages; we will try to compare them fairly in \S
\ref{ssec:compparadigm}.

\subsection{Zealous arithmetic}
\label{ssec:zealous}

Zealous arithmetic is by far the most common implementation of 
$p$-adic numbers in usual mathematical software: \magma, \sage and 
\pari use it for instance. It appears as the exact analogue of the 
interval arithmetic of the real setting: we replace $p$-adic numbers by 
intervals. The benefit is of course that intervals can be represented 
without errors by finite objects and then can be stored and manipulated 
by computers.

\subsubsection{Intervals and the big-$O$ notation}

Recall that we have defined in \S \ref{sssec:tree} (see Definitions 
\ref{def:intervalZp} and \ref{def:intervalQp}) the notion of interval: 
a bounded interval of $\Qp$ is by definition a closed ball lying inside
$\Qp$, \emph{i.e.} a subset of the form:
$$I_{N,a} = \big\{ \, x \in \Qp 
\quad \text{s.t.} \quad 
|x - a| \leq p^{-N} \, \big\}.$$
The condition $|x - a| \leq p^{-N}$ can be rephrased in a more algebraic 
way since it is equivalent to the condition $x \equiv a \pmod{p^N \Zp}$, 
\emph{i.e.} $x \in a + p^N \Zp$. In short $I_{N,a} = a + p^N \Zp$. In 
symbolic computation, we generally prefer using the big-$O$ notation and 
write:
$$I_{N,a} = a + O(p^N).$$
The following result is easy but fundamental for our propose.

\begin{prop}
\label{prop:intervalcomputer}
Each bounded interval $I$ of $\Qp$ may be written as:
$$I = p^v s + O(p^N)$$
where $(N,v,s)$ is either of the form $(N,N,0)$ or a triple of relative 
integers with $v < N$, $s \in [0, p^{N-v})$ and $\gcd(s,p) = 1$.
Moreover this writing is unique.

In particular, bounded intervals of $\Qp$ are representable by exact
values on computers.
\end{prop}

\begin{proof}
It is a direct consequence of the fact that
$I_{N,a}$ depends only on the class of $a$ modulo $p^N$.
\end{proof}

The interval $a + O(p^N)$ 
has a very suggestive interpretation if we are thinking at $p$-adic 
numbers in terms of infinite sequences of digits.
Indeed, write 
$$a = a_v p^v + a_{v+1} p^{v+1} + \cdots + a_{N-1} p^{N-1} + \cdots$$ 
with $0 \leq a_i < p$ and agree to define $a_i = 0$ for 
$i < v$. A $p$-adic number $x$ then lies in $a + O(p^N)$ if its $i$-th 
digit is $a_i$ for all $i < N$. Therefore one may think of the notation 
$a + O(p^N)$ as a $p$-adic number of the shape:
$$a_v p^v + a_{v+1} p^{v+1} + \cdots + a_{N-1} p^{N-1} + {}?\:p^N 
+{} ?\:p^{N+1} + \cdots$$
where the digits at all positions $\geq N$ are unspecified. This 
description enlightens Proposition \ref{prop:intervalcomputer} 
which should become absolutely obvious now.

\medskip

We conclude this paragraph by introducing some additional vocabulary.

\begin{deftn}
Let $I = a + O(p^N)$ be an interval. We define:

\medskip

\noindent
\begin{tabular}{c@{\hspace{1ex}}lll}
$\bullet$ &
the \emph{absolute precision} of $I$: & $\abs(I) = N$; \smallskip \\
$\bullet$ &
the \emph{valuation} of $I$: & $\val(I) = \val(a)$ & if $0 \not\in I$, \\
 && $\val(I) = N$ & otherwise; \smallskip \\
$\bullet$ &
the \emph{relative precision} of $I$: & 
\multicolumn{2}{l}{$\rel(I) = \abs(I) - \val(I)$.}
\end{tabular}
\end{deftn}

\smallskip

We remark that the valuation of $I$ is the integer $v$ of 
Proposition~\ref{prop:intervalcomputer}. It is also always the smallest 
valuation of an element of $I$. Moreover, coming back to the 
interpretation of $p$-adic numbers as infinite sequences of digits, we 
observe that the relative precision of $a + O(p^N)$ is the number of 
known digits (with rightmost zeroes omitted) of the family of $p$-adic 
numbers it represents.

\subsubsection{The arithmetics of intervals}

Since intervals are defined as subsets of $\Qp$, basic arithmetic
operations on them are defined in a straightforward way. For example, 
the sum of the intervals $I$ and $J$ is the subset of $\Qp$ consisting of 
all the elements of the form $a + b$ with $a \in I$ and $b \in J$. As with real 
intervals, it is easy to write down explicit formulas giving the results
of the basic operations performed on $p$-adic intervals.

\begin{prop}
\label{prop:arithinterval}
Let $I = a + O(p^N)$ and $I' = a' + O(p^{N'})$ be two bounded intervals
of $\Qp$. Set $v = \val(I)$ and $v' = \val(I')$. We have:
\begin{align}
I + I' & = a + a' + O\big(p^{\min(N,\,N')}\big) \label{eq:precadd} \\
I - I' & = a - a' + O\big(p^{\min(N,\,N')}\big) \label{eq:precsub} \\
I \times I' & = aa' + O\big(p^{\min(v+N',\,N+v')}\big) \label{eq:precmul} \\
I \div I' & = \frac a{a'} + O\big(p^{\min(v + N' -2v',\, N-v')}\big) \label{eq:precdiv}
\end{align}
where in the last equality, we have further assumed that $0 \not\in I'$.
\end{prop}

\begin{rem}
\label{rem:arithinterval}
Focusing on the precision, we observe that the two first 
formulae stated in Proposition \ref{prop:arithinterval} immediately
imply $\abs(I + I') = \abs(I - I') = \min(\abs(I), \, \abs(I'))$.
Concerning multiplication and division, the results look ugly at 
first glance. They are however much more pretty if we translate them 
in terms of relative precision; indeed, they become simply
$\rel(I\times I') = \rel(I \div I') = \min(\rel(I), \, \rel(I'))$
which is parallel to the case of addition and subtraction and
certainly easier to remember.
\end{rem}

\begin{proof}[Proof of Proposition~\ref{prop:arithinterval}]
The proofs of Eq.~\eqref{eq:precadd} and Eq.~\eqref{eq:precsub} are easy
and left as an exercise to the reader. We then move directly to
multiplication. Let $x \in I \times I'$. By definition, there exist
$h \in p^N \Zp$ and $h' \in p^{N'}\Zp$ such that:
$$x = (a + h) (a' + h')
= aa' + ah' + a'h + hh'.$$
Moreover, coming back to the definition of the valuation of an
interval, we find $v \leq \val(a)$. The term $ah'$ is then divisible 
by $p^{v+N'}$. Similarly $a'h$ is divisible by $p^{N+v'}$. Finally
$hh'$ is divisible by $p^{v+v'}$; it is then \emph{a fortiori} also
divisible by $p^{v+N'}$ because $v \leq N$. Thus $ah'+ a'h + hh'$
is divisible by $p^{\min(v+N',\,N+v')}$ and we have proved one
inclusion:
$$I \times I' \subset aa' + O\big(p^{\min(v+N',\,N+v')}\big).$$
Conversely, up to swapping $I$ and $I'$, we may assume that $\min
(v+N',\,N+v') = v+N'$. Up to changing $a$, we may further suppose
that $\val(a) = v$. Pick now $y \in aa' + O(p^{v+N'})$, \emph{i.e.}
$y = aa' + h$ for some $h$ divisible by $p^{v+N'}$. 
Thanks to our assumption on $a$, we have $\val(\frac h a) = \val(h) - 
\val(a) \geq v + N' - v = N'$, \emph{i.e.} $p^{N'}$ divides
$\frac h a$. Writing $y = a \times (a' + \frac h a)$ then shows 
that $y \in I \times I'$. The converse inclusion follows and
Eq.~\eqref{eq:precmul} is established.

We now assume that $0 \not\in I'$ and start the proof of 
Eq.~\eqref{eq:precdiv}. We define $1 \div I'$ as the set of 
inverses of elements of $I'$. We claim that:
\begin{equation}
\label{eq:oneoverIp}
1 \div I' = \frac 1{a'} + O(p^{N'-2v'}).
\end{equation}
In other to prove the latter relation, we write $I' = a' \times 
(1 + O(p^{N-v'}))$. We then notice that the function $x \mapsto
x^{-1}$ induces a involution from $1 + O(p^{N-v'})$ to itself; in
particular $1 \div (1 + O(p^{N-v'})) = 1 + O(p^{N-v'})$. Dividing
both sides by $a'$, we get \eqref{eq:oneoverIp}. We finally
derive that
$I \div I' = I \times \big(\frac 1{a'} + O(p^{N'-2v'})\big)$.
Eq.~\eqref{eq:precdiv} then follows from Eq.~\eqref{eq:precmul}.
\end{proof}

Using Proposition~\ref{prop:arithinterval}, it is more or less 
straightforward to implement addition, subtraction, multiplication and 
division on intervals when they are represented as triples $(N,v,s)$ as 
in Proposition~\ref{prop:intervalcomputer} (be careful however at the 
conditions on $v$ and $s$). This yields the basis of the \emph{zealous
arithmetic}. 

\subsubsection{Newton iteration}
\label{sssec:intervalNewton}

A tricky point when dealing with zealous arithmetic concerns Newton
iteration. To illustrate it, let us examine the example of the
computation of a square root of $c$ over $\Q_2$ for
$$c = 1 + 2^{3} + 2^{4} + 2^{5} + 2^{10} + 2^{13} + 2^{16} + 2^{17} + 
      2^{18} + 2^{19} + O(2^{20}).$$
We observe that $c \equiv 1 \pmod 8$. By Hensel's Lemma, $c$
has then a unique square root in $\Q_2$ which is congruent to $1$ 
modulo $4$. Let us denote it by $\sqrt c$. Following \S \ref{sssec:sqrt}, 
we compute $\sqrt c$ using Newton iteration: if $(x_i)_{i \geq 0}$ is
the recursive sequence defined by
$$x_0 = 1 \quad ; \quad x_{i+1} = \frac 1 2 \cdot \left( x_i + \frac 
c{x_i}\right), \, i = 0, 1, 2, \ldots$$
we know that $x_i$ and $\sqrt c$ share the same $(2^i + 1)$ final
digits. Here is the result we get if we compute the first $x_i$'s
using zealous arithmetic:

$$\begin{array}{lr}
x_1 = & 
 {\color{black!70}\ldots1111001001000011}{\color{purple}101} \\
x_2 = & 
 {\color{black!70}\ldots0011101000111}{\color{purple}10101} \\
x_3 = & 
 {\color{black!70}\ldots10111010}{\color{purple}001010101} \\
x_4 = & 
 {\color{black!70}\ldots}{\color{purple}0111010001010101} \\
x_5 = & 
 {\color{black!70}\ldots}{\color{purple}111010001010101} \\
x_6 = & 
 {\color{black!70}\ldots}{\color{purple}11010001010101} 
\end{array}$$

\noindent
Here the $2^i{+}1$ rightmost digits of $x_i$ (which are the correct
digits of $\sqrt c$) are colored in purple and the dots represent
the unknown digits, that is the digits which are absorbed by the
$O(-)$. We observe that the precision decreases by $1$ digit at 
each iteration; this is of course due to the division by $2$ in the 
recurrence defining the $x_i$'s.
The maximal number of correct digits is obtained for $x_4$ with 
$16$ correct digits. After this step, the result stabilizes but 
the precision continues to decrease. Combining naively zealous 
arithmetic and Newton iteration, we have then managed to compute 
$\sqrt c$ at precision $O(2^{16})$.

\medskip

It turns out that the losses of precision we have just highlighted has 
no intrinsic meaning but is just a consequence of zealous 
arithmetic. We will now explain that it is possible to slightly modify 
Newton iteration in order to completely eliminate this 
unpleasant phenomenon. The starting point of your argument is 
Remark~\ref{rem:Hensel} which tells us that Newton iteration still 
converges at the same rate to $\sqrt c$ as soon as the sequence $(x_i)$ 
satisfies the weaker condition:
$$\left|x_{i+1} - 
\frac 1 2 \cdot \left( x_i + \frac
c{x_i}\right) \right| \leq 2^{2^{i+1}+1}.$$
In other words, we have a complete freedom on the choice of the 
``non-purple'' digits of $x_{i+1}$. In particular, if some digit of 
$x_{i+1}$ was not computed because the precision on $x_i$ was too poor, 
we can just assign freely our favorite value (\emph{e.g.} $0$) to this 
digit without having to fear unpleasant consequences. Even better, we 
can assign $0$ to each ``non-purple'' digit of $x_i$ as well; this will
not affect the final result and leads to computation with smaller
integers. 
Proceeding this way, we obtain the following sequence:

$$\begin{array}{lrr}
x_1: & 
 {\color{black!70}\ldots00000000000000000}{\color{purple}101} \\
x_2: & 
 {\color{black!70}\ldots000000000000000}{\color{purple}10101} \\
x_3: & 
 {\color{black!70}\ldots00000000000}{\color{purple}001010101} \\
x_4: & 
 {\color{black!70}\ldots000}{\color{purple}10111010001010101} \\
x_5: & 
 {\color{black!70}\ldots}{\color{purple}1010111010001010101}
\end{array}$$

\noindent
and we obtain the final result with $19$ correct digits instead of $16$. 
Be very careful: we cannot assign arbitrarily the $20$-th digit of $x_5$ 
because it is a ``purple'' digit and not a ``non-purple'' one.
More precisely if we do it, we have to write it in black (of course, we 
cannot decide randomly what are the correct digits of $\sqrt c$) and a 
new iteration of the Newton scheme again loses this $20$-th digit
because of the division by $2$.
In that case, one may wonder why we did not try to assign a $21$-st digit to 
$x_4$ in order to get one more correct digit in $x_5$. This sounds like
a good idea but does not work because the input $c$ --- on which we do 
not have any freedom --- appears in the recurrence formula defining the 
$x_i$'s and is given at precision $O(2^{20})$. For this reason, each 
$x_i$ cannot be computed with a better precision than $O(2^{19})$.

In fact, we can easily convince ourselves that $O(2^{19})$ is the 
optimal precision for $\sqrt c$. Indeed $c$ can be lifted at precision
$O(2^{21})$ either to:
\begin{align*}
c_1 & = 1 + 2^{3} + 2^{4} + 2^{5} + 2^{10} + 2^{13} + 2^{16} + 2^{17} + 
        2^{18} + 2^{19} + O(2^{21}) \\
\text{or} \quad 
c_2 & = 1 + 2^{3} + 2^{4} + 2^{5} + 2^{10} + 2^{13} + 2^{16} + 2^{17} + 
        2^{18} + 2^{19} + 2^{20} + O(2^{21}).
\end{align*}
The square roots of these liftings are:
\begin{align*}
\sqrt{c_1} & = 1 + 2^2 + 2^4 + 2^6 + 2^{10} + 2^{12} + 2^{13} + 2^{14} + 
               2^{16} + 2^{18} + O(2^{20}) \\
\sqrt{c_2} & = 1 + 2^2 + 2^4 + 2^6 + 2^{10} + 2^{12} + 2^{13} + 2^{14} + 
               2^{16} + 2^{18} + 2^{19} + O(2^{20})
\end{align*}
and we now see clearly that the $19$-th digit of $\sqrt c$ is
affected by the $20$-th digit of $c$.

\subsection{Lazy and relaxed arithmetic}
\label{ssec:lazy}

The very basic idea behind lazy arithmetic is the following: in every 
day life, we are not working with all $p$-adic numbers but only with 
\emph{computable} $p$-adic numbers, \emph{i.e.} $p$-adic numbers for
which there exists a program that outputs the sequence of its digits.
A very natural idea is then to use these programs to represent $p$-adic
numbers. By operating on programs, one should then be able to perform
additions, multiplications, \emph{etc}. of $p$-adic numbers.

In this subsection, we examine further this idea. In \S 
\ref{sssec:lazynaive}, we adopt a very naive point of view, insisting on 
ideas and not taking care of doability/performances. We hope that this 
will help the reader to understand more quickly the mechanisms behind the 
notion of lazy $p$-adic numbers. We will then move to complexity 
questions and will focus on the problem of designing a framework 
allowing our algorithms to take advantage of fast multiplication 
algorithms for integers (as 
Karatsuba's algorithm, Schönhage--Strassen's algorithm 
\cite[\S 8]{GaGe03} or Fürer's algorithm and its improvements 
\cite{Fu09,HaHoLe15}). This will lead us to report on the theory of 
\emph{relaxed algorithms} introduced recently by Van der Hoeven and his
followers~\cite{Ho97,Ho02,Ho07,BeHoLe11,BeLe12,Le13}.

Lazy $p$-adic numbers have been implemented in the software \mathemagix.

\subsubsection{Lazy $p$-adic numbers}
\label{sssec:lazynaive}

\begin{deftn}
A \emph{lazy $p$-adic number} is a program \ttx that takes as
input an integer $N$ and outputs a \emph{rational number} $\ttx
(N)$ with the property that:
$$\big|\ttx(N+1) - \ttx(N)\big|_p \leq p^{-N}$$
for all $N$.
\end{deftn}

The above condition implies that the sequence $\ttx(N)$ is a Cauchy 
sequence in $\Qp$ and therefore converges. We call its limit the 
\emph{value} of $\ttx$ and denote it by $\Value(\ttx)$. Thanks to 
ultrametricity, we obtain $|\Value(\ttx) - \ttx(N)| \leq p^{-N}$ for all 
$N$. In other words, $\ttx(N)$ is nothing but an approximation of 
$\Value(\ttx)$ at precision $O(p^N)$.

The $p$-adic numbers that arise as values of lazy $p$-adic numbers
are called \emph{computable}.
We observe that there are only countably many lazy $p$-adic numbers;
the values they take in $\Qp$ then form a countable set as well. Since $\Qp$
is uncountable (it is equipotent to the set of functions $\N \to
\{0, 1, \ldots, p{-}1\}$), it then exists many uncomputable $p$-adic
numbers.

Our aim is to use lazy $p$-adic numbers to implement actual (computable) 
$p$-adic numbers. In order to do so, we need (at least) to answer the 
following two questions:
\begin{itemize}
\renewcommand{\itemsep}{0pt}
\item Given (the source code of) two lazy $p$-adic numbers \ttx and 
\tty, can we decide algorithmically whether $\Value(\ttx) = \Value(\tty)$
or not?
\item Can we lift standard operations on $p$-adic numbers at the level
of lazy $p$-adic numbers?
\end{itemize}
Unfortunately, the first question admits a negative answer; it is yet
another consequence of Turing's halting problem. For our purpose, this 
means that it is impossible to implement equality at the level of
lazy $p$-adic numbers. Inequalities however can always be detected. In
order to explain this, we need a lemma.

\begin{lem}
\label{lem:vallazy}
Let \ttx be a lazy $p$-adic number and set $x = \Value(\ttx)$.
For all integers $N$, we have:
$$\begin{array}{r@{\hspace{0.5ex}}ll}
\val_p(x) & = \val_p(\ttx(N)) & \text{if } \val_p(\ttx(N)) < N \smallskip \\
\val_p(x) & \geq N & \text{otherwise.}
\end{array}$$
\end{lem}

\begin{proof}
Recall that the norm is related to the valuation through the formula
$|a| = p^{-\val_p(a)}$ for all $a \in \Qp$. Moreover by definition,
we know that $|x - \ttx(N)| \leq p^{-N}$. If $\val_p
(\ttx(N)) \geq N$, we derive $|\ttx(N)| \leq p^{-N}$ and thus $|x|
\leq p^{-N}$ by the ultrametric inequality.
On the contrary if $\val_p (\ttx(N)) < N$, write $x = (x - \ttx(N))
+ \ttx(N)$ and observe that the two summands have different norms.
Hence $|x| = \min\big(|x - \ttx(N)|, |\ttx(N)|) = \ttx(N)$ and we
are done.
\end{proof}

\noindent
Lemma \ref{lem:vallazy} implies that $\Value(\ttx) = \Value(\tty)$ if and only if 
$\val_p(\ttx(N) - \tty(N)) \geq N$ for all integers $N$. Thus, assuming 
that we have access to a routine $\val_p$ computing the $p$-adic 
valuation of a rational number, we can write down the function 
\texttt{is\_equal} as follows (Python style\footnote{We emphasize that all the 
procedures we are going to write are written in pseudo-code in Python 
\emph{style} but certainly not in pure Python: they definitely do not
``compile'' in Python.}):

\begin{lstlisting}
    def is_equal(x,y):
        for N in (*$\N$*):
            v = (*$\val_p$*)(x(N) - y(N))
            if v < N: return False
        return True
\end{lstlisting}

\noindent
We observe that this function stops if and only if it answers False,
\emph{i.e.} inequalities are detectable but equalities are not.

\medskip

We now move to the second question and try to define standard operations 
at the level of lazy $p$-adic numbers. Let us start with addition. 
Given two computable $p$-adic numbers $x$ and $y$ represented by the 
programs \ttx and \tty respectively, it is actually easy to build a 
third program \texttt{x\_plus\_y} representing the sum $x+y$. Here it is:

\begin{lstlisting}
    def x_plus_y(N):
        return x(N) + y(N)
\end{lstlisting}

The case of multiplication is a bit more subtle because of precision.
Indeed going back to Eq.~\eqref{eq:precmul}, we see that, in order to
compute the product $xy$ at precision $O(p^N)$, we need to know $x$
and $y$ at precision $O(p^{N-\val(y)})$ and $O(p^{N-\val(x)})$
respectively. We thus need to compute first the valuation of $x$ and
$y$. Following Proposition~\ref{lem:vallazy}, we write down the
following procedure:

\begin{lstlisting}
    def val(x):
        for N in (*$\N$*):
            v = (*$\val_p$*)(x(N))
            if v < N: return v
        return Infinity
\end{lstlisting}

\noindent
However, we observe one more time that the function \texttt{val} can 
never stop; precisely it stops if and only if $x$ does not vanish.
For application to multiplication, it is nevertheless not an issue.
Recall that we needed to know $\val(x)$ because we wanted to compute
$y$ at precision $O(p^{N-\val(x)})$. But obviously computing $y$ at
higher precision will do the job as well. Hence it is in fact enough
to compute an integer $v_x$ with the guarantee that $\val(x) \geq
v_x$. By Lemma~\ref{lem:vallazy}, an acceptable value of $v_x$ is
$\min(0, \val_p(\ttx(0)))$ (or more generally $\min(i, \val_p(\ttx(i)))$ 
for any value of $i$).
A lazy $p$-adic number whose value if $xy$ is then given by the
following program:

\begin{lstlisting}
    def x_mul_y(N):
        vx = min(0, (*$\val_p$*)(x(0))
        vy = min(0, (*$\val_p$*)(y(0))
        return x(N-vy) * y(N-vx)
\end{lstlisting}

The case of division is similar except that we cannot do the same
trick for the denominator: looking at Eq.~\eqref{eq:precdiv}, we 
conclude that we do not need a lower bound on the valuation of
the denominator but an upper bound.
This is not that surprising since we easily imagine that the division
becomes more and more ``difficult'' when the denominator gets closer 
and closer to $0$. Taking the argument to its limit, we notice that
a program which is supposed to perform a division should be able, as 
a byproduct, to detect whether the divisor is zero or not. But we 
have already seen that the latter is impossible.
We are then reduced to code the division as follows:

\begin{lstlisting}
    def x_div_y(N):
        vx = min(0, (*$\val_p$*)(x(0))
        vy = val(y)  # This step might never stop
        return x(N + vy) / y(N + 2*vy - vx)
\end{lstlisting}

\noindent
keeping in mind that the routine never stops when the denominator 
vanishes.

\subsubsection{Relaxed arithmetic}

The algorithms we have sketched in \S\ref{sssec:lazynaive} are very 
naive and inefficient. In particular, they redo many times a lot of 
computations. For example, consider the procedure \texttt{x\_plus\_y} we 
have designed previously and observe that if we had already computed the 
sum $x + y$ at precision $O(p^N)$ and we now ask for the same sum $x + 
y$ at precision $O(p^{N+1})$, the computation restarts from scratch 
without taking advantage of the fact that only one digit of the result 
remains unknown.

One option for fixing this issue consists basically in implementing a 
``sparse cache'': we decide in advance to compute and cache the 
$\ttx(N)$'s only for a few values of $N$'s (typically the powers of $2$) 
and, on the input $N$, we output $\ttx(N')$ for a ``good'' $N' \geq N$. 
Concretely (a weak form of) this idea is implemented by maintaining two 
global variables \texttt{current\_precision[x]} and 
\texttt{current\_approximation[x]} (attached to each lazy $p$-adic 
number $x$) which are always related by the equation:
\begin{center}
\tt
current\_approximation[x] = x(current\_precision[x])
\end{center}
If we are asking for $\ttx(N)$ for some $N$ which is not greater than 
\texttt{current\_approximation[x]}, we output 
\texttt{current\_precision[x]} without launching a new computation. 
If, instead, $N$ is greater than \texttt{current\_approximation[x]}, we are 
obliged to redo the computation but we take the opportunity to double
the current precision (at least). Here is the corresponding code:

\begin{lstlisting}
    def x_with_sparse_cache(N):
        if current_precision[x] < N:
            current_precision[x] = max(N, 2*current_precision[x])
            current_approximation[x] = x(current_precision[x])
        return current_approximation[x]
\end{lstlisting}

\noindent
This solution is rather satisfying but it has nevertheless several 
serious disadvantages: it often outputs too large results\footnote{Of 
course it is always possible to reduce the result modulo $p^N$ but this 
extra operation has a non-negligible cost.} (because they are too 
accurate) and often does unnecessary computations.

Another option was developed more recently first by van der Hoeven in 
the context of formal power series~\cite{Ho97,Ho02,Ho07} and then by 
Berthomieu, Lebreton, Lecerf and van der Hoeven for $p$-adic numbers 
\cite{BeHoLe11,BeLe12,Le13}. It is the so-called \emph{relaxed 
arithmetic} that we are going to expose now. For simplicity we shall 
only cover the case of $p$-adic \emph{integers}, \emph{i.e.} $\Zp$. 
Extending the theory to $\Qp$ is more or less straightforward but needs 
an additional study of the precision which makes the exposition more 
technical without real benefit.

\paragraph{Data structure and addition}

The framework in which relaxed arithmetic takes place is a bit different 
and more sophisticated than the framework we have used until now. First 
of all, relaxed arithmetic does not view a $p$-adic number as a sequence 
of more and more accurate approximations but as the sequence of its 
digits. Moreover it needs a richer data structure in order to provide
enough facilities for designing its algorithms. For this reason, it is
preferable to encode $p$-adic numbers by classes which may handle its own
internal variables and its own methods.

We first introduce the class \texttt{RelaxedPAdicInteger} which is an 
abstract class for all relaxed $p$-adic integers (\emph{i.e.} relaxed 
$p$-adic integers must be all instances of a derived class of 
\texttt{RelaxedPAdicInteger}). The class \texttt{RelaxedPAdicInteger} is 
defined as follows:

\begin{lstlisting}
    class RelaxedPAdicInteger:
        # (*\color{comment}\textrm{Variable}*)
        digits           # (*\color{comment}\textrm{list of (already computed) digits}*)

        # (*\color{comment}\textrm{Virtual method}*)
        def next(self)   # (*\color{comment}\textrm{compute the next digit and append it to the list}*)
\end{lstlisting}

\noindent
We insist on the fact that the method \texttt{next} is an internal
(private) method which is not supposed to be called outside the class.
Of course, although it does not appear above, we assume that the class 
\texttt{RelaxedPAdicInteger} is endowed with several additional public methods 
making the interface user-friendly. For instance if \ttx is a relaxed 
$p$-adic integer inheriting from \texttt{RelaxedPAdicInteger}, we shall often
use the construction $\ttx[N]$ to access to the $N$-th digit of \ttx.
In pure Python, this functionality can be implemented by adding to the
class a method \texttt{\_\_getitem\_\_} written as follows:

\begin{lstlisting}
        def __getitem__(self,N):
            n = len(self.digits)   # (*\color{comment}\textrm{Index of the first uncomputed digit}*)
            while n < N+1:         # (*\color{comment}\textrm{We compute the digits until the position \ttN}*)
                self.next()
                n += 1
            return self.digits[N]  # (*\color{comment}\textrm{We return the \ttN-th digit}*)
\end{lstlisting}

In order to highlight the flexibility of the construction, let us 
explain as a warm-up how addition is implemented in this framework.
One actually just follows the schoolbook algorithm: the $n$-th digit
of a sum is obtained by adding the $n$-th digit of the summands plus
possibly a carry. We add an additional local variable to the class in
order to store the current carry and end up with the following
implementation:

\begin{lstlisting}
    class x_plus_y(RelaxedPAdicInteger):
        # (*\color{comment}\textrm{Additional variable}*)
        carry            # (*\color{comment}\textrm{variable for storing the current carry}*)

        def next(self):  # (*\color{comment}\textrm{compute the next digit and append it to the list}*)
            n = len(self.digits)         # (*\color{comment}\textrm{index of the first uncomputed digit}*)
            s = x[n] + y[n] + self.carry # (*\color{comment}\textrm{perform the addition}*)
            self.digits[n] = s % p       # (*\color{comment}\textrm{compute and store the new digit}*)
            self.carry = s // p          # (*\color{comment}\textrm{update the carry}*)
\end{lstlisting}

\noindent
That's all. Although this does not appear in the pseudo-code above, the 
relaxed $p$-adic integers \ttx and \tty (which are supposed to be 
instances of the class \texttt{RelaxedPAdicInteger}) should be passed as 
attributes to the constructor of the class.

The subtraction is performed similarly (exercise left to the reader).

\paragraph{Multiplication}

\emph{A priori}, multiplication can be treated in a similar fashion. 
Given $x$ and $y$ two $p$-adic integers with digits $x_i$'s and $y_j$'s, 
the $n$-th digit of a product $xy$ is obtained by adding the 
contribution of all cross products $x_i y_j$ for $i+j=n$ plus a possible 
carry. This approach leads to the following straightforward 
implementation:

\begin{lstlisting}
    class x_mul_y(RelaxedPAdicInteger):
        # (*\color{comment}\textrm{Additional variable}*)
        carry            # (*\color{comment}\textrm{variable for storing the current carry}*)

        def next(self):  # (*\color{comment}\textrm{compute the next digit and append it to the list}*)
            n = len(self.digits)
            s = self.carry
            for i in 0,1,...,n:
                s += x[i] * y[n-i]
            self.digits[n] = s % p
            self.carry = s // p
\end{lstlisting}

\noindent
Nevertheless this strategy has one important defect: it does not take 
advantage of the fast multiplication algorithms for integers, that is, it 
computes the first $N$ digits of $xy$ in $O(N^2 \log p)$ 
bit operations while we would have expected only $\softO(N \log p)$.

\medskip

The relaxed multiplication algorithm from~\cite{BeHoLe11} fixes this drawback. The 
rough idea behind it is to gather digits in $x$ and $y$ as follows:
\begin{align*}
x & = x_0 + (x_1 + x_2 p) p + (x_3 + x_4 p + x_5 p^2 + x_6 p^3) p^3 + \cdots \\
y & = y_0 + (y_1 + y_2 p) p + (y_3 + y_4 p + y_5 p^2 + y_6 p^3) p^3 + \cdots
\end{align*}
and to implement a kind of block multiplication. More precisely, we 
materialize the situation by a grid drawn in the half plane (see 
Figure~\ref{fig:relaxedmul}).
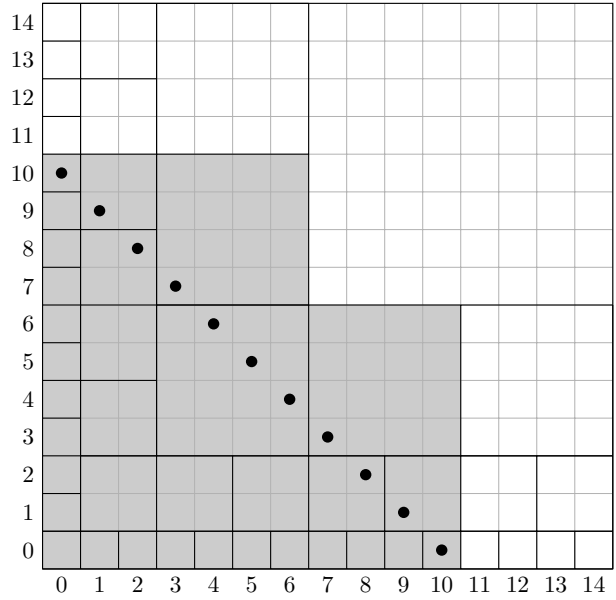
\begin{figure}
\hfill
\begin{tikzpicture}[scale=0.5]
\fill[black!20] (0,0)--(11,0)--(11,7)--(7,7)--(7,11)--(0,11)--cycle;
\draw[very thin,black!30] (0,0) grid (15,15);
\draw (0,0) rectangle (15,15);
\draw (0,0) grid (15,1);
\draw (0,1) grid (1,15);
\draw[xshift=1cm,yshift=1cm,step=2] (0,0) grid (14,2);
\draw[xshift=1cm,yshift=1cm,step=2] (0,2) grid (2,14);
\draw[xshift=3cm,yshift=3cm,step=4] (0,0) grid (12,4);
\draw[xshift=3cm,yshift=3cm,step=4] (0,4) grid (4,12);
\fill (10.5,0.5) circle (1.5mm);
\fill (9.5,1.5) circle (1.5mm);
\fill (8.5,2.5) circle (1.5mm);
\fill (7.5,3.5) circle (1.5mm);
\fill (6.5,4.5) circle (1.5mm);
\fill (5.5,5.5) circle (1.5mm);
\fill (4.5,6.5) circle (1.5mm);
\fill (3.5,7.5) circle (1.5mm);
\fill (2.5,8.5) circle (1.5mm);
\fill (1.5,9.5) circle (1.5mm);
\fill (0.5,10.5) circle (1.5mm);

\node[below,scale=0.8] at (0.5,0) { $0$ };
\node[below,scale=0.8] at (1.5,0) { $1$ };
\node[below,scale=0.8] at (2.5,0) { $2$ };
\node[below,scale=0.8] at (3.5,0) { $3$ };
\node[below,scale=0.8] at (4.5,0) { $4$ };
\node[below,scale=0.8] at (5.5,0) { $5$ };
\node[below,scale=0.8] at (6.5,0) { $6$ };
\node[below,scale=0.8] at (7.5,0) { $7$ };
\node[below,scale=0.8] at (8.5,0) { $8$ };
\node[below,scale=0.8] at (9.5,0) { $9$ };
\node[below,scale=0.8] at (10.5,0) { $10$ };
\node[below,scale=0.8] at (11.5,0) { $11$ };
\node[below,scale=0.8] at (12.5,0) { $12$ };
\node[below,scale=0.8] at (13.5,0) { $13$ };
\node[below,scale=0.8] at (14.5,0) { $14$ };
\node[left,scale=0.8] at (0,0.5) { $0$ };
\node[left,scale=0.8] at (0,1.5) { $1$ };
\node[left,scale=0.8] at (0,2.5) { $2$ };
\node[left,scale=0.8] at (0,3.5) { $3$ };
\node[left,scale=0.8] at (0,4.5) { $4$ };
\node[left,scale=0.8] at (0,5.5) { $5$ };
\node[left,scale=0.8] at (0,6.5) { $6$ };
\node[left,scale=0.8] at (0,7.5) { $7$ };
\node[left,scale=0.8] at (0,8.5) { $8$ };
\node[left,scale=0.8] at (0,9.5) { $9$ };
\node[left,scale=0.8] at (0,10.5) { $10$ };
\node[left,scale=0.8] at (0,11.5) { $11$ };
\node[left,scale=0.8] at (0,12.5) { $12$ };
\node[left,scale=0.8] at (0,13.5) { $13$ };
\node[left,scale=0.8] at (0,14.5) { $14$ };
\end{tikzpicture}
\hfill\null

\caption{Relaxed multiplication scheme}
\label{fig:relaxedmul}
\end{figure}
The coordinate on the $x$-axis (resp. the $y$-axis) corresponds to the 
integers $i$'s (resp. $j$'s) that serve as indices for the positions of 
the digits of the $p$-adic integer $x$ (resp. $y$). As for the cell 
$(i,j)$, it represents the product $x_i y_j$.
The computation of the $n$-th digit of $xy$ needs to know all the digits 
appearing on the anti-diagonal $i+j = n$; we refer again to
Figure~\ref{fig:relaxedmul} where this anti-diagonal is displayed for $n=10$.

We now pave the grid using larger and larger squares as follows. We pave 
the first row\footnote{It is of course the row corresponding to $j=1$ 
which is drawn on the bottom in Figure~\ref{fig:relaxedmul}.} and the 
first column by squares of size $1$. Once this has been done, we forget
temporarily the cells that have been already paved (\emph{i.e.} the 
cells lying on the first row or the first column), we 
group the two next lines and the two next columns together and we pave 
them with squares of size $2$. We continue this process and double the
size of the squares at each iteration. The obtained result is displayed
on Figure~\ref{fig:relaxedmul}.
This construction exhibits the following remarkable property:

\begin{lem}
\label{lem:paving}
If $(i,j)$ is a cell located on a square of size $2^\ell$ of the paving, 
then $2^\ell < i$ and $2^\ell < j$.
\end{lem}

\begin{proof}
The squares of size $2^\ell$ are all located by construction on the 
region where both coordinates $i$ and $j$ are greater than or equal to 
$1 + 2 + 2^2 + \cdots + 2^{\ell-1} = 2^\ell - 1$. The Lemma follows from 
this observation.
\end{proof}

\noindent
As an example, look at the cell $(7,3)$ on Figure~\ref{fig:relaxedmul}; 
it is the bottom left corner of a square of size $4$.
Let us denote by
$S_{i,j}$ the square of the paving with bottom left corner located at
position $(i,j)$. For such any $(i,j)$, we define:
$$C_{i,j}(t) = \big(x_i + t x_{i+1} + \cdots + t^\ell x_{i+2^\ell-1}\big)
\times \big(y_j + t y_{j+1} + \cdots + t^\ell y_{j+2^\ell-1}\big) \in
\Z[t]$$
where $2^\ell$ is the size of $S_{i,j}$ and $t$ is a new formal
variable.
We are now ready to explain the concrete idea behind the relaxed 
multiplication algorithm. In the \texttt{next} procedure, when we 
encounter a pair $(i,j=n{-}i)$ we distinguish between two cases:
\begin{itemize}
\renewcommand{\itemsep}{0pt}
\item if $(i,j)$ is the position of a cell located at the bottom left
corner of square of the paving, we add the contribution of $C_{i,j}(p)$
(instead of adding the sole contribution of $x_i y_j$),
\item otherwise, we do nothing (since the contribution of $x_i y_j$ 
has already been taken into account previously).
\end{itemize}
Here are several several important remarks. First, by 
Lemma~\ref{lem:paving}, the computation of the $C_{i,j}(p)$'s in the 
first step only requires the digits of $x$ and $y$ up to the position 
$n$. Second, notice that listing all the pairs $(i,j)$ fitting into 
the first case is easy. Indeed, notice that $(i,j)$ is the bottom left 
corner of a paving square of size $2^\ell$ if and only if $i$ and $j$ 
are both congruent to $2^\ell - 1$ modulo $2^\ell$ and one of them is 
actually equal to $2^\ell - 1$. Hence, for a given nonnegative integer 
$\ell$, there exist at most two such cells $(i,j)$ for which $i+j = n$ 
and they are moreover given by explicit formulas.

At that point, one may wonder why we have introduced the polynomials 
$C_{i,j}(t)$'s instead of working only with the values $C_{i,j}(p)$'s. 
The reason is technical: this modification is needed to get a good 
control on the size of the carries\footnote{Another option, used in 
\cite{BeHoLe11}, would be to build an algorithm for performing 
operations on integers written in base $p$. This can be achieved for 
instance using Kronecker substitution~\cite[Corollary~8.27]{GaGe03}.}. 
We will elaborate on this later (see Remark \ref{rem:carry} below).
By introducing the construction $\ttx[i, \ldots, j]$ for giving access
to the polynomial
$$x_i + x_{i+1} t + x_{i+2} t^2 + \cdots + x_j t^{j-i} \in \Z[t]$$
the above ideas translate into the concrete algorithm of 
Figure~\ref{fig:relaxedmulimpl}.
\begin{figure}
\begin{lstlisting}
    class x_mul_y(RelaxedPAdicInteger):
        # (*\rm\color{comment}Additional variable*)
        carry            # (*\rm\color{comment}variable in $\Z[t]$ for storing the current carry*)

        def next(self):  # (*\rm\color{comment}compute the next digit and append it to the list*)
            n = len(self.digits)
            m = n + 2; (*\rm$\ell$*) = 0; s = 0
            while m > 1:
                # (*\rm\color{comment}The contribution of the first square of size $2^\ell$*)
                s += x[2^(*\rm$\ell$*) - 1, ..., 2^((*\rm$\ell$*)+1) - 2]
                   * y[(m-1)*2^(*\rm$\ell$*) - 1, ..., m*2^(*\rm$\ell$*) - 2]
                # (*\rm\color{comment}The contribution of the second square*)
                if m > 2:   # (*\rm\color{comment}Case where the two squares are indeed not the same*)
                    s += y[2^(*\rm$\ell$*) - 1, ..., 2^((*\rm$\ell$*)+1) - 2]
                       * x[(m-1)*2^(*\rm$\ell$*) - 1, ..., m*2^(*\rm$\ell$*) - 2]
                if m (*\rm is odd*): break
                m = m // 2
                (*\rm$\ell$*) += 1
            s += self.carry
            self.digits[n] = s(0) % p
            self.carry = (s(0) // p) + (s // t)
\end{lstlisting}

\caption{An implementation of the relaxed multiplication}
\label{fig:relaxedmulimpl}
\end{figure}

We now briefly sketch its complexity analysis.
The first observation is that the while loop of the \texttt{next} method 
of Figure \ref{fig:relaxedmulimpl} is repeated until $\ell$ reaches the 
$2$-adic valuation of $n+2$. Writing $v = \val_2(n+2)$, the total 
complexity for the execution of the while loop therefore stays within: 
$$\sum_{\ell=0}^v \softO(2^\ell \log p) = \softO(2^v \log p)$$
if we use fast algorithms for polynomial multiplication~\cite[\S 8]{GaGe03}.
Moreover at the end of the while loop, the degree of $s$ remains 
bounded by $2^v$ and its coefficients have all at most $O(\log(np))$
binary digits.
The second step, which is more technical, consists in bounding the 
size of the carry. The key point is to notice that after the computation
of the $(n{-}1)$-st digit, the carry takes the form:
$$c + (c_0 + c_1 t + c_2 t^2 + \cdots + c_n t^n)$$
where the $c_k$'s come from the contribution of the cells $(i',j')$ with 
$i' + j' \geq n$ that have been already discovered (\emph{i.e.} belonging 
to a paving square $S_{i,j}$ for which $C_{i,j}(t)$ has already been
computed) and $c$ comes from the carries appearing while computing the
$t^{n-1}$-term of the product
$(x_0 + x_1 t + x_2 t^2 + \cdots + x_n t^n) \times
(y_0 + y_1 t + y_2 t^2 + \cdots + y_n t^n)$.
Once this has been noticed, we deduce that the $c_k$'s are
all bounded from above by $n\cdot (p{-}1)^2$ while:
$$c \leq \frac 1 {p^n} \sum_{i'+j' \leq n} x_{i'} y_{j'}
\leq np\cdot (p{-}1).$$
Hence the carry itself is a polynomial of degree at most $n$ and each
of its coefficients has at most $O(\log(pn))$ binary digits.

\begin{rem}
\label{rem:carry}
While the growth of the coefficients of the carry remains under 
control, its value at $p$ may --- and does --- grow exponentially fast 
with respect to $n$. That is the reason why we had to switch to 
polynomials.
\end{rem}

Noticing finally that the division by $t$ on polynomials is free (it is 
just a shift of the coefficients), we find that the last three lines 
of the \texttt{next} method displayed on Figure
\ref{fig:relaxedmulimpl} have a cost of $\softO(2^v \log(np))$ bit
operations. The total complexity of the computation of the $n$-th
digit then stays within $\softO(2^v \log(np))$ bit operations as well.

Summing up all these contributions up to $N$, we find that the relaxed 
multiplication algorithm computes the first $N$ digits of a product in 
$\softO(N \log p)$ bit operations.
We refer to~\cite{BeHoLe11} for a more precise complexity analysis (of a 
slightly different version of the algorithm we have presented here).

\paragraph{Computation of fixed points}

A quite interesting application of the relaxed approach is the 
possibility to define and compute very easily fixed points of 
contraction mappings defined over $\Zp$ (or more generally $\Zp^d$ for 
some $d$). Before going into this, we need to define the notion of
\emph{function} over relaxed $p$-adic numbers. 
Recall that a relaxed $p$-adic integer is an instance of 
\texttt{RelaxedPAdicInteger}. A function over relaxed $p$-adic numbers 
can then be defined as a class deriving from 
\texttt{RelaxedPAdicInteger} endowed with a constructor accepting \ttx 
as parameter and constructing $F(\ttx)$. Actually we have already seen
such examples: the classes \texttt{x\_plus\_y} and \texttt{x\_mul\_y}
we have designed before fit exactly in this framework.
Below is yet another easy example modeling the function $s : x \mapsto 
1 + px$.

\begin{lstlisting}
    class S(RelaxedPAdicInteger):
        def next(self):
            n = len(self.digits)
            if n == 0: self.digits[0] = 1
            else:      self.digits[n] = x[n-1]
\end{lstlisting}

\begin{deftn}
Let $F$ be a function defined over relaxed $p$-adic integers.
We say that $F$ is a \emph{contraction} if its \texttt{next} method 
computes the $n$-th digit by making only use of the first $n{-}1$ digits 
of \ttx.
\end{deftn}

\noindent
Concretely $F$ is a contraction if its \texttt{next} method never
calls $\ttx[i]$ for some $i \geq n$ when $n$ denotes (as usual) the 
position of the digit the \texttt{next} method is currently computing.
For example, the function \texttt{S} introduced above is a contraction.

\begin{rem}
If $F$ is a contraction modeling a function $f : \Zp \to \Zp$,
the first $n$ digits of $f(x)$ depends only on the first $n{-}1$ digits
of $x$. This implies that $f$ is a $\frac 1 p$-contraction in the 
usual sense:
$$\textstyle 
|f(x) - f(y)| \leq \frac 1 p \cdot |x-y|.$$
Note that the converse is not necessarily true: it may happen that $f$
is indeed a contraction but is modeled by the function $F$ which misses
this property.
\end{rem}

If $F$ is a contraction modeling a function $f : \Zp \to \Zp$, the 
Banach fixed point Theorem implies that $f$ has a unique fixed point in 
$\Zp$. Moreover this sequence is the limit of any sequence 
$(x_n)_{n\geq0}$ solution to the recurrence $x_{n+1} = f(x_n)$. This 
property translates immediately in the world of programs: it says that 
we can compute the fixed point of $f$ just by replacing each occurrence 
of \ttx by the current instance of the class (\texttt{self}) in the 
\texttt{next} function of $F$. 
As an example, applying this treatment to the function $S$, we get:

\begin{lstlisting}
    class SFixed(RelaxedPAdicInteger):
        def next(self):
            n = len(self.digits)
            if n == 0: self.digits[0] = 1
            else:      self.digits[n] = self[n-1]
\end{lstlisting}

\noindent
which is\footnote{Or more precisely: ``whose any instance is''} a 
relaxed $p$-adic number representing the fixed point of $s$, namely
$\frac 1{1-p}$.

\paragraph{Division}

The above toy example has a real interest because it can be generalized 
to handle arbitrary divisions $a \div b$ where $b$ is invertible in $\Zp$.
We assume first that $b$ is congruent to $1$ modulo $p$, \emph{i.e} $b = 
1 + pb'$ for some $b' \in \Zp$. The assumption means that the $0$-th 
digit of $b$ is 1; as for the digits of $b'$, they are those of $b$ 
shifted by one position. We introduce the affine function $f : \Zp \to 
\Zp$ defined by:
$$f(x) = a + (1-b) x = a - p b' x.$$
It is a contraction which can be modeled (using relaxed multiplication 
and relaxed subtraction) by a contraction $F$ defined over relaxed
$p$-adic integers. Applying the recipe discussed above, we then build
a new relaxed $p$-adic integer that models the unique fixed point of
$x$, which is $\frac a {1+pb'} = \frac a b$. We have therefore computed
the division of $a$ by $b$ in the relaxed world.

For a general $b$, one can proceed as follows. Let $b_0$ be the $0$-th 
digit of $b$. We compute an integer $c$ such that $b_0 c \equiv 1 \pmod 
p$ and view $c$ as a $p$-adic number. Since $b \equiv b_0 \pmod p$, we 
get $bc \equiv 1 \pmod p$ as well. Moreover the fraction $\frac a b$ 
rewrites $\frac a b = \frac{ac}{bc}$ and we can now apply the method 
above. Using this techniques, a general relaxed division reduces to one
inversion modulo $p$ and two relaxed multiplication for the preparation 
plus one more relaxed multiplication and one relaxed subtraction for 
the computation of the fixed point.

\subsection{Floating-point arithmetic}
\label{ssec:pfloat}

The most common representation of real numbers on computers is
by far floating-point arithmetic which is normalized by the
IEEE 754-2008 Standard~\cite{IEEE08,Mu09}.
Let us recall very briefly that the rough idea behind the floating-point 
arithmetic is to select a finite fixed subset $\calR$ of $\R$ 
(the so-called set of \emph{floating-point numbers}) on which we
define the operators $\oplus$ and $\otimes$ which are supposed to mimic
the usual addition and multiplication of reals.

In the $p$-adic world, this approach has an analogue which was first 
proposed (as far as we know) in a unpublished note by Kedlaya and Roe 
in 2008~\cite{KeRo08}. Unfortunately it seems that it has never been implemented 
yet\footnote{Although an implementation was very recently proposed by 
Roe for integration in \sagenoref; see 
\url{https://trac.sagemath.org/ticket/20348}.}. The aim of the 
subsection is to report on Kedlaya and Roe's proposal.

\subsubsection{Definition of $p$-adic floating-point numbers}

The construction of $p$-adic floating-point numbers depends on the 
initial choice of three positive integers $N$, $e_\min$, $e_\max$ that 
we fix now until the end of \S \ref{ssec:pfloat}. We assume $e_\min < 
e_\max$ and $N \geq 1$.
The integer $N$ is called the \emph{precision} of the system while 
$e_\min$ and $e_\max$ are respectively the minimal and the maximal 
\emph{exponent} of the system.

\begin{deftn}
A \emph{$p$-adic floating-point number} is either 
\begin{itemize}
\renewcommand{\itemsep}{0pt}
\item a $p$-adic number $x \in \Qp$ of the shape $p^e s$ with the 
conditions:
\begin{equation}
\label{eq:condpFP}
\textstyle
e_\min \leq e \leq e_\max \quad ; \quad
-\frac{p^N-1} 2 < s \leq \frac{p^N-1} 2 \quad ; \quad
\gcd(s,p) = 1
\end{equation}
\item the number $0 \in \Qp$,
\item the special symbol $\infty$,
\item the special symbol $\nan$ (not a number).
\end{itemize}
\end{deftn}

\noindent
Note that the conditions~\eqref{eq:condpFP} imply that a $p$-adic 
floating-point number $x \in \Qp$ of the shape $p^e s$ cannot be zero
and moreover that $e$ and $s$ are uniquely determined by $x$. Indeed, 
since $s$ is prime with $p$, it cannot vanish and $e$ has to be equal
to the $p$-adic valuation of $x$. It is thus determined by $x$ and,
consequently, so is $s = p^{-e} \: x$.

\begin{deftn}
Given a $p$-adic floating-point number $x$ of the shape $x = p^e s$,
\begin{itemize}
\renewcommand{\itemsep}{0pt}
\item the integer $e$ is called the \emph{exponent} of $x$, 
\item the integer $s$ is called the \emph{significand} (or sometimes
the \emph{mentissa}) of $x$.
\end{itemize}
\end{deftn}

\paragraph{Representation of $p$-adic floating-point numbers.}

We represent $p$-adic floating-point numbers on machine by pairs
of integers as follows:
\begin{itemize}
\renewcommand{\itemsep}{0pt}
\item the $p$-adic number $p^e s$ is represented by $(e,s)$,
\item the $p$-adic number $0$ is represented by $(e_\max,0)$,
\item the symbol $\infty$ is represented by $(e_\min{-}1, 1)$,
\item the symbol $\nan$ is represented by $(e_\min{-}1, 0)$.
\end{itemize}

\begin{rem}
The above definitions omit several components of the IEEE 
standard. First, they do not include subnormal numbers. These are 
meant to allow for gradual underflow to zero, which is reasonable to 
provide when working in base $2$ with relatively small $e$. However, for 
$e$ as large as the bit length of a modern machine word (\emph{i.e.} no 
less than 32), this benefit is sufficiently minor that it does not 
outweigh the added complexity needed to handle subnormal numbers 
properly.
Second, the above definitions do not include signed infinities, because $p$-adic numbers 
do not have a meaningful notion of sign. For similar reasons, the comparison
operators $<, \leq, >, \geq$ are not defined for $p$-adic numbers.
Third, they do not provide for multiple types of \nan's. For instance, in 
the IEEE specification, a distinction is made between signaling \nan's 
(those which raise an exception upon any operation) and quiet \nan's.
\end{rem}

For some computations, it may be useful to allow some $p$-adic floating-point 
numbers to be represented in more than one way. We thus define an 
operation called \emph{normalization} on the set of pairs of integers 
$(e,s)$ as follows:
\begin{itemize}
\renewcommand{\itemsep}{0pt}
\item if $e < e_\min$ and $s=0$, return $(e_\min{-}1,0)$ (which represents $\nan$);
\item if $e < e_\min$ and $s \neq 0$, return $(e_\min{-}1,1)$ (which represents $\infty$);
\item if $e > e_\max$, return $(e_\max,0) = 0$ (which represents $0$);
\item if $e_\min \leq e \leq e_\max$ and $s=0$, return $(e_\max,0) = 0$;
\item if $e_\min \leq e \leq e_\max$ and $s \neq 0$, let $k \in 
\{0,\dots,m-1\}$ be the largest integer such that $p^k$ divides $s$ and 
return the normalization of $(e+k, s')$ where $s'$ is the unique integer
such that $-\frac{p^N-1} 2< s' \leq \frac{p^N-1} 2$ and
$s' \equiv p^{-k} m \pmod{p^N}$.
\end{itemize}

\noindent
For many operations, it is safe to leave an unnormalized pair $(e,s)$ 
as long as $s \neq 0$ and $e_\min \leq e \leq e_\max - \val_p(s)$ since
$p^e s$ is then guaranteed to equal the value of the normalization of 
$(e,s)$. However, one should normalize before testing for equality.

\subsubsection{Operations on $p$-adic floats}

Let $\QpFP$ be the set of $p$-adic floating-point numbers. Define also 
$\QpNan = \Qp \sqcup \{\infty, \nan\}$. Clearly $\QpFP \subset \QpNan$.

\begin{lem}
\label{lem:roundingFP}
Given $x \in \Qp$ with $e_\min \leq \val(x) \leq e_\max$, there
exists a unique element $y \in \Qp \cap \QpFP$ such that
$|x-y| \leq |x| \cdot p^{-N}$.
\end{lem}

\begin{proof}
We write $x = p^v \cdot x_0$ where $v = \val(x)$ and $x_0$ is
invertible in $\Zp$. Set $e = v$ and let $s$ be the unique integer
of the range $\big(\!-\frac{p^N-1} 2,  \frac{p^N-1} 2\big]$ which is
congruent to $x_0$ modulo $p^N$. Define $y = p^e s$; it is a $p$-adic
floating-point number by the assumptions of the Lemma. Moreover:
$$|x - y| = p^{-v} \cdot |x_0 - s| \leq p^{-v-N} = |x| \cdot p^{-N}.$$
The existence is proved. 

As for uniqueness, remark first that $y = 0$ cannot do the job. Assume now 
that $y' = p^{e'} s'$ is a $p$-adic floating-point number written in 
normalized form which satisfies the required inequality. First notice 
that $e' = v$ necessarily. Indeed, from $e' > v$, we would deduce $|x - 
y'| = |x| > |x| \cdot p^{-N}$. Similarly from $e_i < v$, we 
would get $|x - y'| = p^{-e_1} > |x|$.
Consequently $|x-y'| = |x| \cdot |x_0 - s'|$ and our assumption on $y'$
translates to the congruence $x_0 \equiv s' \pmod {p^N}$. Therefore we
derive $s \equiv s' \pmod{p^N}$ and then $s = s'$ since both $s$ and $s'$
have to lie in the interval $\big(\!-\frac{p^N-1} 2,  \frac{p^N-1} 2\big]$.
Hence $y = y'$ and we are done.
\end{proof}

\begin{deftn}
The \emph{rounding function} is the function 
$o : \QpNan \to \QpFP$ defined as follows:
\begin{itemize}
\renewcommand{\itemsep}{0pt}
\item if $x \in \{\infty, \nan\}$, then $o(x) = x$;
\item if $x \in \Qp$ and $\val(x) < e_\min$,
then $o(x) = \infty$ (\emph{overflow});
\item if $x \in \Qp$ and $\val(x) > e_\max$,
then $o(x) = 0$ (\emph{underflow});
\item if $x \in \Qp$ and $e_\min \leq \val(x) \leq e_\max$,
then $o(x)$ is the unique element of $\QpFP$ such that
$|x - o(x)| \leq |x| \cdot p^{-N}$.
\end{itemize}
\end{deftn}

\begin{rem}
We emphasize that overflow (resp. underflow) occurs when $e$ is 
small (resp. large) which is the exact opposite of the real case.
This is \emph{not} a typo; the reason behind that is just that, in
the $p$-adic world, $p^N$ goes to $0$ when $N$ goes to $+\infty$.
\end{rem}

\begin{rem}
Note that the uniqueness in the last case means that there is no 
analogue of the notion of a rounding mode in real arithmetic.
\end{rem}

Given any $k$-ary function $f : (\QpNan)^k \to \QpNan$, we define its 
\emph{compliant approximation} $f_\FP : (\QpFP)^k \to \QpFP$ by:
$$f_\FP(x_1, \ldots, x_k) = o \big(f(x_1, \ldots, x_k)\big).$$
We are now in position to specify the four basic arithmetic operations 
(addition, subtraction, multiplication, division) on $\QpFP$. In order 
to do so, we need first to extend them on $\QpNan$. Let us agree to set
(here $\div$ stands for the division operator):

$$\begin{array}{ll}
\forall x \in \QpNan, 
  & x \pm \nan = \nan \pm x = \nan \smallskip \\
\forall x \in \QpNan \backslash \{\nan\}, 
  & x \pm \infty = \infty \pm x = \infty \medskip \\

\forall x \in \QpNan,
  & x \times \nan = \nan \times x = \nan \smallskip \\
\forall x \in \QpNan \backslash \{0, \nan\}, 
  & x \times \infty = \infty \times x = \infty \smallskip \\
  & 0 \times \infty = \infty \times 0 = \nan \medskip \\

\forall x \in \QpNan,
  & x \div \nan = \nan \div x = \nan \smallskip \\
\forall x \in \QpNan \backslash \{0, \nan\}, 
  & x \div 0 = \infty, \quad 0 \div 0 = \nan \smallskip \\
\forall x \in \QpNan \backslash \{\infty, \nan\}, 
  & x \div \infty = 0, \quad \infty \div \infty = \nan \smallskip \\
\end{array}$$
Addition, subtraction, multiplication and division 
on floating-point $p$-adic numbers are, by definition, the operations
$+_\FP$, $-_\FP$, $\times_\FP$ and $\div_\FP$ respectively.

\smallskip

Let us examine concretely how the computation goes at the level of
$p$-adic floating-point numbers. We start with multiplication with
is the easiest operator to handle. The formula we derive immediately
from the definition is:
$$\begin{array}{r@{\hspace{0.5ex}}ll}
(p^{e_1} s_1) \: \times_\FP \: (p^{e_2} s_2)
  & = p^{e_1 + e_2} \cdot (s_1 s_2 \text{ mod } p^N)
  & \text{if } e_1 + e_2 \leq e_\max \smallskip \\
  & = 0
  & \text{otherwise}
\end{array}$$
where $a \text{ mod } p^N$ is the unique representation of $a$ modulo
$p^N$ lying in the interval $\big(\!-\frac{p^N-1} 2,  \frac{p^N-1} 2\big]$.
We emphasize moreover that if $(e_1, s_1)$ and $(e_2, s_2)$ are both
written in normalized form then so is $(e_1 + e_2, s_1 s_2 \text{ mod } p^N)$ (under the 
assumption $e_1 + e_2 \leq e_\max$). Indeed the product of two numbers
which are coprime to $p$ is coprime to $p$ as well.
The case of the division is similar:
$$\begin{array}{r@{\hspace{0.5ex}}ll}
(p^{e_1} s_1) \: \div_\FP \: (p^{e_2} s_2)
  & = p^{e_1 - e_2} \cdot (s_1 s_2^{-1} \text{ mod } p^N)
  & \text{if } e_1 - e_2 \geq e_\min \smallskip \\
  & = \infty
  & \text{otherwise.}
\end{array}$$
The cases of addition and subtraction are more subtle because they 
can introduce cancellations of the rightmost digits. When $e_1
\neq e_2$, this cannot happen and the formula writes as follows:
$$\begin{array}{r@{\hspace{0.5ex}}ll}
(p^{e_1} s_1) \: +_\FP \: (p^{e_2} s_2)
  & = p^{e_1} \cdot (s_1 + p^{e_2 - e_1} s_2 \text{ mod } p^N)
  & \text{if } e_1 < e_2 \smallskip \\
  & = p^{e_2} \cdot (p^{e_1-e_2} s_1 + s_2 \text{ mod } p^N)
  & \text{if } e_2 < e_1.
\end{array}$$
On the contrary, when $e_1 = e_2$, the naive formula
$(p^e s_1) \: +_\FP \: (p^e s_2)
  = p^e \cdot (s_1 + s_2 \text{ mod } p^N)$
might fail because $s_1 + s_2$ can have a huge valuation. Instead 
we introduce $v = \val_p(s_1 + s_2)$ and the correct formula writes:
$$\begin{array}{r@{\hspace{0.5ex}}ll}
(p^e s_1) \: +_\FP \: (p^e s_2)
  & = p^{e+v} \cdot \big(\frac{s_1 + s_2}{p^v} \text{ mod } p^N\big)
  & \text{if } v \leq e_\max - e \smallskip \\
  & = 0
  & \text{otherwise.}
\end{array}$$
As an alternative, we may prefer using non-normalized representations
of $p$-adic floating-point numbers; it is indeed an option but it
requires to be \emph{very} careful with underflows.

\subsection{Comparison between paradigms}
\label{ssec:compparadigm}

We have just presented three quite different ways of thinking of
$p$-adic numbers from the computational point of view. Of course, 
each of them has its own advantages and disadvantages regarding
flexibility, rapidity, accuracy, \emph{etc.}
In this subsection we compare them, trying as much as possible to
support our arguments with many examples.

\subsubsection{Zealous arithmetic \textsc{vs.} lazy/relaxed arithmetic}
\label{sssec:zealousVSrelaxed}

The main difference between zealous and lazy/relaxed arithmetic is
certainly that they are designed for different uses! In the zealous
perspective, the inputs are given with a certain precision which is
supposed to be fixed once for all: one can imagine for instance that
the inputs come from measures\footnote{Through I have to confess that
I have never seen a relevant physical $p$-adic measure yet.} or that
computing them requires a lot of effort and is then not an option.
So the precision on the inputs is immutable and the goal of zealous
arithmetic is to propagate the precision to the output (trying to
be as sharp as possible).

On the other hand, the point of view of lazy/relaxed arithmetic is the 
exact opposite: instead of fixing the precision on the inputs, we fix a 
\emph{target} precision and we propagate it back to the inputs. Indeed 
suppose, as an easy example, that we have to evaluate the product $xy$. In 
the lazy/relaxed world, this is done by creating and returning the 
program \texttt{x\_mul\_y} \emph{without running it}.
The execution (\emph{i.e.} the actual computation of the product) is 
postponed until somebody asks for the value of $xy$ at a certain 
precision $O(p^N)$ (through the call $\texttt{x\_mul\_y}[N]$). At that 
moment, the program runs and infers the needed precision on the inputs 
$x$ and $y$.

Except maybe in very particular situations, the point of view of 
lazy/relaxed arithmetic is certainly much closer to the needs of the 
mathematician. Indeed, it is quite rare to have to work with $p$-adic
numbers that do not come from a previous computation. Mathematicians
are more relaxed than zealous!

Apart from that, still regarding precision, the zealous and the 
lazy/relaxed approaches are both based on the arithmetic of intervals 
and, for this reason, they exhibit similar behaviors in the following 
sense. Assume that we are given a certain function $f : \Zp \to \Zp$ to 
be evaluated and define the two numbers $\ell_1$ and $\ell_2$ (depending 
\emph{a priori} on extra parameters) as follows:
\begin{itemize}
\renewcommand{\itemsep}{0pt}
\item when we ask for the value of $f(a + O(p^N))$, the zealous 
approach answers $b + O(p^{N-\ell_1})$ (loss of $\ell_1$ digits of 
precision),
\item when we ask for the value of $f(a)$ at precision $O(p^N)$, 
the lazy/relaxed technique requires the computation of $a$ at precision 
$O(p^{N+\ell_2})$.
\end{itemize}
It turns out that $\ell_1$ and $\ell_2$ do not very much depend 
neither on $a$, nor on $N$. Moreover they are almost 
equal\footnote{We refer to \S \ref{sssec:preclemmaatwork} for a 
clarification of this statement based on a theoretical study (which
is itself based on the theory of $p$-adic precision we shall develop
in \S \ref{ssec:foundpadicprec}).}. On the 
other hand, we underline that this value $\ell_1 \approx \ell_2$ is very 
dependent on the function $f$ and even on the algorithm we use for the 
evaluation. We will present and discuss this last remark in \S 
\ref{sssec:compexamples} below.

\medskip

The lazy/relaxed arithmetic has nevertheless several more or less 
annoying disadvantages. First of all, it is
certainly more difficult to implement efficiently (although the 
implementation in \mathemagix is very well polished and quite 
competitive; we refer to~\cite{BeHoLe11} for timings). It is also 
supposed to be a bit slower; for instance, the relaxed multiplication 
looses asymptotically a factor $\log N$ (for the precision $O(p^N)$) 
compared to the zealous approach.
Another more serious issue is memory. Indeed the relaxed arithmetic 
needs to keep stored all intermediate results by design. As an easy 
toy example, let us have a quick look at the following function

\begin{lstlisting}
    def nth_term(n)
        u = (*$a$*)
        for i in (*$1, 2, \ldots, n$*):
            u = (*$f$*)(u)
        return u
\end{lstlisting}

\noindent
that computes the $n$-th term of a recursive sequence defined by its 
initial value $u_0 = a$ and the recurrence $u_{i+1} = f(u_i)$. Here $a$ 
and $f$ are given data.
In zealous arithmetic, the above implementation requires to store only 
one single value of the sequence $(u_i)_{i \geq 0}$ at the same time 
(assuming that we can rely on a good garbage collector); indeed at the 
$i$-th iteration of the loop, the value of $u_i$ is computed and stored 
in the variable \texttt{u} while the previous value of \texttt{u}, 
\emph{i.e.} the value of $u_{i-1}$, is destroyed.
On the other hand, in relaxed arithmetic, the variable representing 
$u_i$ stores the definition of $u_i$, including then the value of 
$u_{i-1}$. The relaxed $p$-adic number returns by the function 
\texttt{nth\_term} is then a huge structure in which all the relaxed 
$p$-adic numbers $u_1, u_2, \ldots, u_n$ have to appear. When we will
then ask for the computation of the first $N$ digits of $u_n$, the
relaxed machinery will start to work and compute --- and store --- 
the values of all the $u_i$'s at the desired precision. This is the price 
to pay in order to be able to compute one more digit of $u_n$ without 
having to redo many calculations.

\subsubsection{Interval arithmetic \textsc{vs.} floating-point arithmetic}
\label{sssec:intervalVSfloat}

$p$-adic floating-point arithmetic has a very serious limitation for 
its use in mathematics: the results it outputs are \emph{not proved} 
and even often wrong! Precisely, the most significand digits of the 
output are very likely to be correct while the least significand digits are 
very likely to be incorrect... and we have \emph{a priori} no way to know 
which digits are correct and which ones are not.
On the other hand, at least in the real setting, interval arithmetic 
is known for its tendency to yield pessimistic enclosures. 
What about the $p$-adic 
case? At first, one might have expected that ultrametricity may help. 
Indeed ultrametricity seems to tell that rounding errors do not 
accumulate as it does in the real world. Unfortunately this simple 
rationale is too naive: in practice, $p$-adic interval arithmetic
does overestimate the losses of precision exactly as real interval
arithmetic does.

The causes are multiple and complex but some of them can be isolated. 
The first source of loss of precision comes from the situation where we 
add two $p$-adic numbers known at different precision 
(\emph{e.g.} two $p$-adic numbers of different sizes). As a toy example, 
consider the function $f : \Zp^2 \to \Zp^2$ mapping $(x,y)$ to $(x+y, 
x-y)$ (it is a similarity in the $p$-adic plane) and suppose that we
want to evaluate $f \circ f$ on the entry $x = 1 + O(p^2)$, $y = 1
+ O(p^{20})$. We then code the following function:
\begin{lstlisting}
    def F(x,y):
        return x+y, x-y
\end{lstlisting}
and call $\ttF(\ttF(\ttx,\tty))$.
Let us have a look at precision. According to 
Proposition~\ref{prop:arithinterval}, the first call $\ttF(\ttx,\tty)$ 
returns the pair $(2 + O(p^2), O(p^2))$ and the final result is then 
$(2 + O(p^2), 2 + O(p^2))$. On the other hand, observing that $f \circ f$
is the mapping taking $(x,y)$ to $(2x,2y)$. Using this, we end up with:
$$\begin{array}{r@{\hspace{0.5ex}}ll}
f \circ f(x,y) & = \big(2 + O(p^2), 2 + O(p^{20})\big) 
  & \text{if } p > 2 \smallskip \\
& = \big(2 + O(2^3), 2 + O(2^{21})\big) 
  & \text{if } p = 2
\end{array}$$
which is much more accurate. 
Interval arithmetic misses the simplification and consequently loses 
accuracy.

A second more subtle source of inaccuracy comes from the fact that 
interval arithmetic often misses \emph{gain} of precision. A very basic 
example is the computation of $px$; written this way, the absolute 
precision increases by $1$. However if, for some good 
reason\footnote{For example, imagine that all the values $x, 2x, 3x, 
\ldots, px$ are needed.}, the product $px$ does not appear explicitly 
but instead is computed by adding $x$ to itself $p$ times, 
interval arithmetic will not see the gain of precision.
It turns out that similar phenomena appear for the multiplicative 
analogue of this example, \emph{i.e.} the computation of $x^p$ for $x = 
a + O(p^N)$. For simplicity assume that $\val(x) = 0$ and $N \geq 
1$. Using again Proposition~\ref{prop:arithinterval}, we find that 
zealous arithmetic leads to the result $x^p = a^p + O(p^N)$. However, 
surprisingly, this result is not optimal (regarding precision) as 
shown by the next lemma.

\begin{lem}
\label{lem:ppower}
If $a \equiv b \pmod{p^N}$, then $a^p \equiv b^p \pmod{p^{N+1}}$.
\end{lem}

\begin{proof}
Write $b = a + p^Nc$ with $c \in \Zp$ and compute:
$$b^p = (a + p^Nc)^p = \sum_{i=0}^p \binom p i a^{p-i} (p^Nc)^{i}.$$
When $i \geq 2$, the corresponding term is divisible by $p^{2N}$ while 
when $i=1$, it equals $p^{N+1} a^{p-1} c$ and is therefore apparently
divisible by $p^{N+1}$. As a consequence $b^p \equiv a^p \pmod{p^{N+1}}$
as claimed.
\end{proof}

\noindent
According to Lemma \ref{lem:ppower}, the correct precision for $x^p$ is (at least) 
$O(p^{N+1})$ which is not detected by arithmetic interval. \sage fixes
this issue by an \emph{ad-hoc} implementation of the power function 
which knows about Lemma~\ref{lem:ppower}. However similar behaviors 
happen in quite a lot of different other situations and they of course 
cannot be all fixed by \emph{ad-hoc} patches.

\begin{rem}
Recall that we have seen in \S \ref{sssec:intervalNewton} that the 
computation of square roots in $\Q_2$ looses one digit of precision. 
It is absolutely coherent with the result above which says that the 
computation of squares gains one digit of precision.
\end{rem}

As a conclusion, a very common situation where $p$-adic floating-point 
arithmetic can be very helpful is that situation where the mathematician 
is experimenting, trying to understand how the new mathematical 
objects he has just introduced behave. At the moment, he does not really need 
proofs\footnote{Well, it is of course often better to have proofs... but 
it is maybe too early.}; he needs fast computations and plausible 
accuracy, exactly what $p$-adic floating-point arithmetic can offer.

\subsubsection{Comparison by everyday life examples}
\label{sssec:compexamples}

We analyze many examples and compare for each of them the accuracy we 
get with $p$-adic floating-point arithmetic on the one hand and with interval 
arithmetic (restricting ourselves to the zealous approach for 
simplicity) on the other hand. Examples are picked as basic and very
common primitives in linear algebra and commutative algebra.

\paragraph{Determinant}

Our first example concerns the computation of the determinant of a 
square matrix with entries in $\Zp$. In order to be as accurate as
possible, we shall always use a division-free algorithm (see for
instance~\cite{Bi11} or~\cite{KaVi05} and the references therein for 
faster solutions).
When $M$ is a generic matrix there is actually not so much to say:
if the entries of the matrix are given with $N$ significand digits,
the same holds for the determinant and this precision is optimal.
Nonetheless, quite interesting phenomena show up for matrices having 
a special form. 
In the sequel, we will examine the case of matrices $M$ of the form $M 
= P D Q$ where $P, Q \in \GL_d(\Zp)$ and $D$ is a diagonal matrix with 
diagonal entries $p^{a_1}, \ldots, p^{a_d}$. We assume that the three
matrices $P$, $D$ and $Q$ are given at precision $O(p^N)$ for some $N$.
Here is a concrete example (picked at random) with $p=2$, $d=4$ and
$N = 10$:
$$\begin{array}{r@{\hspace{0.5ex}}l}
\multicolumn{2}{c}{
  a_1 = 0 \quad ; \quad
  a_2 = 2 \quad ; \quad
  a_3 = 3 \quad ; \quad
  a_4 = 5} \medskip \\
P & = 
\small \left(\begin{array}{rrrr}
\ldots1111100100 & \ldots0110110101 & \ldots0101011000 & \ldots1101010001 \\
\ldots1010001101 & \ldots0110011001 & \ldots1101111000 & \ldots1010100100 \\
\ldots1011101111 & \ldots0100100111 & \ldots0000111101 & \ldots0010010001 \\
\ldots1011111001 & \ldots1000100011 & \ldots1100100110 & \ldots0111100011
\end{array}\right) \medskip \\
Q & = 
\small \left(\begin{array}{rrrr}
\ldots1010110001 & \ldots0010011111 & \ldots1010010010 & \ldots1010001001 \\
\ldots1111101111 & \ldots0111100101 & \ldots0110101000 & \ldots0111100000 \\
\ldots0111110100 & \ldots0010010101 & \ldots0000101111 & \ldots1001100010 \\
\ldots0101111111 & \ldots0101110111 & \ldots1110000011 & \ldots1110000110
\end{array}\right)
\end{array}$$
so that:
\begin{equation}
\label{eq:matrixM}
\begin{array}{r@{\hspace{0.5ex}}l}
M & = 
\small \left(\begin{array}{rrrr}
\ldots0101110000 & \ldots0011100000 & \ldots1011001000 & \ldots0011000100 \\
\ldots1101011001 & \ldots1101000111 & \ldots0111001010 & \ldots0101110101 \\
\ldots0111100011 & \ldots1011100101 & \ldots0010100110 & \ldots1111110111 \\
\ldots0000111101 & \ldots1101110011 & \ldots0011010010 & \ldots1001100001
\end{array}\right)
\end{array}
\end{equation}

\begin{rem}
Each entry of $M$ is known at precision $O(2^{10})$ as well. Indeed any 
permutation of $M$ by an element $H \in 2^{10} M_4(\Z_2)$ can be induced 
by a perturbation of $D$ by the element $P^{-1} H Q^{-1}$ which lies in 
$2^{10} M_4(\Z_2)$ as well because $P$ and $Q$ have their inverses in 
$M_4(\Z_2)$.
\end{rem}

Here are the values for the determinant of $M$ computed according to
the two strategies we want to compare:
\begin{center}
\begin{tabular}{ll}
Interval (zealous) arithmetic: & $\det M = O(2^{10})$ \smallskip \\
Floating-point arithmetic: & $\det M = 2^{10} \times \ldots0001001101$
\end{tabular}
\end{center}
We observe the interval arithmetic does not manage to decide whether
$\det M$ vanishes or not. On the other hand, the floating-point approach
outputs a result with $10$ significand digits by design; the point is
that we do not know \emph{a priori} whether these digits are correct or not.
In our particular case, we can however answer this question by computing 
$\det M$ as the product $\det P \cdot \det D \cdot \det Q = 2^{10} 
\det P \cdot \det Q$ using zealous arithmetic. We find this way:
$$\det M = 2^{10} \times \ldots 01101$$
which means that the result computed by floating-point arithmetic has (at
least) $5$ correct digits. This is actually optimal as we shall see
later in \S \ref{sssec:diffexamples}.

\paragraph{Characteristic polynomial}

We now move to the characteristic polynomial keeping first the same
matrix $M$. Here are the results we get:
$$\begin{array}{cr@{\hspace{0.2ex}}r@{\hspace{0.5ex}}l@{\hspace{0.2ex}}r@{\hspace{0.5ex}}l}
\multicolumn{6}{l}{\text{Interval (zealous) arithmetic:}}\smallskip\\
\hspace{5mm}\null
& \chi_M(X) = X^4 + {}& (\ldots 0001000010) & X^3 
      + {}& (\ldots 1000101100) & X^2 \\
  & {}+ {}& (\ldots 0011100000) & X^{\phantom{1}} 
      + {}& (\ldots 0000000000) \medskip \\
\multicolumn{6}{l}{\text{Floating-point arithmetic:}}\smallskip\\
& \chi_M(X) = X^4 + {}& (\ldots 0{\color{purple}0001000010}) & X^3 
      + {}& (\ldots 11{\color{purple}1000101100}) & X^2 \\
  & {}+ {}& (\ldots 110{\color{purple}100011100000}) & X^{\phantom{1}} 
      + {}& (2^{10} \times \ldots 00010{\color{purple}01101}) \\
\end{array}$$
where, in the second case, the correct digits are written in 
purple\footnote{We decided which digits are correct simply
by computing at higher precision (with zealous arithmetic in order
to get a guaranteed result).}.
We observe again that, although the computation is not proved 
mathematically, the floating-point arithmetic outputs more accurate
results. As we shall see later (see \S \ref{sssec:diffexamples}) 
the accuracy it gives is even optimal for this particular example.

We consider now the matrix $N = I_4 + M$ where $I_4$ is the identity
matrix of size $4$.
The computation of $\chi_N$ leads to the following results:
$$\begin{array}{cr@{\hspace{0.2ex}}r@{\hspace{0.5ex}}l@{\hspace{0.2ex}}r@{\hspace{0.5ex}}l}
\multicolumn{6}{l}{\text{Interval (zealous) arithmetic:}}\smallskip\\
\hspace{5mm}\null
& \chi_N(X) = X^4 + {}& (\ldots 0000111110) & X^3 
      + {}& (\ldots 0101101100) & X^2 \\
  & {}+ {}& (\ldots 0101001010) & X^{\phantom{1}} 
      + {}& (\ldots 0100001011) \medskip \\
\multicolumn{6}{l}{\text{Floating-point arithmetic:}}\smallskip\\
& \chi_N(X) = X^4 + {}& (\ldots0{\color{purple}0000111110}) & X^3 
      + {}& (\ldots00{\color{purple}0101101100}) & X^2 \\
  & {}+ {}& (\ldots1{\color{purple}0101001010}) & X^{\phantom{1}} 
      + {}& (\ldots {\color{purple}0100001011}) \\
\end{array}$$
On that example, we remark that interval arithmetic is as accurate
as floating-point arithmetic.
More interesting is the evaluation of $\chi_N$ at $1$, which is 
nothing but the determinant of $M$ we have already computed before.
The values we get are the following:
\begin{center}
\begin{tabular}{ll}
Interval (zealous) arithmetic: & $\chi_N(1) = O(2^{10})$ \smallskip \\
Floating-point arithmetic: & $\chi_N(1) = 0$
\end{tabular}
\end{center}
Although there is no surprise with interval arithmetic, we observe
that floating-point arithmetic now fails to produce any correct digit.

\paragraph{LU factorization}

LU factorization is a classical tool in linear algebra which serves
as primitives for many problems.
We refer to~\cite[\S 2]{AbLo04} for an introduction to the topic. 
Recall briefly that a LU factorization of a matrix $M$ is a 
decomposition of the form $M = LU$ where $L$ (resp. $U$) is a lower triangular 
(resp. upper triangular) matrix. In the sequel, we require moreover that 
the diagonal entries of $L$ are all equal to $1$.
With the normalization, one can prove that any matrix $M$ defined over a 
field admits a \emph{unique} LU factorization as soon as all its principal 
minors\footnote{Recall the $i$-th principal minor of $M$ is the 
determinant of the submatrix of $M$ obtained by selecting the first $i$
rows and first $i$ columns.} do not vanish. 

LU factorizations can be computed using standard Gaussian elimination: 
starting from $M$, we first multiply the first column by the appropriate 
scalar in order to make the top left entry equal to $1$ and use it as 
pivot to cancel all coefficients on the first line by operating on 
columns. We then get a matrix of the shape:
$$\left(\begin{matrix}
1 & 0 & \cdots & 0 \\
\star & \star & \cdots & \star \\
\vdots & \vdots & & \vdots \\
\star & \star & \cdots & \star 
\end{matrix}\right)$$
and we continue this process with the submatrix obtained by deleting
the first row and the first column. The matrix we get at the end is
the $L$-part of the LU factorization of $M$.

We propose to explore the numerical stability of LU factorization 
\emph{via} Gaussian elimination in the $p$-adic case. Let us start with 
an example. Consider first the same matrix $M$ as above and write its LU 
factorization $M = LU$ (one can check that its principal minors do not 
vanish). Performing Gaussian elimination as described above within the 
framework of zealous arithmetic on the one hand and within the framework 
of floating-point arithmetic on the other hand, we end up with the 
following matrices:
\begin{equation}
\label{eq:resultLU}
\begin{array}{cr@{\hspace{0.5ex}}l}
\multicolumn{3}{l}{\text{Interval (zealous) arithmetic:}}\smallskip\\
\hspace{5mm}\null 
& L & = 
\small \left(\begin{array}{rrr@{\hspace{2em}}r}
1 & 0 & 0 & 0 \\
2^{-4} \times \ldots001111 & 1 & 0 & 0 \\
2^{-4} \times \ldots010101 & \ldots100011 & 1 & 0 \\
2^{-4} \times \ldots001011 & \ldots010101 & \ldots110 & 1
\end{array}\right) \medskip \\
\multicolumn{3}{l}{\text{Floating-point arithmetic:}}\smallskip\\
\hspace{5mm}\null 
& L & = 
\small \left(\begin{array}{rrr@{\hspace{2em}}r}
1 & 0 & 0 & 0 \\
2^{-4} \times \ldots1010{\color{purple}001111} & 1 & 0 & 0 \\
2^{-4} \times \ldots0110{\color{purple}010101} & \ldots{0\color{purple}011100011} & 1 & 0 \\
2^{-4} \times \ldots0101{\color{purple}001011} & \ldots{\color{purple}0111010101} & 0001{\color{purple}0110110} & 1
\end{array}\right)
\end{array}
\end{equation}
As before the correct digits are displayed in purple. We remark once
again that the accuracy of $p$-adic floating-point arithmetic is better 
than that of interval arithmetic (it is actually even optimal on that
particular example as we shall see in \S \ref{sssec:diffexamples}).
More precisely, we note that the
zealous precision is sharp on the first column but the gap increases
when we are moving to the right.

\begin{rem}
It is somehow classical~\cite[\S 1.4]{Ho75} that the entries of $L$ and $U$ can 
all be expressed as the quotient of two appropriate minors of $M$ 
(Cramer-like formulas). Evaluating such expressions, it is possible to 
compute the LU factorization and stay sharp on precision within the 
zealous framework. The drawback is of course complexity since evaluating 
two determinants for each entry of $L$ and $U$ is clearly very 
time-consuming. A stable and efficient algorithm, combining the 
advantages of the two approaches, is designed in~\cite{Ca12}.
\end{rem}

\paragraph{Bézout coefficients and Euclidean algorithm}

We now move to polynomials and examine the computation of
Bézout coefficients \emph{via} the Euclidean algorithm. We pick at
random two monic polynomials $P$ and $Q$ of degree $4$ over $\Z_2$:

$$\begin{array}{r@{\hspace{0.2ex}}r@{\hspace{0.5ex}}l@{\hspace{0.2ex}}r@{\hspace{0.5ex}}l}
P = X^4 + {}& (\ldots 1101111111) & X^3 
        + {}& (\ldots 0011110011) & X^2 \\
      {}+ {}& (\ldots 1001001100) & X^{\phantom{1}} 
        + {}& (\ldots 0010111010) \medskip \\
Q = X^4 + {}& (\ldots 0101001011) & X^3 
        + {}& (\ldots 0111001111) & X^2 \\
      {}+ {}& (\ldots 0100010000) & X^{\phantom{1}} 
        + {}& (\ldots 1101000111) 
\end{array}$$

We can check that $P$ and $Q$ are coprime modulo $p$ (observe that
$P + Q \equiv 1 \pmod p$); they are therefore \emph{a fortiori} coprime
in $\Qp[X]$. By Bézout Theorem, there exist two polynomials 
$U, V \in \Z_2[X]$ such that $UP + VQ = 1$. Moreover these polynomials 
are uniquely determined if we require $\deg U, \deg V \leq 3$. Computing 
them with the extended Euclidean algorithm (without any kind of 
algorithmic optimization), we get:
$$\begin{array}{cr@{\hspace{0.2ex}}r@{\hspace{0.5ex}}l@{\hspace{0.2ex}}r@{\hspace{0.5ex}}l}
\multicolumn{6}{l}{\text{Interval (zealous) arithmetic:}}\smallskip\\
\hspace{1cm}\null
&U = {}& (\ldots 101100) & X^3 
   + {}& (\ldots 101100) & X^2 \\
&{}+ {}& (\ldots 100) & X^{\phantom{1}} 
   + {}& (\ldots 1011) \medskip \\
&V = {}& (\ldots 010100) & X^3 
   + {}& (\ldots 100100) & X^2 \\
&{}+ {}& (\ldots 100) & X^{\phantom{1}} 
   + {}& (\ldots 101) \medskip \\
\multicolumn{6}{l}{\text{Floating-point arithmetic:}}\smallskip\\
&U = {}& (\ldots 101{\color{purple}011101100}) & X^3 
   + {}& (\ldots 111{\color{purple}100101100}) & X^2 \\
&{}+ {}& (\ldots 100000{\color{purple}110100}) & X^{\phantom{1}} 
   + {}& (\ldots 01{\color{purple}10001011}) \medskip \\
&V = {}& (\ldots 010{\color{purple}100010100}) & X^3 
   + {}& (\ldots 00{\color{purple}1011100100}) & X^2 \\
&{}+ {}& (\ldots 00010{\color{purple}0111100}) & X^{\phantom{1}} 
   + {}& (\ldots 111{\color{purple}1100101})
\end{array}$$
Floating-point arithmetic provides again better accuracy, though far for 
being optimal this time. Indeed the theory of subresultants 
\cite[\S 4.1]{Wi96}
shows that the coefficients of $U$ and $V$ can be all expressed as 
quotients of some minors of the Sylvester matrix of $(P,Q)$ by the 
resultant of $P$ and $Q$, denoted hereafter by $\Res(P,Q)$. From the 
fact that $P$ and $Q$ are coprime modulo $p$, we deduce that $\Res(P,Q)$ 
does not vanish modulo $p$. Thus $\val_p(\Res(P,Q)) = 0$ and dividing by
$\Res(P,Q)$ does not decrease the absolute precision according to
Proposition~\ref{prop:arithinterval}. As a consequence, following this
path and staying within the zealous framework, we can calculate
the values of $U$ and $V$ at precision $O(2^{10})$. Here is the result
we get:
$$\begin{array}{r@{\hspace{0.2ex}}r@{\hspace{0.5ex}}
l@{\hspace{0.2ex}}r@{\hspace{0.5ex}}
l@{\hspace{0.2ex}}r@{\hspace{0.5ex}}
l@{\hspace{0.2ex}}r@{\hspace{0.5ex}}}
U = {}& (\ldots 0011101100) & X^3 
  + {}& (\ldots 0100101100) & X^2 
{}+ {}& (\ldots 1101110100) & X^{\phantom{1}} 
  + {}& (\ldots 0010001011) \medskip \\
V = {}& (\ldots 1100010100) & X^3 
  + {}& (\ldots 1011100100) & X^2 
{}+ {}& (\ldots 0110111100) & X^{\phantom{1}} 
  + {}& (\ldots 0001100101).
\end{array}$$
We observe that the latter values (in addition to being proved) are 
even more accurate that the result which was computed ``naively'' using 
floating-point arithmetic. The drawback is complexity since evaluating
many determinants requires a lot of time. 
The theory of $p$-adic precision we are going to introduce in the next
section (see \S \ref{sec:precision} and, more especially, \S
\ref{sssec:adaptiveprecision}) provides the tools for writing a
stabilized version of Euclidean algorithm that computes provably $U$ 
and $V$ at (almost) optimal precision. We refer to~\cite{Ca17} for more
details (the material developed in  \S \ref{sec:precision} is necessary 
to read this reference).

\paragraph{Polynomial evaluation and interpolation}

Polynomial evaluation and polynomial interpolation are very classical 
and useful primitives involved in many algorithms in symbolic 
computation. In this last paragraph, we examine how they behave from
the point of view of precision. We focus on a very basic (but already
very instructive) example.
We consider the following two procedures:
\begin{itemize}
\renewcommand{\itemsep}{0pt}
\item \texttt{evaluation}: it takes as input a polynomial $P \in \Qp[X]$ 
of degree at most $d$ and outputs $P(0), P(1), \ldots, P(d)$;
\item \texttt{interpolation}: it takes as input a list of values $y_0, 
\ldots, y_d \in \Qp$ and returns the interpolation polynomial $P \in 
\Qp[X]$ of degree at most $d$ such that $P(i) = y_i$ for $i \in \{0,1,\ldots, d\}$.
\end{itemize}
Algorithms (and notably fast algorithms) for these tasks abound in the
literature (see for instance~\cite[\S 10]{GaGe03}). For our propose, we choose naive algorithms: we implement
\texttt{evaluation} by evaluating separately the $P(i)$'s (using Hörner
scheme say) and we implement \texttt{interpolation} using the method
of divided differences~\cite[\S 2]{Hi56}. Under our assumptions 
(interpolation at the first integers), it turns out that it takes a 
particularly simple form that we make explicit now.

Define the \emph{difference operator} $\Delta$ on $\Qp[X]$ by $\Delta 
A(X) = A(X+1) - A(X)$. The values taken by $A$ at the integers are 
related to the values $\Delta^n A(0)$ by a simple closed formula, as
shown by the next lemma.

\begin{lem}
\label{lem:mahler}
For all polynomials $A \in \Qp[X]$ of degree at most $d$, we have:
$$A(X) = \sum_{n=0}^d \Delta^n A(0) \cdot \binom X n
\quad \text{where} \quad
\binom X n = \frac{X(X-1)\cdots (X-n+1)}{n!}.$$
\end{lem}

\begin{proof}
We proceed by induction on $d$. When $d = 0$, the lemma is trivial.
We assume now that it holds for all polynomials $A$ of degree at most
$d$ and consider a polynomial $A \in \Qp[X]$ with $\deg A = d{+}1$.
We define $B(X) = \sum_{n=0}^d \Delta^n A(0) \cdot \binom X n$. 
Remarking that $\Delta \binom X n = 
\binom X{n-1}$ for all positive integers $n$, we derive:
$$\Delta B (X) = \sum_{n=1}^d \Delta^n A(0) \cdot \binom X {n-1}
= \sum_{n=0}^{d-1} \Delta^{n+1} A(0) \cdot \binom X n.$$
From the induction hypothesis, we deduce $\Delta B = \Delta A$ (since 
$\Delta A$ has degree at most $d$). Furthermore, going back to the
definition of $B$, we find $A(0) = B(0)$. All together, these two
relations imply $A = B$ and the induction goes.
\end{proof}

\begin{rem}
Lemma~\ref{lem:mahler} extends by continuity to all continuous
functions $f$ on $\Zp$. In this generality, it states that any
continuous function $f : \Zp \to \Zp$ can be uniquely written as
a convergent series of the shape:
$$f(x) = \sum_{n=0}^\infty a_n \cdot \binom x n$$
where the $a_n$'s lie in $\Zp$ and converge to $0$ when $i$ goes
to infinity. The $a_n$'s are moreover uniquely determined: we have
$a_n = \Delta^n f(0)$. They are called the \emph{Mahler coefficients}
of $f$. We refer to~\cite{Ma58} for much more details on this topic.
\end{rem}

From Lemma~\ref{lem:mahler}, we easily derive an algorithm for our 
interpolation problem. Given $y_0, \ldots, y_d$, we define the 
``divided'' differences $y_{n,i}$ for $0 \leq i \leq n \leq d$ by 
(decreasing) induction on $n$ by $y_{d,i} = y_i$ and $y_{n-1,i} = 
y_{n,i+1} - y_{n,i}$. These quantities can be easily computed.
Moreover, thanks to Lemma~\ref{lem:mahler}, the interpolation 
polynomial $P$ we are looking for writes
$P(X) = \sum_{n=0}^d y_{n,0} \binom x n$.
This provides an algorithm for computing the interpolation polynomial.

Let us now try to apply successively \texttt{evaluation} and
\texttt{interpolation} taking as input a random polynomial $P$ of 
degree $d=8$ with coefficients in $\Zp$ (for $p = 2$ as usual):
$$\begin{array}{r@{\hspace{0.2ex}}r@{\hspace{0.5ex}}l@{\hspace{0.2ex}}r@{\hspace{0.5ex}}l@{\hspace{0.2ex}}r@{\hspace{0.5ex}}l}
P = {}& (\ldots 0111001110) & X^8 
  + {}& (\ldots 0101010001) & X^7 
  + {}& (\ldots 1000001100) & X^6 \\
{}+ {}& (\ldots 1010001101) & X^5 
  + {}& (\ldots 1111000100) & X^4
  + {}& (\ldots 0011101101) & X^3 \\
{}+ {}& (\ldots 1010010111) & X^2
  + {}& (\ldots 0011011010) & X^{\phantom{1}} 
  + {}& (\ldots 0001011110) 
\end{array}$$
The results we get are:
$$\begin{array}{cr@{\hspace{0.2ex}}r@{\hspace{0.5ex}}l@{\hspace{0.2ex}}r@{\hspace{0.5ex}}l@{\hspace{0.2ex}}r@{\hspace{0.5ex}}l}
\multicolumn{8}{l}{\text{Interval (zealous) arithmetic:}}\smallskip\\
\hspace{1cm}\null
&    {}& (\ldots 110) & X^8 
   + {}& (\ldots 001) & X^7 
   + {}& (\ldots 100) & X^6 \\
&{}+ {}& (\ldots 101) & X^5 
   + {}& (\ldots 100) & X^4
   + {}& (\ldots 1101) & X^3 \\
&{}+ {}& (\ldots 10111) & X^2
   + {}& (\ldots 1011010) & X^{\phantom{1}} 
   + {}& (\ldots 0001011110) \medskip \\
\multicolumn{8}{l}{\text{Floating-point arithmetic:}}\smallskip\\
\hspace{1cm}\null
&    {}& (\ldots 0000{\color{purple}1001110}) & X^8 
   + {}& (\ldots 101{\color{purple}1010001}) & X^7 
   + {}& (\ldots 00101{\color{purple}0001100}) & X^6 \\
&{}+ {}& (\ldots 0{\color{purple}010001101}) & X^5 
   + {}& (\ldots 0000{\color{purple}11000100}) & X^4
   + {}& (\ldots 010{\color{purple}11011101}) & X^3 \\
&{}+ {}& (\ldots 0{\color{purple}010010111}) & X^2
   + {}& (\ldots 11{\color{purple}011011010}) & X^{\phantom{1}} 
   + {}& (\ldots 0{\color{purple}0001011110}) 
\end{array}$$
The result computed by floating-point arithmetic is again a bit more
accurate. However this observation is not valid anymore where the degree
$d$ gets larger.
\begin{figure}
$$\scriptsize
\begin{array}{r@{\hspace{0.2ex}}
r@{\hspace{0.5ex}}l@{\hspace{0.2ex}}
r@{\hspace{0.5ex}}l@{\hspace{0.2ex}}
r@{\hspace{0.5ex}}l@{\hspace{0.2ex}}r@{\hspace{0.5ex}}l}
\multicolumn{9}{l}{\text{Initial polynomial:}}\smallskip\\
     {}& (\ldots 0101101001) & X^{19}
   + {}& (\ldots 1101000011) & X^{18} 
   + {}& (\ldots 0011001110) & X^{17} 
   + {}& (\ldots 1001011010) & X^{16} \\
 {}+ {}& (\ldots 0011100111) & X^{15} 
   + {}& (\ldots 0110101110) & X^{14}
   + {}& (\ldots 0111111001) & X^{13}
   + {}& (\ldots 1011010111) & X^{12} \\
 {}+ {}& (\ldots 0100000100) & X^{11}
   + {}& (\ldots 0000110000) & X^{10}
   + {}& (\ldots 1110101010) & X^{9\phantom{0}}
   + {}& (\ldots 1111101100) & X^{8} \\
 {}+ {}& (\ldots 0100010001) & X^{7\phantom{0}}
   + {}& (\ldots 0101010000) & X^{6\phantom{0}}
   + {}& (\ldots 0111101111) & X^{5\phantom{0}}
   + {}& (\ldots 1100010011) & X^{4} \\
 {}+ {}& (\ldots 0100000001) & X^{3\phantom{0}}
   + {}& (\ldots 1000010010) & X^{2\phantom{0}}
   + {}& (\ldots 0000100000) & X^{\phantom{01}} 
   + {}& (\ldots 0001111110) \medskip \\
\multicolumn{9}{l}{\text{Interval (zealous) arithmetic:}}\smallskip\\
       & O(2^{-6}) & X^{19}
   + {}& O(2^{-6}) & X^{18}
   + {}& O(2^{-6}) & X^{17}
   + {}& O(2^{-5}) & X^{16} \\
 {}+ {}& O(2^{-6}) & X^{15}
   + {}& O(2^{-6}) & X^{14}
   + {}& O(2^{-6}) & X^{13}
   + {}& O(2^{-5}) & X^{12} \\
 {}+ {}& O(2^{-6}) & X^{11}
   + {}& O(2^{-6}) & X^{10}
   + {}& O(2^{-6}) & X^{9\phantom{0}}
   + {}& O(2^{-5}) & X^{8} \\
 {}+ {}& O(2^{-4}) & X^{7\phantom{0}}
   + {}& O(2^{-3}) & X^{6\phantom{0}}
   + {}& O(2^{-2}) & X^{5\phantom{0}}
   + {}& O(2^{-1}) & X^{4} \\
 {}+ {}& (\ldots 1) & X^{3\phantom{0}}
   + {}& (\ldots 010) & X^{2\phantom{0}}
   + {}& (\ldots 100000) & X^{\phantom{01}} 
   + {}& (\ldots 0001111110) \medskip \\
\multicolumn{9}{l}{\text{Floating-point arithmetic:}}\smallskip\\
     {}& (2^{-3} \times \ldots 1110011011) & X^{19}
   + {}& (2^{-5} \times \ldots 0000000011) & X^{18} 
   + {}& (2^{-3} \times \ldots 0001011111) & X^{17} 
   + {}& (2^{-5} \times \ldots 1100111101) & X^{16} \\ 
 {}+ {}& (\ldots 11111100110) & X^{15}
   + {}& (2^{-4} \times \ldots 0110100011) & X^{14} 
   + {}& (2^{-2} \times \ldots 0000010011) & X^{13} 
   + {}& (2^{-4} \times \ldots 1010001101) & X^{12} \\ 
 {}+ {}& (2^{-3} \times \ldots 0010000011) & X^{11}
   + {}& (2^{-5} \times \ldots 0100101111) & X^{10} 
   + {}& (2^{-3} \times \ldots 0000110011) & X^{9\phantom{0}} 
   + {}& (2^{-5} \times \ldots 1010101001) & X^{8} \\ 
 {}+ {}& (2^{-2} \times \ldots 0010000101) & X^{7\phantom{0}}
   + {}& (2^{-2} \times \ldots 1101100111) & X^{6\phantom{0}} 
   + {}& (\ldots 110{\color{purple}1101111}) & X^{5\phantom{0}} 
   + {}& (2^{-1} \times \ldots 0011100111) & X^{4} \\ 
 {}+ {}& (\ldots 01011101{\color{purple}01}) & X^{3\phantom{0}}
   + {}& (\ldots 11011101{\color{purple}010}) & X^{2\phantom{0}} 
   + {}& (\ldots 000000001{\color{purple}100000}) & X^{\phantom{01}}
   + {}& (\ldots 0{\color{purple}0001111110})
\end{array}$$

\caption{Evaluation and re-interpolation of a polynomial of degree $19$
over $\Z_2$}
\label{fig:evalinterpol}
\end{figure}
Figure~\ref{fig:evalinterpol} shows an example with a polynomial
(picked at random) of degree $d = 19$. We see that almost all digits
are incorrect, many coefficients have negative valuations, \emph{etc.}
For this problem, floating-point arithmetic is then not well suited.

One may wonder whether another algorithm, specially designed for 
stability, would lead to better results. Unfortunately, the answer is 
negative: we shall see later (in \S \ref{sssec:diffexamples}) that 
the problem of polynomial evaluation and interpolation is \emph{very} 
ill-conditioned in the $p$-adic setting, so that numerical methods are 
ineffective.

\section{The art of tracking $p$-adic precision}
\label{sec:precision}

In many examples presented in \S \ref{sssec:compexamples}, we have often 
recognized similar behaviors: interval arithmetic often overestimates 
the losses of precision while floating-point arithmetic provides more 
accurate --- but unproved --- results. Understanding precisely the
origin of these phenomena is a quite stimulating question (that has
been widely studied in the real setting).

Recently Caruso, Roe and Vaccon~\cite{CaRoVa14} proposed a general theory for 
dealing with precision in the $p$-adic setting and this way provided
powerful tools for attacking the aforementioned question. Their rough 
idea was to develop an analogue of interval arithmetic in higher 
dimensions. In other words, instead of attaching a precision $O(p^N)$ to 
each $p$-adic variable, they group variables and attach to the 
collection of all of them a \emph{unique} \emph{global} precision datum 
materialized by an ``ellipsoid'' in some normed $p$-adic vector space. 
The magic of ultrametricity then operates: ellipsoids are rather easy to 
deal with and behave very well with respect to tracking of precision.

In this section, we report on Caruso, Roe and Vaccon's work. 
\S \ref{ssec:foundpadicprec} is dedicated to the foundations of their 
theory; it is mostly of mathematical nature and deals with $p$-adic 
analysis in several variables. It culminates with the statement (and
the proof) of the precision Lemma (Theorem \ref{theo:preclemma}).
Some first applications are discussed in \S \ref{ssec:optimalprec}
where the precision Lemma is used for finding the maximal precision
one can expect in many concrete situations. Finally we propose in
\S \ref{ssec:latticearith} general methods for reaching this optimal 
precision and apply them in concrete cases.

\subsection{Foundations of the theory of $p$-adic precision}
\label{ssec:foundpadicprec}

The aforementioned theory of $p$-adic precision is based on a single 
result of $p$-adic analysis --- the so-called \emph{precision Lemma} --- 
controlling how ellipsoids transform under mappings of class $C^1$. In 
fact the terminology ``ellipsoid'' is not quite appropriate to the 
$p$-adic setting (though very suggestive for the comparison with the 
real setting) because vector spaces over $\Qp$ are not equipped with 
some $L^2$-norm. Mathematicians then prefer using the term ``lattice'' 
because, as we shall see below, the $p$-adic lattices behave like usual 
$\Z$-lattices in $\R$-vector spaces.

This subsection is organized as follows. We first introduce the notion 
of $p$-adic lattice (\S\S 
\ref{sssec:lattices}--\ref{sssec:complattices}) together with the 
necessary material of $p$-adic analysis (\S \ref{sssec:classC1}). After 
this preparation, \S \ref{sssec:preclemma} is devoted to the precision 
Lemma: we state it and prove it. Applications to $p$-adic precision 
will be discussed in the next subsections (\S \ref{ssec:optimalprec}
and \S \ref{ssec:latticearith}).

\subsubsection{Lattices in finite-dimensional $p$-adic vector spaces}
\label{sssec:lattices}

Let $E$ be a finite-dimensional vector space over $\Qp$.
A (ultrametric) \emph{norm} on $E$ is a mapping $\Vert \cdot \Vert_E :
E \to \R^+$ satisfying the usual requirements:
\begin{enumerate}[(i)]
\renewcommand{\itemsep}{0pt}
\item $\Vert x \Vert_E = 0$ if and only if $x = 0$;
\item $\Vert \lambda x \Vert_E = |\lambda| \cdot \Vert x \Vert_E$;
\item \label{item:triangleineq}
$\Vert x + y \Vert_E \leq \max\big( \Vert x \Vert_E, 
\Vert y \Vert_E \big)$.
\end{enumerate}
Here $x$ and $y$ refer to elements of $E$ while $\lambda$ refers
to a scalar in $\Qp$. We notice that, without further assumption, 
one can prove that equality holds in \eqref{item:triangleineq} as
soon as $\Vert x \Vert_E \neq \Vert y \Vert_E$: all triangles are 
isosceles in all normed $p$-adic vector spaces! Indeed, by symmetry, we 
may assume $\Vert x \Vert_E < \Vert y \Vert_E$. Now we remark that from 
$\Vert x + y \Vert_E < \Vert y \Vert_E$, we would deduce:
$$\Vert y \Vert_E 
  \leq \max\big(\Vert x+y\Vert_E, \Vert {-}x \Vert_E\big) 
  = \max\big(\Vert x+y\Vert_E, \Vert x \Vert_E\big) 
  < \Vert y \Vert_E$$
which is a contradiction. Our assumption was then absurd, meaning that 
$\Vert x + y \Vert_E = \Vert y \Vert_E$.

Given a real number $r$, we let $B_E(r)$ denote the closed ball in $E$ 
of centre $0$ and radius $r$, \emph{i.e.}:
$$B_E(r) = \big\{\, x \in E \,\text{ s.t. } \Vert x \Vert_E \leq r 
\,\big\}.$$
It is worth remarking that $B_E(r)$ is a module over $\Zp$. Indeed,
on the one hand, multiplying by a scalar in $\Zp$ does not increase
the norm (since elements of $\Zp$ have norm at most $1$) and, on the 
other hand, the ultrametric triangular inequality implies that $E$ is 
stable under addition. Balls in $p$-adic vector spaces have then two
faces: one is analytic and one is algebraic.
Being able to switch between these two points of view is often very
powerful.

The very basic (but still very important) example of a normed 
$\Qp$-vector space is $\Qp^d$ itself endowed with the infinite norm
defined by
$\Vert (x_1, x_2, \ldots, x_d) \Vert_\infty 
= \max \big(|x_1|, |x_2|, \ldots, |x_d|\big)$ for $x_1, \ldots, x_d
\in \Qp$.
The unit ball of $\big(\Qp^d, \Vert \cdot \Vert_\infty\big)$ is $\Zp^d$
and, more generally, $B_{\Qp^d}(r) = p^n \Zp^d$ where $n$ is the unique
relative integer defined by $p^{-n} \leq r < p^{-(n{-}1)}$. We notice
that they are indeed $\Zp$-modules.
Other standard norms over $\R^d$ do not have a direct $p$-adic analogue
since they violate the ultrametric triangular inequality.

\medskip

The general notion of lattice is modeled on balls:

\begin{deftn}
\label{def:latticeana}
Let $E$ be a finite-dimensional vector space over $\Qp$.
A \emph{lattice} in $E$ is a subset of $E$ of the form $B_E(1)$ 
for some $p$-adic norm on $E$.
\end{deftn}

\noindent
Remark that if $H$ is a lattice in $E$ and $a$ is a non zero scalar
in $\Qp$, then $aH$ is a lattice in $E$ as well. Indeed if $H$ is
the closed unit ball for some norm $\Vert \cdot \Vert_E$, then $aH$ 
is the closed unit ball for $\frac 1{|a|} \Vert \cdot \Vert_E$.
More generally, if $H \subset E$ is a lattice and $f : E \to E$ is
a bijective linear transformation, then $f(H)$ is also a lattice.
A consequence of Proposition \ref{prop:nakayama} below is that all
lattices take the form $f(H_0)$ where $H_0$ is a fixed lattice and
$f$ vary in $\GL(E)$.

Lattices might be thought of as ``ellipsoids'' centered at $0$. 
They form a class of special neighborhoods of $0$ that we will use 
afterwards to model the precision (see \S \ref{ssec:latticearith}).
From this perspective, they will 
appear as the natural generalization to higher dimension of the notion 
of bounded interval (centered at $0$) of $\Qp$ we have introduced in \S 
\ref{sssec:tree} and already widely used in \S \ref{ssec:zealous} to 
model precision in zealous arithmetic.

Figure~\ref{fig:lattice} shows a possible picture of a lattice drawn in 
the $p$-adic plane $\Qp^2$ (endowed with the infinite norm). 
Nevertheless, we need of course to be very careful with such 
representations because the topology of $\Qp$ has nothing to do with the 
topology of the paper sheet (or the screen). In particular, it is quite 
difficult to reflect ultrametricity.

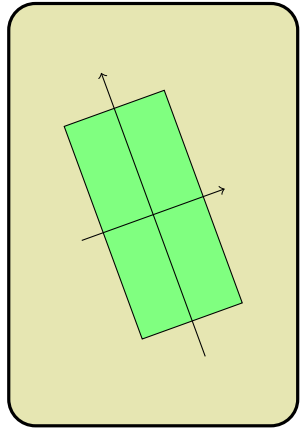
\begin{figure}
\hfill
\begin{tikzpicture}
\begin{scope}[beige, rounded corners=10pt]
\fill (0.1,0.2) rectangle (3.9,5.8);
\end{scope}

\begin{scope}[xshift=2cm,yshift=3cm,rotate=20]
\draw[black,fill=green!50] (-0.7,-1.5) rectangle (0.7,1.5);
\draw[black,->] (-1,0)--(1,0);
\draw[black,->] (0,-2)--(0,2);
\end{scope}

\begin{scope}[very thick, rounded corners=10pt]
\draw (0.1,0.2) rectangle (3.9,5.8);
\end{scope}
\end{tikzpicture}
\hfill \null

\caption{Picture of a lattice in the ultrametric world}
\label{fig:lattice}
\end{figure}

\medskip

Here is another purely algebraic definition of lattices which justifies
the wording (compare with the case of $\Z$-lattice in $\R$-vector
spaces).

\begin{deftn}
\label{def:latticealg}
Let $E$ be a finite-dimensional vector space over $\Qp$.
A \emph{lattice} in $E$ is a $\Zp$-module generated by a basis of $E$ 
over $\Qp$.
\end{deftn}

\noindent
The fact that a lattice in the sense of Definition \ref{def:latticealg}
is a lattice in the sense of Definition \ref{def:latticeana} is rather
easy: if $(e_1, \ldots, e_d)$ is a basis of $E$ over $\Qp$, the 
$\Zp$-span of the $e_i$'s is the closed unit ball for the norm 
$\Vert \cdot \Vert_E$ defined by:
$$\Vert \lambda_1 e_i + \cdots + \lambda_d e_d \Vert_E =
\max\big(|\lambda_1|, \ldots, |\lambda_d|\big).$$
As for the converse, it follows from the next proposition.

\begin{prop}
\label{prop:nakayama}
Let $E$ be a $d$-dimensional normed vector space over $\Qp$.
There exists a basis $(e_1, \ldots, e_d)$ of $E$ over $\Qp$ such 
that $B_E(1)$ is the $\Zp$-module generated by the $e_i$'s.
\end{prop}

\begin{proof}
Set $L = B_E(1)$. We have already seen that $L$ is a module over $\Zp$; 
thus the quotient $L/pL$ makes sense and is a vector space over $\Fp = 
\Z/p\Z$. Consider $(\bar e_i)_{i \in I}$ a basis of $L/pL$. (We shall 
see later that $L/pL$ is finite dimensional of dimension $d$ over $\Fp$ 
but, for now, we do not know this and we do not assume anything on the 
set $I$ indexing the basis.) For all $i \in I$, consider $e_i \in L$
which reduces to $\bar e_i$ modulo $p$.

We first claim that the family $(e_i)_{i \in I}$ (where $e_i$ is
considered as an element
of $E$) is free over $\Qp$. Indeed consider a relation of the form
$\sum_{i \in I} \lambda_i e_i = 0$ with $\lambda_i \in \Qp$ and
$\lambda_i = 0$ for almost all $i$. Assume by contradiction that
this relation is non trivial. Then $v = \min_i \, \val_p(\lambda_i)$ is
finite. Up to multiplying the $\lambda_i$'s by $p^{-v}$, we may assume 
that $\lambda_i \in \Zp$ for all $i \in I$ and that there exists
at least one index $i$ for which $\lambda_i$ does not reduce to $0$
modulo $p$.
Now reducing our dependency relation modulo $p$, we get $\sum_{i \in I} 
\bar \lambda_i \bar e_i = 0$ where $\bar \lambda_i \in \Fp$ is the class 
of $\lambda_i$ modulo $p$. Since the family $(\bar e_i)_{i \in I}$ is free 
over $\Fp$ by construction, we derive $\bar \lambda_i = 0$ which
contradicts our assumption.

It follows from the freedom of the $e_i$'s that $I$ is finite of 
cardinality at most $d$. Let us then write $I = \{1, 2, \ldots, d'\}$ 
with $d' \leq d$. We now prove that $(e_1, \ldots, e_{d'})$ generates $L$ 
over $\Zp$. Let then $x \in L$. Using that the $\bar e_i$'s generate 
$L/pL$, we find $p$-adic integers $\lambda_{i,1}$ ($0 \leq i \leq d'$)
such that:
\begin{align}
x & \equiv \lambda_{1,1} e_1 + 
\lambda_{2,1} e_2 + \cdots + \lambda_{d',1} e_{d'} \pmod{pL} \nonumber\\
\text{\emph{i.e.}} \quad 
x & = \lambda_{1,1} e_1 +
\lambda_{2,1} e_2 + \cdots + \lambda_{d',1} e_{d'} + p x_1
\label{eq:xmodp}
\end{align}
for some $x_1 \in L$. Applying the same reasoning to $x_1$ and 
re-injecting in Eq.~\eqref{eq:xmodp}, we end up with an equality of
the form $x = \lambda_{1,2} e_1 +
\lambda_{2,2} e_2 + \cdots + \lambda_{d',2} e_{d'} + p x_2$
with $x_2 \in L$ and $\lambda_{i,2} \equiv \lambda_{i,1} \pmod p$
for all $i$. Continuing this process we construct $d'$ sequences 
$(\lambda_{i,n})_{n \geq 0}$ with the property that
$$x \equiv \lambda_{1,n} e_1 + 
\lambda_{2,n} e_2 + \cdots + \lambda_{d',n} e_{d'} \pmod{p^nL}$$
and $\lambda_{i,n+1} \equiv \lambda_{i,n} \pmod {p^n}$ for all $n$
and $i$. The latter congruence shows that, for all $i$, the sequence 
$(\lambda_{i,n})_{n \geq 0}$ is Cauchy and then converges
to some $\lambda_i \in \Zp$. Passing to the limit, we find that these 
$\lambda_i$'s furthermore satisfy $x = \lambda_1 e_1 + \cdots + 
\lambda_{d'} e_{d'}$.

It now remains to prove that the $e_i$'s generate $E$ over $\Qp$.
For this, remark that any vector $x \in E$ can be written as $x =
p^v y$ with $\Vert y \Vert_E \leq 1$, \emph{i.e.} $y \in B_E(1)$.
By what we have done before, we know that $y$ can be written as a
linear combination of the $e_i$'s. Hence the same holds for $x$.
\end{proof}

\begin{rem}
Proposition~\ref{prop:nakayama} is a particular case of Nakayama's
Lemma (which is a classical result of commutative
algebra). We refer to~\cite{Ei95} for a general introduction to 
commutative algebra including an exposition of Nakayama's
Lemma in a much more general context.
\end{rem}

Proposition~\ref{prop:nakayama} has other important consequences. It 
shows for instance that the inclusion of a lattice $H$ in the ambient 
space $E$ is homeomorphic to the inclusion of $\Zp^d$ in $\Qp^d$ 
(where $d$ is the dimension of $E$). In particular lattices are all
open and compact at the same time.

Three other consequences are enumerated in the next corollary.

\begin{cor}
\label{cor:complete}
\begin{enumerate}[(i)]
\renewcommand{\itemsep}{0pt}
\item
All finite dimension normed vector spaces over $\Qp$ are complete.
\item
All norms over a given finite dimension vector space over $\Qp$
are equivalent.
\item
All linear applications between finite dimensional vector spaces over 
$\Qp$ are continuous.
\end{enumerate}
\end{cor}

\begin{proof}
Let $E$ be a finite dimension normal vector space over $\Qp$. Use
Proposition \ref{prop:nakayama} to 
pick a basis $e_1, \ldots, e_d$ of $E$ whose $\Zp$-span
is $B_E(1)$. We claim that an element $x \in E$ lies in the ball
$B_E(p^{-n})$ if and only if all its coordinates on the basis $(e_1,
\ldots, e_d)$ are divisible by $p^n$. Indeed the latter assertion
implies easily the former by factoring out $p^n$. Conversely, let
$x \in B_E(p^{-n})$. Then $p^{-n} x$ has norm at most $1$ and thus 
can be written as
$p^{-n} x = \lambda_1 e_1 + \cdots + \lambda_d e_d$
with $\lambda_i \in \Zp$ for all $i$. Multiplying by $p^n$ on both 
side and identifying coefficients, we get the claim.

Let $(x_n)_{n \geq 0}$ be a Cauchy sequence with values in $E$. For
all $n$, write $x_n = \lambda_{n,1} e_1 + \cdots + \lambda_{n,d} e_d$
with $\lambda_{n,i} \in \Qp$. It follows from the result we have 
proved in the previous paragraph that the sequences $(\lambda_{n,i})
_{n \geq 0}$ are Cauchy for all $i$. They thus converge and hence so 
does $(x_n)_{n \geq 0}$. This proves (i).

The two other assertions are easy (after Proposition 
\ref{prop:nakayama}) and left to the reader.
\end{proof}

\subsubsection{Computation with lattices}
\label{sssec:complattices}

The algebraic side of lattices provides the tools for representing and 
manipulating $p$-adic lattices on computers (at least if the underlying 
vector space $E$ is reasonable) in a quite similar fashion as usual
integral lattices are represented and manipulated \emph{via} integral
matrices.

For simplicity, let us expose the theory in the case where $E = \Qp^d$ 
(endowed with the infinite norm). We then represent a lattice $H \subset 
E$ by the matrix whose row vectors form a 
basis of $L$ (in the sense of Definition \ref{def:latticealg}). 
Equivalently, we can take the matrix, in the canonical basis, of the 
linear transformation $f$ mentioned just above Definition~\ref{def:latticeana}.

For example, if $d = 4$ and $H$ is generated by the vectors:
\begin{align*}
e_1 & = 
\big(\ldots0101110000,\, \ldots0011100000,\, \ldots1011001000,\, \ldots0011000100\big) \\
e_2 & = 
\big(\ldots1101011001,\, \ldots1101000111,\, \ldots0111001010,\, \ldots0101110101\big) \\
e_3 & = 
\big(\ldots0111100011,\, \ldots1011100101,\, \ldots0010100110,\, \ldots1111110111\big) \\
e_4 & = 
\big(\ldots0000111101,\, \ldots1101110011,\, \ldots0011010010,\, \ldots1001100001\big)
\end{align*}
we build the matrix
\begin{equation}
\label{eq:examplelattice}
\begin{array}{r@{\hspace{0.5ex}}l}
M & = 
\small \left(\begin{array}{rrrr}
\ldots0101110000 & \ldots0011100000 & \ldots1011001000 & \ldots0011000100 \\
\ldots1101011001 & \ldots1101000111 & \ldots0111001010 & \ldots0101110101 \\
\ldots0111100011 & \ldots1011100101 & \ldots0010100110 & \ldots1111110111 \\
\ldots0000111101 & \ldots1101110011 & \ldots0011010010 & \ldots1001100001
\end{array}\right)
\end{array}
\end{equation}

\begin{rem}
By convention, our vectors will always be row vectors.
\end{rem}

\noindent
The matrix $M$ we obtain this way is rather nice but we can further 
simplify it using Hermite reduction~\cite[\S 2.4]{Co93}. 
In the $p$-adic setting, Hermite reduction takes the following form.

\begin{theo}[$p$-adic Hermite normal form]
\label{theo:hermitered}
Any matrix $M \in \GL_d(\Qp)$ can be uniquely written as a product
$M = U A$ where $U \in \GL_d(\Zp)$ and $A$ has the shape:
$$A = \left(\begin{matrix}
p^{n_1} & a_{1,2} & \cdots & \cdots & a_{1,d} \\
0 & p^{n_2} & \ddots & & \vdots \\
\vdots & \ddots & \ddots & \ddots & \vdots \\
\vdots & & \ddots & p^{n_{d-1}} & a_{d-1,d} \\
0 & \cdots & \cdots & 0 & p^{n_d} 
\end{matrix} \right)$$
where the $n_i$'s are relative integers and the $a_{i,j}$'s are 
rational numbers of the form $a_{i,j} = \frac {b_{i,j}}{p^{v_{i,j}}}$
with $0 \leq b_{i,j} < p^{n_j + v_{i,j}}$.

The matrix $A$ is called the \emph{Hermite normal form} of $M$.
\end{theo}

\begin{rem}
The left multiplication by the matrix $U$ corresponds to operations on
the rows of $M$, that are operations on the vectors $e_i$.
The invertibility of $U$ over $\Zp$ ensures that the row vectors of
$A$ continue to generate the lattice $H$ we have started with.
\end{rem}

The proof of Theorem~\ref{theo:hermitered} is constructive and can be 
done by row-echelonizing the matrix $M$. Instead of writing it down for 
a general $M$, let us just show how it works on the example 
\eqref{eq:examplelattice}. We first select in the first column an entry 
with minimal valuation, we then move it to the top left corner by swapping rows 
and we normalize it so that it becomes a power of $p$ by rescaling the 
first row by the appropriate invertible element of $\Zp$.
In our example, one can for instance choose the second entry of the 
first column. After swap and renormalization, we get:
$$\small \left(\begin{array}{rrrr}
               1 & \ldots1110011111 & \ldots0011011010 & \ldots1101111101 \\
\ldots0101110000 & \ldots0011100000 & \ldots1011001000 & \ldots0011000100 \\
\ldots0111100011 & \ldots1011100101 & \ldots0010100110 & \ldots1111110111 \\
\ldots0000111101 & \ldots1101110011 & \ldots0011010010 & \ldots1001100001
\end{array}\right)$$
We now use the top left entry as pivot to clear the other entries of
the first column, obtaining this way the new matrix:
$$\small \left(\begin{array}{rrrr}
1 & \ldots1110011111 & \ldots0011011010 & \ldots1101111101 \\
0 & \ldots0001010000 & \ldots0101101000 & \ldots0100010100 \\
0 & \ldots0111101000 & \ldots0101011000 & \ldots1100100000 \\
0 & \ldots1010010000 & \ldots0011100000 & \ldots0011011000
\end{array}\right)$$
Remark that these row operations do not affect the $\Zp$-span
of the row vectors (\emph{i.e.} they correspond to a transformation
matrix $U$ which lies in $\GL_d(\Zp)$).
We continue this process with the $3 \times 3$ matrix obtained by erasing the first row
and the first column. Since we have forgotten the first row, the smallest
valuation of an entry of the second column is now $3$ and the corresponding
entry is located on the third row. We then swap the second and the third
rows, rescale the (new) second row in order to put $2^3$ on the diagonal
and use this value $2^3$ as pivot to cancel the remaining entries on the
second column. After these operations, we find:
$$\small \left(\begin{array}{rrrr}
1 & \ldots1110011111 & \ldots0011011010 & \ldots1101111101 \\
0 & 2^3 & \ldots0000111000 & \ldots0110100000 \\
0 &   0 & \ldots1100111000 & \ldots0011010100 \\
0 &   0 & \ldots1011110000 & \ldots0001011000
\end{array}\right)$$
Iterating again one time this process, we arrive at:
$$\small \left(\begin{array}{rrrr}
1 & \ldots1110011111 & \ldots0011011010 & \ldots1101111101 \\
0 & 2^3 & \ldots0000111000 & \ldots0110100000 \\
0 &   0 & 2^3 & \ldots000001100 \\
0 &   0 &   0 & \ldots111110000
\end{array}\right)$$
Interestingly, observe that the precision on the last two entries of the
last column has decreased by one digit. This is due to the fact that the
pivot $2^3$ was not the element with the smallest valuation on its
\emph{row}. This loss of precision may cause troubles only when the 
initial matrix $M$ is known at precision $O(p^N)$ with $N \leq \max_i\, 
n_i$; in practice such a situation very rarely happen and will never
appear in this course. 

The next step of Hermite reduction is the normalization of the bottom
right entry:
$$\small \left(\begin{array}{rrrr}
1 & \ldots1110011111 & \ldots0011011010 & \ldots1101111101 \\
0 & 2^3 & \ldots0000111000 & \ldots0110100000 \\
0 &   0 & 2^3 & \ldots000001100 \\
0 &   0 &   0 & 2^4
\end{array}\right)$$
It remains now to clean up the upper triangular part of the matrix. For this, we
proceed again column by column by using the pivots on the diagonal. Of
course there is nothing to do for the first column. We now reduce the 
$(1,2)$ entry modulo $2^3$ which is the pivot located on the same column.
In order to do so, we add to the first row the appropriate multiple 
of the second row. We obtain this way the new matrix:
$$\small \left(\begin{array}{rrrr}
1 &   7 & \ldots1110110010 & \ldots0010011101 \\
0 & 2^3 & \ldots0000111000 & \ldots0110100000 \\
0 &   0 & 2^3 & \ldots000001100 \\
0 &   0 &   0 & 2^4
\end{array}\right)$$
We emphasize that the $(1,2)$ entry of the matrix is now \emph{exact}!
Repeating this procedure several times we end up finally with
$$\small \left(\begin{array}{rrrr}
1 &   7 &   2 &   5 \\
0 & 2^3 &   0 &  12 \\
0 &   0 & 2^3 &  12 \\
0 &   0 &   0 & 2^4
\end{array}\right)$$
which is the expected Hermite normal form.

\begin{rem}
The algorithmic problem of computing the Hermite normal form has been 
widely studied over the integers and we now know much more efficient 
algorithms (taking advantage of fast matrix multiplication) to solve it 
\cite{KaVi05}; these algorithms extend without trouble to the case 
of $\Zp$ (with the same complexity). Other algorithms specially 
designed for the $p$-adic setting are also available~\cite{Ca12}.
\end{rem}

Here is a remarkable corollary of Theorem~\ref{theo:hermitered}
(compare with Proposition \ref{prop:intervalcomputer}):

\begin{cor}
Lattices in $\Qp^d$ are representable on computers by exact data.
\end{cor}

Operations on lattices can be performed on Hermite forms without 
difficulty. The sum of two lattices, for example, is computed by 
concatenating the two corresponding Hermite matrices and re-echelonizing. 
In a similar fashion, one can compute the image of a lattice under a 
surjective\footnote{Surjectivity ensures that the image remains a 
lattice in the codomain.} linear mapping $f$: we apply $f$ to each 
generator and echelonize the (possibly rectangular) matrix obtained this 
way.

\subsubsection{A small excursion into $p$-adic analysis}
\label{sssec:classC1}

Another important ingredient we will need is the notion of 
differentiable functions in the $p$-adic world. All the material 
presented in this section is classical. We refer to~\cite[\S I.4]{Sc11} 
and~\cite{Co10} for a much more detailed and systematic exposition.

\paragraph{The case of univariate functions}

The notion of differentiability at a given point comes with no surprise: 
given an open subset $U \subset \Qp$, $x \in U$ and a function $f : U 
\to \Qp$, we say that $f$ is differentiable at $x$ if $\frac{f(x) - 
f(y)} {x-y}$ has a limit when $y$ tends to $x$. 
When $f$ is differentiable at $x$, we define:
$$f'(x) = \lim_{y \to x} \, \frac{f(x) - f(y)} {x-y}.$$
Alternatively one can define the derivative using Taylor expansion: one
says that $f$ is differentiable at $x$ is there exists $f'(x) \in \Qp$
for which:
\begin{equation}
\label{eq:defdiff}
\big|f(y) - f(x) - (y{-}x) f'(x)\big| = o\big(|y{-}x|\big)
\quad \text{for } y \to x.
\end{equation}
Usual formulas for differentiating sums, products, composed functions,
\emph{etc.} extend \emph{verbatim} to the $p$-adic case.

The fact that $\Qp$ is (highly) disconnected has however unpleasant 
consequences. For instance, the vanishing of $f'$ on an interval does
not imply the constancy of $f$. Indeed, consider as an example the
indicator function of $\Zp$: it is a locally constant function defined 
on $\Qp$. It is then differentiable with zero derivative but it
is not constant on $\Qp$.
One can cook up even worse examples. Look for instance at the 
function $f : \Zp \to \Zp$ defined by:
$$a_0 + a_1 p + \cdots + a_n p^n + \cdots \quad \mapsto \quad
a_0 + a_1 p^2 + \cdots + a_n p^{2n} + \cdots$$
where the $a_i$'s are the digits, \emph{i.e.} they lie in the range
$[0,p{-}1]$. A straightforward computation shows that $|f(x) - f(y)| 
= |x-y|^2$, \emph{i.e.} $f$ is $2$-Hölder. In particular, $f$ is
differentiable everywhere and $f'$ vanishes on $\Zp$. Yet $f$ is 
apparently injective and \emph{a fortiori} not constant on any
interval.

\medskip

The notion of $p$-adic function of class $C^1$ is more subtle. Indeed 
the notion of differentiable function with continuous derivative is of 
course well defined but not that interesting in the $p$-adic setting. 
Precisely, this definition is too weak and encompasses many 
``irregular'' functions. A more flexible definition is the
following.

\begin{deftn}
\label{def:classC1}
Let $U$ be an open subset of $\Qp$ and $f : U \to \Qp$ be a function.
We say that $f$ is of class $C^1$ on $U$ if there exists a function
$f' : U \to \Qp$ satisfying the following property: for all $a \in U$ 
and all $\varepsilon > 0$, there exists a neighborhood $U_{a,\varepsilon}
\subset U$ of $a$ on which:
\begin{equation}
\label{eq:defC1uni}
\big|f(y) - f(x) - (y{-}x) f'(a)\big| \leq  \varepsilon \cdot 
|y{-}x|.
\end{equation}
\end{deftn}

Given a function $f : U \to \Qp$ of class $C^1$ on $U$, it is clear
that $f$ is differentiable on $U$ and that the function $f'$ appearing in
Definition \ref{def:classC1} has to be the derivative of $f$. Moreover
$f'$ is necessarily continuous on $U$. Indeed consider $a \in U$ and
$\varepsilon > 0$. Let $b \in U_{a,\varepsilon}$. 
The intersection $U = U_{a,\varepsilon} \cap U_{b,\varepsilon}$ is 
open and not empty. It then contains at least two different points,
say $x$ and $y$. It follows from the definition that:
\begin{align*}
\big|f(y) - f(x) - (y{-}x) f'(a)\big| 
 & \leq \varepsilon \cdot |y-x| \smallskip \\ 
\text{and} \quad
\big|f(y) - f(x) - (y{-}x) f'(b)\big| 
 & \leq \varepsilon \cdot |y-x|. 
\end{align*}
Combining these inequalities, we derive
$\big| f'(b) - f'(a) \big| \leq \varepsilon$
which proves the continuity of $f'$ at $x$.
In the real setting, the continuity of $f'$ implies conversely the inequality 
\eqref{eq:defC1uni}\footnote{Indeed, by the mean value theorem we can write 
$f(y) - f(x) = (y-x)f'(g(x,y))$ for some $g(x,y) \in [x,y]$. Thus $|f(y) 
- f(x) - (y-x) f'(a)| = |y-x| \cdot |f'(g(x,y)) - f'(a)|$ and the 
conclusion follows from Heine--Cantor Theorem.} but this implication
fails in the $p$-adic world. A counter-example is given by the function
$f : \Zp \to \Qp$ defined by:
\begin{align*}
x & 
= a_v p^v + a_{v+1} p^{v+1} + \cdots + a_{2v} p^{2v} + a_{2v+1} p^{2v+1} + \cdots \\
\mapsto \quad f(x) & 
= a_v p^{2v} + \big(a_{2v} p^{2v} + a_{2v+1} p^{2v+2} + a_{2v+2} p^{2v+4} + \cdots \big)
\end{align*}
where the $a_i$'s are integers between $0$ and $p{-}1$ with $a_v \neq 0$.
One checks that $|f(x)| \leq |x|^2$ and $|f(x) - f(y)| = 
\frac{|x-y|^2}{|x|^2}$ if $|x-y| \leq |x|^2 < 1$. These inequalities 
ensure that $f$ is differentiable on $\Zp$ and that its derivative 
vanishes everywhere (and thus is continuous). On the other hand when 
$|x-y| = |x|^2 < 1$, we have $|f(x) - f(y)| = |x - y|$, showing that $f$ 
cannot be of class $C^1$.

\begin{rem}
\label{rem:altdefC1uni}
Alternatively, following~\cite{Co10}, one may define the notion of class 
$C^1$ for a function $f : U \to F$ by requiring the existence of a 
covering $(U_i)_{i \in I}$ of $U$ and of \emph{continuous} real-valued 
functions $\varepsilon_i : \R^+ \to \R^+$ with 
$\varepsilon_i(0) = 0$ such that:
\begin{equation}
\label{eq:defC1uni2}
\forall i \in I, \quad \forall x, y \in U_i, \quad 
\big|f(y) - f(x) - (y{-}x) f'(x)\big| \leq  |y{-}x| \cdot
\varepsilon_i\big(|y{-}x|\big)
\end{equation}
(compare with \eqref{eq:defC1uni}). In brief, a function $f$ is of class 
$C^1$ when the estimation \eqref{eq:defdiff} is \emph{locally uniform} 
on $x$. When the domain $U$ is compact (\emph{e.g.} $U = \Zp$), one 
can remove the word ``locally'', \emph{i.e.} one can forget about the 
covering and just take $I = \{\star\}$ and $U_\star = U$.
\end{rem}

\begin{proof}[Proof of the equivalence between the two definitions]
Assume first that $f$ satisfies the definition of Remark
\ref{rem:altdefC1uni}. Let $a \in U$ and $\varepsilon > 0$. We have
$a \in U_i$ for some $i$. Let $\delta$ be a positive real number such 
that $\varepsilon_i < \varepsilon$ on the interval $[0, \delta)$.
Define $U_{a, \varepsilon}$ as the intersection of $U_i$ with the
open ball of centre $a$ and radius $\delta$. On $U_{a,\varepsilon}$,
the estimation:
\begin{equation}
\label{eq:C1equiv}
\big|f(y) - f(x) - (y{-}x) f'(x)\big| \leq  \varepsilon \cdot 
|y{-}x|
\end{equation}
holds. It remains then to relate $f'(x)$ to $f'(a)$. In order to
do so, we write:
\begin{align*}
\big|f(x) - f(a) - (x{-}a) f'(a)\big| & \leq \varepsilon\cdot |x{-}a| \\
\big|f(a) - f(x) - (a{-}x) f'(x)\big| & \leq \varepsilon\cdot |a{-}x|.
\end{align*}
Combining these inequalities, we obtain $|f'(x) - f'(a)| \leq 
\varepsilon$. Re-injecting this new input in \eqref{eq:C1equiv}, 
we finally get \eqref{eq:defC1uni} as desired.

Conversely, assume that $f$ is of class $C^1$ in the sense of Definition
\ref{def:classC1}. Since the problem is local on the domain, we may
assume that $U$ is an open ball. Up to translating and rescaling $f$,
one may further suppose without loss of generality that $U = \Zp$. In
particular, note that $U$ is compact. We define the function 
$\varepsilon$ on $(0,\infty)$ by:
$$\varepsilon(\delta) = \sup_{\substack{x,y \in \Zp \\ 0 < |y{-}x| \leq
\delta}} \frac{|f(y)-f(x)-f'(x)(y{-}x)|}{|y{-}x|}.$$
Compacity ensures that the supremum is finite, so that $\varepsilon$
is well defined. We have to show that $\varepsilon$ goes to $0$ when 
$\delta$ goes to $0$. Let $\varepsilon' > 0$. By assumption, for all
$a \in \Zp$, there exists an open neighborhood $U_{a,\varepsilon'}$ of 
$a$ on which $\frac{|f(y)-f(x)-f'(x)(y{-}x)|}{|y{-}x|} \leq \varepsilon'$.
Up to shrinking $U_{a,\varepsilon'}$, one may assume that
$U_{a,\varepsilon'} = a + p^{n_a} \Zp$ for some positive integer $n_a$.
Now observe that the family of all $U_{a,\varepsilon'}$ when $a$ varies 
is a covering of $\Zp$; by compacity, one can extract from it a finite 
subcovering $(U_{a_i,\varepsilon'})_{1 \leq i \leq m}$ that continues 
to cover $\Zp$. If $n$ denotes the supremum of the $n_{a_i}$'s ($1 
\leq i \leq m$), we thus derive $\varepsilon(\delta) \leq \varepsilon'$
for $\delta \leq p^{-n}$ and we are done.
\end{proof}

\paragraph{The case of multivariate functions}

Let $E$ and $F$ be two finite dimensional normed vector spaces over
$\Qp$. We denote by $\calL(E,F)$ the space of $\Qp$-linear mappings 
from $E$ to $F$. The definition of differentiability and ``of class
$C^1$'' is mimicked from the univariate case.

\begin{deftn}
Let $U$ be an open subset of $E$.

A function $f : U \to F$ is \emph{differentiable} at the point 
$x \in U$ if there exists a linear mapping $df_x \in \calL(E,F)$
such that:
$$\Vert f(y) - f(x) - df_x(y{-}x)\Vert_F = o\big(\Vert y{-}x\Vert_E\big)
\quad \text{when} \quad y \to x.$$

A function $f : U \to F$ is \emph{of class $C^1$} on $U$ if there exists 
a function $df : U \to \calL(E,F)$, $x \mapsto df_x$ satisfying the 
following property: for all $v \in U$ and all $\varepsilon > 0$, there 
exists a neighborhood $U_{a,\varepsilon} \subset U$ of $a$ on which:
\begin{equation}
\label{eq:defC1}
\Vert f(y) - f(x) - df_a(y{-}x)\Vert_F \leq  \varepsilon \cdot 
\Vert y{-}x\Vert_E.
\end{equation}
\end{deftn}

Of course, if $f$ is of class $C^1$ on $U$, it is differentiable at 
every point $x \in U$. When $E = \Qp^n$ and $F = \Qp^m$ (or more
generally when $E$ and $F$ are equipped with distinguished bases)
the matrix of the linear mapping $df_x$ is
the Jacobian matrix $J(f)_x$ defined by:
$$J(f)_x = 
\left(\frac{\partial f_j}{\partial x_i}(x)\right)_
{1 \leq i \leq n, \, 1 \leq j \leq m}$$
where $x_1, \ldots, x_n$ are the coordinates of $E$ and $f_1, \ldots,
f_m$ are the components of $f$.

\begin{rem}
\label{rem:altdefC1multi}
Similarly to the univariate case (see Remark \ref{rem:altdefC1uni}), a function 
$f : U \to F$ is of class $C^1$ if and only if there exist a covering 
$(U_i)_{i \in I}$ of $U$ and some \emph{continuous} real-valued
functions $\varepsilon_i : \R^+ \to \R^+$ with $\varepsilon_i(0) = 0$ 
such that:
\begin{equation}
\label{eq:defC1multi2}
\forall i \in I, \quad \forall x, y \in U_i, \quad
\Vert f(y) - f(x) - (y{-}x) f'(x)\Vert_F \leq  \Vert y{-}x\Vert_E \cdot
\varepsilon_i\big(\Vert y{-}x\Vert_E \big).
\end{equation}
and one can just take $I = \{\star\}$ and $U_\star = U$ when $U$ is compact.
\end{rem}

\subsubsection{The precision Lemma}
\label{sssec:preclemma}

The precision Lemma is a result of $p$-adic analysis controlling how 
lattices transform under sufficiently regular mappings. Before stating 
it, we need a definition.

\begin{deftn}
\label{def:roundedsmall}
Let $E$ be a normed vector space over $\Qp$. 
Given $\rho \in (0,1]$ and $r > 0$, we say that a lattice $H \subset E$ 
is $\rho$-rounded, $r$-small if
$B_E(\rho r) \subset H \subset B_E(r)$.
\end{deftn}

When $E = \Qp^d$ (endowed with the infinite norm), one can determine 
$\rho$ and $r$ satisfying the conditions of Definition 
\ref{def:roundedsmall} by looking at any matrix $M$ representing the 
lattice $H$ (see \S \ref{sssec:complattices}). Indeed, if $n$ is an 
integer for which the matrix $p^n M$ has all its entries in $\Zp$, the 
lattice $p^n H$ is included in $\Zp^n = B_E(1)$, meaning that $H \subset 
B_E(p^n)$. Similarly $B_E(p^{-m}) \subset H$ whenever $m$ is an integer 
such that $p^m M^{-1} \in M_d(\Zp)$.
Therefore if $n$ is the smallest valuation of an entry of $M$ and $m$
is the smallest valuation of an entry of $M^{-1}$, the corresponding
lattice $H$ is $(p^{n+m})$-bounded, $p^n$-small.

\begin{theo}[Precision Lemma]
\label{theo:preclemma}
Let $E$ and $F$ be two finite-dimensional $p$-adic normed vector spaces
and let $f : U \rightarrow F$ be a function of class $C^1$ defined on an 
open subset $U$ of $E$.

Let $v \in U$ be such that $df_v$ is surjective.
Then, for all $\rho \in (0, 1]$, there exists a positive real number 
$\delta$ such that, for any $r \in (0,\delta)$, any lattice 
$\rho$-bounded $r$-small lattice $H$ satisfies:
\begin{equation}
\label{eq:preclemma}
f(v+H) = f(v) + df_v(H).
\end{equation}
\end{theo}

Unfortunately the precision Lemma is a bit technical; understanding its 
precise content is then probably not easy at first glance. In order to 
help the reader, let us say that the most important part of the 
precision Lemma is the conclusion, namely Eq.~\eqref{eq:preclemma}. This 
equation explains how $f$ transforms a ``shifted'' lattice (\emph{i.e.} 
a lattice translated by some vectors) and teaches us that $f$ transforms 
a shifted lattice into another shifted lattice! From people coming from 
the real world, this result should appear as really amazing: in the real 
case, the image of an ellipsoid under a function $f$ is in general 
definitely \emph{not} another ellipsoid (unless $f$ is affine). In the 
$p$-adic case, this happens for any function $f$ of class $C^1$ with 
\emph{surjective differential} (this assumption is important) and almost 
any lattice $H$ (the assumptions on $H$ are actually rather weak though 
technical). This is the magic of ultrametricity.

Below, we give the proof of the precision Lemma and discuss several 
extensions. We advise the reader who is more interested by applications 
(than by computations with $\varepsilon$) to skip the end of the 
subsection and go directly to \S \ref{ssec:optimalprec}, page 
\pageref{ssec:optimalprec}.

\begin{proof}[Proof of Theorem~\ref{theo:preclemma}]
Without loss of generality, we may assume $v=0$ and $f(0)=0$. 
From the surjectivity of $df_0$, we derive that there exists a positive 
constant $C$ such that $B_F(1) \subset df_0(B_E(C))$.
Pick $\varepsilon < \frac\rho C$, and choose a positive real number 
$\delta$ for which
\begin{equation}
\label{eq:C1preclemma}
\Vert f(b) - f(a) - df_0(b{-}a) \Vert_F \leq 
\varepsilon \cdot \Vert b{-}a\Vert_E
\end{equation}
on the ball $B_E(\delta)$.
Let $r \in (0, \delta)$. We suppose that $H$ is a lattice with
$B_E(\rho r) \subset H \subset B_E(r).$
We seek to show that $f$ maps $H$ surjectively onto $df_0 (H)$. We first 
prove that $f(H) \subset df_0 (H)$. Suppose $x \in H$. Applying
Eq.~\eqref{eq:C1preclemma} with $a = 0$ and $b = x$, we get
$\Vert f(x)-df_0(x) \Vert_F \leq \varepsilon \Vert x \Vert_E$.
Setting $y=f(x)-df_0(x)$, we have $\Vert y \Vert_F \leq
\varepsilon r$. The definition of $C$ implies that $B_F(\varepsilon r)
\subset df_0 (B_E(\rho r))$. Thus there exists $x' \in B_E(\rho r)$ such 
that $df_0 (x') =y$. Then $f(x)= df_0 (x-x') \in df_0 (H)$.

We now prove surjectivity. Let $y \in df_0 (H)$. Let $x_0 \in H$
be such that $y = df_0 (x_0)$. We inductively define two sequences
$(x_n)$ and $(z_n)$ by the following cross requirements:
\begin{itemize}
\renewcommand{\itemsep}{0pt}
\item $z_n$ is an element of $E$ satisfying $df_0(z_n) = y -
f(x_n)$ and $\Vert z_n \Vert_E \leq C \cdot \Vert y - f(x_n) \Vert_F$, and
\item $x_{n+1}=x_n+z_n$.
\end{itemize}
For convenience, let us also define $x_{-1} = 0$ and $z_{-1}=x_0.$ We claim that the
sequences $(x_n)$ and $(z_n)$ are well defined and take their values in
$H$. We do so by induction, assuming that $x_{n-1}$ and $x_n$ belong to $H$
and showing that $z_n$ and $x_{n+1}$ do as well. Noticing that
\begin{equation}
\label{eq:mainlemma}
\begin{aligned}
y - f(x_n) &= f(x_{n-1}) + df_0(z_{n-1}) - f(x_n) \\
&= f(x_{n-1}) - f(x_n) - df_0(x_{n-1} - x_n)
\end{aligned}
\end{equation}
we deduce using differentiability that
$\Vert y - f(x_n) \Vert_F \leq \varepsilon \cdot \Vert x_n - x_{n-1}
\Vert_E$.
Since we are assuming that $x_{n-1}$ and $x_n$ lie in $H \subset
B_E(r)$, we find $\Vert y - f(x_n) \Vert \leq \varepsilon r$. Thus
$\Vert z_n \Vert \leq C \cdot \varepsilon r < \rho r$ and then
$z_n \in H$. From the relation $x_{n+1} = x_n + z_n$, we finally
deduce $x_{n+1} \in H$.

Using \eqref{eq:mainlemma} and differentiability at $0$ once more,
we get
$$\Vert y  - f(x_n) \Vert
\leq \varepsilon \cdot \Vert z_{n-1} \Vert \leq \varepsilon C \cdot \Vert
y - f(x_{n-1}) \Vert,$$
for all $n > 0$.  Therefore, $\Vert y - f(x_n) \Vert = O(a^n)$ and
$\Vert z_n \Vert = O(a^n)$ for $a = \varepsilon C < \rho \leq 1$.
These conditions show that $(x_n)_{n \geq 0}$ is a Cauchy sequence, which converges 
since $E$ is complete (see Corollary \ref{cor:complete}). Write $x$ for the 
limit of the $x_n$; we have $x \in H$ because $H$ is closed (see again
Corollary \ref{cor:complete}). Moreover, 
$f$ is continuous on $H \subseteq U_\varepsilon$ since it is 
differentiable, and thus $y=f(x)$.
\end{proof}

For the applications we have in mind, Theorem \ref{theo:preclemma} is 
actually a bit too weak because it is not effective: the value of 
$\delta$ is not explicit whereas we will often need it in concrete 
applications. In the most general case, it seems difficult to say much 
more about $\delta$. Nevertheless, there are many cases of interest 
where the dependence on $\delta$ can be made explicit.
The simplest such case is that of multivariate polynomials and is
covered by the next proposition.

\begin{prop}
\label{prop:preclemmapoly}
Let $f : \Qp^n \to \Qp^m$ be a function whose coordinates are all 
multivariate polynomials with coefficients in $\Zp$. Let $v \in \Zp^n$ 
such that $df_v$ is surjective. Then Eq.~\eqref{eq:preclemma} holds as 
soon as $H \subset B_E(r)$ and $B_F(p r^2) \subset df_v(H)$ for some 
positive real number $r$.
\end{prop}

\begin{rem}
The case where $f$ is a multivariate polynomial with coefficients in
$\Qp$ reduces to the above proposition by multiplying $f$ by an
appropriate constant. Similarly, if the point $v$ does not lie in
$\Zp^n$ but in $\Qp^n$, we can shift and rescale the polynomial $f$ 
in order to place ourselves within the scope of application of
Proposition \ref{prop:preclemmapoly}.
\end{rem}

The proof of Proposition \ref{prop:preclemmapoly} is based on the
following Lemma.

\begin{lem}
\label{lem:multipolyC1}
Let $g : \Qp^n \to \Qp^m$ be a function whose coordinates are all
multivariate polynomials with coefficients in $\Zp$. Then
$$\Vert g(b) - g(a) - dg_0(b{-}a) \Vert \leq \max 
  \big(\Vert a \Vert, \Vert b \Vert \big) \cdot \Vert b{-}a \Vert$$
for all $a,b \in \Zp^n$.
\end{lem}

\begin{proof}
We may assume $m = 1$. By linearity, we may further assume that
$g(x_1, \ldots, x_n) = x_1^{\alpha_1} \cdots x_n^{\alpha_n}$. If the
sum of the $\alpha_i$'s is at most $1$, the quantity $g(b) - g(a) - 
dg_0(b{-}a)$ vanishes and the lemma is clear. Otherwise the differential
$dg_0$ vanishes and we write:
$$\begin{array}{l}
g(b_1, \ldots, b_n) - g(a_1, \ldots, a_n) \smallskip \\
\hspace{2em}
=  (b_1^{\alpha_1} - a_1^{\alpha_1}) b_2^{\alpha_2} \cdots b_n^{\alpha_n} 
 + a_1^{\alpha_1} (b_2^{\alpha_2} - a_2^{\alpha_2}) b_3^{\alpha_3} \cdots b_n^{\alpha_n} 
 + \cdots 
 + a_1^{\alpha_1} \cdots a_{n-1}^{\alpha_{n-1}} (a_n^{\alpha_n} - b_n^{\alpha_n}).
\end{array}$$
Noting that $b_i^{\alpha_i} - a_i^{\alpha_i} = 
(b_i - a_i) \sum_{k=0}^{\alpha_i-1} b_i^k a_i^{\alpha_i-1-k}$ and
remembering that $\alpha_1 + \cdots + \alpha_n \geq 2$, we get 
the announced result by applying the ultrametric triangular
inequality.
\end{proof}

\begin{proof}[Proof of Proposition \ref{prop:preclemmapoly}]
Let $H$ be a lattice satisfying the assumptions of the proposition. 
Let $n$ be the largest relative integer for which $p^n \leq r$. 
Then $B_E(r) = B_E(p^n)$ and $B_F(p^{2n+1}) \subset B_F(p r^2)$, so 
that $H \subset B_E(p^n)$ and $B_F(p^{2n+1}) \subset df_v(H)$.

Set $H' = p^n H$; clearly $H'$ is a lattice contained in $B_E(1)$. Let 
$\varphi : \Qp^n \to \Qp^n$ be a linear mapping taking bijectively 
$B_E(1)$ to $H'$. Consider the composite $g = f \circ (v + \varphi)$. 
From $H' \subset B_E(1)$, we deduce that the entries of the matrix of 
$\varphi$ (in the canonical basis) lie in $\Zp$. Since moreover $v \in
\Zp^n$, we deduce that all the coordinates of $g$ are given by
multivariate polynomials with coefficients in $\Zp$. Furthermore, we
have:
\begin{align*}
f(v+H) & = f(v + p^{-n} H') = f(v + \varphi(B_E(p^n)) = g(B_E(p^n)) \\
\text{and}\quad
f(v) + df_v(H) & = f(v) + df_v \circ \varphi(B_E(p^n)) = g(0) + dg_0(B_E(p^n)).
\end{align*}
We then have to prove that $g(B_E(p^n)) = g(0) + dg_0(B_E(p^n))$
knowing that $g$ is a multivariate polynomial function with coefficients
in $\Zp$ and $B_F(p^{n+1}) \subset dg_0(B_E(1))$. This can be done by
following the lines of the proof of Theorem \ref{theo:preclemma} and
refining Eq.~\eqref{eq:C1preclemma} by the estimation of Lemma
\ref{lem:multipolyC1}.
\end{proof}

Other results concerning more general $f$ and building on the machinery 
of Newton polygons are available in the literature. We refer the reader 
to~\cite{CaRoVa14} for the case of locally analytic functions and to 
\cite{CaRoVa15} for the case of solutions of partial differential equations of 
order $1$.

\subsection{Optimal precision and stability of algorithms}
\label{ssec:optimalprec}

The precision Lemma (Theorem \ref{theo:preclemma}) provides the 
mathematical tools for finding the intrinsic optimal loss/gain of 
precision of a given problem and then for studying the stability 
of numerical $p$-adic algorithms.

\subsubsection{The precision Lemma at work}
\label{sssec:preclemmaatwork}

Consider an ``abstract $p$-adic problem'' encoded by a mathematical 
function $\varphi : U \to F$ where $U$ is an open subset in a finite 
dimensional $p$-adic vector space $E$ and $F$ is another finite
dimensional $p$-adic vector space.
The modelling is straightforward: the domain $U$ is the space of 
inputs, the codomain $F$ is the space of outputs and $\varphi$ is 
the mapping sending an input to the expected output.
Many concrete problems fit into this framework: for instance $\varphi$ 
can be the function mapping a pair of $p$-adic numbers to their sum
or their product, but it can also be the function mapping a polynomial
(of a given degree) to its derivative, a matrix (of given dimensions) 
to its inverse, a family of multivariate polynomials (of given size and 
degrees) to its reduced Gröbner basis, \emph{etc.}

In order to apply the precision Lemma, we shall assume that $\varphi$ is 
of class $C^1$. This hypothesis is not restrictive at all in practice 
since in many common cases, the function $\varphi$ has much more 
regularity than that: it is often obtained by composing sums, products, 
divisions and if-statements so that it is locally given by 
multivariate rational fractions.
We suppose furthermore that $E$ and $F$ are endowed with distinguished 
bases $(e_1, \ldots, e_n)$ and $(f_1, \ldots, f_m)$ respectively. This
choice models the way the elements of $E$ and $F$ are represented on the 
computer. For instance if $E = \Qp[X]_{< n}$ (the vector space of
polynomials of degree less than $n$) and polynomials are represented
internally by the list of their coefficients, we will just choose the
canonical basis. On the contrary, if polynomials are represented as
linear combinations of the shape:
$$\lambda_0 + \lambda_1 X + \lambda_2 X(X{-}1) + \cdots +
\lambda_{n-1} X (X{-}1) \cdots (X{-}n{+}2)$$
then the distinguished basis we will choose is $(1, X, X(X{-}1), \ldots,
X (X{-}1) \cdots (X{-}n{+}2))$.

\paragraph{Context of zealous arithmetic}

Recall that, in the zealous point of view, $p$-adic numbers are 
modeled by intervals of the form $a + O(p^N)$. As a consequence the 
input of the problem we are studying will not be an actual element of 
$U$ but a quantity of the form:
$$\big(a_1 + O\big(p^{N_1}\big) \big) \cdot e_1 + 
\big(a_2 + O\big(p^{N_2}\big) \big) \cdot e_2 + \cdots +
\big(a_n + O\big(p^{N_n}\big) \big) \cdot e_n$$
which can be rewritten as the shifted lattice $v + H$ with:
\begin{align*}
v & = a_1 e_1 + a_2 e_2 + \cdots + a_n e_n \\
H & = \Span_{\Zp} \big( p^{N_1} e_1, \, p^{N_2} e_2, \cdots, \,
p^{N_n} e_n \big).
\end{align*}
Similarly the output takes the form:
$$\big(b_1 + O\big(p^{M_1}\big) \big) \cdot f_1 + 
\big(b_2 + O\big(p^{M_2}\big) \big) \cdot f_2 + \cdots +
\big(b_m + O\big(p^{M_m}\big) \big) \cdot f_m$$
where the $b_j$'s are the coordinates of $\varphi(v)$ and then do not 
depend (fortunately) on the algorithm we are using to evaluate 
$\varphi$. On the other hand, the $M_j$'s may --- and do --- depend on
it. The question we address here can be formulated as follows: what is 
the maximal value one can expect for the $M_j$'s? In other words: what 
is the optimal precision one can expect on the output (in terms of the 
initial precision we had on the inputs)?

The precision Lemma provides a satisfying answer to this question.
Indeed, for $j \in \{1, \ldots, m\}$, let $\pr_{F,j} : F \to \Qp$ be
the linear mapping taking a vector to its $j$-th coordinate (in the
distinguished basis $(f_1, \ldots, f_m)$ we have fixed once for all). 
From the precision Lemma applied to the composite $\varphi_j = \pr_{F,j} 
\circ \varphi$, we get:
\begin{equation}
\label{eq:preclemmazealous}
\varphi_j(v+H) =
\varphi_j(v) + d\varphi_{j,v}(H) = b_j + d\varphi_{j,v}(H) 
\end{equation}
as soon as $d\varphi_{j,v}$ is surjective (and $H$ satisfies some
additional mild assumptions that we will ignore for now). 
Under these assumptions, Eq.~\eqref{eq:preclemmazealous} holds and 
the equality signs in it show that the optimal precision $O(p^{M_j})$ we 
were looking for is nothing but $d\varphi_{j,v}(H)$; note that the 
latter is a lattice in $\Qp$ and then necessarily takes the form 
$p^{M_j}\Zp = O(p^{M_j})$ for some relative integer $M_j$.
Note moreover that the surjectivity of $d\varphi_{j,v}$ is equivalent
to its non-vanishing (since $d\varphi_{j,v}$ takes its value in a
one-dimensional vector space). 
Furthermore, by the chain rule, we have $d\varphi_{j,v} = \pr_{F,j} 
\circ d\varphi_v$ because $\pr_{F,j}$ is linear and then agrees with
its differential at any point.

Finding the optimal precision $O(p^{M_j})$ can then be done along the
following lines: we first compute the Jacobian matrix $J(\varphi)_v$
of $\varphi$ at the point $v$ and form the product:
\begin{equation}
\label{eq:jacobianzealous}
A = \left( \begin{matrix}
p^{N_1} & & \\ & \ddots & \\ & & p^{N_n}
\end{matrix}\right) \cdot J(\varphi)_v
\end{equation}
whose rows generate $d\varphi_{j,v}(H)$. The integer $M_j$ appears
then as the smallest valuation of an entry of the $j$-th column of
$A$ (unless this column vanishes in which case the precision Lemma
cannot be applied). Of course, the computation of the Jacobian matrix 
can be boring and/or time-consuming; we shall see however in
\S \ref{sssec:diffexamples} below that, in many situations, simple
mathematical arguments provide closed formulas for $J(\varphi)_v$,
avoiding this way their direct computation.

\subparagraph{The notion of diffused digits of precision.}

One important remark is that the lattice:
$$H' = \Span_{\Zp} \big( p^{M_1} f_1, \, p^{M_2} f_2, \cdots, \,
p^{M_m} e_m \big) \subset F$$
(for the $M_j$'s we have constructed) does satisfy $f(v+H) \subset f(v) 
+ H'$ but does \emph{not} satisfy $f(v+H) = f(v) + H'$ in general.
In other words, although we cannot expect more correct digits on each 
component of the output separately, the precision on the output is 
globally not optimal. Below is a very simple example illustrating this 
point.

\begin{ex*}
Assume that we want to evaluate an affine polynomial at the two points 
$0$ and $p$. The function $\varphi$ modeling this problem is:
$$\varphi : \Qp[X]_{\leq 1} \to \Qp^2, 
\quad P(X) = a X + b \mapsto (P(0), P(p)) = (b,\,ap{+}b).$$
It is a linear function so its Jacobian is easy to compute. We find
$J(\varphi) = \Big( {\small \begin{matrix} 0 & p \\ 1 & 1
\end{matrix}} \Big)$.
We assume furthermore that the inputs $a$ and $b$ are both given at
precision $O(p^N)$. We then find that the optimal precision on $P(0)$ and 
$P(p)$ is $O(p^N)$ as well. Nevertheless the determinant of $J(\varphi)$
is apparently $p$; hence $J(\varphi)$ is not invertible in $M_2(\Zp)$ 
and its row vectors generate a lattice which is strictly smaller than 
$\Zp^2$. 
What does happen concretely? The main observation is that the difference 
$P(p) - P(0) = ap$ can be computed at precision $O(p^{N+1})$ because
of the multiplication by $p$. There is thus one more digit that we know
but this digit is ``diffused'' between $P(0)$ and $P(p)$ and only appears
when we look at the difference.
\end{ex*}

The phenomenon enlighten in the above example is quite general: \emph{we 
cannot in general reach the optimal precision by just giving individual 
precision on each coordinate of the output (we say that some digits of 
precision are diffused over several coordinates)... but this becomes 
always\footnote{At least when the precision Lemma applies...} possible 
after a suitable base change (which recombines the diffused digits of 
precision into actual digits).} In order to make this sentence more
precise, we 
first formulate a rigorous definition concerning the diffusion of the 
digits of precision.

\begin{deftn}
\label{def:diffuseddigits}
Let $F$ be a finite dimensional vector space over $\Qp$ with basis
$(f_1, \ldots, f_m)$.

\smallskip

\noindent
We say that a lattice $H$ is \emph{diagonal} with respect to 
$(f_1, \ldots, f_m)$ if it has the shape
$H = \Span_{\Zp} \big(p^{\nu_1} f_1, \ldots, p^{\nu_m} f_m\big)$ for
some relative integers $\nu_1, \ldots, \nu_m$.

\smallskip

\noindent
The \emph{number of diffused digits} (with respect to $(f_1, \ldots,
f_m)$) of a given lattice $H$ is the 
logarithm in base $p$ of the index of $H$ in the smallest diagonal
lattice containing $H$.
\end{deftn}

Roughly speaking, the number of diffused digits of a lattice $H$ with 
respect to the basis $(f_1, \ldots, f_m)$ quantifies the quality of the 
basis for expressing the precision encoded by $H$: if $H$ has 
no diffused digit, it is represented by a diagonal matrix and the 
precision splits properly into each coordinate. On the contrary, if 
$H$ has diffused digits, writing the precision in the system of 
coordinates associated with $(f_1, \ldots, f_m)$ leads to non optimal 
results.
We remark furthermore that there always exists a basis in which the 
number of diffused digits of a given lattice $H$ is zero: it suffices to 
take a basis whose $\Zp$-span is $H$ (such a basis always exists by 
Proposition \ref{prop:nakayama}).

Before going further, let us emphasize that, given a set of generators 
of $H$, it is not difficult to compute effectively its number of 
diffused digits of precision. Indeed, let $M$ be a matrix whose row
vectors span $H$.
Pick in addition a diagonal lattice $H' = \Span_{\Zp} \big(p^{\nu_1} f_1, \ldots, 
p^{\nu_m} f_m\big)$ with $H \subset H'$. By projecting on the $j$-th 
coordinate we see that $\nu_j$ is at least the smallest valuation $v_j$ of 
an entry of the $j$-th column of $M$. Conversely, one easily checks that 
the diagonal lattice $\Span_{\Zp} \big(p^{v_1} f_1, \ldots, p^{v_m} 
f_m\big)$ contains $H$. Thus the smallest diagonal matrix containing $H$ 
is $\Span_{\Zp} \big(p^{v_1} f_1, \ldots, p^{v_m} f_m\big)$. It follows 
that the number of diffused digits of precision of $H$ with respect to 
the basis $(f_1, \ldots, f_m)$ is given by:
\begin{equation}
\label{eq:numberdiffused}
(v_1 + v_2 + \cdots + v_m) - \val_p(\det M).
\end{equation}
If $M$ is in Hermite normal form (see \S \ref{sssec:complattices}) this 
quantity vanishes if and only if $M$ is indeed diagonal.

\medskip

We now go back to our initial question: we want to analyze the optimal 
precision for the computation of $\varphi(v)$ for $v$ is given at 
precision $H = \Span_{\Zp} \big( p^{N_1} e_1, \, p^{N_2} e_2, \cdots, \, 
p^{N_n} e_n \big)$. The precision Lemma applied with $\varphi$ itself
(assuming that the hypotheses of the precision Lemma are fulfilled
of course) yields the equality:
$$\varphi(v+H) = \varphi(v) + d\varphi_v(H).$$
The number of diffused digits of precision resulting from the computation 
of $\varphi(v)$ is then the number of diffused digits of the lattice
$d\varphi_v(H)$. If we are working in the basis $(f_1, \ldots, f_n)$,
this number can be large, meaning that writing down the precision in 
this basis can be weak. However, if we are allowing a change of basis on 
the codomain, we can always reduce the number of diffused digits to $0$; 
in such a basis, the precision can be made optimal!

\paragraph{Context of lazy/relaxed arithmetic}

As we have explained in \S \ref{sssec:zealousVSrelaxed}, the point of 
view of lazy/relaxed arithmetic differs from the zealous one on the 
following point: instead of fixing the precision on the input and 
propagating it to the outputs, it fixes a target precision and infers 
from it the precision needed on the inputs.
Of course, in order to save time and memory, we would like to avoid 
the computation of unnecessary digits. Regarding precision, the
question then becomes: what is the minimal number of digits we have
to compute on the inputs in order to ensure a correct output at the
requested precision?

This question translates into our precision language as follows: we fix 
some relative integers $M_1, \ldots, M_m$ and look for relative
integers $N_1, \ldots, N_n$ for which:
\begin{equation}
\label{eq:lazypreclemma}
\varphi(v + H) \subset \varphi(v) + H'
\end{equation}
with $H = \Span_{\Zp} \big(p^{N_1} f_1, \ldots, p^{N_n} e_n\big)$
(the unknown) and 
$H' = \Span_{\Zp} \big(p^{M_1} f_1, \ldots, p^{M_m} f_m\big)$.
Assuming that the precision Lemma applies for $\varphi$ and $H$, 
Eq.~\eqref{eq:lazypreclemma} is equivalent to $d\varphi_v(H) \subset H'$,
\emph{i.e.} $H \subset d\varphi_v^{-1}(H')$. The problem then reduces
to find the largest diagonal lattice sitting in $d\varphi_v^{-1}(H')$.

The comparison with the zealous situation is instructive: in the zealous 
case, one had the lattice $d \varphi_v(H_{\zealous})$ and we were 
looking for the smallest diagonal lattice containing it. The problem is
now ``reversed'' and can actually be related to the zealous problem
using duality. Before going further, we need to recall basic facts about
duality and lattices.

\medskip

Given $V$ a finite dimensional $\Qp$-vector space, we define its 
\emph{dual} $V^\star$ as the space of $\Qp$-linear forms on $V$. We
recall that there is a canonical isomorphism between $V$ and its
bidual $V^{\star\star}$; it maps an element $x \in V$ to the linear
form taking $\ell \in V^\star$ to $\ell(x)$. 
We recall also that each basis $(v_1, \ldots, v_d)$ of $V$ has a
dual basis $(v_1^\star, \ldots, v_d^\star)$ defined by $v_j^\star(v_i)
= 1$ if $i = j$ and $v_j^\star(v_i) = 0$ otherwise.

If $L \subset V$ is any $\Zp$-module, its dual $L^\star$ is defined by 
as the subset of $V^\star$ consisting of linear forms $\ell : V \to \Qp$ 
for which $\ell(L) \subset \Zp$. Clearly $L^\star$ is a sub-$\Zp$-module 
of $V^\star$. One can check that the construction $L \mapsto L^\star$ is 
involutive (\emph{i.e.} $L^{\star\star} = L$), order-reversing 
(\emph{i.e.} if $L_1 \subset L_2$ then $L_2^\star \subset L_1^\star$) 
and preserves diagonal lattices (\emph{i.e.} if $L$ is a diagonal 
lattice with respect to some basis then $L^\star$ is a diagonal lattice 
with respect to the dual basis).

\begin{rem}
\label{rem:duallattice}
If $L$ is a lattice, so is $L^\star$. More precisely, if
$(v_1, \ldots, v_d)$ is a basis of $V$ and $M$ is a square matrix whose 
rows span $L$ (with respect to the above basis), it is easily checked 
that $L^\star$ is spanned by the rows of ${}^{\text{t}}\! M^{-1}$ (with 
respect to the dual basis). On the contrary if $L$ does not generate 
$V$ as a $\Qp$-vector space, the dual space $L^\star$ contains a
$\Qp$-line; conversely if $L$ contains a $\Qp$-line then $L^\star$
does not generate $V^\star$ over $\Qp$.
\end{rem}

Applying the above formalism to our situation, we find that the 
condition $H \subset d\varphi_v^{-1}(H')$ we want to ensure is 
equivalent to $d\varphi_v^{-1}(H')^\star \subset H^\star$. Our problem 
then reduces to finding the smallest diagonal lattice containing 
$d\varphi_v^{-1}(H')^\star$. By the analysis we have done in the
zealous context, the 
integer $N_i$ we are looking for then appears as the smallest valuation 
of an entry of $i$-th column of a matrix whose rows span the space 
$d\varphi_v^{-1}(H')^\star$ (with respect to the dual basis $(e_1^\star, 
\ldots, e_n^\star)$). Computing $d\varphi_v^{-1}(H')^\star$ using
Remark \ref{rem:duallattice} (and taking care of the possible kernel of 
$d\varphi_v$), we finally find that $N_i$ is the opposite of the 
smallest valuation of an entry of the $i$-th \emph{row} of the matrix:
$$J(\varphi)_v \cdot
\left( \begin{matrix}
p^{-M_1} & & \\ & \ddots & \\ & & p^{-M_m}
\end{matrix}\right)$$
where $J(\varphi)_v$ is the Jacobian matrix of $\varphi$ at $v$ (compare 
with \eqref{eq:jacobianzealous}).

\subsubsection{Everyday life examples revisited}
\label{sssec:diffexamples}

We revisit the examples of \S \ref{sssec:compexamples} in light of the 
theory of $p$-adic precision.

\paragraph{The $p$-power map}

As a training example, let us first examine the computation of the 
$p$-power map (though technically it was not treated in \S 
\ref{sssec:compexamples} but at the end of \S 
\ref{sssec:intervalVSfloat}). Recall that we have seen that raising an 
invertible $p$-adic integer $x$ at the power $p$ gains one digit of 
precision: if $x$ is given at precision $O(p^N)$, one can compute $x^p$ 
at precision $O(p^{N+1})$.

We introduce the function $\varphi : \Qp \to \Qp$ taking $x$ to $x^p$. 
It is apparently of class $C^1$ and its differential at $x$ is the 
multiplication by $\varphi'(x) = p x^{p-1}$. The Jacobian matrix of
$\varphi$ at $x$ is then the $1 \times 1$ matrix whose unique entry
is $p x^{p-1}$. By the results of \S \ref{sssec:preclemmaatwork} 
(see particularly Eq.~\eqref{eq:preclemmazealous}), the optimal 
precision on $x^p$ is the valuation of 
$$(p x^{p-1}) \times p^N = p^{N+1} x^{p-1}.$$
We recover this way the exponent $N{+}1$ thanks to the factor $p$ 
which appears in the Jacobian. Unlike Lemma \ref{lem:ppower},
the method we have just applied here teaches us in addition that
the exponent $N{+}1$ cannot be improved.

\paragraph{Determinant}

We take here $\varphi = \det : M_d(\Qp) \to \Qp$. We endow $M_d(\Qp)$ 
with its canonical basis and let $x_{i,j}$ denote the corresponding 
coordinate functions: for $A \in M_d(\Zp)$, $x_{i,j}(A)$ is the $(i,j)$ 
entry of $A$.
Fix temporarily a pair $(i,j)$. Expanding the determinant according 
to the $i$-th row, we find that $\varphi$ is an affine function of
$x_{i,j}$ whose linear term is $(-1)^{i+j} x_{i,j} \cdot \det G_{i,j}$
where $G_{i,j}$ is the matrix obtained from the ``generic'' matrix
$G = (x_{i,j})_{1 \leq i,j \leq d}$ by erasing the $i$-th row and the 
$j$-th column. Therefore
$\frac{\partial \varphi}{\partial x_{i,j}} = (-1)^{i+j} \det G_{i,j}$.
Evaluating this identity at some matrix $M \in M_d(\Qp)$, we get
the well-known formula:
\begin{equation}
\label{eq:partialdet}
\frac{\partial \varphi}{\partial x_{i,j}}(M) = (-1)^{i+j} \det M_{i,j}
\end{equation}
where $M_{i,j}$ is the matrix obtained from $M$ by erasing the $i$-th
row and the $j$-th column. The Jacobian matrix of $\varphi$ at $M$ is 
then the column matrix (whose rows are indexed by the pairs $(i,j)$
for $1 \leq i,j \leq d$) with entries $(-1)^{i+j} \det M_{i,j}$.
According to the results of \S \ref{sssec:preclemmaatwork} (see 
especially Eq.~\eqref{eq:preclemmazealous}), the optimal precision
for $\det M$, when the $(i,j)$ entry of $M$ is given at precision
$O(p^{N_{i,j}})$, is:
\begin{equation}
\label{eq:optdet}
O\big(p^{N'}\big) \quad \text{for} \quad
N' = \min_{i,j} \big( N_{i,j} + \val(M_{i,j})\big).
\end{equation}
Remember of course that this conclusion is only valid when the
assumptions of the precision Lemma are all fulfilled. Nonetheless,
relying on Proposition \ref{prop:preclemmapoly} (after having noticed that the 
determinant is a multivariate polynomial function), we can further 
write down explicit sufficient conditions for this to hold. These
conditions take a particularly simple form when all the entries of
$M$ are given at the \emph{same} precision $O(p^N)$: in this case,
we find that the precision \eqref{eq:optdet} is indeed the optimal
precision as soon as $N' > \val(\det M)$, \emph{i.e.} as soon as it 
is possible to ensure that $\det M$ does not vanish.

\medskip

The general result above applies in particular to the matrix $M$ 
considered as a running example in \S \ref{sssec:compexamples} (see 
Eq.~\eqref{eq:matrixM}):
\begin{equation}
\label{eq:matrixM2}
\begin{array}{r@{\hspace{0.5ex}}l}
M & = 
\small \left(\begin{array}{rrrr}
\ldots0101110000 & \ldots0011100000 & \ldots1011001000 & \ldots0011000100 \\
\ldots1101011001 & \ldots1101000111 & \ldots0111001010 & \ldots0101110101 \\
\ldots0111100011 & \ldots1011100101 & \ldots0010100110 & \ldots1111110111 \\
\ldots0000111101 & \ldots1101110011 & \ldots0011010010 & \ldots1001100001
\end{array}\right) \in M_4(\Z_2).
\end{array}
\end{equation}
Recall that all the entries of $M$ were given at the same precision
$O(2^{10})$.
A simple computation (using zealous arithmetic) shows that the comatrix 
of $M$ (\emph{i.e.} the matrix whose $(i,j)$ entry is $(-1)^{i+j} \det 
M_{i,j}$) is:
$$\begin{array}{r@{\hspace{0.5ex}}l}
\Com(M) & =
\small \left(\begin{array}{rrrr}
\ldots1111000000 & \ldots0001100000 & \ldots1101000000 & \ldots1001100000 \\
\ldots0111000000 & \ldots0011100000 & \ldots1001000000 & \ldots1011100000 \\
\ldots0111000000 & \ldots0001100000 & \ldots0111000000 & \ldots0011100000 \\
\ldots1010000000 & \ldots0111000000 & \ldots1110000000 & \ldots0011000000
\end{array}\right)
\end{array}$$
We observe that the minimal valuation of its entries is $5$ (it
is reached on the second column), meaning that the optimal precision
on $\det M$ is $O(2^{15})$. Coming back to \S \ref{sssec:compexamples}, 
we see that the interval floating-point arithmetic reached this
accuracy whereas zealous interval arithmetic
missed many digits (without further help).

\subparagraph{Computation of determinants with optimal precision.}

It is worth remarking that one can design a simple algorithm that computes 
the determinant of a $p$-adic matrix at the optimal precision 
\emph{within the framework of zealous arithmetic} when the coefficients 
of the input matrix are all given at the same precision $O(p^N)$. This 
algorithm is based on a variation of the Smith normal form that we state 
now.

\begin{prop}
\label{prop:Smith}
Any matrix $M \in M_d(\Qp)$ admits a factorization of the following
form:
\begin{equation}
\label{eq:Smith}
M = P \cdot 
\left(\begin{matrix} a_1 & & \\ & \ddots & \\ & & a_d \end{matrix}\right)
\cdot Q
\end{equation}
where $P$ and $Q$ are have determinant $\pm 1$, the matrix in the middle 
is diagonal and its diagonal entries satisfy $\val(a_1) \leq \cdots \leq 
\val(a_d)$.
\end{prop}

\noindent
The factorization \eqref{eq:Smith} is quite interesting for us for two 
reasons. First of all, it reveals the smallest valuation of an entry
of the comatrix of $M$: it is 
\begin{equation}
\label{eq:optprecdet}
v = \val(a_1) + \val(a_2) + \cdots + \val(a_{d-1}).
\end{equation}
Indeed, taking exterior powers, one proves the entries of $\Com(M)$ 
and $\Com(P^{-1} M Q^{-1})$ span the same $\Zp$-submodule of $\Qp$.
Moreover, since $P^{-1} M Q^{-1}$ is diagonal, its comatrix is easily
computed; we find:
$$\Com(P^{-1} M Q^{-1}) =
\left(\begin{matrix} b_1 & & \\ & \ddots & \\ & & b_d \end{matrix}\right)
\quad \text{with} \quad
b_i = a_1 \cdots a_{i-1} a_{i+1} \cdots a_d$$
(recall that $\Com(A) = \det A \cdot A^{-1}$ if $A$ is invertible)
and we are done. 
The direct consequence of this calculation is the following: if the 
entries of $M$ are all given at precision $O(p^N)$, the optimal 
precision on $\det M$ is $O(p^{N+v})$ (assuming $v < N$).

The second decisive advantage of the factorization \eqref{eq:Smith} is 
that it can be computed without any loss of precision.
The underlying algorithm basically follows the lines of the Hermite
reduction we have presented in \S \ref{sssec:complattices}, except on
the three following points:
\begin{itemize}
\renewcommand{\itemsep}{0pt}
\item instead of choosing the pivot as an entry of smallest valuation 
on the working column, we choose as an entry of smallest valuation on
the whole matrix,
\item we do not only clear the entries behind the pivot but also the
entries located on the right of the pivot,
\item in order to keep the transformation matrices of determinant
$\pm 1$, we do not rescale rows.
\end{itemize}
For example, starting with our favorite matrix $M$ (see 
Eq.~\eqref{eq:matrixM2} above), we first select the $(2,1)$ entry
(which has valuation $0$), we put it on the top left corner and use
it as pivot to clear the entries on the first row and the first
column. Doing so, we obtain the intermediate form:
$$\small \left(\begin{array}{rrrr}
\ldots1101011001 & 0 & 0 & 0 \\
0 & \ldots0001010000 & \ldots0101101000 & \ldots0100010100 \\
0 & \ldots0111101000 & \ldots0101011000 & \ldots1100100000 \\
0 & \ldots1010010000 & \ldots0011100000 & \ldots0110011000
\end{array}\right)$$
and we retain that we have swapped two rows, implying that the
determinant of the above matrix is the opposite to the determinant
of $M$. We now select the $(2,4)$ entry (which as valuation $2$) 
and continue the reduction:
$$\small \left(\begin{array}{rrrr}
\ldots1101011001 & 0 & 0 & 0 \\
0 & \ldots0100010100 & 0 & 0 \\
0 & 0 & \ldots0100011000 & \ldots0101101000 \\
0 & 0 & \ldots0000110000 & \ldots0000110000
\end{array}\right)$$
We observe that we have not lost any precision in this step (all 
the entries of the matrix above are known at precision $O(2^{10})$),
and this 
even if our pivot had positive valuation. It is actually a general
phenomenon due to the fact that we always choose the pivot of
smallest valuation. Continuing this process, we end up with the
matrix:
$$\small \left(\begin{array}{rrrr}
\ldots1101011001 & 0 & 0 & 0 \\
0 & \ldots0100010100 & 0 & 0 \\
0 & 0 & \ldots0100011000 & 0 \\
0 & 0 & 0 & \ldots1101100000
\end{array}\right)$$
Now multiplying the diagonal entries, we find
$\det M = 2^{10} \times ...01101$ with precision $O(2^{15})$!

\medskip

For a general $M$, the same precision analysis works. Indeed, when the 
entries of $M$ are all given at precision $O(p^N)$, the reduction 
algorithm we have sketched above outputs a diagonal matrix whose 
determinant is the same as the one of $M$ and whose diagonal entries 
$a_1, \ldots, a_d$ are all known at precision $O(p^N)$. By Proposition 
\ref{prop:arithinterval} (see also Remark \ref{rem:arithinterval}),
zealous arithmetic can compute the product $a_1 a_2 \cdots a_d$ 
at precision $O(p^{N+w})$ with 
$$w = \val(a_1) + \cdots + \val(a_d) - \max_i \val(a_i).$$
We then recover the optimal precision given by Eq.~\eqref{eq:optprecdet}.

\paragraph{Characteristic polynomial}

The case of characteristic polynomials has many similarities with
that of determinants.
We introduce the function $\varphi : M_d(\Qp) \to \Qp[X]_{< d}$
that maps a matrix $M$ to $\varphi(M) = \det(X I_d - M) - X^d$ where
$I_d$ is the identity matrix of size $d$. The renormalization 
(consisting in subtracting $X^d$) is harmless but needed in order to 
ensure that $\varphi$ takes its values in a vector space (and not an
affine space) and has a surjective differential almost everywhere.
The partial derivatives of $\varphi$ are:
$$\frac{\partial \varphi}{\partial x_{i,j}}(M) = (-1)^{i+j} \cdot
\det (X I_d - M)_{i,j}$$
where $(X I_d - M)_{i,j}$ is the matrix obtained from $X I_d {-} M$ 
by deleting the $i$-th row and the $j$-th column (compare with 
Eq.~\eqref{eq:partialdet}). The Jacobian matrix $J(\varphi)_M$ of 
$\varphi$ at $M$
is then a matrix with $d^2$ rows, indexed by pairs $(i,j)$ with
$1 \leq i,j \leq d$, and $d$ columns whose row with index $(i,j)$
contains the coefficients of the polynomial $(-1)^{i+j} \cdot
\det (X I_d - M)_{i,j}$. 
Using again the results of \S\ref{sssec:preclemmaatwork}, the optimal
precision on each coefficient of the characteristic polynomial of
$M$ can be read off from $J(\varphi)_M$.

\medskip

Let us explore further our running example: the matrix $M$ given by 
Eq.~\eqref{eq:matrixM2} whose entries are all known at precision 
$O(2^{10})$. A direct computation leads to the Jacobian matrix 
displayed on the left of Figure~\ref{fig:jaccharpoly}.
\begin{figure}
{\scriptsize\hfill
$\begin{array}{cc@{}r@{\hspace{1.5em}}r@{\hspace{1.5em}}r@{\hspace{1.5em}}r@{\hspace{0.2ex}}c}
(i,j) && X^3 & X^2 & X & 1 \\
\rule[3ex]{0pt}{\baselineskip}
(1,1) &
\raisebox{0.2ex}{\multirow{16}{*}{$\left(\parbox[c][39ex]{0cm}{}\right.$}}
 &1 & \ldots0110110010 & \ldots0111111000 & \ldots0001000000 &
\raisebox{0.2ex}{\multirow{16}{*}{$\left.\parbox[c][39ex]{0cm}{}\right)$}} \\
(1,2) &&0 & \ldots0011100000 & \ldots1011010100 & \ldots1110100000 \\
(1,3) &&0 & \ldots1011001000 & \ldots1001001000 & \ldots0011000000 \\
(1,4) &&0 & \ldots0011000100 & \ldots1111100100 & \ldots0110100000 \\
(2,1) &&0 & \ldots1101011001 & \ldots1010010000 & \ldots1001000000 \\
(2,2) &&1 & \ldots1110001001 & \ldots1001001100 & \ldots1100100000 \\
(2,3) &&0 & \ldots0111001010 & \ldots0110011000 & \ldots0111000000 \\
(2,4) &&0 & \ldots0101110101 & \ldots0111111100 & \ldots0100100000 \\
(3,1) &&0 & \ldots0111100011 & \ldots1010000000 & \ldots1001000000 \\
(3,2) &&0 & \ldots1011100101 & \ldots1110100000 & \ldots1110100000 \\
(3,3) &&1 & \ldots0011101000 & \ldots1001000100 & \ldots1001000000 \\
(3,4) &&0 & \ldots1111110111 & \ldots1111100100 & \ldots1100100000 \\
(4,1) &&0 & \ldots0000111101 & \ldots0010111000 & \ldots0110000000 \\
(4,2) &&0 & \ldots1101110011 & \ldots1101011000 & \ldots1001000000 \\
(4,3) &&0 & \ldots0011010010 & \ldots1101001000 & \ldots0010000000 \\
(4,4) &&1 & \ldots1010100011 & \ldots0111010000 & \ldots1101000000
\end{array}$
\quad $\stackrel{\text{\tiny Hermite}}{\longrightarrow}$ \quad
$\begin{array}{c@{}r@{\hspace{1.5em}}r@{\hspace{1.5em}}r@{\hspace{1.5em}}r@{\hspace{0.5ex}}c}
& X^3 & X^2 & X & 1 \\
\rule[3ex]{0pt}{\baselineskip}
\raisebox{0.2ex}{\multirow{16}{*}{$\left(\parbox[c][39ex]{0cm}{}\right.$}}
& 1 & 0 & 0 & 0 &
\raisebox{0.2ex}{\multirow{16}{*}{$\left.\parbox[c][39ex]{0cm}{}\right)$}} \\
&{\color{black!50} 0} & 1 & 0 & 0 \\
&{\color{black!50} 0} & {\color{black!50} 0} & 2^2 & 0 \\
&{\color{black!50} 0} & {\color{black!50} 0} & {\color{black!50} 0} & 2^5 \\
&{\color{black!50} 0} & {\color{black!50} 0} & {\color{black!50} 0} & {\color{black!50} 0} \\
&{\color{black!50} 0} & {\color{black!50} 0} & {\color{black!50} 0} & {\color{black!50} 0} \\
&{\color{black!50} 0} & {\color{black!50} 0} & {\color{black!50} 0} & {\color{black!50} 0} \\
&{\color{black!50} 0} & {\color{black!50} 0} & {\color{black!50} 0} & {\color{black!50} 0} \\
&{\color{black!50} 0} & {\color{black!50} 0} & {\color{black!50} 0} & {\color{black!50} 0} \\
&{\color{black!50} 0} & {\color{black!50} 0} & {\color{black!50} 0} & {\color{black!50} 0} \\
&{\color{black!50} 0} & {\color{black!50} 0} & {\color{black!50} 0} & {\color{black!50} 0} \\
&{\color{black!50} 0} & {\color{black!50} 0} & {\color{black!50} 0} & {\color{black!50} 0} \\
&{\color{black!50} 0} & {\color{black!50} 0} & {\color{black!50} 0} & {\color{black!50} 0} \\
&{\color{black!50} 0} & {\color{black!50} 0} & {\color{black!50} 0} & {\color{black!50} 0} \\
&{\color{black!50} 0} & {\color{black!50} 0} & {\color{black!50} 0} & {\color{black!50} 0} \\
&{\color{black!50} 0} & {\color{black!50} 0} & {\color{black!50} 0} & {\color{black!50} 0} 
\end{array}$}
\hfill\null

\caption{The Jacobian of the characteristic polynomial at $M$}
\label{fig:jaccharpoly} 
\end{figure} 
We observe that the minimal valuation of an entry of the column of $X^j$ 
is $0$, $0$, $2$ and $5$ when $j$ is $3$, $2$, $1$ and $0$ respectively. 
Consequently the optimal precision on the coefficients of $X^3$, $X^2$, 
$X$ and on the constant coefficient are $O(2^{10})$, $O(2^{10})$, 
$O(2^{12})$ and $O(2^{15})$ respectively. Coming back to \S 
\ref{sssec:compexamples}, we notice that floating-point arithmetic
found all the correct digits.
We observe moreover that the Hermite normal form of the Jacobian matrix
(shown on the right on Figure~\ref{fig:jaccharpoly}) is diagonal. There
is therefore no diffused digit of precision for this example; in other
words the precision written as
$O(p^{10}) X^3 + O(p^{10}) X^2 + O(p^{12}) X + O(p^{15})$
is sharp.

\begin{rem}
Making more extensive tests, we conclude that $p$-adic floating point
arithemtic was pretty lucky with the above example: in general, it 
indeed sometimes finds all (or almost all) relevant digits but it may
also sometimes not be more accurate than zealous arithmetic. We refer
to \cite{CaRoVa17} for precise statistics thereupon.
\end{rem}

Let us now have a look at the matrix $N = I_4 + M$. Playing the same
game as before, we obtain the Jacobian matrix displayed on the left of
Figure~\ref{fig:jaccharpoly2}.
\begin{figure}
\hfill
{\scriptsize
$\begin{array}{cc@{}r@{\hspace{1.5em}}r@{\hspace{1.5em}}r@{\hspace{1.5em}}r@{\hspace{0.2ex}}c}
(i,j) && X^3 & X^2 & X & 1 \\
\rule[3ex]{0pt}{\baselineskip}
(1,1) &
\raisebox{0.2ex}{\multirow{16}{*}{$\left(\parbox[c][39ex]{0cm}{}\right.$}}
 &1 & \ldots0110101111 & \ldots1010010111 & \ldots1111111001 &
\raisebox{0.2ex}{\multirow{16}{*}{$\left.\parbox[c][39ex]{0cm}{}\right)$}} \\
(1,2)&&0 & \ldots0011100000 & \ldots0100010100 & \ldots0110101100 \\
(1,3)&&0 & \ldots1011001000 & \ldots0010111000 & \ldots0101000000 \\
(1,4)&&0 & \ldots0011000100 & \ldots1001011100 & \ldots1010000000 \\
(2,1)&&0 & \ldots1101011001 & \ldots1111011110 & \ldots1100001001 \\
(2,2)&&1 & \ldots1110000110 & \ldots1100111101 & \ldots0001011100 \\
(2,3)&&0 & \ldots0111001010 & \ldots1000000100 & \ldots0111110010 \\
(2,4)&&0 & \ldots0101110101 & \ldots1100010010 & \ldots0010011001 \\
(3,1)&&0 & \ldots0111100011 & \ldots1010111010 & \ldots0110100011 \\
(3,2)&&0 & \ldots1011100101 & \ldots0111010110 & \ldots1011100101 \\
(3,3)&&1 & \ldots0011100101 & \ldots0001110111 & \ldots0011100011 \\
(3,4)&&0 & \ldots1111110111 & \ldots1111110110 & \ldots1100110011 \\
(4,1)&&0 & \ldots0000111101 & \ldots0000111110 & \ldots0100000101 \\
(4,2)&&0 & \ldots1101110011 & \ldots0001110010 & \ldots1001011011 \\
(4,3)&&0 & \ldots0011010010 & \ldots0110100100 & \ldots1000001010 \\
(4,4)&&1 & \ldots1010100000 & \ldots0010001101 & \ldots0000010010
\end{array}$
\quad $\stackrel{\text{\tiny Hermite}}{\longrightarrow}$ \quad
$\begin{array}{c@{}r@{\hspace{1.5em}}r@{\hspace{1.5em}}r@{\hspace{1.5em}}r@{\hspace{0.5ex}}c}
& X^3 & X^2 & X & 1 \\
\rule[3ex]{0pt}{\baselineskip}
\raisebox{0.2ex}{\multirow{16}{*}{$\left(\parbox[c][39ex]{0cm}{}\right.$}}
& 1 & 0 & 1 & 30 &
\raisebox{0.2ex}{\multirow{16}{*}{$\left.\parbox[c][39ex]{0cm}{}\right)$}} \\
&{\color{black!50} 0} & 1 & 2 & 29 \\
&{\color{black!50} 0} & {\color{black!50} 0} & 2^2 & 28 \\
&{\color{black!50} 0} & {\color{black!50} 0} & {\color{black!50} 0} & 2^5 \\
&{\color{black!50} 0} & {\color{black!50} 0} & {\color{black!50} 0} & {\color{black!50} 0} \\
&{\color{black!50} 0} & {\color{black!50} 0} & {\color{black!50} 0} & {\color{black!50} 0} \\
&{\color{black!50} 0} & {\color{black!50} 0} & {\color{black!50} 0} & {\color{black!50} 0} \\
&{\color{black!50} 0} & {\color{black!50} 0} & {\color{black!50} 0} & {\color{black!50} 0} \\
&{\color{black!50} 0} & {\color{black!50} 0} & {\color{black!50} 0} & {\color{black!50} 0} \\
&{\color{black!50} 0} & {\color{black!50} 0} & {\color{black!50} 0} & {\color{black!50} 0} \\
&{\color{black!50} 0} & {\color{black!50} 0} & {\color{black!50} 0} & {\color{black!50} 0} \\
&{\color{black!50} 0} & {\color{black!50} 0} & {\color{black!50} 0} & {\color{black!50} 0} \\
&{\color{black!50} 0} & {\color{black!50} 0} & {\color{black!50} 0} & {\color{black!50} 0} \\
&{\color{black!50} 0} & {\color{black!50} 0} & {\color{black!50} 0} & {\color{black!50} 0} \\
&{\color{black!50} 0} & {\color{black!50} 0} & {\color{black!50} 0} & {\color{black!50} 0} \\
&{\color{black!50} 0} & {\color{black!50} 0} & {\color{black!50} 0} & {\color{black!50} 0} 
\end{array}$}
\hfill\null

\caption{The Jacobian of the characteristic polynomial at $I_4{+}M$}
\label{fig:jaccharpoly2}
\end{figure}
Each column of this matrix contains an entry of valuation zero. The 
optimal precision on each individual coefficient of the characteristic 
polynomial of $N$ is then no more than $O(p^{10})$. However the Hermite 
reduced form of the Jacobian is now no longer diagonal, meaning that 
diffused digits of precision do appear. One can moreover count them 
using Eq.~\eqref{eq:numberdiffused} (applied to the Hermite normal
form of the Jacobian): we find $7$. This means that the precision
$$O(p^{10}) X^3 + O(p^{10}) X^2 + O(p^{10}) X + O(p^{10})$$
is not sharp; more precisely, after a suitable base change on $\Qp[X]_{< 
d}$, it should be possible to visualize seven more digits. 
Given that $\chi_N(X) = \chi_M(X{-}1)$, this base change is of course
the one which is induced by the change of variable $X \mapsto X{-}1$.

\paragraph{LU factorization}

Given a positive integer $d$, define $L_d(\Zp)$ as the subspace of 
$M_d(\Zp)$ consisting of lower triangular matrices with zero diagonal.
Similarly let $U_d(\Zp)$ be the subspace space of $M_d(\Zp)$ 
consisting of upper triangular matrices. Clearly $L_d(\Zp)$ and 
$U_d(\Zp)$ are $p$-adic vector spaces of respective dimensions 
$\frac{d(d{-}1)}2$ and $\frac{d(d{+}1)}2$. Let $\calU$ be the open 
subset of $M_d(\Zp)$ defined by the non-vanishing of all principal 
minors. We define two mappings:
$$\varphi : \calU \to L_d(\Zp), \, M \mapsto L{-}I_d 
\quad \text{and} \quad 
\psi : \calU \to L_d(\Zp), \, M \mapsto U$$
where $M = LU$ is the (unique) LU factorization of $M$.
In order to compute the differentials of $\varphi$ and 
$\psi$\footnote{These two functions are indeed differentiable. One can 
prove this by noticing for instance that the entries of $L$ and $U$ are 
given by explicit multivariate rational fractions~\cite[\S 1.4]{Ho75}.}, we
simply differentiate the defining relation $M = LU$. We get this
way the relation $dM = dL \cdot U + L \cdot dU$. We rewrite it as:
$$L^{-1} \cdot dM \cdot U^{-1} = 
\big(L^{-1} \cdot dL\big) + \big(dU \cdot U^{-1}\big)$$
and observe that the first summand of the right hand side is strictly lower 
triangular whereas the second summand of upper triangular. We
therefore derive
$dL = L \cdot \text{Lo}\big(L^{-1} \cdot dM \cdot U^{-1}\big)$ and
$dU = \text{Up}\big(L^{-1} \cdot dM \cdot U^{-1}\big) \cdot U$
where $\text{Lo}$ (resp. $\text{Up}$) is the projection on the first
factor (resp. on the second factor) of the decomposition $M_d(\Zp)
= L_d(\Zp) \oplus U_d(\Zp)$. 
The differentials of $\varphi$ and $\psi$ at $M \in \calU$ are 
then given by the linear mappings:
$$\begin{array}{rl}
d \varphi_M : M_d(\Zp) \to L_d(\Zp), &
dM \mapsto L \cdot \text{Lo}\big(L^{-1} \cdot dM \cdot U^{-1}\big) \smallskip \\
d \psi_M : M_d(\Zp) \to U_d(\Zp), &
dM \mapsto \text{Up}\big(L^{-1} \cdot dM \cdot U^{-1}\big) \cdot U 
\end{array}$$
where $L$ and $U$ are the ``L-part'' and the ``U-part'' of the LU 
factorization of $M$ respectively.

\medskip

As an example, take again the matrix $M$ given by 
Eq.~\eqref{eq:matrixM2}. The Jacobian matrix of $\varphi$ at this point
is the matrix displayed on Figure~\ref{fig:jacLU}.
\begin{figure}
\hfill{\tiny
$\begin{array}{cc@{\hspace{1ex}}r@{\hspace{1ex}}r@{\hspace{1ex}}r@{\hspace{2ex}}r@{\hspace{0ex}}r@{\hspace{4ex}}r@{\hspace{2ex}}c}
 && (2,1) & (3,1) & (3,2) & (4,1) & (4,2) & (4,3) \\
\rule[3ex]{0pt}{\baselineskip}
(1,1) &
\raisebox{0.2ex}{\multirow{16}{*}{$\left(\parbox[c][39ex]{0cm}{}\right.$}} &
 \ldots11,\!00110111 & \ldots11,\!01001101 & \ldots1100101\phantom{,\!0} & \ldots00,\!10010011 & \ldots1100010 & \ldots001,\!01\phantom{0} &
\raisebox{0.2ex}{\multirow{16}{*}{$\left.\parbox[c][39ex]{0cm}{}\right)$}} \\
(1,2) && 0\phantom{,\!00000000} & 0\phantom{,\!00000000} & \ldots110101,\!1 & 0\phantom{,\!00000000} & \ldots011111 & \ldots100,\!01\phantom{0} \\
(1,3) && 0\phantom{,\!00000000} & 0\phantom{,\!00000000} & 0\phantom{,\!0} & 0\phantom{,\!00000000} & 0 & \ldots1011,\!1\phantom{00} \\
(1,4) && 0\phantom{,\!00000000} & 0\phantom{,\!00000000} & 0\phantom{,\!0} & 0\phantom{,\!00000000} & 0 & 0\phantom{,\!000} \\
(2,1) && \ldots111010,\!0111\phantom{0000} & \ldots000000\phantom{,\!00000000} & \ldots0111001110\phantom{,\!0} & \ldots000\phantom{,\!00000000} & \ldots11100010 & \ldots0110,\!111 \\
(2,2) && 0\phantom{,\!00000000} & 0\phantom{,\!00000000} & \ldots0100001001\phantom{,\!0} & 0\phantom{,\!00000000} & \ldots0011111 & \ldots1111,\!011 \\
(2,3) && 0\phantom{,\!00000000} & 0\phantom{,\!00000000} & 0\phantom{,\!0} & 0\phantom{,\!00000000} & 0 & \ldots10111,\!01\phantom{0} \\
(2,4) && 0\phantom{,\!00000000} & 0\phantom{,\!00000000} & 0\phantom{,\!0} & 0\phantom{,\!00000000} & 0 & 0\phantom{,\!000} \\
(3,1) && 0\phantom{,\!00000000} & \ldots111010,\!0111\phantom{0000} & \ldots00110100110\phantom{,\!0} & \ldots000\phantom{,\!00000000} & \ldots00000000 & \ldots1111,\!11\phantom{0} \\
(3,2) && 0\phantom{,\!00000000} & 0\phantom{,\!00000000} & \ldots0111011101\phantom{,\!0} & 0\phantom{,\!00000000} & \ldots0000000 & \ldots1000,\!11\phantom{0} \\
(3,3) && 0\phantom{,\!00000000} & 0\phantom{,\!00000000} & 0\phantom{,\!0} & 0\phantom{,\!00000000} & 0 & \ldots11010,\!1\phantom{00} \\
(3,4) && 0\phantom{,\!00000000} & 0\phantom{,\!00000000} & 0\phantom{,\!0} & 0\phantom{,\!00000000} & 0 & 0\phantom{,\!000} \\
(4,1) && 0\phantom{,\!00000000} & 0\phantom{,\!00000000} & 0\phantom{,\!0} & \ldots111010,\!0111\phantom{0000} & \ldots00110100110 & \ldots1010,\!011 \\
(4,2) && 0\phantom{,\!00000000} & 0\phantom{,\!00000000} & 0\phantom{,\!0} & 0\phantom{,\!00000000} & \ldots0111011101 & \ldots0100,\!111 \\
(4,3) && 0\phantom{,\!00000000} & 0\phantom{,\!00000000} & 0\phantom{,\!0} & 0\phantom{,\!00000000} & 0 & \ldots00100,\!01\phantom{0} \\
(4,4) && 0\phantom{,\!00000000} & 0\phantom{,\!00000000} & 0\phantom{,\!0} & 0\phantom{,\!00000000} & 0 & 0\phantom{,\!000} \\ \\
&&\raisebox{0.2ex}{
 \multirow{6}{*}{$\stackrel{\text{\tiny Hermite}}{\longrightarrow}\quad\left(\parbox[c][17ex]{0cm}{}\right.$}}\hspace{2ex}
 2^{-8} & 2^{-8} \times 11 & 0 & 2^{-8} \times 5 & 0 & 0 &
\raisebox{0.2ex}{\multirow{6}{*}{$\left.\parbox[c][17ex]{0cm}{}\right)$}} \\
&& {\color{black!50} 0} & 2^{-4} & 0 & 0 & 0 & 0 \\
&& {\color{black!50} 0} & {\color{black!50} 0} & 2^{-1} & 0 & 0 & 2^{-3} \\
&& {\color{black!50} 0} & {\color{black!50} 0} & {\color{black!50} 0} & 2^{-4} & 0 & 2^{-3} \\
&& {\color{black!50} 0} & {\color{black!50} 0} & {\color{black!50} 0} & {\color{black!50} 0} & 1 & 2^{-3} \\
&& {\color{black!50} 0} & {\color{black!50} 0} & {\color{black!50} 0} & {\color{black!50} 0} & {\color{black!50} 0} & 2^{-2} 
\end{array}$}
\hfill\null

\caption{The Jacobian of the L-part of the LU factorization at $M$}
\label{fig:jacLU}
\end{figure}
Looking at the minimal valuation of the entries of $J(\varphi)_M$ column 
by column, we find the optimal precision for each entry of the matrix
$L$ (defined as the L-part $L$ of the LU factorization of $M$):
\begin{equation}
\label{eq:optprecLU}
\left(\begin{matrix}
{\color{black!50} -} & {\color{black!50} -} & {\color{black!50} -} & {\color{black!50} -} \\
O(2^2) & {\color{black!50} -} & {\color{black!50} -} & {\color{black!50} -} \\
O(2^2) & O(2^9) & {\color{black!50} -} & {\color{black!50} -} \\
O(2^2) & O(2^{10}) & O(2^7) & {\color{black!50} -} \\
\end{matrix}\right)
\end{equation}
(we recall that the entries of $M$ were all initially given at precision 
$O(2^{10})$). Comparing with the numerical results obtained in \S 
\ref{sssec:compexamples} (see Eq.~\eqref{eq:resultLU}), we see that 
floating-point arithmetic, one more time, found all the relevant 
digits.

A closer look at Figure~\ref{fig:jacLU} shows that the lattice 
$d \varphi_M(2^{10} M_4(\Z_2)) = 
2^{10} \cdot d \varphi_M(M_4(\Z_2))$
has exactly $9$ diffused digits (apply Eq.~\eqref{eq:numberdiffused}).
The precision \eqref{eq:optprecLU} is then globally not optimal although 
it is on each entry separately. For example, if $\ell_{i,j}$ denotes the 
$(i,j)$ entry of $L$, the linear combination $\ell_{3,2} - 11 \cdot 
\ell_{3,1}$ can be known at precision $O(2^6)$ although $\ell_{3,2}$ and
$\ell_{3,1}$ cannot be known at higher precision than $O(2^2)$.
It is remarkable that floating-point arithmetic actually ``saw'' 
these diffused digits; indeed, taking the values computed in
floating-point arithmetic for $\ell_{3,2}$ and $\ell_{3,1}$, we compute:
$$\ell_{3,2} - 11 \cdot \ell_{3,1} = \ldots0000010111.$$
It turns out that the last six digits of the latter value are 
correct!

\begin{rem}
\label{rem:precLU}
For a general input $d \times d$ matrix $M$ whose principal minors have 
valuation $v_1, \ldots, v_d$, we expect that zealous arithmetic looses 
about $\Omega(v_1 + \cdots + v_d)$ significant digits (for the 
computation of the LU factorization of $M$ using standard Gaussian
elimination) while floating-point arithmetics looses only $O(\max(v_1, 
\ldots, v_d))$ significant digits. If $M$ is picked at random in the
space $M_d(\Zp)$ (equipped with its natural Haar measure), the former 
is $O(\frac d{p-1})$ on average while the second is $O(\log_p d)$ on
average~\cite{Ca12}.
\end{rem}

\paragraph{Bézout coefficients}

We now move to commutative algebra. We address the question of the 
computation of the Bézout coefficients of two monic polynomials of 
degree $d$ which are supposed to be coprime.
Let $\mathcal U$ be the subset of $\Qp[X]_{< d} \times \Qp[X]_{< d}$ 
consisting of pairs $(\tilde P, \tilde Q)$ for which $P = X^d + \tilde 
P$ and $Q = X^d + \tilde Q$ are coprime. Observe that $\mathcal U$ is defined by
the non-vanishing of some resultant and so is open. We introduce the
function:
$$\begin{array}{rcl}
\varphi : \quad \calU & \longrightarrow & \Qp[X]_{< d} \times \Qp[X]_{< d} \smallskip \\
(\tilde P, \tilde Q) & \mapsto & (U,V) \quad \text{s.t.} \quad UP + VQ = 1
\end{array}$$
for $P = X^d + \tilde P$ and $Q = X^d + \tilde Q$. (Note that $U$
and $V$ are uniquely determined thanks to the conditions on the degree.)
Once again, the differential of $\varphi$ can be easily computed
by differentiating the defining relation $UP + VQ = 1$; indeed doing so,
we immediately get:
\begin{equation}
\label{eq:diffBezoutrel}
\big(dU \cdot P\big) + \big(dV \cdot Q\big) = dR
\end{equation}
for $dR = - U {\cdot} dP - V {\cdot} dQ$. Eq~\eqref{eq:diffBezoutrel} gives
$dU {\cdot} P \equiv dR \pmod Q$ for what we derive $dU \equiv U {\cdot} dR
\pmod Q$. Comparing degrees we find $dU = U {\cdot} dR \text{ mod } Q$.
Similarly $dV = V {\cdot} dR \text{ mod } P$. 
The differential of $\varphi$ at a pair $(\tilde P, \tilde Q)$ is 
then the linear mapping:
$$\begin{array}{rcl}
d \varphi_{(\tilde P, \tilde Q)} :
\quad \Qp[X]_{< d} \times \Qp[X]_{< d} & \longrightarrow & \Qp[X]_{< d} \times \Qp[X]_{< d} \smallskip \\
(dP, dQ) & \mapsto & (dU, dV) = \big(U{\cdot}dR \text{ mod } Q, \, V{\cdot}dR \text{ mod } P\big) \smallskip \\
& & \hspace{5em}\text{where} \quad dR = - U \cdot dP - V \cdot dQ.
\end{array}$$
It is important to note that $d \varphi_{(\tilde P, \tilde Q)}$ is never
injective because it maps $(-V, U)$ to $0$. Consequently, it is never
surjective either and one cannot apply the precision Lemma to it. The
easiest way to fix this issue is to decompose $\varphi$ as $\varphi = 
(\varphi_U, \varphi_V)$ and to apply the precision Lemma to each 
component separately (though this option leads to less accurate results).

\medskip

For the particular example we have studied in \S 
\ref{sssec:compexamples}, namely:
$$\begin{array}{r@{\hspace{0.2ex}}r@{\hspace{0.5ex}}l@{\hspace{0.2ex}}r@{\hspace{0.5ex}}l}
P = X^4 + {}& (\ldots 1101111111) & X^3 
        + {}& (\ldots 0011110011) & X^2 \\
      {}+ {}& (\ldots 1001001100) & X^{\phantom{1}} 
        + {}& (\ldots 0010111010) \medskip \\
Q = X^4 + {}& (\ldots 0101001011) & X^3 
        + {}& (\ldots 0111001111) & X^2 \\
      {}+ {}& (\ldots 0100010000) & X^{\phantom{1}} 
        + {}& (\ldots 1101000111) 
\end{array}$$
the Jacobian matrix we find is shown on Figure~\ref{fig:jacBezout}.
\begin{figure}
\hfill{\tiny
$\begin{array}{r@{\hspace{0.3ex}}lc@{\hspace{1ex}}rrrr@{\hspace{3ex}}|@{\hspace{3ex}}rrrr@{\hspace{2ex}}c}
&&& \multicolumn{4}{c|@{\hspace{3ex}}}{dU} & \multicolumn{4}{c}{dV} \\
\rule[3ex]{0pt}{\baselineskip}
(dP,&dQ) && X^3 & X^2 & X & 1 & X^3 & X^2 & X & 1 \\
\rule[3ex]{0pt}{\baselineskip}
(X^3,&0) &
\raisebox{0.2ex}{\multirow{8}{*}{$\left(\parbox[c][22ex]{0cm}{}\right.$}} &
           \ldots0100111 & \ldots0001000 & \ldots1000000 & \ldots1010000 & \ldots1011001 & \ldots0100000 & \ldots1000000 & \ldots0100000 &
\raisebox{0.2ex}{\multirow{8}{*}{$\left.\parbox[c][22ex]{0cm}{}\right)$}} \\
(X^2,&0) &&\ldots1010000 & \ldots0010111 & \ldots0111000 & \ldots1000000 & \ldots0110000 & \ldots0101001 & \ldots0110000 & \ldots0000000 \\
(X,&0)   &&\ldots1000000 & \ldots0010000 & \ldots1010111 & \ldots0111000 & \ldots1000000 & \ldots1110000 & \ldots1101001 & \ldots0110000 \\
(1,&0)   &&\ldots1111000 & \ldots1101000 & \ldots0011000 & \ldots1010111 & \ldots0001000 & \ldots0111000 & \ldots0001000 & \ldots1001001 \\
\rule[3ex]{0pt}{\baselineskip}
(0,&X^3) &&\ldots1011001 & \ldots1111000 & \ldots0010000 & \ldots1110000 & \ldots0100111 & \ldots1100000 & \ldots1100000 & \ldots1100000 \\
(0,&X^2) &&\ldots1110000 & \ldots0101001 & \ldots0001000 & \ldots0010000 & \ldots0010000 & \ldots0010111 & \ldots0010000 & \ldots0100000 \\
(0,&X)   &&\ldots0010000 & \ldots0100000 & \ldots0011001 & \ldots0001000 & \ldots1110000 & \ldots0100000 & \ldots1100111 & \ldots1010000 \\
(0,&1)   &&\ldots1001000 & \ldots0101000 & \ldots1011000 & \ldots0011001 & \ldots0111000 & \ldots0111000 & \ldots1001000 & \ldots0000111 
\end{array}$}
\hfill\null

\caption{The Jacobian of the ``Bézout function'' at $(P,Q)$}
\label{fig:jacBezout}
\end{figure}
Each column of this matrix contains an entry of valuation zero, meaning 
that the optimal precision on each coefficient of $U$ and $V$ is 
$O(2^{10})$. Observe that it was not reached by interval arithmetic or 
floating-point arithmetic!
Continuing our analysis, we further observe that the lattice generated 
by the $dU$-part (resp. the $dV$-part) of $J(\varphi)_{(\tilde P, \tilde 
Q)}$ has no diffused digit. The precision
$$O(2^{10}) X^3 + O(2^{10}) X^2 + O(2^{10}) X + O(2^{10})$$
on both $U$ and $V$ is then optimal (but the joint precision induced 
on the pair $(U,V)$ is not).

Another option for fixing the default of surjectivity of $d 
\varphi_{(P,Q)}$ is to observe that $\varphi$ in fact takes its value in 
the hyperplane $H$ of $\Qp[X]_{< d} \times \Qp[X]_{< d}$ consisting of 
pairs of polynomials $(P,Q)$ whose coefficients in degree $d{-}1$ agree. 
Restricting the codomain of $\varphi$ to $H$, we then obtain a 
well-defined mapping whose differential is almost everywhere surjective. 
Applying the precision Lemma, we find that the pair $(U,V)$ has $16$ 
diffused digits!

\begin{rem}
A careful study of the precision in Euclidean algorithm is presented 
in~\cite{Ca17}: we prove in this reference that if $P$ and $Q$ are two
monic polynomials of degree $d$, Euclidean algorithm executed within the 
framework of zealous arithmetic (resp. floating-point arithmetic) looses 
$\Omega(v_0 + \cdots + v_{d{-}1})$ (resp. $O(\max(v_0, \ldots, 
v_{d{-}1}))$) digits on each coefficient where $v_j$ is the valuation of 
the $j$-th scalar subresultant of $P$ and $Q$. Moreover, for random 
polynomials, we have by \cite[Corollary~3.6]{Ca17}:
\begin{align*}
\E\big[v_0 + \cdots + v_{d{-}1}\big] & 
\textstyle \geq \frac d{p-1} \smallskip \\
\E\big[\max(v_0, \ldots, v_{d{-}1})\big] & 
\textstyle \leq \log_p d + O\big(\sqrt{\log_p d}\big).
\end{align*}
Compare this result with the case of LU factorization
(see Remark~\ref{rem:precLU}).
\end{rem}

\paragraph{Polynomial evaluation and interpolation}

Recall that the last example considered in \S \ref{sssec:compexamples}
was about evaluation and interpolation of polynomials: precisely, 
starting with a polynomial $P \in \Zp[X]$ of degree $d$, we first
evaluated it at the points $0, 1, \ldots, d$ and then reconstructed
it by interpolation from the values $P(0), P(1), \ldots, P(d)$.
We have observed that this problem seemed to be numerically highly 
unstable: for example, for $d = 19$ and an initial polynomial $P$ given 
at precision $O(2^{10})$, we were only able to reconstruct a few
number of digits and even found not integral coefficients in many 
places (see Figure~\ref{fig:evalinterpol}). We propose here to give
a theoretical explanation of this phenomenon.

Consider the function $\varphi : \Qp[X]_{\leq d} \to \Qp^{d+1}$, $P 
\mapsto (P(0), P(1), \ldots, P(d))$. Note that it is linear, so that $d 
\varphi_P = \varphi$ for all $P$. 
The numerical stability of our evaluation problem is governed by
the lattice
$H = \varphi(p^N \Zp[X]_{\leq d}) = p^N \cdot \varphi(\Zp[X]_{\leq d})$
where $O(p^N)$ is the initial precision we have on each coefficient of
$P$ (assuming for simplicity that it is the same for all coefficients).
Applying $\varphi$ to the standard basis of $\Zp[X]_{\leq d}$, we find
that $H$ is generated by the row vectors of the Vandermonde matrix:
$$M = \left( \begin{matrix}
1 & 1 & 1 & \cdots & 1 \\
0 & 1 & 2 & \cdots & d \\
0 & 1 & 2^2 & \cdots & d^2 \\
\vdots & \vdots & \vdots & & \vdots \\
0 & 1 & 2^d & \cdots & d^d
\end{matrix} \right).$$
According to Eq.~\eqref{eq:numberdiffused}, the number of diffused
digits of $H$ is then equal to the $p$-adic valuation of the
determinant of $M$ whose value is:
$$\det M = \prod_{0 \leq i < j \leq d} (j-i) = 1! \times 2! \times 
\cdots \times d!.$$
In order to estimate its $p$-adic valuation, we recall the following
result.

\begin{prop}[Legendre's formula]
For all positive integer $n$, we have:
$$\val_p(n!) = \Big\lfloor\frac n p\Big\rfloor + \Big\lfloor\frac n {p^2}\Big\rfloor + \cdots +
\Big\lfloor\frac n {p^k}\Big\rfloor + \cdots$$
where $\left\lfloor\cdot\right\rfloor$ is the usual floor function.
\end{prop}

\begin{proof}
Exercise.
\end{proof}

It follows in particular from Legendre's formula that $\val_p(i!) \geq
\frac i p - 1$, from what we derive:
$$\val_p(\det M) \geq \sum_{i=1}^d {\textstyle \big(\frac i p - 1\big)}
= \frac{d(d+1)}{2p} - d.$$
When $d$ is large compared to $p$, the number of diffused digits
is much larger than the 
number of $p$-adic numbers on which these diffused digits ``diffuse'' 
(which are the $d{+}1$ values $P(0), \ldots, P(d)$). They
will then need to have an influence at high 
precision\footnote{Roughly speaking, if we want to dispatch $n$ 
(diffused) digits 
between $m$ places, we have to increase the precision by at least $\lceil
\frac n m\rceil$ digits; this is the pigeonhole principle.}. This is why it is so 
difficult to get a sharp precision on the tuple $(P(0), \ldots, P(d))$. Concretely the 
$n$-th digits of the coefficients of $P$ may influence the $(n+\frac 
d{2p})$-th digit of some linear combination of the $P(i)$'s and 
conversely the $n$-th digits of the $P(i)$'s may influence the $(n-\frac 
d{2p})$-th digits of the coefficients of $P$. In other words, in order 
to be able to recover all the coefficients of $P$ at the initial 
precision $O(p^N)$, one should have used at least $N+\frac d{2p}$ digits 
of precision (in the model of floating-point arithmetic).

\subparagraph{Evaluation at general points.}

Until now, we have only considered evaluation at the first integers. We 
may wonder whether to what extend the instability we have observed above 
is related to this particular points. 
It turns out that it is totally independant and that similar behaviors 
show up with any set of evaluation points. Indeed, let $a_0, \ldots, 
a_d$ be pairwise distinct elements of $\Zp$ and consider the function 
$\varphi : \Qp[X]_{\leq d} \to \Qp^{d+1}$, $P \mapsto (P(a_0), P(a_1), 
\ldots, P(a_d))$.
It is linear and its Jacobian at any point is the following Vandermonde 
matrix:
$$\left( \begin{matrix}
1 & 1 & 1 & \cdots & 1 \\
a_0 & a_1 & a_2 & \cdots & a_d \\
a_0^2 & a_1^2 & a_2^2 & \cdots & a_d^2 \\
\vdots & \vdots & \vdots & & \vdots \\
a_0^d & a_1^d & a_2^d & \cdots & a_d^d
\end{matrix} \right)$$
whose determinant is
$V(a_0, \ldots, a_d) = \prod_{0\leq i < j \leq d} (a_i - a_j)$.
Thanks to Eq.~\eqref{eq:numberdiffused}, the number of diffused digits 
of the lattice $\varphi(p^N \Zp[X]_{\leq d})$ is the valuation of 
$V(a_0, \ldots, a_d)$. The next lemma ensures that this number is always
as large as it was in the particular case we have considered at first
(\emph{i.e.} $a_i = i$ for all $i$).

\begin{lem}
For all $a_0, \ldots, a_d \in \Zp$, we have:
$$\val\big(V(a_0, \ldots, a_d)\big) \geq 
\val\big(1! \times 2! \times \cdots \times d!\big).$$
\end{lem}

\begin{proof}
Given a finite subset $A \subset \Zp$, we denote by $\nu(A)$ the 
valuation of $V(x_1, \ldots, x_n)$ where the $x_i$'s are the elements of 
$A$. We note that $V(x_1, \ldots, x_n)$ depends up to a sign to the way 
the elements of $A$ are enumerated but $\nu(A)$ depends only on $A$. We
are going to prove by induction on the cardinality $n$ of $A$ that
$\nu(A) \geq \val\big(1! \times 2! \times \cdots \times (n{-}1)!\big)$.

If $A$ has cardinality $2$, the result is obvious. Let us now consider 
a finite subset $A$ of $\Zp$ of cardinality $n$. For $i \in \{1, \ldots,
p\}$, let $A_i$ be the subset of $A$ consisting of elements which are
congruent to $i$ modulo $p$ and write $n_i = \text{Card } A_i$. Define
in addition $B_i$ as the set of $\frac{x-r} p$ for $x$ varying in $A_i$.
Clearly $B_i$ is a subset of $\Zp$ of cardinality $n_i$.
Moreover, we easily check that: 
\begin{equation}
\label{eq:valvdm}
\nu(A) = \sum_{i=1}^p \nu(A_i) 
= \sum_{i=1}^p \left(\frac{n_i \: (n_i{-}1)} 2 + \nu(B_i)\right).
\end{equation}
Define the function $v : \N \to \N$,
$n \mapsto \frac{n(n{-}1)} 2 + \val(1!) + \cdots + \val((n{-}1)!)$.
It follows from the induction hypothesis that $\nu(A) \geq v(n_1) + 
\cdots + v(n_p)$. Observe that $f(n{+}1) - f(n) = n + \val(n!)$ and
deduce from this that $f(n) + f(m) \leq f(n{+}1) + f(m{-}1)$ as soon
as $m \geq n+2$. Applying again and again this inequality and noting
in addition that $n_1 + \cdots + n_p = n$, we end up with:
\begin{equation}
\label{eq:convexityv}
v(n_1) + \cdots + v(n_p) \geq r \cdot v(q{+}1) + (p{-}r) \cdot v(q)
\end{equation}
where $q$ and $r$ are the quotient and the remainder of the Euclidean
division of $n$ by $p$ respectively. 
It is enough to prove that the right hand side of 
Eq.~\eqref{eq:convexityv} is equal to $\val(1!) + \cdots + 
\val((n{-}1)!)$. This follows from Eq.~\eqref{eq:valvdm} applied with
the particular set $A = \{0, 1, \ldots, n{-}1\}$.
\end{proof}

\begin{rem}
The above proof shows in a similar fashion 
that if $a_0, \ldots, a_d$ are integers, then $V(a_0, \ldots, a_d)$
is divisible by $1! \times 2! \times \cdots \times d!$.
\end{rem}

The conclusion of the above analysis is that the 
evaluation-interpolation strategy --- which is used for instance in FFT 
--- cannot be used in the $p$-adic setting without further care. 
Precisely, from this point of view, one should really think of $\Zp$ as 
if it were a ring of characteristic $p$ and reuse all the techniques 
developed in the case of finite fields for avoiding divisions by $p$.

\subsubsection{Newton iteration}
\label{sssec:NewtonC2}

We have already seen that Hensel's Lemma is a powerful tool in the 
$p$-adic context. However, as highlighted in \S 
\ref{sssec:intervalNewton}, it may have a strange behavior regarding 
precision. Below we use the precision Lemma in order to study carefully 
this phenomenon. We also take the opportunity to extend the Newton 
scheme (on which Hensel's Lemma is based) to the framework of 
multivariate functions of class $C^2$.

\paragraph{Functions of class $C^2$ in the $p$-adic setting}

Recall that, given two $\Qp$-vector spaces $E$ and $F$, we denote by
$\calL(E,F)$ the space of linear mappings from $E$ to $F$. Similarly,
Given three $\Qp$-vector spaces $E_1$, $E_2$ and $F$, we denote by 
$\calB(E_1{\times}E_2, F)$ the space of bilinear functions 
$E_1{\times}E_2 \to F$. We recall that there exist canonical 
isomorphisms --- the so-called \emph{curryfication isomorphisms} --- 
between $\calB(E_1{\times}E_2, F)$, $\calL(E_1, \calL(E_2,F))$ and 
$\calL(E_2, \calL(E_1,F))$. When $E_1 = E_2 = E$, we define further 
$\calB^\symm(E{\times}E)$ as the subspace of $\calB(E{\times}E, F)$ 
consisting of \emph{symmetric} bilinear functions.

We recall also that when $E$ and $F$ are endowed with norms, the
space $\calL(E,F)$ inherits the norm operator defined by 
$\Vert f \Vert_{\calL(E,F)} = \sup_{x \in B_E(1)} \Vert f(x) \Vert_F$.
Similarly, we endow $\calB(E_1{\times}E_2,F)$ with the norm
$\Vert \cdot \Vert_{\calB(E_1{\times}E_2,F)}$ defined by
$\Vert b \Vert_{\calB(E_1{\times}E_2,F)} = \sup_{x_1 \in B_{E_1}(1), 
x_2 \in B_{E_2}(1)} \Vert b(x_1,x_2) \Vert_F$. The curryfication
isomorphisms do preserve the norm.

The definition of ``being of class $C^2$'' is inspired from the 
alternative definition of ``being of class $C^1$'' provided by Remark 
\ref{rem:altdefC1multi} (see also Remark \ref{rem:altdefC1uni}).

\begin{deftn}
\label{def:classC2}
Let $E$ and $F$ be two normed finite dimensional vector spaces and let 
$U$ be an open subset of $E$.
A function $f : U \to F$ is of \emph{class $C^2$} if there exist:
\begin{itemize}
\renewcommand{\itemsep}{0pt}
\item a function $df : U \to \calL(E,F)$, $x \mapsto df_x$
\item a function $d^2 f: U \to \calB^\symm(E{\times}E, F)$, $x \mapsto d^2 f_x$
\item a covering $(U_i)_{i \in I}$ of $U$, and
\item for all $i \in I$, a \emph{continuous} real-valued function 
$\varepsilon_i : \R^+ \to \R^+$ vanishing at $0$
\end{itemize}
such that, for all $i \in I$ and all $x,y \in U_i$:
\begin{equation}
\label{eq:defC2}
\textstyle
\Vert f(y) - f(x) - df_x(y{-}x) - \frac 1 2 d^2 f_x(y{-}x, y{-}x) \Vert_F \leq 
\Vert y{-}x \Vert_E^2 \cdot \varepsilon_i\big(\Vert y{-}x \Vert_E\big)
\end{equation}
\end{deftn}

\begin{prop}
Let $E$ and $F$ be two normed finite dimensional vector spaces and let 
$U$ be an open subset of $E$. Let $f : U \to F$ be a function of class 
$C^2$. Then:
\begin{enumerate}[(i)]
\renewcommand{\itemsep}{0pt}
\item
The function $f$ is of class $C^1$ on $U$ and its differential is the
function $df$ of Definition \ref{def:classC2};
\item
the function $df$ is of class $C^1$ on $U$ as well and its 
differential is the function $d^2 f$ (viewed as a element of $\calL(E, 
\calL(E,F))$) of Definition \ref{def:classC2}.
\end{enumerate}
\end{prop}

\begin{proof}
We consider the data $df$, $d^2 f$, $U_i$, $\varepsilon_i$ given by 
Definition \ref{def:classC2}. Without loss of generality, we may assume 
that the $U_i$'s are all balls and that the $\varepsilon_i$'s are all 
non-decreasing functions. By this additional assumption, $U_i$ is
compact; the function $x \mapsto \Vert d^2 f_x \Vert_{\calB(E{\times}E,
F)}$ is then bounded on $U_i$ by a constant $C_i$. Therefore
Eq.~\eqref{eq:defC2} implies:
$$\Vert f(y) - f(x) - df_x(y{-}x) \Vert_F \leq 
\Vert y{-}x \Vert_E^2 \cdot \max\big(C_i, \varepsilon_i\big(\Vert y{-}x \Vert_E\big)\big)$$
for all $x,y \in U_i$. The first assertion of the Proposition follows.

We now prove the second assertion. We fix $i \in I$. Let $x,y \in U_i$. 
We set $\delta = \Vert y{-}x \Vert_E$,
\begin{align*}
\ell & = df_x - df_y - d^2 f_x(y-x) \in \calL(E,F) \\
\text{and} \quad
b & = \textstyle \frac 1 2 \big(d^2 f_x - d^2 f_y\big) \in \calB^\symm(E{\times}E, F).
\end{align*}
Observe that, thanks to our assumption on $U_i$ (and ultrametricity), 
the ball $B_E(y,\delta)$ of centre $y$ and radius $\delta$ is included 
in $U_i$. For $z \in B_E(y,\delta)$, we define:
\begin{align*}
\delta_{z,x} & \textstyle 
  = f(z) - f(x) - df_x(z{-}x) - \frac 1 2 d^2 f_x(z{-}x, z{-}x) \smallskip \\
\delta_{z,y} & \textstyle
  = f(z) - f(y) - df_y(z{-}y) - \frac 1 2 d^2 f_y(z{-}y, z{-}y) \smallskip \\
\delta_{y,x} & \textstyle
  = f(y) - f(x) - df_x(y{-}x) - \frac 1 2 d^2 f_x(y{-}x, y{-}x).
\end{align*}
By our assumptions, these three vectors have norm at most $\delta^2 
\cdot \varepsilon_i(\delta)$.
A simple computation using linearity and bilinearity yields 
$\delta_{z,x} - \delta_{z,y} - \delta_{y,x} = \ell(z{-}y) - 
b(z{-}y,z{-}y)$. Consequently, we get the estimation:
$$\Vert \ell(h) - b(h,h) \Vert_F \leq \delta^2 \cdot \varepsilon_i(\delta)$$
for all $h \in E$ with $\Vert h \Vert_E \leq \delta$. Applying it with
$(p{+}1)h$ in place of $h$, we find:
$$\Vert \ell(h) - (p{+}1) \cdot b(h,h) \Vert_F \leq \delta^2 \cdot \varepsilon_i(\delta)$$
and combining now the two above inequalities, we end up with
$\Vert b(h,h) \Vert_F \leq p \delta^2 \cdot \varepsilon_i(\delta)$ and
so $\Vert \ell(h) \Vert_F \leq p \delta^2 \cdot \varepsilon_i(\delta)$
as well. This inequality being true for all $h \in B_E(\alpha)$, we 
derive $\Vert \ell \Vert_{\calL(E,F)} \leq p^2 \delta \cdot 
\varepsilon_i(\delta)$. This proves the differentiability of $df$
together with the fact that its differential is $d^2 f$.
\end{proof}

\paragraph{Hensel's Lemma for functions of class $C^2$}

We want to develop an analogue of Hensel's Lemma (see Theorem 
\ref{th:Hensel}) for a function $f : U \to F$ of class $C^2$ whose 
domain $U$ is an open subset of a finite dimensional $p$-adic vector 
space $E$ and whose codomain $F$ is another $p$-adic vector space of the same 
dimension. For simplicity, we assume further that $U = B_E(1)$ (the
general case is deduced from this one by introducing coverings).
Under this additional assumption, the function $d^2 f$ is bounded, 
\emph{i.e.} there exists a constant $C$ such that
$\Vert d^2 f_x \Vert_{\calB(E{\times}E, F)} \leq C$ for all $x \in B_E(1)$.
Plugging this into Eq.~\eqref{eq:defC2} and possibly enlarging a bit
$C$, we derive the two estimations:
\begin{align}
\forall x, y \in B_E(1), \quad 
& \Vert f(y) - f(x) - df_x(y{-}x) \Vert_F \leq 
  C \cdot \Vert y{-}x \Vert_E^2 \label{eq:Calmost2} \medskip \\
& \Vert df_y - df_x \Vert_{\calL(E,F)} \leq 
  C \cdot \Vert y{-}x \Vert_E \label{eq:Calmost2diff}
\end{align}
which are the key points for making Newton iteration work. 
We consider a point $v \in B_E(1)$ for which:
\begin{equation}
\label{eq:hypHensel}
df_v \text{ is invertible}
\quad \text{and} \quad
\Vert f(v) \Vert_F \cdot \Vert df_v^{-1} \Vert_{\calL(F,E)}^2 
\,<\, C^{-1} \,\leq\, 
\Vert df_v^{-1} \Vert_{\calL(F,E)}.
\end{equation}
We set $r = \big(C \cdot \Vert df_v^{-1} \Vert_{\calL(F,E)}\big)^{-1}$ 
and denote by $V$ the open ball of centre $v$ and radius $r$. Observe
that $r \leq 1$, so that $V \subset B_E(1)$.

\begin{lem}
\label{lem:HenselC2}
Under the assumptions \eqref{eq:Calmost2}, \eqref{eq:Calmost2diff}
and \eqref{eq:hypHensel}, the linear function $df_x$ is invertible
and $\Vert df_x^{-1} \Vert_{\calL(F,E)} = \Vert df_v^{-1} \Vert_{\calL(F,E)}$
for all $x \in V$. Moreover, the Newton iteration mapping 
$\calN : V \to E$, $x \mapsto x - df_x^{-1}(f(x))$
takes its values in $V$ and satisfies:
\begin{equation}
\label{eq:estimNewton}
\Vert f(\calN(x)) \Vert_F \leq
C \cdot \Vert df_v^{-1} \Vert_{\calL(F,E)}^2 \cdot \Vert f(x) \Vert_F^2.
\end{equation}
\end{lem}

\begin{proof}
Let $x \in V$. Set $g = df_v - df_x$.
By the inequality \eqref{eq:Calmost2diff}, we have:
$$\Vert g \Vert_{\calL(E,F)} = 
\Vert df_v - df_x \Vert_{\calL(E,F)} \leq C \cdot
\Vert v{-}x \Vert_E < C r = \big(\Vert df_v^{-1} \Vert_{\calL(F,E)}\big)^{-1}.$$
Write
$df_x = df_v - g = df_v \circ (\id_E - df_v^{-1} 
\circ g)$. The above estimation shows that the norm of $df_v^{-1} 
\circ g$ is strictly less than $1$. Therefore the series
$\sum_{n=0}^\infty (df_v^{-1} \circ g)^{\circ n}$ converges to
an inverse of $(\id_E - df_v^{-1} \circ g)$. In particular 
$(\id_E - df_v^{-1} \circ g)$ is invertible and so is $df_x$.
Remark in addition that the norm of $\sum_{n=0}^\infty (df_v^{-1} 
\circ g)^{\circ n}$ is at most $1$ and so
$\Vert df_x^{-1} \Vert_{\calL(F,E)} \leq \Vert df_v^{-1} \Vert_{\calL(F,E)}$.
Reverting the roles of $v$ and $x$, we find even better that equality 
does hold.
Applying now \eqref{eq:Calmost2}, we get:
$$\Vert f(v) - f(x) - df_x(v{-}x) \Vert_F 
  \leq C \cdot \Vert v{-}x \Vert_E^2
  < C r^2 = \frac r {\Vert df_v^{-1} \Vert_{\calL(F,E)}^2}.$$
Applying $df_x^{-1}$, we deduce $\Vert df_x^{-1}(f(v)) - df_x^{-1}(f(x)) 
- (v{-}x) \Vert_F < r$. Notice furthermore that $\Vert v{-}x\Vert_E <
r$ by assumption and $\Vert df_x^{-1}(f(v)) \Vert_F \leq \Vert df_v^{-1} 
\Vert_{\calL(F,E)} \cdot \Vert f(v) \Vert_F < r$.
Consequently we find $\Vert df_x^{-1}(f(x)) \Vert_F < r$ as well, for
what we derive $\calN(x) \in V$. Eq.~\eqref{eq:estimNewton} finally
follows by applying again Eq.~\eqref{eq:Calmost2} with $x$ and 
$\calN(x)$.
\end{proof}

\begin{cor}[Hensel's Lemma for function of class $C^2$]
\label{cor:HenselC2}
Under the assumptions \eqref{eq:Calmost2}, \eqref{eq:Calmost2diff}
and \eqref{eq:hypHensel}, the sequence $(x_i)_{i \geq 0}$ defined by 
the recurrence
$$x_0 = v \quad ; 
\quad x_{i+1} = \calN(x_i) = x_i - df_{x_i}^{-1}(f(x_i))$$
is well defined and converges to $x_\infty \in V$ such that
$f(x_\infty) = 0$.
The rate of convergence is given by:
$$\Vert x_i - x_\infty \Vert_E \leq 
r \cdot \Big(C \cdot \Vert f(v) \Vert_F \cdot \Vert df_v^{-1} \Vert_{\calL(F,E)}^2 \Big)^{2^i} =
\frac {\Big(C \cdot \Vert f(v) \Vert_F \cdot \Vert df_v^{-1} \Vert_{\calL(F,E)}^2 \Big)^{2^i}}
{C \cdot \Vert df_v^{-1} \Vert_{\calL(F,E)}}.$$
Moreover $x_\infty$ is the unique solution in $V$ to the equation $f(x) 
= 0$.
\end{cor}

\begin{proof}
We define $\rho = C \cdot 
\Vert f(v) \Vert_F \cdot \Vert df_v^{-1} \Vert_{\calL(F,E)}^2$
By Lemma \eqref{lem:HenselC2}, all the $x_i$'s lie in $V$ and:
\begin{equation}
\label{eq:normfxi}
\Vert f(x_{i+1}) \Vert_F \leq 
C \cdot \Vert df_v^{-1} \Vert_{\calL(F,E)}^2 \cdot \Vert f(x_i) \Vert_F^2
\end{equation}
for all $i$. By induction, we derive $\Vert f(x_i) \Vert_F \leq
\big(C \cdot \Vert df_v^{-1} \Vert_{\calL(F,E)}^2\big)^{-1} 
\cdot \rho^{2^i}$. Therefore:
\begin{equation}
\label{eq:diffxi}
\Vert x_{i+1} - x_i \Vert_E =
\Vert df_{x_i}^{-1}(f(x_i)) \Vert_E \leq
\Vert df_v^{-1} \Vert_{\calL(F,E)} \cdot \Vert f(x_i) \Vert_F \leq
r \cdot \rho^{2^i}.
\end{equation}
The sequence $(x_i)_{i \geq 0}$ is a Cauchy sequence and therefore 
converges to some $x_\infty$ for which $f(x_\infty) = 0$.
Eq.~\eqref{eq:diffxi} implies moreover the announced rate of convergence 
together with the fact that $x_\infty$ belongs to $V$.
It then only remains to prove the uniqueness of the solution to the 
equation $f(x) = 0$ in $V$. For 
this, assume that $y \in V$ satisfies $f(y) = 0$. 
Instantiating Eq.~\eqref{eq:Calmost2} with $x = x_\infty$, we obtain
$\Vert df_{x_\infty}(y{-}x_\infty) \Vert_F \leq C \cdot 
\Vert y{-}x_\infty \Vert_E^2$. Hence:
$$\Vert y{-}x_\infty \Vert_E 
= \Vert df_{x_\infty}^{-1}\big(df_{x_\infty}(y{-}x_\infty)\big) \Vert_E 
\leq C \cdot \Vert df_v^{-1} \Vert_{\calL(F,E)} \cdot
  \Vert y{-}x_\infty \Vert_E^2 
= r^{-1} \cdot \Vert y{-}x_\infty \Vert_E^2.$$
Since moreover $\Vert y - x_\infty \Vert_E$ has to be strictly less
than $r$ by our assumptions, we derive $y = x_\infty$ and uniqueness
follows.
\end{proof}

\begin{rem}
\label{rem:HenselC2}
Hensel's Lemma is stable in the sense that its conclusion remains
correct for a sequence $(x_i)_{i \geq 0}$ satisfying the weaker
recurrence:
$$\quad x_{i+1} = x_i - df_{x_i}^{-1}(f(x_i))\, +\, (\text{some
 small perturbation})$$
as soon as the perturbation is small enough to continue to ensure 
that the estimation \eqref{eq:normfxi} holds.
\end{rem}

\paragraph{Precision}

In practice, the function $f$ of which we want to find a root is often 
not exact but given with some uncertainty; think typically of the case
where $f$ is a polynomial (of given degree) with coefficients in $\Qp$
given at some finite precision.
In order to model this, we introduce $C^2(B_E(1),F)$, the set 
of functions $f : B_E(1) \to F$ of class $C^2$. We endow it with the
$C^2$-norm defined by
$$\Vert f \Vert_{C^2} = \max \big( 
\Vert f \Vert_\infty,
\Vert df \Vert_\infty,
\Vert d^2f \Vert_\infty\big)$$
where the infinite norm of a function is defined as usual as the
supremum of the norms of its values.
We consider in addition a finite dimensional subspace $\calF$ of 
$C^2(B_E(1),F)$. The uncertainty on $f$ will be modeled by some lattice 
in $\calF$.

We fix $v \in B_E(1)$ together with two positive real numbers $C$ and 
$r$. We assume $r \leq 1$. Let $V$ be the open ball in $E$ of centre $v$ 
and radius $r$; clearly $V \subset B_E(1)$.
Let $\calU$ be the open subset of $\calF$ consisting of function $f$
for which Eqs.~\eqref{eq:Calmost2}, \eqref{eq:Calmost2diff} and 
\eqref{eq:hypHensel} hold and $\Vert df_v^{-1} \Vert_{\calL(F,E)}
= \frac r C$. By Hensel's Lemma (Corollary \ref{cor:HenselC2}), any
$f \in \calU$ has a unique zero in $V$. The function:
$$\begin{array}{rcl}
\calZ : \quad \calU & \longrightarrow & V \\
f & \mapsto & x \quad \text{s.t.} \quad x\in V \text{ and } f(x) = 0
\end{array}$$
is then well defined. In the idea of applying the precision Lemma, we
study its differentiability.

\begin{lem}
\label{lem:diffZ}
The function $\calZ$ is differentiable on $\calU$. 
Moreover, its differential at $f \in \calU$ is the linear mapping 
$d\calZ_f : \varphi \mapsto - df_{\calZ(f)}^{-1}(\varphi(\calZ(f)))$.
\end{lem}

\begin{proof}
The lemma can be seen as an application of the $p$-adic implicit 
functions Theorem (see~\cite[Proposition~4.3]{Sc11}). Below, we give 
a direct proof avoiding it.
Let $f, g \in \calU$. We set $x = \calZ(f)$, $\varphi = g{-}f$
and $y = x - df_x^{-1}(\varphi(x))$.
Since $f(x)$ vanishes, Eq.~\eqref{eq:Calmost2} reads:
\begin{equation}
\label{eq:diffZ1}
\Vert g(y) - (\id_E - dg_x \circ df_x^{-1})(\varphi(x)) \Vert_F \leq
C \cdot \Vert y{-}x \Vert_E^2 \leq 
C \cdot \Vert df_x^{-1} \Vert_{\calL(F,E)}^2 \cdot \Vert \varphi \Vert_\infty^2.
\end{equation}
From the identity $\id_E - dg_x \circ df_x^{-1} = (df_x - dg_x) \circ 
df_x^{-1}$, we deduce moreover that 
\begin{equation}
\label{eq:diffZ2}
\Vert \id_E - dg_x \circ df_x^{-1}\Vert_{\calL(E,E)} \leq
\Vert df_x^{-1} \Vert_{\calL(F,E)} \cdot 
\Vert df_x - dg_x \Vert_{\calL(E,F)} \leq \Vert df_x^{-1} \Vert_{\calL(F,E)}
\cdot \Vert d \varphi \Vert_\infty.
\end{equation}
Combining \eqref{eq:diffZ1} and \eqref{eq:diffZ2}, we derive
$\Vert g(y) \Vert_F \leq \Vert df_x^{-1} \Vert_{\calL(F,E)}
\cdot \Vert \varphi \Vert_{C^2}^2$. Corollary \ref{cor:HenselC2}
then implies 
$\Vert y - \calZ(g) \Vert_E \leq \Vert df_x^{-1} \Vert_{\calL(F,E)}^2
\cdot \Vert \varphi \Vert_{C^2}^2 =  \Vert df_v^{-1} \Vert_{\calL(F,E)}^2
\cdot \Vert \varphi \Vert_{C^2}^2$, which proves the lemma.
\end{proof}

Applying the precision Lemma, we end up with the next corollary.

\begin{cor}
If the uncertainty on $f$ is given by a (sufficiently rounded and small) 
lattice $\calH \subset \calU$, then the optimal precision on $x = \calZ(f)$ 
is $df_x^{-1} \circ \ev_x(\calH)$ where $\ev_x : \calH \to F$ is the 
evaluation morphism at $x$.
\end{cor}

\noindent
It is quite instructive to compare this result with the optimal
precision we get after a single iteration of the Newton scheme.
In order to do so, we introduce the map:
$$\begin{array}{rcl}
\calN : \quad \calU \times V & \longrightarrow & V \\
(f,x) & \mapsto & f(x) - df_x^{-1}(f(x))
\end{array}$$
which is well defined thanks to Lemma~\ref{lem:HenselC2}. It is moreover 
clear that $\calN$ is differentiable since it appears as a composite of 
differentiable functions.
A straightforward (but tedious) calculation shows that its differential
at $(f,x)$ is given by:
$$d \calN_{(f,x)} : (\varphi, \xi) \mapsto
-df_x^{-1}\big(\varphi(x)\big) + df_x^{-1}\big(d^2 f_x(\xi, f(x))\big)
+ df_x^{-1}\big( d\varphi_x(f(x))\big).$$
In particular, we observe that if $f(x) = 0$, the last two terms
vanish as well, so that the linear mapping $d \calN_{(f,x)}$ does 
not depend on the second variable $\xi$ and agrees with $d \calZ_f$.
In the language of precision, this means that the optimal precision 
on $\calZ(f)$ is governed by the last iteration of the Newton scheme.

In practice, this result suggests the following strategy: \emph{in order 
to compute $\calZ(f)$, we start with some $v$ satisfying the 
requirements of Hensel's Lemma (Corollary \ref{cor:HenselC2}), we 
iterate the Newton scheme \emph{without taking care (so much) of 
precision} until the obtained approximation looks correct and we finally 
perform a last iteration reintroducing the machinery for tracking 
precision.}
By Remark~\ref{rem:HenselC2}, the omission of precision tracking is 
harmless while the equality $d \calZ_f = d \calN_{(f, \calZ(f))}$ shows 
(under mild assumptions on the lattice $\calH$) that the precision we 
will get at the end (using the above strategy) will be very good because 
the loss of precision will just come from the last iteration and so will 
be limited. 

This behavior is very well illustrated by the example of the 
computation of square roots detailed in \S \ref{sssec:intervalNewton}.
(We then encourage the reader to read it and study it again in light of 
the discussion of this paragraph.)
More recently Lairez and Vaccon applied this strategy for the 
computation of the solutions of $p$-adic differential equations with 
separation of variables~\cite{LaVa16}; they managed to improve this way a
former algorithm by Lercier and Sirvent for computing isogenies between 
elliptic curves in positive characteristic~\cite{LeSi08}.

\subsection{Lattice-based methods for tracking precision}
\label{ssec:latticearith}

Until now, we have explained how the precision Lemma can be used for 
finding the optimal precision and analyzing this way the stability of 
algorithms. It turns out that the precision Lemma is also useful for 
\emph{stabilizing} algorithms, \emph{i.e.} for producing a stable 
procedure by altering slightly another given procedure which might be 
highly unstable.

\subsubsection{The method of adaptive precision}
\label{sssec:adaptiveprecision}

We place ourselves in the following general setting. Let 
\texttt{Phi} be a routine that computes a mathematical function $\varphi 
: U \to V$ whose domain $U$ is an open subset of a finite dimensional 
$\Qp$-vector space $E$ (\emph{i.e.} \texttt{Phi} takes as input a tuple 
of $p$-adic numbers but may possibly fails for particular instances 
lying a closed subset) and whose codomain $V$ is an open subset of
another finite dimensional $\Qp$-vector space (\emph{i.e.} \texttt{Phi} 
outputs a tuple of $p$-adic numbers as well).
We assume that $f$ is of class $C^1$ on $U$. We moreover fix an input $x 
\in U$ for which the differential $df_x$ is surjective.

\begin{rem}
While the $C^1$-assumption is usually harmless, the surjectivity
assumption is often more serious. 
For instance, observe that it implies that the dimension of $F$ 
(\emph{i.e.} the size of the output) is not greater than the dimension 
of $E$ (\emph{i.e.} the size of the input). Clearly there exist many 
interesting algorithms that do not satisfy this requirement and 
therefore do not fit into our framework.
One can always however workaround this issue as follows. We write the 
space of outputs $F$ as a direct sum $F = F_1 \oplus \cdots \oplus F_m$ 
and decompose the procedure \texttt{Phi} accordingly, \emph{i.e.} for 
each $i$, we introduce the algorithm $\texttt{Phi}_i$ that outputs only 
the $i$-th part of $\texttt{Phi}$. Clearly $\texttt{Phi}_i$ is modeled 
by the mathematical function $\varphi_i = \pr_i \circ \varphi$ where 
$\pr_i : F \to F_i$ is the canonical projection. As a consequence 
$d\varphi_{i,x} = \pr_i \circ d \varphi_x$; if the $F_i$'s are small 
enough, it will then be unlikely that one of them is not surjective.
For instance, in the special case where the $F_i$'s are all lines,
the writing $F = F_1 \oplus \cdots \oplus F_m$ corresponds to the
choice of a basis of $F$ and the surjectivity of the $d\varphi_{i,x}$'s
is equivalent to the fact that each column of the Jacobian matrix has
a nonzero entry (which is clearly much weaker than requiring its
surjectivity).
\end{rem}

Our aim is to compute $\varphi(x)$ using the procedure \texttt{Phi}; we 
would like moreover to be sharp in terms of precision and get proved results. We 
have seen that neither zealous arithmetic, nor floating-point arithmetic 
can ensure these two requirements at the same time: zealous arithmetic is 
usually not sharp (see \S \ref{sssec:compexamples}, \S 
\ref{sssec:diffexamples} for many examples) while floating-point 
arithmetic is not proved.
The case of lazy/relaxed arithmetic is a bit aside but it is not quite 
satisfactory either. Indeed recall that in the lazy approach, we are 
fixing a target precision; the output of \texttt{Phi} will then be sharp 
by design. The issue is elsewhere and comes from the fact that the 
lazy/relaxed machinery will generally compute the input $x$ at a higher 
precision than needed, consuming then more resources than necessary (see 
\S \ref{sssec:preclemmaatwork} for a theoretical discussion about this
and/or \S \ref{sssec:somos} below for a concrete example).

\paragraph{Context of zealous arithmetic: precision on inputs}

We assume that the input $x$ is given at some precision which is 
represented by a lattice $H$ in $E$: if the $i$-th coordinate of $x$ is 
given at precision $O(p^{N_i})$, the lattice $H$ is the diagonal lattice
$$H = p^{N_1} \Zp \oplus p^{N_2} \Zp \oplus \cdots \oplus
p^{N_d} \Zp \quad \text{(with $d = \dim E$)}$$
(see also \S \ref{sssec:preclemmaatwork}). We underline nevertheless 
that non diagonal lattices $H$ (corresponding to precision data with
diffused digits, see Definition \ref{def:diffuseddigits}) are also 
permitted. We assume that $H$ is nice enough so that the precision 
Lemma applies, giving:
\begin{equation}
\label{eq:adapt}
\varphi(x+H) = \varphi(x) + d\varphi_x(H).
\end{equation}

\subparagraph{Separation of approximation and precision.}

The formula above strongly suggests to split the computation of 
$\varphi(x{+}H)$ into two independent parts corresponding to the two 
summands. Concerning precision, we have to compute the lattice $d 
\varphi_x(H)$. 
One option for this 
is to rely on automatic differentiation techniques~\cite{BaBrChDi00,Ne10}; 
this approach however often leads to costly computations and 
really affects the complexity of our algorithm. 
For this reason, alternative strategies are often preferable (when 
available). The simplest one is probably to precompute $d \varphi_x$ by 
hand (we have seen in \S \ref{sssec:diffexamples} that it is tractable 
in many situations) and plug the obtained result into our algorithm.
On the other hand, we observe that $d \varphi_x(H)$ is a lattice in $F$ 
and is therefore encoded by a $\dim F \times \dim F$ matrix, which might 
be quite a large object (compare with $\dim F$ which is the size of 
the output). Therefore, if we do not need to be \emph{extremely} careful
on precision, we could prefer replacing $d \varphi_x(H)$ by a slightly
larger lattice $H_\max$ which is more easily representable, \emph{e.g.}
$H_\max$ could be a diagonal lattice.

For now on, we assume that we have at our disposal a lattice $H_\max$ 
containing $d \varphi_x(H)$ and we focus on the computation of the first 
term of Eq.~\eqref{eq:adapt}, that is the approximation.
Since $d \varphi_x(H)$ is stable under addition, 
Eq.~\eqref{eq:adapt} yields $\varphi(x+H) = y + d\varphi_x(H) 
\subset y + H_\max$ for any $y$ lying in $\varphi(x{+}H)$. Using now
that $H_\max$ is stable under addition as well, we end up with 
$\varphi(x+H) \subset y + H_\max$ as soon as $y \in \varphi(x{+}H) + 
H_\max$.
To conclude, it is then enough to 
compute the value of \emph{any} element of $\varphi(x{+}H)$ at precision 
$H_\max$. For doing so, we can rely on zealous arithmetic: we increase 
sufficiently the precision on $x$ (picking arbitrary values for digits) 
so that $\varphi(x)$ can be computed at precision at least $H_\max$ and 
we output $\varphi(x)$ at precision $H_\max$. 
In more conceptual terms, what we do is to choose an element $x' \in 
x{+}H$ together with a diagonal lattice $H' \subset H$ having the 
property that zealous arithmetic computes $y$ and $H'_\zealous$ 
such that $f(x'+H') \subset y + H'_\zealous$ and 
$H'_\zealous \subset H_\max$. We thus have:
$$y \, \in \, \varphi(x'+H') + H'_\zealous 
\, \subset \, \varphi(x+H) + H_\max$$
as wanted (see also Figure~\ref{fig:adaptive1}).
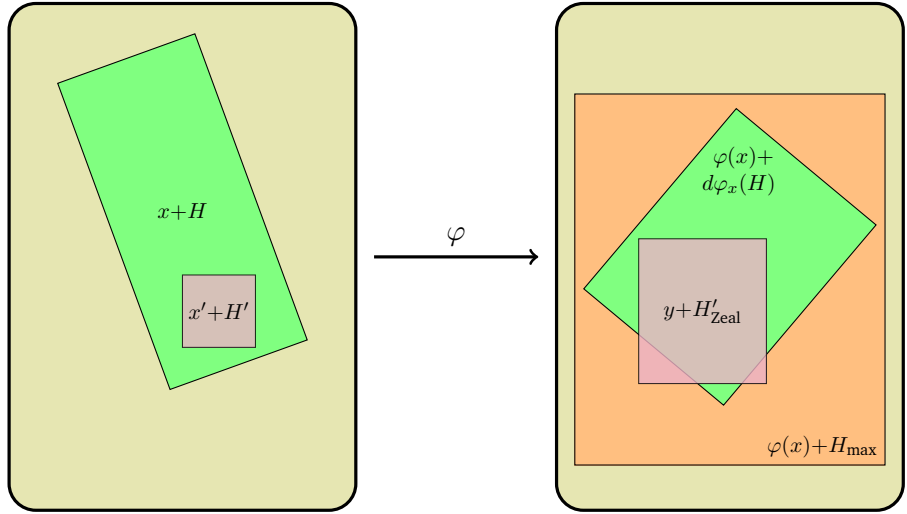
\begin{figure}
\hfill
\begin{tikzpicture}[scale=1.2]

\begin{scope}[beige, rounded corners=10pt]
\fill (0.1,0.2) rectangle (3.9,5.8);
\fill (6.1,0.2) rectangle (9.9,5.8);
\end{scope}

\draw[->,very thick] (4.1,3)--(5.9,3)
  node[midway,above] { $\varphi$ };

\begin{scope}[xshift=2cm,yshift=3.5cm,rotate=20]
\draw[black,fill=green!50] (-0.8,-1.8) rectangle (0.8,1.8);
\node[scale=0.75] at (0,0) { $x{+}H$ };
\end{scope}
\draw[black,opacity=0.8,fill=purple!30] (2,2.8) rectangle (2.8,2);
\node[scale=0.75] at (2.4,2.4) { $x'{+}H'$ };

\begin{scope}[xshift=8cm,yshift=3cm]
\draw[black,fill=orange!50] (-1.7,-2.3) rectangle (1.7,1.8);
\node[above left,scale=0.75] at (1.7,-2.3) { $\varphi(x){+}H_\max$ };
\draw[black,fill=green!50,rotate=-40] (-1,-1.3) rectangle (1,1.3);
\node[scale=0.75] at (0.15,1.1) { $\varphi(x){+}{}$ };
\node[scale=0.75] at (0.1,0.8) { $d\varphi_x(H)$ };
\draw[black,opacity=0.8,fill=purple!30] (-1,-1.4) rectangle (0.4,0.2);
\node[scale=0.75] at (-0.3,-0.6) { $y{+}H'_\zealous$ };
\end{scope}

\begin{scope}[very thick, rounded corners=10pt]
\draw (0.1,0.2) rectangle (3.9,5.8);
\draw (6.1,0.2) rectangle (9.9,5.8);
\end{scope}
\end{tikzpicture}
\hfill \null

\caption{The method of adaptive precision: first attempt}
\label{fig:adaptive1}
\end{figure}
These ideas lead to the following implementation:

\begin{lstlisting}
    def Phi_stabilized_v1(x):
        lift x(*\textrm{ at precision }$H'$*)
        y = Phi(x)    # (*\color{comment}\textrm{here }$\tty$\textrm{ is computed at precision }$H'_{\zealous}$*)
        return y(*\textrm{ at precision} $H_{\max}$*)
\end{lstlisting}

\noindent
The method presented above works but has a serious drawback: in 
practice, if often happens that the precision encoded by $H'$ 
is very high, leading then to very costly computations. In the next 
paragraph, we explain a strategy for dealing with this issue.

\subparagraph{Viewing \textrm{\tt Phi} as a sequence of steps.}

Roughly speaking, the problem with the above approach comes from the 
fact that the precision is tracked using interval arithmetic all along 
the execution of \texttt{Phi}. In order to avoid this, the idea consists 
in decomposing \texttt{Phi} into several steps; we then let interval 
arithmetic do the job of tracking precision within each individual step 
but we readjust the precision (using informations coming from the 
precision Lemma) each time we switch from one step to the next one. 
Concretely the readjustment is made by applying the method of the
previous paragraph to each individual step.

Let us now go further into the details. As said above, we view the routine 
\texttt{Phi} as a finite sequence of steps $\texttt{Step}_1, \ldots, 
\texttt{Step}_n$. Let $\sigma_i : U_{i-1} \to U_i$ be the mathematical 
function modeling the $i$-th step; its codomain $U_i$ is the space of 
outputs of the $i$-th step, that is the space of ``active'' variables 
after the execution of $i$-th step. Mathematically, it is an open 
subset in a finite dimensional $\Qp$-vector space $E_i$. For $i \in \{0, 
\ldots, n\}$, we set $\varphi_i = \sigma_i \circ \cdots \circ \sigma_1$ 
and $x_i = \varphi_i(x)$; they are respectively the function modeling 
the execution of the first $i$ steps of \texttt{Phi} and the state of 
the memory after the execution of the $i$-th step of \texttt{Phi} on the 
input $x$.
In what follows, we always assume that $d\varphi_{i,x}$ is 
surjective\footnote{This assumption is rather restrictive but really 
simplifies the discussion.} and, even more, that the precision Lemma 
applies with $\varphi_i$, $x$ and $H$, meaning that, for all $i$:
\begin{equation}
\label{eq:adapt2}
\varphi_i(x+H) = x_i + d \varphi_{i,x}(H).
\end{equation}
We define $H_i = d \varphi_{i,x}(H)$ and we assume that we
are given a ``simple'' lattice $H_{i,\max}$ containing $H_i$.
However, for the application we have in mind, this datum will not be 
sufficient; indeed, we shall need to be \emph{extremely} careful with
precision at intermediate levels. In order to do so \emph{without having 
to manipulate the ``complicated'' lattice $H_i$}, we assume that we are 
given in addition another ``simple'' lattice $H_{i,\min}$ which is 
\emph{contained} in $H_i$. With these notations, the stabilized version
of \texttt{Phi} takes the following schematic form:

\begin{lstlisting}[label={algo:phistabzealousv2}]
    def Phi_stabilized_v2(x):
        (*$\tty_0$*) = x
        for (*$i$*) in (*$1, 2, \ldots, n$*):
            lift (*$\tty_{i-1}$\textrm{ at enough precision}*)
            (*$\tty_i$*) = (*$\texttt{Step}_i(\tty_{i-1})$*)  # (*\color{comment}\textrm{here we want }$\tty_i$\textrm{ to be computed at precision at least }$H_{i,\min}$*)
        return (*$\tty_n$\textrm{ at precision} $H_{n,\max}$*)
\end{lstlisting}

\noindent
The locution ``enough precision'' in the above code means that we
want to ensure that zealous arithmetics is able to compute $\tty_i$ 
at precision $H_{i,\min}$. Let $H'_{i-1}$ be a lattice encoding such
an acceptable precision. Rigorously, it can be defined as follows. We
set $y_0 = x$ and for $i \in \{1,\ldots, n\}$, we define inductively 
$H'_{i-1}$, $y_i$ and $H'_{i,\zealous}$ by requiring that the routine 
$\texttt{Step}_i$ called on the input $y_{i-1}$ given at precision 
$H'_{i-1}$ computes $y_i$ at precision $H'_{i,\zealous}$ with 
$H'_{i,\zealous} \subset H_{i,\min}$ (see Figure~\ref{fig:adaptive2}). 
\begin{figure}
\hfill
\begin{tikzpicture}[scale=1.2]

\begin{scope}[beige, rounded corners=10pt]
\fill (0.1,0.2) rectangle (3.9,5.8);
\fill (5.1,0.2) rectangle (8.9,5.8);
\end{scope}

\draw[->,very thick] (-0.9,3)--(-0.1,3)
  node[midway,above] { $\sigma_{i-1}$ };
\draw[->,very thick] (4.1,3)--(4.9,3)
  node[midway,above] { $\sigma_i$ };
\draw[->,very thick] (9.1,3)--(9.9,3)
  node[midway,above] { $\sigma_{i+1}$ };

\begin{scope}[xshift=2cm,yshift=3cm,rotate=10]
\draw[black,fill=green!50] (-1.3,-2.3) rectangle (1.3,2.3);
\node[scale=0.75] at (0,1.9) { $x_{i-1}{+}H_{i-1}$ };
\end{scope}

\draw[black,fill=red!50] (1.2,1) rectangle (3,4);
\node[scale=0.75] at (2.1,3.6) { $y_{i-1}{+}H_{i-1,\min}$ };
\draw[black,opacity=0.8,fill=purple!30] (1.5,1.5) rectangle (2.9,2.5);
\node[scale=0.75] at (2.2,2) { $y_{i-1}{+}H'_{i-1}$ };

\begin{scope}[xshift=7cm,yshift=3cm]
\draw[black,fill=green!50,rotate=-40] (-1,-1.3) rectangle (1,1.3);
\node[scale=0.75] at (0.15,1.1) { $x_i{+}H_i$ };
\draw[black,fill=red!50] (-0.8,-0.8) rectangle (0.6,0.4);
\node[scale=0.75] at (-0.1,0.1) { $y_i{+}H_{i,\min}$ };
\draw[black,opacity=0.8,fill=purple!30] (-0.7,-0.7) rectangle (0.5,-0.3);
\node[scale=0.75] at (-0.1,-0.5) { $y_i{+}H'_{i,\zealous}$ };
\end{scope}

\begin{scope}[very thick, rounded corners=10pt]
\draw (0.1,0.2) rectangle (3.9,5.8);
\draw (5.1,0.2) rectangle (8.9,5.8);
\end{scope}
\end{tikzpicture}
\hfill \null

\caption{The method of adaptive precision: second attempt}
\label{fig:adaptive2}
\end{figure}
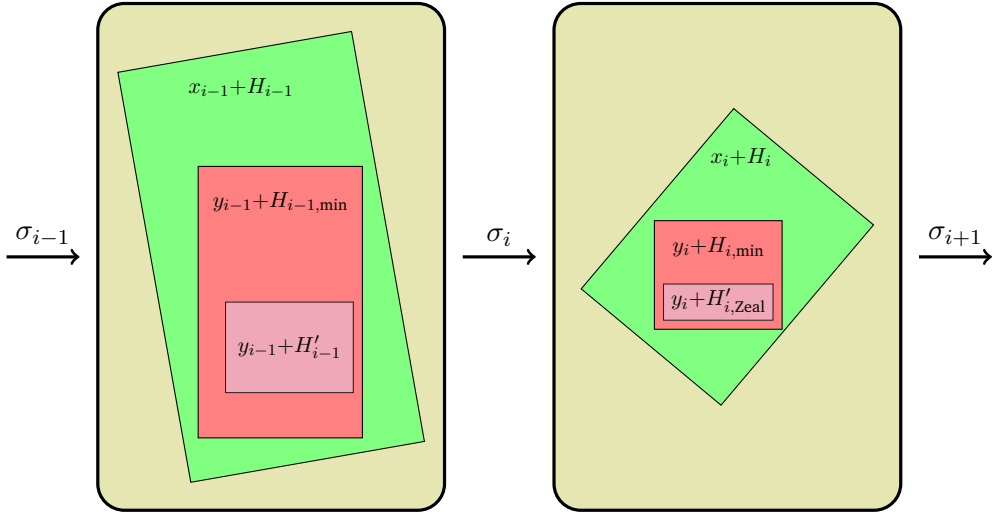
Of course the value $y_i$ we have just defined corresponds 
exactly to the variable $\tty_i$ of the procedure 
\texttt{Phi\_stabilized}. The following lemma is the key for proving the 
correctness of the algorithm \texttt{Phi\_stabilized\_v2}.

\begin{lem}
\label{lem:adaptcorrect}
For all $i \in \{1,\ldots,n\}$, we have $y_i \in x_i + H_i$.
\end{lem}

\begin{proof}
We argue by induction on $i$. The initialization is obvious given that 
$y_0 = x = x_0$. We now assume that $y_{i-1} \in x_{i-1} + H_{i-1}$. 
Applying $\sigma_i$, we derive
\begin{equation}
\label{eq:adapt2step}
\sigma_i(y_{i-1}) \in \sigma_i(x_{i-1} + H_{i-1}) = 
\sigma_i \circ \varphi_{i-1}(x + H) = \varphi_i (x+H) = x_i + H_i
\end{equation}
thanks to Eq.~\eqref{eq:adapt2} applied twice. Moreover, by
construction, we know that $\sigma_i(y_{i-1}) \in y_i + H'_{i,\zealous}
\subset y_i + H_{i,\min} \subset y_i + H_i$. As a consequence the
difference $\sigma_i(y_{i-1}) - y_i$ lies in the lattice $H_i$; this
can be rewritten again as $y_i \in \sigma_i(y_{i-1}) + H_i$. Combining
with Eq.~\eqref{eq:adapt2step}, we get $y_i \in x_i + H_i$ as desired.
\end{proof}

\noindent
We derive from Lemma~\ref{lem:adaptcorrect} that $x_n + H_n = y_n +
H_n$. Combining this equality with Eq.~\eqref{eq:adapt2}, we obtain
$\varphi(x+H) = x_n + H_n = y_n + H_n \subset y_n + H_{\max,n}$.
This inclusion exactly means that \texttt{Phi\_stabilized} called on 
the input $x$ given at precision $H$ outputs an approximation of 
$\varphi(x)$ at precision $H_{\max,n}$; in other words, the algorithm 
\texttt{Phi\_stabilized} is correct.

\smallskip

The main advantage of the algorithm \texttt{Phi\_stabilized\_v2} 
(compared to \texttt{Phi\_stabilized\_v1}) is that the precision encoded 
by the lattices $H'_{i-1}$ are usually acceptable (compared to the 
precision encoded by $H'$). Actual computations are then carried out
at a reasonable precision and then consume moderated resources.

\begin{rem}
Interestingly, observe that the spirit of algorithm 
\texttt{Phi\_stabilized\_v2} is close to the behavior of $p$-adic 
floating-point arithmetic; precisely, in $p$-adic floating-point 
arithmetic, we are decomposing \texttt{Phi} as a sequence of elementary 
steps (each of them consists of a single arithmetic operation) and the 
lattices $H_{i,\min}$ are chosen in such a way that the number of 
significand digits of each $p$-adic variable is kept constant. For 
this reason, the method of adaptive precision can be also used to 
derive \emph{proofs} on the results obtained \emph{via} $p$-adic 
floating-point arithmetic.
\end{rem}

\paragraph{Context of lazy/relaxed arithmetic: target precision}

We recall that the question addressed in the context of lazy/relaxed 
arithmetic is formulated as follows: we fix a target precision on the 
output and we want to avoid as much as possible the computation with
unnecessary digits of the 
input, that are the digits that do not influence the output at the 
requested precision. For simplicity, it is convenient to assume that the 
output of \texttt{Phi} consists of a single $p$-adic 
number\footnote{This assumption is harmless since one can always reduce 
to this case by projection on each coordinate.}, that is $F = \Qp$. The 
target precision is then given by a single integer $N$.

We recall from \S \ref{sssec:preclemmaatwork} that, if $H'$ is a lattice 
in $E$, then the condition:
\begin{equation}
\label{eq:adaptlazy}
\varphi(x+H') \subset \varphi(x) + p^N \Zp
\end{equation}
ensures that the knowledge of $x$ at precision $H'$ is enough to compute 
$\varphi(x)$ at precision $O(p^N)$. 
Moreover, assuming that $H'$ is nice enough so that the precision Lemma 
applies, the inclusion \eqref{eq:adaptlazy} is equivalent to $H' \subset 
d \varphi_x^{-1}(p^N \Zp)$. We draw the attention of the reader to the
fact that $d \varphi_x^{-1}(p^N \Zp)$ itself is usually \emph{not} a 
lattice because $d \varphi_x^{-1}$ often has a nontrivial kernel.
Anyway, 
the above analysis gives a theoretical answer to the question we are 
interested in: the unnecessary digits on the inputs are those lying in 
some nice enough lattice contained in $d\varphi_x^{-1}(p^N \Zp)$.

Mimicking what we have done before in the zealous context, we can use the 
above result to write down a first stabilized version of \texttt{Phi}. 
Set $H = d\varphi_x^{-1} (\Zp)$ and choose a lattice $H_\min$ contained 
in $H$. Clearly $p^N H_\min$ is a lattice contained in $d\varphi_x^{-1} 
(p^N \Zp)$. Moreover, examining closely the statement of the precision 
Lemma, we find that there exists an integer $N_0$ such that:
\begin{equation}
\label{eq:adaptlazyv1}
\varphi(x + p^N H) = \varphi(x) + p^N d \varphi_x(H) \subset
\varphi(x) + p^N \Zp
\end{equation}
for all $N \geq N_0$. The first stabilized version of \texttt{Phi}
then can be described as follows.

\begin{lstlisting}
    def Phi_stabilized_v1(x):
        def Phi_of_x(N):
            xapp = x((*$p^{\max(N,N_0)} H$*))    # (*\color{comment}\textrm{Evaluation of $\ttx$ at precision $p^{\min(N,N_0)} H$}*)
            return Phi(xapp) % (*$p^N$*)
        return Phi_of_x
\end{lstlisting}

\begin{rem}
In the code above, the variable \texttt{x} is a tuple of lazy $p$-adic 
numbers (as described in \S \ref{ssec:lazy}), not a tuple of relaxed 
$p$-adic numbers. Making the method of adaptive precision work within 
the framework of relaxed arithmetic is possible but tedious; for 
simplicity, we shall omit these technicalities and restrict 
ourselves to lazy $p$-adic numbers in this course.

The attentive reader has certainly noticed that the value $N_0$ has 
appeared explicitly in the function \texttt{Phi\_stabilized\_v1} above. 
This implies that we need to precompute it. This has to be done by hand 
\emph{a priori} by explicating the constants in the precision Lemma in 
the special case we are considering; we do not hide that this may 
require some effort (through general techniques for that are 
available~\cite{CaRoVa14,CaRoVa15}).
\end{rem}

Eq.~\eqref{eq:adaptlazyv1} shows that the nested function 
\texttt{Phi\_of\_x} returns the correct answer.
The main benefit of \texttt{Phi\_stabilized\_v1} is that it asks for the 
computation of the input at sharp precision. Nevertheless, it carries 
all the computations with exact rational numbers and so still consumes 
a lot of resources.

In order to get rid of this disadvantage, we can follow the strategy 
already used with success in the zealous context: we decompose 
\texttt{Phi} into $n$ steps $\texttt{Step}_1, \ldots, \texttt{Step}_n$. 
Again we call $\sigma_i : U_{i-1} \to U_i$ the mathematical function 
modeling $\texttt{Step}_i$. We set $\varphi_i = \sigma_i \circ \cdots 
\circ \sigma_1$, $x_i = \varphi_i(x)$ and introduce moreover the 
function $\psi_i = \sigma_n \circ \cdots \circ \sigma_{i+1}$. Clearly 
$\varphi = \psi_i \circ \varphi_i$ for all $i$. Interestingly, observe 
that the relation $d \varphi_x = d \psi_{i,x_i} \circ d \varphi_{i,x}$ 
implies that all the differentials $d \psi_{i,x_i}$'s are surjective 
without any further assumption. For each $i \in \{0,\ldots, n\}$, we 
define $H_i = d \psi_{i,x_i}^{-1} (\Zp)$ (which is generally \emph{not} 
a lattice) and choose a lattice $H_{i,\min}$ contained in $H_i$. The 
optimized form of the stabilized version of \texttt{Phi} then writes as 
follows.

\begin{lstlisting}
    def Phi_stabilized_v2(x):
        def Phi_of_x(N):
            (*$\tty_0$*) = x((*$p^N H_{0,\min}$*))   # (*\color{comment}\textrm{Evaluation of $\ttx$ at precision $p^N H_{0,\min}$}*)
            for (*$i$*) in (*$1, 2, \ldots, n$*):
                (*$\tty_i$*) = (*$\texttt{Step}_i(\tty_{i-1})$*) % (*$p^N H_{i,\min}$*)
            return (*$\tty_n$*) % (*$p^N$*)
        return Phi_of_x
\end{lstlisting}

\noindent
Roughly speaking, the above optimization allows us to ``do the modulo'' 
after each step, avoiding then the explosion of the size of the 
intermediate values.

\begin{prop}
\label{prop:adaptlazy}
There exists an integer $N_0$ (which can be made explicit with some
effort) such that the nested function \textrm{\tt Phi\_of\_x} outputs a
correct answer on the inputs $N \geq N_0$.
\end{prop}

\begin{rem}
We underline that Proposition~\ref{prop:adaptlazy} does not assume the 
surjectivity of the differentials $d \varphi_{i,x}$; recall that this 
assumption was unavoidable in the case of zealous arithmetic and 
appeared as a serious limitation of the method. We have to mention 
however that the analysis of the algorithm \texttt{Phi\_stabilized\_v2} 
is easier under this additional hypothesis (see \S \ref{sssec:somos} 
below for a concrete example).
\end{rem}

\begin{proof}[Sketch of the proof of Proposition~\ref{prop:adaptlazy}]
Given an integer $N$, we let $(y_{N,i})_{0 \leq i \leq n}$ be a 
sequence such that $y_{N,0} = x$ and $y_{N,i} = \sigma_i(y_{N,i-1}) +
p^N H_{i,\min}$.
The value $y_{N,i}$ corresponds to the value of the variable $\tty_i$ 
computed by the function \texttt{Phi\_of\_x} called on the input $N$. 
Moreover, for all fixed $i$, the sequence $(y_{N,i})_{N \geq 0}$ 
converges to $x_i$ when $N$ goes to infinity. Since $\psi_i$ is a 
function of class $C^1$, this implies that $d\psi_{i,y_{N,i}}$ converges 
to $d\psi_{i,x_i}$ when $N$ goes to infinity. From this, we derive that
$d \psi_{i,y_{N,i}}(H_{i,\min}) = d \psi_{i,x_i}(H_{i,\min})$
for $N$ large enough. On the other hand, analyzing closely the proof of the precision 
Lemma, we find that there exists an integer $N_0$ such that:
\begin{align}
\psi_i\big(y_{N,i} + p^N H_{i,\min}\big) 
& = \psi_i(y_{N,i}) + d \psi_{i,y_{N,i}}\big(p^N H_{i,\min}\big) \nonumber \\
& = \psi_i(y_{N,i}) + p^N d \psi_{i,x_i}\big(H_{i,\min}\big)
\label{eq:adaptlazyv2}
\end{align}
for all $i$ and all $N \geq N_0$. We claim that \textrm{\tt Phi\_of\_x} 
outputs a correct answer (\emph{i.e.} that $y_{N,n} \equiv \varphi(x) 
\pmod{p^N}$) whenever $N \geq N_0$. More precisely, we are going to prove by
induction on $i$ that $\psi_i(y_{i,N}) \equiv \varphi(x) \pmod{p^N}$
for $i \in \{0, 1, \ldots, n\}$ and $N \geq N_0$. This is obvious
when $i = 0$ given that $\psi_0 = \varphi$ and $y_{0,N} = x$. Let us
now assume $\psi_{i-1}(y_{i-1,N}) \equiv \varphi(x) \pmod{p^N}$. 
Noting than $\psi_i \circ \sigma_i = \psi_{i-1}$, we derive from
$\sigma_i(y_{N,i-1}) \in y_{N,i} + p^N H_{i,\min}$ that:
\begin{align*}
\psi_{i-1}(y_{N,i-1}) \in \psi_i\big(y_{N,i} + p^N H_{i,\min}\big)
& = \psi_i(y_{N,i}) + p^N d \psi_{i,x_i}\big(H_{i,\min}\big) \\
& \subset \psi_i(y_{N,i}) + p^N \Zp
\end{align*}
using Eq.~\eqref{eq:adaptlazyv2}.
Consequently $\psi_i(y_{N,i}) \equiv \psi_{i-1}(y_{N,i-1}) 
\pmod{p^N}$ and we are done thanks to the induction hypothesis.
\end{proof}

\subsubsection{Example: The Somos 4 sequence}
\label{sssec:somos}

The method of adaptive precision (presented in \S 
\ref{sssec:adaptiveprecision}) has been used with success in several 
concrete contexts: the computation of the solutions of some $p$-adic 
differential equations~\cite{LaVa16} and the computation of 
\textsc{gcd}'s and Bézout coefficients \emph{via} the usual extended 
Euclidean algorithm~\cite{Ca17}.
Below, we detail another example, initially pointed out by Buhler and 
Kedlaya~\cite{BuKe12}, which is simpler (it actually looks like a toy 
example) but already shows up all the interesting phenomena: it is the 
evaluation of the $p$-adic Somos 4 sequences.

\paragraph{The Somos 4 sequence: definition, properties and calculations}

The Somos 4 sequences were first introduced in~\cite{So89} by Somos.
It is a recursive sequence defined by the recurrence:
$$\begin{array}{l}
u_1 = a \quad ; \quad 
u_2 = b \quad ; \quad 
u_3 = c \quad ; \quad 
u_4 = d \medskip \\
\displaystyle
u_{n+4} = \frac{u_{n+1} u_{n+3} + u_{n+2}^2}{u_n} 
\quad \text{for} \quad n \geq 1
\end{array}$$
where $a$, $b$, $c$ and $d$ are given initial values.
\begin{figure}
\tiny \hfill
$\begin{array}{r@{\hspace{0.5ex}}r@{\hspace{0.2ex}}l}
u_5 = & 
\frac 1 a &  \big(c^{2} + b d\big)
\medskip \\
u_6 = & 
\frac 1{ab} &  \big(c^{3} + b c d + a d^{2}\big)
\medskip \\
u_7 = & 
\frac 1{a^2 bc} &  \big(b c^{4} + 2 b^{2} c^{2} d + a c^{3} d + b^{3} d^{2} + a b c d^{2} + a^{2} d^{3}\big)
\medskip \\
u_8 = & 
\frac 1{a^3b^2cd} &  \big(b^{2} c^{6} + a c^{7} + 3 b^{3} c^{4} d + 3 a b c^{5} d + 3 b^{4} c^{2} d^{2} + 3 a b^{2} c^{3} d^{2} + 2 a^{2} c^{4} d^{2} + b^{5} d^{3} + a b^{3} c d^{3} + 3 a^{2} b c^{2} d^{3} + a^{2} b^{2} d^{4} + a^{3} c d^{4}\big)
\medskip \\
u_9 = & 
\frac 1{a^3b^2c^2d} &  \big(b^{2} c^{8} + a c^{9} + 4 b^{3} c^{6} d + 3 a b c^{7} d + 6 b^{4} c^{4} d^{2} + 6 a b^{2} c^{5} d^{2} + 3 a^{2} c^{6} d^{2} + 4 b^{5} c^{2} d^{3} + 7 a b^{3} c^{3} d^{3} + 6 a^{2} b c^{4} d^{3} + b^{6} d^{4} \\
& & \hspace{5mm}
{}+ 3 a b^{4} c d^{4} + 5 a^{2} b^{2} c^{2} d^{4} + 3 a^{3} c^{3} d^{4} + 2 a^{2} b^{3} d^{5} + 3 a^{3} b c d^{5} + a^{4} d^{6}\big)
\medskip \\
u_{10} = & 
\frac 1{a^5b^3c^3d^2} & \big(b^{4} c^{10} + 2 a b^{2} c^{11} + a^{2} c^{12} + 5 b^{5} c^{8} d + 11 a b^{3} c^{9} d + 6 a^{2} b c^{10} d + 10 b^{6} c^{6} d^{2} + 24 a b^{4} c^{7} d^{2} + 17 a^{2} b^{2} c^{8} d^{2} + 4 a^{3} c^{9} d^{2} \\
& \phantom{\frac 1{a^5b^3c^3d^2}} & \hspace{5mm}
{}+ 10 b^{7} c^{4} d^{3} + 26 a b^{5} c^{5} d^{3} + 28 a^{2} b^{3} c^{6} d^{3} + 15 a^{3} b c^{7} d^{3} + 5 b^{8} c^{2} d^{4} + 14 a b^{6} c^{3} d^{4} + 27 a^{2} b^{4} c^{4} d^{4} + 24 a^{3} b^{2} c^{5} d^{4} \\
& \phantom{\frac 1{a^5b^3c^3d^2}} & \hspace{5mm}
{}+ 6 a^{4} c^{6} d^{4} + b^{9} d^{5} + 3 a b^{7} c d^{5} + 14 a^{2} b^{5} c^{2} d^{5} + 19 a^{3} b^{3} c^{3} d^{5} + 12 a^{4} b c^{4} d^{5} + 3 a^{2} b^{6} d^{6} + 6 a^{3} b^{4} c d^{6} + 9 a^{4} b^{2} c^{2} d^{6} \\
& \phantom{\frac 1{a^5b^3c^3d^2}} & \hspace{5mm}
{}+ 4 a^{5} c^{3} d^{6} + 3 a^{4} b^{3} d^{7} + 3 a^{5} b c d^{7} + a^{6} d^{8}\big)
\end{array}$
\hfill \null

\caption{The first terms of the generic SOMOS 4 sequence}
\label{fig:somos}
\end{figure}
The first values of the generic Somos 4 sequence (that is the Somos
4 sequence where the initial values are left as indeterminates) are
shown on Figure~\ref{fig:somos}.
We observe a quite surprising property: the denominators of the first 
$u_i$'s are all monomials. We insist on the fact that this behavior
is \emph{a priori} not expected at all. For instance the definition of
$u_9$ involves a division by $u_5$ but a miraculous simplification
occurs.
This phenomenon is in fact general. Fomin and Zelevinsky proved in 
\cite{FoZe02} (as an application of their theory of cluster algebras) 
that the Somos 4 sequence exhibits the \emph{Laurent phenomenon}: for all $n$, 
the term $u_n$ of the generic Somos 4 sequence is a Laurent polynomial
in $a$, $b$, $c$, $d$, that is a polynomial in $a$, $b$, $c$, $d$ and
their inverses.

We address the question of the computation of the $n$-th term (for $n$ 
large) of a Somos 4 sequence whose initial values $a$, $b$, $c$ and $d$ 
are invertible elements of $\Zp$. At least two options are available: 
(1)~we unroll the recurrence starting from the given initial values 
and (2)~we precompute the $n$-th term of the generic Somos 4 
sequence and, in a second time, we plug the values of $a$, $b$, $c$ and 
$d$. Concretely, they correspond to the following codes respectively:

\begin{lstlisting}
    def somos_option1(a,b,c,d,n):     # (*\color{comment}\textrm{We assume }$n \geq 5$*)
        x,y,z,t = a,b,c,d
        for i in (*$1, \, 2, \, \ldots, \, \texttt{n}{-}4$*):
            x,y,z,t = y, z, t, (y*t + z*z)/x
        return t
\end{lstlisting}

\begin{lstlisting}
    def somos_option2(a,b,c,d,n):     # (*\color{comment}\textrm{We assume }$n \geq 5$*)
        X,Y,Z,T = A,B,C,D             # (*\color{comment}\textrm{{\tt A}, {\tt B}, {\tt C}, {\tt D} are formal variables}*)
        for i in (*$1, \, 2, \, \ldots, \, \texttt{n}{-}4$*):
            X,Y,Z,T = Y, Z, T, (Y*T + Z*Z)/X
        return T(A=a, B=b, C=c, D=d)
\end{lstlisting}

The second option has of course a terrible complexity because the size 
of the terms of the generic Somos 4 sequence grows very rapidly (see 
Figure~\ref{fig:somos}). However, if we are mostly interested in 
numerical stability, this option is rather attractive; indeed, we deduce 
from the Laurent phenomenon that if the initial values $a$, $b$, $c$ and 
$d$ (which are supposed to be invertible in $\Zp$) are given at 
precision $O(p^N)$ then all the next terms will be computed at precision 
$O(p^N)$ as well. 
On the contrary, the first option may lead to quite important losses of
precision. 
\begin{figure}
\renewcommand{\arraystretch}{1.2}
\hfill \small
$\begin{array}{|c|p{2.6cm}|p{2.6cm}|p{2.6cm}||p{2.6cm}|}
\cline{2-5} 
\multicolumn{1}{c|}{\phantom{\Big(}}
& \hfill\text{Laurent method}\hfill\null 
& \hfill\text{Zealous arith.}\hfill\null 
& \hfill\text{Float. arith.}\hfill\null 
& \hfill\text{Relaxed arith.}\hfill\null \\
\hline
u_1 & \hfill$\ldots0000000001$ & \hfill$\ldots0000000001$ &  \hfill$\ldots0000000001$ & \hfill --- \hfill\null \\
u_2 & \hfill$\ldots0000000001$ & \hfill$\ldots0000000001$ &  \hfill$\ldots0000000001$ & \hfill --- \hfill\null \\
u_3 & \hfill$\ldots0000000001$ & \hfill$\ldots0000000001$ &  \hfill$\ldots0000000001$ & \hfill --- \hfill\null \\
u_4 & \hfill$\ldots0000000001$ & \hfill$\ldots0000000001$ &  \hfill$\ldots0000000001$ & \hfill --- \hfill\null \\
u_5 & \hfill$\ldots0000000010$ & \hfill$\ldots0000000010$ &  \hfill$\ldots0000000010$ & {\scriptsize precision on $u_1$:} \hfill $10$ \\
u_6 & \hfill$\ldots0000000011$ & \hfill$\ldots0000000011$ &  \hfill$\ldots0000000011$ & {\scriptsize precision on $u_1$:} \hfill $10$ \\
u_7 & \hfill$\ldots0000000111$ & \hfill$\ldots0000000111$ &  \hfill$\ldots0000000111$ & {\scriptsize precision on $u_1$:} \hfill $10$ \\
u_8 & \hfill$\ldots0000010111$ & \hfill$\ldots0000010111$ &  \hfill$\ldots0000010111$ & {\scriptsize precision on $u_1$:} \hfill $10$ \\
u_9 & \hfill$\ldots0000111011$ &  \hfill$\ldots000111011$ &  \hfill$\ldots0000111011$ & {\scriptsize precision on $u_1$:} \hfill $11$ \\
u_{10}
    & \hfill$\ldots0100111010$ &  \hfill$\ldots100111010$ & \hfill$\ldots10100111010$ & {\scriptsize precision on $u_1$:} \hfill $11$ \\
u_{50} 
    & \hfill$\ldots0000011010$ &          \hfill$\ldots0$ & \hfill$\ldots11000011010$ & {\scriptsize precision on $u_1$:} \hfill $19$ \\
u_{54} 
    & \hfill$\ldots0100101001$ &  \hfill \textsc{boom} \hfill\null &  \hfill$\ldots1100101001$ & {\scriptsize precision on $u_1$:} \hfill $20$ \\
u_{500} 
    & \hfill$\ldots1111110010$ &  \hfill --- \hfill\null  & \hfill$\ldots01111110010$ & {\scriptsize precision on $u_1$:} \hfill $109$ \\
\hline
\end{array}$
\hfill\null

\caption{The Somos 4 sequence with initial values $1, 1, 1, 1 \in \Z_2$}
\label{fig:somos1111}
\end{figure}
\begin{figure}
\renewcommand{\arraystretch}{1.2}
\hfill \small
$\begin{array}{|c|p{2.6cm}|p{2.6cm}|p{2.6cm}||p{2.6cm}|}
\cline{2-5}
\multicolumn{1}{c|}{\phantom{\Big(}}
& \hfill\text{Laurent method}\hfill\null 
& \hfill\text{Zealous arith.}\hfill\null 
& \hfill\text{Float. arith.}\hfill\null 
& \hfill\text{Relaxed arith.}\hfill\null \\
\hline
u_1 & \hfill$\ldots0000000001$ & \hfill$\ldots0000000001$ &  \hfill$\ldots0000000001$ & \hfill --- \hfill\null \\
u_2 & \hfill$\ldots0000000001$ & \hfill$\ldots0000000001$ &  \hfill$\ldots0000000001$ & \hfill --- \hfill\null \\
u_3 & \hfill$\ldots0000000001$ & \hfill$\ldots0000000001$ &  \hfill$\ldots0000000001$ & \hfill --- \hfill\null \\
u_4 & \hfill$\ldots0000000011$ & \hfill$\ldots0000000011$ &  \hfill$\ldots0000000011$ & \hfill --- \hfill\null \\
u_5 & \hfill$\ldots0000000100$ & \hfill$\ldots0000000100$ &  \hfill$\ldots000000000100$ & {\scriptsize precision on $u_1$:} \hfill $10$ \\
u_6 & \hfill$\ldots0000001101$ & \hfill$\ldots0000001101$ &  \hfill$\ldots0000001101$ & {\scriptsize precision on $u_1$:} \hfill $10$ \\
u_7 & \hfill$\ldots0000110111$ & \hfill$\ldots0000110111$ &  \hfill$\ldots0000110111$ & {\scriptsize precision on $u_1$:} \hfill $10$ \\
u_8 & \hfill$\ldots0111010111$ & \hfill$\ldots0111010111$ &  \hfill$\ldots0111010111$ & {\scriptsize precision on $u_1$:} \hfill $10$ \\
u_9 & \hfill$\ldots0111101111$ & \hfill$\ldots11101111$ &  \hfill$\ldots1111101111$ & {\scriptsize precision on $u_1$:} \hfill $12$ \\
u_{10}
    & \hfill$\ldots0000010010$ & \hfill$\ldots00010010$ &  \hfill$\ldots11000010010$ & {\scriptsize precision on $u_1$:} \hfill $12$ \\
u_{11} 
    & \hfill$\ldots1000111001$ & \hfill$\ldots00111001$ &  \hfill$\ldots0000111001$ & {\scriptsize precision on $u_1$:} \hfill $12$ \\
u_{12}
    & \hfill$\ldots0011111101$ & \hfill$\ldots11111101$ &  \hfill$\ldots0011111101$ & {\scriptsize precision on $u_1$:} \hfill $12$ \\
u_{13}
    & \hfill$\ldots0000110101$ & \hfill$\ldots00110101$ &  \hfill$\ldots0000110101$ & {\scriptsize precision on $u_1$:} \hfill $12$ \\
u_{14}
    & \hfill$\ldots1101010011$ & \hfill$\ldots1010011$ &  \hfill$\ldots1101010011$ & {\scriptsize precision on $u_1$:} \hfill $13$ \\
u_{15}
    & \hfill$\ldots0000000000$ & \hfill$\ldots0000000$ &  \hfill$0$ & {\scriptsize precision on $u_1$:} \hfill $13$ \\
u_{16}
    & \hfill$\ldots0101011101$ & \hfill$\ldots1011101$ &  \hfill$\ldots0101011101$ & {\scriptsize precision on $u_1$:} \hfill $13$ \\
u_{17}
    & \hfill$\ldots1001101011$ & \hfill$\ldots1101011$ &  \hfill$\ldots1001101011$ & {\scriptsize precision on $u_1$:} \hfill $13$ \\
u_{18}
    & \hfill$\ldots0011110011$ & \hfill$\ldots1110011$ &  \hfill$\ldots0011110011$ & {\scriptsize precision on $u_1$:} \hfill $13$ \\
u_{19}
    & \hfill$\ldots0000000111$ & \hfill\textsc{boom}\hfill\null &  \hfill\textsc{boom}\hfill\null & {\scriptsize precision on $u_1$:} \hfill $23$ \\
\hline
\end{array}$
\hfill\null

\caption{The Somos 4 sequence with initial values $1, 1, 1, 3 \in \Z_2$}
\label{fig:somos1113}
\end{figure}

The table of Figure~\ref{fig:somos1111} shows the results obtained for 
the particular Somos 4 sequence initialized with $a = b = c = d = 1 \in 
\Z_2$ using various methods. The first column corresponds to the second 
option discussed above; as expected, no losses of precision have to be 
reported. On the second column (resp. the third column), we have 
displayed the results obtained by unrolling the recurrence using the 
algorithm \texttt{somos\_option1} within the framework of zealous 
arithmetic (resp. $p$-adic floating-point arithmetic). We observe that 
the precision decreases rapidly in the second column. For example, one 
digit of precision is lost when computing $u_9$ because of the division 
by $u_5$ which has valuation $1$. These losses of precision continue 
then regularly and, at some point, we find a value which becomes 
indistinguishable from $0$. Dividing by it then produces an error and 
the computations cannot continue. The third column exhibits much better
results though certain digits are still wrong (compare with the first
column).
Finally, the last column concerned lazy/relaxed arithmetic; it shows the 
number of digits of $u_1$ that have been asked for in order to compute 
the term $u_n$ at precision $O(2^{10})$. We observe that this number
increases regularly in parallel with the losses of precision observed
in the zealous context. For instance, the relaxed technology needs one 
more digit of precision on $u_1$ to compute $u_9$ while the zealous
approach loses one digit of precision on $u_9$. There phenomena have of 
course a common explanation: they are both due to the division by $u_5$ 
which has valuation $1$.

The Somos 4 sequence starting with $(a, b, c, d) = (1,1,1,3)$ is even 
more unstable (see Figure~\ref{fig:somos1113}); indeed its $15$-th term 
is divisible by $2^{10}$ and so vanishes at the precision $O(2^{10})$ we 
are considering. Zealous arithmetic and floating-point arithmetic are 
both puzzled after this underflow and crash on the computation of 
$u_{19}$ (which requires a division by $u_{15}$).

\paragraph{Stabilization}

We now apply the theory of adaptive precision, trying to stabilize 
this way the algorithm \texttt{somos\_version1}. The decomposition of 
the routine \texttt{somos\_version1} into steps in straightforward: each 
step corresponds to an interaction of the for loop. They are all modeled 
by the same mathematical function:
$$\begin{array}{rcl}
\sigma : \qquad \Qp^4 & \longrightarrow & \Qp^4 \smallskip \\
(x,y,z,t) & \mapsto & \Big(y,z,t,\frac{yt+z^2}x\Big).
\end{array}$$

\begin{rem}
Of course $\sigma$ is not defined at the points of the form $(0,y,z,t)$. 
In the sequel, it will be nevertheless more convenient to regard 
$\sigma$ as a partially defined function over $\Qp^4$ (than as a 
well-defined function over $\Qp^\star \times \Qp^3$).
\end{rem}

\subparagraph{Differential computations.}

For $i \geq 0$, we set $\varphi_i = \sigma \circ \cdots \circ \sigma$
($i$ times). Clearly $\sigma$ and the $\varphi_i$'s are all functions
of class $C^1$ on their domain of definition. The Jacobian matrix of 
$\sigma$ at the point $(x,y,z,t)$ is easily seen to be:
$$J(\sigma)_{(x,y,z,t)} = 
\left( \begin{array}{c@{\hspace{4ex}}c@{\hspace{4ex}}c@{\hspace{2.5ex}}c}
0 & 0 & 0 & -\frac{yt+z^2}{x^2} \smallskip \\
1 & 0 & 0 & \frac t x \smallskip \\
0 & 1 & 0 & \frac{2z} x \smallskip \\
0 & 0 & 1 & \frac y x 
\end{array} \right).$$
The computation of the Jacobian matrix of $\varphi_i$ is more 
complicated. It is however remarkable that we can easily have access to 
its determinant. Indeed, observe that the determinant of $J(\sigma) 
_{(x,y,z,t)}$ is equal to $\frac {t'} x$ where $t' = \frac{yt+z^2}x$ is 
the last coordinate of $\sigma(x,y,z,t)$. By the chain rule, we then
deduce that:
\begin{equation}
\label{eq:detjacsomos}
\det J(\varphi_i)_{(a,b,c,d)} = 
\frac{u_5}{u_1} \cdot \frac {u_6}{u_2} \cdots \frac{u_{i+4}}{u_i} =
\frac{u_{i+1} \cdot u_{i+2} \cdot u_{i+3} \cdot u_{i+4}}{u_1 \cdot u_2 \cdot u_3 \cdot u_4} =
\frac{u_{i+1} \cdot u_{i+2} \cdot u_{i+3} \cdot u_{i+4}}{abcd}
\end{equation}
where the $u_i$'s are the terms of the Somos 4 sequence initialized 
by $(u_1, u_2, u_3, u_4) = (a,b,c,d)$.
Apart from that, the matrix $J(\varphi_i)_{(a,b,c,d)}$ has another 
remarkable property: if $a,b,c,d$ are invertible in $\Zp$, all the 
entries of $J(\varphi_i)_{(a,b,c,d)}$ lie in $\Zp$; indeed they can
be obtained alternatively as the partial derivatives of the terms of 
the generic Somos 4 sequence in which only divisions by $a$, $b$,
$c$ and $d$ occur. As a consequence, we derive the double inclusion:
\begin{align}
\label{eq:latticesomos}
& p^{v(i)} \cdot \Zp^4 \, \subset \, 
  d \varphi_{i,(a,b,c,d)}\big(\Zp^4\big) \, \subset \, \Zp^4\\
\text{with} \quad &
  v(i) = \val(u_{i+1}) + \val(u_{i+2}) + \val(u_{i+3}) + \val(u_{i+4})
\nonumber
\end{align}
which is the key for making the method of adaptive precision work.

\begin{lem}
\label{lem:somosvi}
For all $i$, there is at most one non invertible element among four
consecutive terms of the Somos 4 sequence.
In particular:
$$v(i) = \max\big(\val(u_{i+1}), \val(u_{i+2}), \val(u_{i+3}),
\val(u_{i+4})\big).$$
\end{lem}

\begin{proof}
It is enough to prove that if $\val(u_{i-3}) = \val(u_{i-2}) = 
\val(u_{i-1}) = 0$ and $\val(u_i) > 0$, then $\val(u_{i+1}) = 
\val(u_{i+2}) = \val(u_{i+3}) = 0$. This can be derived easily from
the defining formula of the Somos 4 sequence.
\end{proof}

\begin{ex*}
As an instructive example, we take the values $(a,b,c,d) = 
(1,1,1,3)$ (given at precision $O(2^{10})$) and $i = 14$ corresponding 
to the annoying situation reported on Figure~\ref{fig:somos1113} where a 
term of the Somos sequence vanishes at the working precision.
An explicit computation yields:
$$\small J(\varphi_{14})_{(1,1,1,3)} = 
\left(\begin{array}{rrrr}
\ldots0001110110 & \ldots0100010110 & \ldots1110111100 & \ldots0111110000 \\
\ldots1100101001 & \ldots1110001010 & \ldots1101010100 & \ldots0111111101 \\
\ldots0001001100 & \ldots0101001110 & \ldots1010110101 & \ldots0011101100 \\
\ldots0000000111 & \ldots0100100101 & \ldots1011100010 & \ldots0101011110
\end{array}\right).$$
According to the results of \S \ref{sssec:preclemmaatwork}, the $j$-th 
column of this matrix (for $1 \leq j \leq 4$) encodes the loss/gain of
precision on the computation of $u_{j+14}$. Precisely, we observe that
each column of $J(\varphi_{14})_{(1,1,1,3)}$ contains an element of 
valuation zero, meaning that the optimal precision on $u_{15}$, 
$u_{16}$, $u_{17}$ and $u_{18}$ separately is $O(2^{10})$. In 
particular, we cannot decide whether $u_{15}$ actually vanishes or 
does not. This sounds quite weird because we were able in any case to 
compute $u_{19}$ using algorithm \texttt{somos\_option1}. The point is 
that the formula given the value of $u_{19}$ has the shape $u_{19} = 
\frac{\text{numerator}}{u_{15}}$ where the numerator is divisible by 
$u_{15}$. The approach based on the Laurent phenomenon views this
simplification and is therefore able to compute $u_{19}$ even when
$u_{15}$ vanishes!

The analysis of the optional precision on the quadruple $(u_{15}, 
u_{16}, u_{17}, u_{18})$ is interesting as well. However, according to 
Eq.~\eqref{eq:detjacsomos}, the determinant of $J(\varphi_{14})_{(1,1, 
1,3)}$ vanishes at the working precision; we cannot then certify its 
surjectivity and it seems that we are in a deadlock. In order to go 
further, we increase a bit the precision on the initial values: for now 
on, we assume that $a$, $b$, $c$ and $d$ (whose values are $1$, $1$, $1$ 
and $3$ respectively) are given at precision $O(2^N)$ with $N > 10$. 
The computation then shows that the determinant of $J(\varphi_{14})_ 
{(1,1,1,3)}$ has valuation $10$. 
Thanks to Eq.~\eqref{eq:numberdiffused} that the quadruple $(u_{15}, 
u_{16}, u_{17}, u_{18})$ has $10$ diffused digits of precision. We
can even be more precise by computing the Hermite normal form of 
$J(\varphi_{14})_{(1,1,1,3)}$, which is:
$$\left( \begin{array}{c@{\hspace{4ex}}c@{\hspace{4ex}}c@{\hspace{4ex}}c}
1 & 0 & 0 & 179 \\
0 & 1 & 0 & 369 \\
0 & 0 & 1 & 818 \\
0 & 0 & 0 & 2^{10} \\
\end{array}\right)$$
meaning that the particular linear combination $179 \: u_{15} + 369 \: u_{16} 
+ 818 \: u_{17} - u_{18}$ could in principle be determined at precision 
$O(2^{N+10})$. Using this, it becomes possible to determine $u_{19}$ at
precision $O(2^N)$ (while a naive computation will lead to precision
$O(2^{N-10})$ because of the division by $u_{15}$).
\end{ex*}

\subparagraph{Explicit precision Lemma.}

Another task we have to do in order to use the theory of adaptive
precision is to make explicit the constants in the precision Lemma
for the particular problem we are studying. In the case of the Somos
4 sequence, the result we need takes the following form.

\begin{prop}
\label{prop:preclemmasomos}
We keep the above notations and assumptions\footnote{In particular,
we continue to assume that the initial values $a$, $b$, $c$ and $d$
are invertible in $\Zp$.}.
For all $N > v(i)$, we have:
$$\varphi_i\big((a,b,c,d) + p^N \Zp^4\big) =
(u_{i+1}, u_{i+2}, u_{i+3}, u_{i+4}) + 
p^N \: d \varphi_{i,(a,b,c,d)}\big(\Zp^4\big).$$
\end{prop}

\begin{proof}
The proof is a slight generalization (to the case of multivariate 
\emph{Laurent} polynomials) of Proposition \ref{prop:preclemmapoly}. 
In fact, the only missing input is the generalization of Lemma 
\ref{lem:multipolyC1} to our setting. Concretely, we then just have 
to establish that: 
\begin{equation}
\label{eq:multilaurentC1}
\Vert \varphi_i(v_2) - \varphi_i(v_1) - d\varphi_{i,v}(v_2{-}v_1) \Vert 
  \leq \max
  \big(\Vert v_1{-}v \Vert, \Vert v_2{-}v \Vert \big) \cdot \Vert v_2{-}v_1 \Vert
\end{equation}
for all integers $i$ and
all vectors $v, v_1, v_2 \in \Zp^4$ whose all coordinates are
invertible. (Here $\Vert \cdot \Vert$ denotes as usual the infinite
norm.) For doing so, we write $\varphi_i$ as the compositum
$\varphi_i = \tilde \varphi_i \circ \iota$ where $\iota$ is the
(partially defined) embedding:
$$\iota : \Qp^4 \to \Qp^4, \quad (x,y,z,t) \mapsto (x,y,z,t, x^{-1},
y^{-1}, z^{-1}, t^{-1})$$
and $\tilde \varphi_i : \Qp^8 \to \Qp^4$ is a function whose each
coordinate is given by a multivariate polynomial with coefficients 
in $\Zp$. We set $w = \iota(v)$, $w_1 = \iota(v_1)$, $w_2 = \iota
(v_2)$ and compute:
\begin{align*}
\varphi_i(v_2) - \varphi_i(v_1) - d\varphi_{i,v}(v_2{-}v_1)
  & = \tilde \varphi_i(w_2) - \tilde \varphi_i(w_1) - d\tilde \varphi_{i,w}
\circ d \iota_v (v_2{-}v_1) \\
  & = \tilde \varphi_i(w_2) - \tilde \varphi_i(w_1) - d\tilde \varphi_{i,w} (w_2{-}w_1) \\
  & \hspace{1.5cm} {}+ d\tilde \varphi_{i,w}( w_2 - w_1 - d \iota_v (v_2{-}v_1)\big).
\end{align*}
From the shape of $\tilde \varphi_i$, we deduce that the norm of its
differential $d\tilde \varphi_{i,w}$ is at most $1$. Therefore:
\begin{align}
&\Vert \varphi_i(v_2) - \varphi_i(v_1) - d\varphi_{i,v}(v_2{-}v_1) \Vert \smallskip \nonumber \\
&\hspace{1cm} \leq \max\big(
  \Vert \tilde \varphi_i(w_2) - \tilde \varphi_i(w_1) - d\tilde \varphi_{i,w} (w_2{-}w_1) \Vert, \,
  \Vert w_2 - w_1 - d \iota_v (v_2{-}v_1) \Vert\big) \smallskip \nonumber \\
&\hspace{1cm} \leq \max\big(
  \Vert w_1{-}w \Vert {\cdot} \Vert w_2{-}w_1 \Vert,\,
  \Vert w_2{-}w \Vert {\cdot} \Vert w_2{-}w_1 \Vert,\,
  \Vert w_2{-}w_1 - d \iota_v (v_2{-}v_1) \Vert\big).
  \label{eq:multilaurentC1step}
\end{align}
the last inequality coming from Lemma \ref{lem:multipolyC1} applied to a 
shift of the function $\tilde \varphi_i$. Given three invertible 
elements $x, x_1, x_2$ in $\Zp$, we have:
$$\left|\frac 1{x_1} - \frac 1{x_2} - \Big(\frac{-1}{x^2}\Big) \cdot (x_1{-}x_2) \right| 
= \left| \frac{(x_1{-}x_2) \cdot (x_1 x_2 - x^2)}{x^2 x_1 x_2} \right|
= |x_1{-}x_2| \cdot |x_1 x_2 - x^2|.$$
Writing $x_1 x_2 - x^2 = x (x_1 - x) + x_1 (x_2 - x)$ shows
that $|x_1 x_2 - x^2| \leq \max\big(|x_1{-}x|, |x_2{-}x|\big)$. Thus:
$$\Vert w_2{-}w_1 - d \iota_v (v_2{-}v_1) \Vert \leq
\Vert v_2{-}v_1 \Vert \cdot \max\big(\Vert v_1{-}v \Vert, \,
\Vert v_2{-}v \Vert\big).$$
Finally, a straightforward computation shows that $\Vert w_1{-}w \Vert = 
\Vert v_1{-}v \Vert$, $\Vert w_2{-}w \Vert = \Vert v_2{-}v \Vert$ and 
$\Vert w_2{-}w_1 \Vert = \Vert v_2{-}v_1 \Vert$. Inserting these inputs
into Eq.~\eqref{eq:multilaurentC1step}, we get Eq.~\eqref{eq:multilaurentC1} 
as desired.
\end{proof}

\subparagraph{Context of zealous arithmetic.}

We fix a positive integer $N$. We assume that the initial values $a$, 
$b$, $c$ and $d$ are given at precision $O(p^N)$. The corresponding 
lattice is $H = p^n \Zp^4$. By Proposition \ref{prop:preclemmasomos}, 
the precision Lemma applies with $\varphi_i$, $x$ and $H$ as soon as $N 
> v(i)$.

We now follow step by step the method of adaptive precision (see \S 
\ref{sssec:adaptiveprecision}). 
We first set $H_i = d\varphi_{i,(a,b,c,d)}(H) = p^N \: 
d\varphi_{i,(a,b,c,d)} (\Zp^4)$. We then have to exhibit two ``simple'' 
lattices $H_{i,\min}$ and $H_{i,\max}$ that approximates $H_i$ from 
below and from above respectively. The answer is given by the inclusions 
\eqref{eq:latticesomos}: we take $H_{i,\min} = p^{N+v(i)} \Zp^4$ and 
$H_{i,\max} = p^N \Zp^4$. Concerning the lattice $H'_{i-1}$, it can be any 
lattice for which the following requirement holds: if the quadruple 
$(x,y,z,t)$ is known at precision $H'_{i-1}$, then zealous 
arithmetic is able to compute $\sigma(x,y,z,t)$ at precision 
$H'_{i,\min} = p^{N+v(i)} \Zp^4$. Since the division by $x$ induces
a loss of precision of at most $\val(x)$ digits, we can safely take
$H'_{i-1} = p^{N+v(i)+\val(u_i)} \Zp^4$.
Instantiating the routine \texttt{Phi\_stabilized\_v2} (page 
\pageref{algo:phistabzealousv2}) to our setting, we end up with the following stabilized 
version of the procedure \texttt{somos\_option1}.

\begin{lstlisting}
    def somos_stabilized(a,b,c,d,n):
        # (*\color{comment}\textrm{We assume }$\texttt{n} \geq 5$*)
        # (*\color{comment}a,b,c,d\textrm{ are \emph{units} in $\Zp$ given at precision $O(p^N)$}*)
        x,y,z,t = a,b,c,d
        for i in (*$1, \, 2, \, \ldots, \, \texttt{n}{-}4$*):
            u = (y*t + z*z)/x   # (*\color{comment}\textrm{Precomputation at small precision}*)
            v = (*$\val_p$*)(y) + (*$\val_p$*)(z) + (*$\val_p$*)(t) + (*$\val_p$*)(u)
            if v(*${}\geq N$*): raise PrecisionError
            lift x,y,z,t (*\textrm{at precision $O(p^{N+\texttt{v}+\val_p(\ttx)})$}*)
            x,y,z,t = y, z, t, (y*t + z*z)/x
        return t (*\textrm{at precision $O(p^N)$}*)
\end{lstlisting}

\noindent
This algorithm always returns a correct answer at precision $O(p^N)$. 
It might fail --- and then it raises an error --- if the precision on the 
entries is not sufficient; by Lemma \ref{lem:somosvi}, this occurs 
exactly when one term of the Somos 4 sequence we are computing is 
indistinguishable from $0$. One possible solution in that case is to 
increase arbitrarily the precision on $a$, $b$, $c$ and $d$, rerun the 
algorithm and finally truncate the obtained result back to precision 
$O(p^N)$.

\subparagraph{Context of lazy/relaxed arithmetic.}

The literal translation to the lazy world of the stabilized algorithm 
we have designed just above is:

\begin{lstlisting}
    def somos_stabilized(a,b,c,d,n):
        # (*\color{comment}\textrm{We assume }$\texttt{n} \geq 5$*)
        # (*\color{comment}a,b,c,d\textrm{ are lazy \emph{invertible} $p$-adic numbers}*)
        def nth_term(N):
            x,y,z,t = a(N),b(N),c(N),d(N)
            for i in (*$1, \, 2, \, \ldots, \, \texttt{n}{-}4$*):
                u = (y*t + z*z) / x
                v = (*$\val_p$*)(y) + (*$\val_p$*)(z) + (*$\val_p$*)(t) + (*$\val_p$*)(u)
                if v(*${}\geq N$*): raise PrecisionError
                x,y,z,t = y % (*$p^{\texttt{N+v}}$*), z % (*$p^{\texttt{N+v}}$*), t % (*$p^{\texttt{N+v}}$*), u % (*$p^{\texttt{N+v}}$*)
            return t % (*$p^{\texttt N}$*)
        return nth_term
\end{lstlisting}

\begin{rem}
Instead of raising an error if $N$ is too small, it is more clever to
rerun the computation at higher precision. We can implement this idea
simply by replacing
the statement ``\texttt{{\color{blue} raise} PrecisionError}'' by
``\texttt{{\color{blue} return} ntn\_term(2*N)}'' in the code above.
\end{rem}

\noindent 
The correctness of this algorithm is proved similarly to the case of 
zealous arithmetic by applying the precision Lemma to the function 
$\varphi_i$ (see Proposition \ref{prop:preclemmasomos}). We notice, in 
particular, that we do not need here Proposition \ref{prop:adaptlazy}; 
this is nice because making explicit the constant $N_0$ of 
Proposition \ref{prop:adaptlazy} is tedious. This simplification is
due to the fact that the differential $d \varphi_{i,(a,b,c,d)}$ is
surjective (actually even bijective).



\addcontentsline{toc}{section}{\refname}
\begin{thebibliography}{99}
\small
\renewcommand{\itemsep}{0pt}
\bibitem{AbLo04}
  Jounaïdi Abdeljaoued and Henri Lombardi.
  \emph{Méthodes matricielles, introduction à la complexité algébrique}.
  Springer Verlag, Berlin, Heidelberg (2004)

\bibitem{Am75}
  Yvette Amice.
  \emph{Les nombres p-adiques}.
  PUF (1975)

\bibitem{BaBrChDi00}
  Micheal Bartholomew-Biggs, Steven Brown, Bruce Christianson and Laurence Dixon.
  \emph{Automatic differentiation of algorithms}.
  Journal of Comp. and App. Math. \textbf{124} (2000), 171--190

\bibitem{pari}
  Christian Batut, Karim Belabas, Dominique Benardi, Henri Cohen and Michel Olivier.
  \emph{User's guide to {PARI-GP}}. {(1985--2013)}.

\bibitem{Be11}
  Laurent Berger.
  \emph{La correspondance de Langlands locale $p$-adique pour $\GL_2(\Qp)$}.
  Astérisque \textbf{339} (2011), 157--180

\bibitem{BO78}
  Pierre Berthelot and Arthur Ogus.
  \emph{Notes on Crystalline Cohomology}.
  Princeton University Press (1978)

\bibitem{BeHoLe11}
  Jérémy Berthomieu, Joris van~der Hoeven, and Grégoire Lecerf.
  \emph{Relaxed algorithms for $p$-adic numbers}.
  {J. Number Theor. Bordeaux} \textbf{23} (2011), 541--577

\bibitem{BeLe12}
  Jérémy Berthomieu and Romain Lebreton. 
  \emph{Relaxed $p$-adic Hensel lifting for algebraic systems}.
  In ISSAC'12 (2012), 59--66

\bibitem{Bh14}
  Bhargav Bhatt.
  \emph{What is... a Perfectoid Space?}
  Notices of the Amer. Math. Soc. \textbf{61} (2014), 1082--1084

\bibitem{Bi11}
  Richard Bird.
  \emph{A simple division-free algorithm for computing determinants}.
  Information Processing Letters \textbf{111} (2011), 1072--1074 

\bibitem{magma}
  Wieb Bosma, John Cannon, and Catherine Payoust.
  \emph{The Magma algebra system. I. The user language.}
  J. Symbolic Comput. \textbf{24} (1997), 235--265

\bibitem{BoChLeSaSc07}
  Alin Bostan, Frédéric Chyzak, Grégoire Lecerf, Bruno Salvy and Éric Schost.
  \emph{Differential equations for algebraic functions}.
  In ISSAC'07 (2007), 25--32

\bibitem{BoGoPeSc05}
  Alin Bostan, Laureano Gonz{\'a}lez-Vega, Hervé Perdry and Éric Schost.
  \emph{From {N}ewton sums to coefficients: complexity issues in
  characteristic $p$}.
  In {\em {MEGA'05}} (2005)


\bibitem{BrCo09}
  Olivier Brinon and Brian Conrad.
  \emph{CMI Summer School notes on p-adic Hodge theory}.
  Available at \url{http://math.stanford.edu/~conrad/papers/notes.pdf} (2009)

\bibitem{BuKe12}
  Joe Buhler and Kiran Kedlaya.
  \emph{Condensation of determinants and the Robbins phenomenon}.
  Microsoft Research Summer Number Theory Day, Redmond (2012),
  available at \url{http://kskedlaya.org/slides/microsoft2012.pdf}

\bibitem{Ca12}
  Xavier Caruso.
  \emph{Random matrices over a DVR and LU factorization}.
  J. Symbolic Comput. \textbf{71} (2015), 98--123

\bibitem{Ca17}
  Xavier Caruso.
  \emph{Numerical stability of Euclidean algorithm over ultrametric fields}.
  to appear at J. Number Theor. Bordeaux

\bibitem{CaRoVa14}
  Xavier Caruso, David Roe and Tristan Vaccon.
  \emph{Tracking $p$-adic precision}.
  LMS J. Comp. and Math. {\bf 17} (2014), 274--294

\bibitem{CaRoVa15}
  Xavier Caruso, David Roe and Tristan Vaccon.
  \emph{$p$-adic stability in linear algebra}.
  In ISSAC'15 (2015)

\bibitem{CaRoVa17}
  Xavier Caruso, David Roe and Tristan Vaccon.
  \emph{Characteristic polynomials of $p$-adic matrices}.
  In preparation

\bibitem{Ch83}
  Man-Duen Choi. 
  \emph{Tricks or treats with the Hilbert matrix.},
  Amer. Math. Monthly (1983), 301--312

\bibitem{Co89}
  John Coates.
  \emph{On $p$-adic $L$-functions}.
  Astérisque \textbf{177} (1989), 33--59

\bibitem{Co93}
  Henri Cohen,
  \emph{A course in computational algebraic number theory},
  Springer Verlag, Berlin (1993)

\bibitem{Co98}
  Pierre Colmez.
  \emph{Integration sur les variétés $p$-adiques}.
  Astérisque \textbf{248} (1998)

\bibitem{Co10}
  Pierre Colmez,
  \emph{Fonctions d'une variable $p$-adique},
  Astérisque \textbf{330} (2010), 13--59

\bibitem{DaMu97}
  Marc Daumas, Jean-Michel Muller~et~al.
  \emph{Qualité des Calculs sur Ordinateur}.
  Masson, Paris (1997)

\bibitem{DvTr89}
  Roberto Dvornicich and Carlo Traverso. 
  \emph{Newton symmetric functions and the arithmetic of algebraically closed fields}. 
  In AAECC-5, LNCS \textbf{356}, Springer, Berlin (1989), 216--224

\bibitem{Ei95}
  David Eisenbud.
  \emph{Commutative Algebra: with a view toward algebraic geometry}.
  Springer Science \& Business Media \textbf{150} (1995)

\bibitem{FoZe02}
  Sergey Fomin and Andrei Zelevinsky.
  \emph{The Laurent phenomenon}.
  {Advances in Applied Math.} \textbf{28} (2002), 119--144

\bibitem{Fu09}
  Martin Fürer. 
  \emph{Faster integer multiplication}.
  SIAM J. Comput. \textbf{39} (2009), 979--1005

\bibitem{SGA}
  Alexander Grothendieck~et~al.
  \emph{Séminaire de géométrie algébrique du Bois-Marie}.
  (1971--1977)

\bibitem{GaHoWeRiKo06}
  Pierrick Gaudry, Thomas Houtmann, Annegret Weng, Christophe Ritzenthaler and David Kohel.
  \emph{The $2$-adic CM method for genus $2$ curves with application to cryptography}.
  In {Asiacrypt 2006}, {LNCS 4284}, 114--129

\bibitem{GaGe03}
  Joachim von zur Gathen and Jürgen Gerhard.
  \emph{Modern Computer Algebra}.
  Cambridge University Press, Cambridge (2003)


\bibitem{Go88}
  Fernando Gouvêa.
  \emph{Arithmetic of $p$-adic modular forms}.
  Lecture Notes in Math. \textbf{1304}, Springer-Verlag, Berlin, New-York (1988)

\bibitem{Go97}
  Fernando Gouvêa.
  \emph{$p$-adic Numbers: An Introduction}.
  Springer (1997)

\bibitem{He97}
  Kurt Hensel. 
  \emph{Über eine neue Begründung der Theorie der algebraischen Zahlen.}
  Jahresbericht der Deutschen Mathematiker-Vereinigung \textbf{6} (1897), 83--88

\bibitem{HaMC91}
  James Hafner and Kevin McCurley.
  \emph{Asymptotically fast triangularization of matrices over rings}. 
  SIAM Journal of Comp. \textbf{20} (1991), 1068--1083

\bibitem{Ha16}
  David Harari.
  \emph{Zéro-cycles et points rationnels sur les fibrations en 
  variétés rationnellement connexes (d'après Harpaz et Wittenberg)}.
  Séminaire Bourbaki, Exp. 1096, Astérisque \textbf{380} (2016), 231--262

\bibitem{HaWi16}
  Yonatan Harpaz and Olivier Wittenberg.
  \emph{On the fibration method for zero-cycles and rational points}.
  Ann. of Math. \textbf{183} (2016), 229--295.

\bibitem{HaHoLe15}
  David Harvey, Joris van der Hoeven and Grégoire Lecerf. 
  \emph{Even faster integer multiplication}.
  J. Complexity \textbf{36} (2015), 1--30

\bibitem{Hi56}
  Francis Hilbebrand.
  \emph{Introduction to Numerical Analysis}.
  McGraw-Hill, New York (1956)

\bibitem{Ho97}
  Joris van~der Hoeven.
  \emph{Lazy multiplication of formal power series}. 
  In ISSAC'97 (1997), 17--20

\bibitem{Ho02}
  Joris van~der Hoeven.
  \emph{Relax, but don't be too lazy}.
  {J. Symbolic Comput.} \textbf{34} (2002), 479--542

\bibitem{Ho07}
  Joris van~der Hoeven.
  \emph{New algorithms for relaxed multiplication}.
  {J. Symbolic Comput.} \textbf{42} (2007), 792--802

\bibitem{mathemagix}
  Joris van der Hoeven, Grégoire Lecerf and Bernard Mourrain 
  \emph{The Mathemagix Language}, 2002--2012

\bibitem{Ho75}
  Alston Householder, 
  \emph{The Theory of Matrices in Numerical Analysis}. 
  (1975)

\bibitem{IEEE08}
  IEEE Computer Society.
  \emph{{IEEE} Standard for Floating-Point Arithmetic, {IEEE} Standard 754-2008}.
  IEEE Computer Society, New York (2008)

\bibitem{KaVi05}
  Erich Kaltofen and Gilles Villard.
  \emph{On the complexity of computing determinants}.
  Comp. Complexity \textbf{13} (2005), 91--130


\bibitem{Ke01}
  Kiran Kedlaya.
  \emph{Counting points on hyperelliptic curves using {M}onsky--{W}ashnitzer cohomology}.
  {J. Ramanujan Math. Soc.} \textbf{16} (2001), 323--338

\bibitem{Ke10}
  Kiran Kedlaya.
  \emph{$p$-adic Differential Equations}.
  Cambridge University Press (2010)

\bibitem{KeRo08}
  Kiran Kedlaya and David Roe.
  \emph{Two specifications for $p$-adic floating-point arithmetic: a
  Sage enhancement proposal}.
  Personal communication

\bibitem{Ko84}
  Neal Koblitz.
  \emph{$p$-adic Numbers, $p$-adic Analysis, and Zeta-Functions}.
  Graduate Texts in Math. \textbf{58}, Berlin, New York, Springer-Verlag (1984)

\bibitem{LaVa16}
  Pierre Lairez and Tristan Vaccon,
  \emph{On $p$-adic differential equations with separation of variables}.
  In ISSAC'16 (2016)

\bibitem{La04}
  Alan Lauder.
  \emph{Deformation theory and the computation of zeta functions}.
  {Proc. London Math. Soc.} \textbf{88} (2004), 565--602

\bibitem{Le13}
  Romain Lebreton. 
  \emph{Relaxed Hensel lifting of triangular sets}. 
  In MEGA’13 (2013)

\bibitem{LeLeLo82}
  Arjen Lenstra, Hendrik Lenstra and László Lovász.
  \emph{Factoring polynomials with rational coefficients}.
  Math. Ann. \textbf{261} (1982), 515--534

\bibitem{LeSi08}
  Reynald Lercier and Thomas Sirvent.
  \emph{On Elkies subgroups of $\ell$-torsion points in elliptic curves
  defined over a finite field}.
  {J. Number Theor. Bordeaux} \textbf{20} (2008), 783--797

\bibitem{LS07}
  Bernard Le Stum.
  \emph{Rigid cohomology}.
  Cambridge University Press (2007)

\bibitem{LiPi94}
  Carla Limongelli and Roberto Pirastu. 
  \emph{Exact solution of linear systems over rational
  numbers by parallel $p$-adic arithmetic}. 
  In Parallel Processing: CONPAR 94-VAPP VI (1994), 313--323

\bibitem{Ma58}
  Kurt Mahler.
  \emph{An interpolation series for continuous functions of a $p$-adic variable}.
  J. reine und angew. Mathematik \textbf{199} (1958), 23--34

\bibitem{Ma71}
  Yuri Manin.
  \emph{Le groupe de Brauer-Grothendieck en géométrie diophantienne}.
  In Actes du Congrès International des Mathématiciens (Nice, 1970), Tome 1, 
  Gauthier-Villars, Paris (1971), 401--411

\bibitem{Mu89}
  Jean-Michel Muller.
  \emph{Arithmétique des ordinateurs}.
  Masson, Paris (1989)

\bibitem{Mu09}
  Jean-Michel Muller~et~al.
  \emph{Handbook of floating-point arithmetic}.
  Birkhauser, Boston (2009)

\bibitem{Ne10}
  Richard Neidinger,
  \emph{Introduction to Automatic Differentiation and MATLAB Object-Oriented Programming}.
  SIAM Review \textbf{52} (2010), 545--563

\bibitem{Ne99}
  Jurgen Neurkich.
  \emph{Algebraic Number Theory}.
  Springer, Berlin, New York (1999)

\bibitem{Ne12}
  Jurgen Neurkich,
  \emph{Class Field Theory. The Bonn lectures}.
  Edited by Alexander Schmidt. Springer, Berlin, London (2013)

\bibitem{Ro20}
  Alain Robert.
  \emph{A course in $p$-adic analysis}.
  Springer Science \& Business Media \textbf{198} (2000)

\bibitem{Ro97}
  R.~Tyrell Rockafellar.
  \emph{Variational Analysis}.
  Grundlehren der Mathematischen Wissenschaften \textbf{317}, Springer-Verlag
  (1997)

\bibitem{Sa72}
  Pierre Samuel.
  \emph{Algebraic theory of numbers}.
  Paris, Hermann (1972)

\bibitem{Sc11}
  Peter Schneider.
  \emph{$p$-Adic Lie groups}.
  {Grundlehren der mathematischen Wissenschaften 344}. Springer,
  Berlin (2011)

\bibitem{Sc12}
  Peter Scholze.
  \emph{Perfectoid spaces: A survey}.
  Current Developments in Mathematics (2012)

\bibitem{Sc14}
  Peter Scholze.
  \emph{Perfectoid spaces and their Applications}.
  Proceedings of the ICM 2014 (2014)

\bibitem{Sc82}
  Arnold Schönhage,
  \emph{The fundamental theorem of algebra in terms of computational complexity}.
  Technical report, Univ. Tübingen (1982)

\bibitem{Se62}
  Jean-Pierre Serre.
  \emph{Corps locaux}.
  Hermann Paris (1962)

\bibitem{Se70}
  Jean-Pierre Serre.
  \emph{Cours d'arithmétique}.
  PUF (1970)

\bibitem{So89}
  Michael Somos.
  \emph{Problem 1470}.
  {Crux Mathematicorum} \textbf{15} (1989), 208--208.

\bibitem{sage}
  William Stein~et~al.
  \emph{Sage Mathematics Software},
  The Sage Development Team, {2005--2017}.

\bibitem{TaWi95}
  Richard Taylor and Andrew Wiles.
  \emph{Ring theoretic properties of certain Hecke algebras}.
  Ann. of Math. \textbf{141} (1995), 553--572

\bibitem{Va15}
  Tristan Vaccon.
  \emph{Précision $p$-adique}.
  PhD Thesis (2015). Available at \url{https://tel.archives-ouvertes.fr/tel-01205269}

\bibitem{We49}
  André Weil.
  \emph{Numbers of solutions of equations in finite fields}.
  Bull. Amer. Math. Soc. \textbf{55} (1949), 497--508

\bibitem{Wi95}
  Andrew Wiles.
  \emph{Modular elliptic curves and Fermat's Last Theorem}.
  Ann. of Math. \textbf{141} (1995), 443--551

\bibitem{Wi96}
  Franz Winkler.
  \emph{Polynomial Algorithms in Computer Algebra}.
  Springer Wien New Work (1996)
\end{thebibliography}
\end{document}